\documentclass[10pt,reqno]{amsart}

\usepackage[utf8x]{inputenc}
\usepackage[english]{babel}
\usepackage{amsmath,amsfonts,amssymb,amsthm,shuffle}
\usepackage[T1]{fontenc}
\usepackage{lmodern}

\usepackage[top=3.5cm,bottom=3.5cm,left=3.6cm,right=3.6cm]{geometry}

\usepackage[dvipsnames]{xcolor}
\usepackage[hyperindex=true,frenchlinks=true,colorlinks=true,
citecolor=Mahogany,linkcolor=DarkOrchid,urlcolor=Tan,linktocpage,
pagebackref = true]{hyperref}

\usepackage{tikz}
\usetikzlibrary{shapes}
\usetikzlibrary{fit}
\usetikzlibrary{decorations.pathmorphing}

\usepackage{dsfont}
\usepackage{wasysym}
\usepackage{cite}
\usepackage{subfigure}
\usepackage{multirow}
\usepackage{enumitem}
\usepackage{multicol}

\title{Pluriassociative and polydendriform algebras}
\date{\today}
\author{Samuele Giraudo}
\address{Laboratoire d'Informatique Gaspard-Monge, Université Paris-Est
    Marne-la-Vallée, 5 boulevard Descartes, Champs-sur-Marne,
    77454 Marne-la-Vallée cedex 2, France}
\email{samuele.giraudo@univ-mlv.fr}

\numberwithin{equation}{subsection}
\renewcommand{\leq}{\leqslant}
\renewcommand{\geq}{\geqslant}

\newtheorem{Theoreme}{Theorem}[subsection]
\newtheorem{Proposition}[Theoreme]{Proposition}
\newtheorem{Lemme}[Theoreme]{Lemma}

\newcommand{\Aca}{\mathcal{A}}
\newcommand{\Cca}{\mathcal{C}}
\newcommand{\Dca}{\mathcal{D}}
\newcommand{\Fca}{\mathcal{F}}
\newcommand{\Gca}{\mathcal{G}}
\newcommand{\Pca}{\mathcal{P}}
\newcommand{\Hca}{\mathcal{H}}
\newcommand{\Oca}{\mathcal{O}}
\newcommand{\Mca}{\mathcal{M}}

\newcommand{\Vca}{\mathcal{V}}
\newcommand{\Zca}{\mathcal{Z}}
\newcommand{\Cfr}{\mathfrak{c}}
\newcommand{\Efr}{\mathfrak{e}}
\newcommand{\Rfr}{\mathfrak{r}}
\newcommand{\Sfr}{\mathfrak{s}}
\newcommand{\Tfr}{\mathfrak{t}}
\newcommand{\Gfr}{\mathfrak{g}}

\newcommand{\Ksf}{{\mathsf{K}}}
\newcommand{\Tbb}{\mathbb{T}}
\newcommand{\Pbb}{\mathbb{P}}
\newcommand{\La}{\mathtt{a}}
\newcommand{\Lb}{\mathtt{b}}
\newcommand{\Lc}{\mathtt{c}}

\newcommand{\K}{\mathbb{K}}
\newcommand{\EnsNat}{\mathbb{N}}
\newcommand{\Gen}{\mathbb{G}}
\newcommand{\GenLibre}{\mathfrak{G}}
\newcommand{\RelLibre}{\mathfrak{R}}
\newcommand{\GenDias}{\GenLibre_{\Dias_\gamma}}
\newcommand{\GenDendr}{\GenLibre_{\Dendr_\gamma}}
\newcommand{\GenTDendr}{\GenLibre_{\TDendr_\gamma}}
\newcommand{\GenAs}{\GenLibre_{\As_\gamma}}
\newcommand{\GenDAs}{\GenLibre_{\DAs_\gamma}}
\newcommand{\GenDup}{\GenLibre_{\Dup_\gamma}}
\newcommand{\GenTrias}{\GenLibre_{\Trias_\gamma}}
\newcommand{\RelDias}{\RelLibre_{\Dias_\gamma}}
\newcommand{\RelDendr}{\RelLibre_{\Dendr_\gamma}}
\newcommand{\RelAs}{\RelLibre_{\As_\gamma}}
\newcommand{\RelDAs}{\RelLibre_{\DAs_\gamma}}
\newcommand{\RelDup}{\RelLibre_{\Dup_\gamma}}
\newcommand{\RelTrias}{\RelLibre_{\Trias_\gamma}}
\newcommand{\RelTDendr}{\RelLibre_{\TDendr_\gamma}}

\newcommand{\Alg}{\Aca}
\newcommand{\AlgLibre}{\Fca}
\newcommand{\Dias}{\mathsf{Dias}}
\newcommand{\Dendr}{\mathsf{Dendr}}
\newcommand{\As}{\mathsf{As}}
\newcommand{\DAs}{\mathsf{DAs}}
\newcommand{\Dup}{\mathsf{Dup}}
\newcommand{\DupDendr}{\mathsf{D}}
\newcommand{\DeuxAs}{\mathit{2as}}
\newcommand{\Trias}{\mathsf{Trias}}
\newcommand{\TDendr}{\mathsf{TDendr}}
\newcommand{\Com}{\mathsf{Com}}
\newcommand{\Zin}{\mathsf{Zin}}
\newcommand{\Leib}{\mathsf{Leib}}
\newcommand{\Lie}{\mathsf{Lie}}
\newcommand{\AlgPos}{\mathsf{Pos}}
\newcommand{\AlgEns}{\mathsf{Sets}}
\newcommand{\AlgMots}{\mathsf{Words}}
\newcommand{\AlgMotsCom}{\mathsf{CWords}}
\newcommand{\AlgMotsMarq}{\mathsf{MWords}}

\newcommand{\Unite}{\mathds{1}}
\newcommand{\OpLibre}{\mathbf{Free}}
\newcommand{\Eval}{\mathrm{eval}}
\newcommand{\T}{\mathsf{T}}
\newcommand{\Vect}{\mathrm{Vect}}
\newcommand{\GDias}{\dashv}
\newcommand{\DDias}{\vdash}
\newcommand{\GDiasA}{\rotatebox[origin=c]{180}{$\Vdash$}}
\newcommand{\DDiasA}{\Vdash}
\newcommand{\GDendrA}{\leftharpoonup}
\newcommand{\DDendrA}{\rightharpoonup}
\newcommand{\GDendr}{\prec}
\newcommand{\DDendr}{\succ}
\newcommand{\MAs}{\star}
\newcommand{\MAsA}{\triangle}
\newcommand{\MDAsA}{\wasylozenge}
\newcommand{\MDAs}{\diamond}
\newcommand{\GDup}{\hookleftarrow}
\newcommand{\DDup}{\hookrightarrow}
\newcommand{\MTrias}{\perp}
\newcommand{\MTDendr}{\wedge}
\newcommand{\OpAsDendrA}{\bullet}
\newcommand{\OpAsDendr}{\odot}
\newcommand{\ProdZin}{\shuffle}
\newcommand{\Min}{\downarrow}
\newcommand{\Max}{\uparrow}
\newcommand{\Augm}{\mathrm{h}}
\newcommand{\Rel}{\leftrightarrow}
\newcommand{\Congr}{\equiv}
\newcommand{\Recr}{\to}
\newcommand{\Racine}{\mathrm{root}}
\newcommand{\Mot}{\mathrm{word}}
\newcommand{\MotT}{\mathrm{wordt}}
\newcommand{\Equerre}{\mathrm{hook}}
\newcommand{\EquerreT}{\mathrm{hookt}}
\newcommand{\OrdDias}{\preccurlyeq}
\newcommand{\Halo}{\mathrm{Halo}}

\newcommand{\Cat}{\mathrm{cat}}
\newcommand{\Nar}{\mathrm{nar}}
\newcommand{\Schr}{\mathrm{schr}}
\newcommand{\MProjVersPluri}{\mathrm{M}}
\newcommand{\Incr}{\mathrm{Incr}}
\newcommand{\Hamming}{\mathrm{ham}}
\newcommand{\Corolle}{\mathrm{cor}}
\newcommand{\Action}{\cdot}

\newcommand{\Cacher}[1]{}
\newcommand{\Sloane}[1]{\href{http://oeis.org/#1}{{\bf #1}}}

\newcommand{\Feuille}{%
\begin{tikzpicture}
    \node(1)[Feuille]at(0,0){};
    \draw[Arete](1)--(0,.2);
\end{tikzpicture}}

\newcommand{\Noeud}{%
\begin{split}
\begin{tikzpicture}[xscale=.5,yscale=.4]
    \node[Feuille](0)at(0.50,-.75){};
    \node[Feuille](2)at(1.5,-.75){};
    \node[Noeud](1)at(1.00,0.00){};
    \draw[Arete](0)--(1);
    \draw[Arete](2)--(1);
    \node(r)at(1.00,.85){};
    \draw[Arete](r)--(1);
\end{tikzpicture}
\end{split}}

\newcommand{\NoeudTexte}{%
\raisebox{-.22em}{\scalebox{.7}{%
\begin{tikzpicture}[xscale=.4,yscale=.35]
    \node[Feuille](0)at(0.50,-.75){};
    \node[Feuille](2)at(1.5,-.75){};
    \node[Noeud](1)at(1.00,0.00){};
    \draw[Arete](0)--(1);
    \draw[Arete](2)--(1);
    \node(r)at(1.00,.85){};
    \draw[Arete](r)--(1);
\end{tikzpicture}}}}

\newcommand{\ArbreBinGDeux}[1]{%
\begin{split}
\begin{tikzpicture}[yscale=.25,xscale=.25]
    \node[Feuille](0)at(0.00,-3.33){};
    \node[Feuille](2)at(2.00,-3.33){};
    \node[Feuille](4)at(4.00,-1.67){};
    \node[Noeud](1)at(1.00,-1.67){};
    \node[Noeud](3)at(3.00,0.00){};
    \draw[Arete](0)--(1);
    \draw[Arete](1)edge[]node[EtiqArete]
            {\begin{math}#1\end{math}}(3);
    \draw[Arete](2)--(1);
    \draw[Arete](4)--(3);
    \node(r)at(3.00,1.25){};
    \draw[Arete](r)--(3);
\end{tikzpicture}
\end{split}}

\newcommand{\ArbreBinDDeux}[1]{%
\begin{split}
\begin{tikzpicture}[yscale=.25,xscale=.25]
    \node[Feuille](0)at(0.00,-1.67){};
    \node[Feuille](2)at(2.00,-3.33){};
    \node[Feuille](4)at(4.00,-3.33){};
    \node[Noeud](1)at(1.00,0.00){};
    \node[Noeud](3)at(3.00,-1.67){};
    \draw[Arete](0)--(1);
    \draw[Arete](2)--(3);
    \draw[Arete](3)edge[]node[EtiqArete]
            {\begin{math}#1\end{math}}(1);
    \draw[Arete](4)--(3);
    \node(r)at(1.00,1.25){};
    \draw[Arete](r)--(1);
\end{tikzpicture}
\end{split}}

\newcommand{\ArbreBinValue}[4]{%
\begin{split}
\begin{tikzpicture}[xscale=.5,yscale=.4]
    \node(0)at(-.50,-1.50){\begin{math}#3\end{math}};
    \node(2)at(2.50,-1.50){\begin{math}#4\end{math}};
    \node[Noeud](1)at(1.00,.50){};
    \draw[Arete](0)edge[]node[EtiqArete]
        {\begin{math}#1\end{math}}(1);
    \draw[Arete](2)edge[]node[EtiqArete]
        {\begin{math}#2\end{math}}(1);
    \node(r)at(1.00,1.5){};
    \draw[Arete](r)--(1);
\end{tikzpicture}
\end{split}}

\newcommand{\ArbreBin}[2]{%
\begin{split}
\begin{tikzpicture}[xscale=.4,yscale=.3]
    \node(0)at(-.50,-1.50){\begin{math}#1\end{math}};
    \node(2)at(2.50,-1.50){\begin{math}#2\end{math}};
    \node[Noeud](1)at(1.00,.50){};
    \draw[Arete](0)--(1);
    \draw[Arete](2)--(1);
    \node(r)at(1.00,1.75){};
    \draw[Arete](r)--(1);
\end{tikzpicture}
\end{split}}

\definecolor{Noir}{RGB}{0,0,0}
\definecolor{Blanc}{RGB}{255,255,255}
\definecolor{Rouge}{RGB}{205,35,38}
\definecolor{Bleu}{RGB}{2,60,195}
\definecolor{Vert}{RGB}{23,163,1}
\definecolor{Violet}{RGB}{181,18,225}
\definecolor{Orange}{RGB}{255,113,15}
\definecolor{Marron}{RGB}{52,46,0}

\tikzstyle{Noeud}=[circle,draw=Bleu!80,fill=Bleu!8,inner sep=1pt,
minimum size=2.5mm,line width=1.25pt,font=\scriptsize]
\tikzstyle{Arete}=[Rouge!80,cap=round,line width=1.25pt]
\tikzstyle{Feuille}=[rectangle,draw=Noir!70,fill=Noir!16,
inner sep=0pt,minimum size=1mm,line width=1.25pt]
\tikzstyle{Clair}=[fill=Blanc]
\tikzstyle{Marque1}=[draw=Vert!100,fill=Vert!30]
\tikzstyle{Marque2}=[draw=Orange!100,fill=Orange!40]
\tikzstyle{Marque3}=[draw=Rouge!100,fill=Rouge!50]
\tikzstyle{EtiqArete}=[regular polygon,regular polygon sides=6,
draw=Vert!80,fill=Vert!2,inner sep=.2pt,line width=.75pt,
minimum size=2.8mm,midway,text=Noir,font=\tiny]
\tikzstyle{NoeudSchr}=[Noeud,draw=Vert!80,fill=Vert!8]
\tikzstyle{NoeudCor}=[Noeud,draw=Marron!80,fill=Marron!8]

\begin{document}

\maketitle

\begin{abstract}
    We introduce, by adopting the point of view and the tools offered by
    the theory of operads, a generalization on a nonnegative integer
    parameter $\gamma$ of diassociative algebras of Loday, called
    $\gamma$-pluriassociative algebras. By Koszul duality of operads, we
    obtain a generalization of dendriform algebras, called
    $\gamma$-polydendriform algebras. In the same manner as dendriform
    algebras are suitable devices to split associative operations into
    two parts, $\gamma$-polydendriform algebras seem adapted structures
    to split associative operations into $2 \gamma$ operations so that
    some partial sums of these operations are associative. We provide a
    complete study of the operads governing our generalizations of the
    diassociative and dendriform operads. Among other, we exhibit several
    presentations by generators and relations, compute their Hilbert
    series, show that they are Koszul, and construct free objects in the
    corresponding categories. We also provide consistent generalizations
    on a nonnegative integer of the duplicial, triassociative and
    tridendriform operads, and of some operads of the operadic butterfly.
\end{abstract}

\begin{scriptsize}
    \tableofcontents
\end{scriptsize}

\section*{Introduction}
Associative algebras play an obvious and primary role in algebraic
combinatorics. In recent years, the study of natural operations on
certain sets of combinatorial objects has given rise to more or less
complicated algebraic structures on the vector spaces spanned by these
sets. A primordial point to observe is that these structures maintain
furthermore many links with combinatorics, combinatorial Hopf algebra
theory, representation theory, and theoretical physics. Let us cite for
instance the algebra of symmetric functions~\cite{Mac95} involving
integer partitions, the algebra of noncommutative symmetric
functions~\cite{GKLLRT94} involving integer compositions, the
Malvenuto-Reutenauer algebra of free quasi-symmetric functions~\cite{MR95}
(see also~\cite{DHT02}) involving permutations, the Loday-Ronco Hopf
algebra of binary trees~\cite{LR98} (see also~\cite{HNT05}), and the
Connes-Kreimer Hopf algebra of forests of rooted trees~\cite{CK98}.
\medskip

There are several ways to understand and to gather information about such
structures. A very fruitful strategy consists in splitting their associative
products $\star$ into two separate operations $\GDendr$ and $\DDendr$ in
such a way that $\star$ turns to be the sum of $\GDendr$ and $\DDendr$.
To be more precise, if $\Vca$ is a vector space endowed with an associative
product $\star$, splitting $\star$ consists in providing two operations
$\GDendr$ and $\DDendr$ defined on $\Vca$ and such that for all elements
$x$ and $y$ of $\Vca$,
\begin{equation}
    x \star y = x \GDendr y + x \DDendr y.
\end{equation}
This splitting property is more concisely denoted by
\begin{equation}
    \star = \GDendr + \DDendr.
\end{equation}
One of the most obvious example occurs by considering the shuffle product
on words. Indeed, this product can be separated into two operations
according to the origin (first or second operand) of the last letter of
the words appearing in the result~\cite{Ree58}. Other main examples
include the split of the shifted shuffle product of permutations of the
Malvenuto-Reutenauer Hopf algebra and of the product of binary trees of
the Loday-Ronco Hopf algebra~\cite{Foi07}. The original formalization
and the germs of generalization of these notions, due to Loday~\cite{Lod01},
lead to the introduction of dendriform algebras. Dendriform algebras
are vector spaces endowed with two operations $\GDendr$ and $\DDendr$
so that $\GDendr + \DDendr$ is associative and satisfy few other relations.
Since any dendriform algebra is a quotient of a certain free dendriform
algebra, the study of free dendriform algebras is worthwhile. Besides,
the description of free dendriform algebras has a nice combinatorial
interpretation involving binary trees and shuffle of binary trees.
\medskip

In recent years, several generalizations of dendriform algebras were
introduced and studied. Among these, one can cite dendriform
trialgebras~\cite{LR04}, quadri-algebras~\cite{AL04},
ennea-algebras~\cite{Ler04}, $m$-dendriform algebras of Leroux~\cite{Ler07},
and $m$-dendriform algebras of Novelli~\cite{Nov14}, all providing new
ways to split associative products into more than two pieces. Besides,
free objects in the corresponding categories of these algebras can be
described by relatively complex combinatorial objects and more or less
tricky operations on these. For instance, free dendriform trialgebras
involve Schröder trees, free quadri-algebras involve noncrossing connected
graphs on a circle, and free $m$-dendriform algebras of Leroux and free
$m$-dendriform algebras of Novelli involves planar rooted trees where
internal nodes have a constant number of children.
\medskip

The theory of operads (see~\cite{LV12} for a complete exposition and
also~\cite{Cha08}) seems to be one of the best tools to put all these
algebraic structures under a same roof. Informally, an operad is a space
of abstract operators that can be composed. The main interest of this
theory is that any operad encodes a category of algebras and working with
an operad amounts to work with the algebras all together of this category.
Moreover, this theory gives a nice translation of connections that may
exist between {\em a priori} two very different sorts of algebras. Indeed,
any morphism between operads gives rise to a functor between the both
encoded categories. We have to point out that operads were first
introduced in the context of algebraic topology~\cite{May72,BV73} but
they are more and more present in combinatorics~\cite{Cha08}.
\medskip

The first goal of this work is to define and justify a new generalization
of dendriform algebras. Our long term primary objective is to develop
new implements to split associative products in smaller pieces. Our main
tool is the Koszul duality of operads, an important part of the theory
introduced by Ginzburg and Kapranov~\cite{GK94}. We use the approach
consisting in considering the diassociative operad $\Dias$~\cite{Lod01},
the Koszul dual of the dendriform operad $\Dendr$, rather that focusing
on $\Dendr$. Since $\Dias$ admits a description far simpler than $\Dendr$,
starting by constructing a generalization of $\Dias$ to obtain a
generalization of $\Dendr$ by Koszul duality is a convenient path to
explore.
\medskip

To obtain a generalization of the diassociative operad, we exploit a
general functorial construction $\T$ introduced by the
author~\cite{Gir12,Gir15} producing an operad from any monoid. We showed
in these papers that this functor $\T$ provides an original construction
for the diassociative operad. In the present paper, we rely on $\T$ to
construct the operads $\Dias_\gamma$, where $\gamma$ is a nonnegative
integer, in such a way that $\Dias_1 = \Dias$. The operads $\Dias_\gamma$,
called $\gamma$-pluriassociative operads, are set-operads involving words
on the alphabet $\{0, 1, \dots, \gamma\}$ with exactly one occurrence of
$0$. Then, by computing the Koszul dual of $\Dias_\gamma$, we obtain the
operads $\Dendr_\gamma$, satisfying $\Dendr_1 = \Dendr$. The operads
$\Dendr_\gamma$ govern the category of the so-called $\gamma$-polydendriform
algebras, that are algebras with $2\gamma$ operations $\GDendrA_a$,
$\DDendrA_a$, $a \in \{1, \dots, \gamma\}$, satisfying some relations.
Free objects in these categories involve binary trees where all edges
connecting two internal nodes are labeled on $\{1, \dots, \gamma\}$.
Moreover, the introduction of $\gamma$-polydendriform algebras offers to
split an associative product $\star$ by
\begin{equation}
    \star = \GDendrA_1 + \DDendrA_1 + \dots + \GDendrA_\gamma
    + \DDendrA_\gamma,
\end{equation}
with, among others, the stiffening conditions that all partial sums
\begin{equation}
    \GDendrA_1 + \DDendrA_1 + \dots + \GDendrA_a + \DDendrA_a
\end{equation}
are associative for all $a \in \{1, \dots, \gamma\}$.
\medskip

This work naturally leads to the consideration and the definition of
numerous operads. Table~\ref{tab:operades_introduites} summarizes some
information about these.
\begin{table}
    \centering
    \begin{tabular}{c|c|c|c}
        Operad & Objects & Dimensions & Symm. \\ \hline \hline
        $\Dias_\gamma$ & Some words on $\{0, 1, \dots, \gamma\}$
            & $n \gamma^{n - 1}$ & No \\ \hline
        $\Dendr_\gamma$ & $\gamma$-edge valued binary trees
            & $\gamma^{n - 1} \frac{1}{n + 1} \binom{2n}{n}$ & No \\ \hline
        $\As_\gamma$ & $\gamma$-corollas & $\gamma$ & No \\ \hline
        $\DAs_\gamma$ & $\gamma$-alternating Schröder trees
            & $\sum\limits_{k = 0}^{n - 2} \gamma^{k + 1}
            (\gamma - 1)^{n - k - 2} \frac{1}{k + 1} \binom{n - 2}{k}
            \binom{n - 1}{k}$ & No \\ \hline
        $\Dup_\gamma$ & $\gamma$-edge valued binary trees
            & $\gamma^{n - 1} \frac{1}{n + 1} \binom{2n}{n}$ & No \\ \hline
        $\Trias_\gamma$ & Some words on $\{0, 1, \dots, \gamma\}$
            & $(\gamma + 1)^n - \gamma^n$ & No \\ \hline
        $\TDendr_\gamma$ & $\gamma$-edge valued Schröder trees
            & $\sum\limits_{k = 0}^{n - 1} (\gamma + 1)^k \gamma^{n - k - 1}
            \frac{1}{k + 1} \binom{n - 1}{k} \binom{n}{k}$ & No \\ \hline
        $\Com_\gamma$ & --- & --- & Yes \\ \hline
        $\Zin_\gamma$ & --- & --- & Yes
    \end{tabular}
    \bigskip
    \caption{The main operads defined in this paper. All these operads
    depend on a nonnegative integer parameter $\gamma$. The shown
    dimensions are the ones of the homogeneous components of
    arities $n \geq 2$ of the operads.}
    \label{tab:operades_introduites}
\end{table}
\medskip

This work is organized as follows. Section~\ref{sec:outils} contains a
conspectus of the tools used in this paper. We recall here the definition of
the construction $\T$~\cite{Gir12,Gir15} and provide a reformulation of
results of Hoffbeck~\cite{Hof10} and Dotsenko and Khoroshkin~\cite{DK10}
to prove that an operad is Koszul by using convergent rewrite rules.
Besides, this part provides self-contained definitions about nonsymmetric
operads, algebras over operads, free operads, rewrite rules on trees,
and Koszul duality. This section ends by some recalls about the
diassociative and dendriform operads.
\medskip

Section~\ref{sec:dias_gamma} is devoted to the introduction and the study
of the operad $\Dias_\gamma$. We begin by detailing the construction of
$\Dias_\gamma$ as a suboperad of the operad obtained by the construction
$\T$ applied on the monoid $\Mca_\gamma$ on $\{0, 1, \dots, \gamma\}$
with the operation
$\max$ as product. More precisely, $\Dias_\gamma$ is defined as the
suboperad of $\T \Mca_\gamma$ generated by the words $0a$ and $a0$ for
all $a \in \{1, \dots, \gamma\}$. We then provide a presentation by
generators and
relations of $\Dias_\gamma$ (Theorem~\ref{thm:presentation_dias_gamma}),
and show that it is a Koszul operad (Theorem~\ref{thm:koszulite_dias_gamma}).
We also establish some more properties of this operad: we compute its
group of symmetries (Proposition~\ref{prop:symetries_dias_gamma}), show
that it is a basic operad in the sense of~\cite{Val07}
(Proposition~\ref{prop:dias_gamma_basique}), and show that it is a
rooted operad in the sense of~\cite{Cha14}
(Proposition~\ref{prop:dias_gamma_basique}). We end this section by
introducing an alternating basis of $\Dias_\gamma$, the $\Ksf$-basis,
defined through a partial ordering relation over the words indexing the
bases of $\Dias_\gamma$. After describing how the partial composition of
$\Dias_\gamma$ expresses over the $\Ksf$-basis
(Theorem~\ref{thm:composition_base_K}), we provide a presentation of
$\Dias_\gamma$ over this basis
(Proposition~\ref{prop:presentation_alternative_dias_gamma}). Despite
the fact that this alternative presentation is more complex than the
original one of $\Dias_\gamma$ provided by
Theorem~\ref{thm:presentation_dias_gamma}, the computation of the Koszul
dual $\Dendr_\gamma$ of $\Dias_\gamma$ from this second presentation
leads to a surprisingly plain presentation of $\Dendr_\gamma$
considered later in Section~\ref{sec:dendr_gamma}.
\medskip

In Section~\ref{sec:algebres_dias_gamma}, algebras over $\Dias_\gamma$,
called $\gamma$-pluriassociative algebras, are studied. The free
$\gamma$-pluriassociative algebra over one generator is described as a
vector space of words on the alphabet $\{0, 1, \dots, \gamma\}$ with
exactly one occurrence of $0$, endowed with $2\gamma$ binary operations
(Proposition~\ref{prop:algebre_dias_gamma_libre}). We next study two
different notions of units in $\gamma$-pluriassociative algebras, the
bar-units and the wire-units, that are generalizations of definitions of
Loday introduced into the context of diassociative algebras~\cite{Lod01}.
We show that the presence of a wire-unit in a $\gamma$-pluriassociative
algebra leads to many consequences on its structure
(Proposition~\ref{prop:unite_fil}). Besides, we describe a general
construction $\MProjVersPluri$ to obtain $\gamma$-pluriassociative
algebras by starting from $\gamma$-multiprojection algebras, that are
algebraic structures with $\gamma$ associative products and endowed with
$\gamma$ endomorphisms with extra relations
(Theorem~\ref{thm:algebre_multiprojections_vers_dias_gamma}). The
main interest of the construction $\MProjVersPluri$ is that
$\gamma$-multiprojection algebras are simpler algebraic structures than
$\gamma$-pluriassociative algebras. The bar-units and wire-units of the
$\gamma$-pluriassociative algebras obtained by this construction are then
studied
(Proposition~\ref{prop:algebre_multiprojections_vers_dias_gamma_proprietes}).
We end this section by listing five examples of $\gamma$-pluriassociative
algebras constructed from $\gamma$-multiprojection algebras, including
the free $\gamma$-pluriassociative algebra over one generator considered
in Section~\ref{subsubsec:algebre_dias_gamma_libre}.
\medskip

Then, the operad $\Dendr_\gamma$ is introduced in Section~\ref{sec:dendr_gamma}
as the Koszul dual of $\Dias_\gamma$
(Theorem~\ref{thm:presentation_dendr_gamma}). Since $\Dias_\gamma$ is a
Koszul operad, $\Dendr_\gamma$ also is, and then, by using results of
Ginzburg and Kapranov~\cite{GK94}, the alternating versions of the
Hilbert series of $\Dias_\gamma$ and $\Dendr_\gamma$ are the inverses
for each other for series composition. This leads to an expression for
the Hilbert series of $\Dendr_\gamma$
(Proposition~\ref{prop:serie_hilbert_dendr_gamma}). Motivated by the
knowledge of the dimensions of $\Dendr_\gamma$, we consider binary trees
where internal edges are labelled on $\{1, \dots, \gamma\}$, called $\gamma$-edge
valued binary trees. These trees form a generalization of the common
binary trees indexing the bases of $\Dendr$, and index the bases of
$\Dendr_\gamma$. We continue the study of this operad by providing a
new presentation obtained by considering the Koszul dual of $\Dias_\gamma$
over its $\Ksf$-basis (Theorem~\ref{thm:autre_presentation_dendr_gamma}).
This presentation of $\Dendr_\gamma$ is very compact since its space of
relations can be expressed only by three sorts of relations
(\eqref{equ:relation_dendr_gamma_1_concise},
\eqref{equ:relation_dendr_gamma_2_concise},
and~\eqref{equ:relation_dendr_gamma_3_concise}), each one involving two
or three terms. We also describe all the associative elements of
$\Dendr_\gamma$ over its two bases
(Propositions~\ref{prop:operateur_associatif_dendr_gamma_autre},
\ref{prop:operateur_associatif_dendr_gamma},
and~\ref{prop:description_operateurs_associatifs_dendr_gamma}). We end
this section by constructing the free $\gamma$-polydendriform algebra
over one generator (Theorem~\ref{thm:algebre_dendr_gamma_libre}). Its
underlying vector space is the vector space of the $\gamma$-edge valued
binary trees and is endowed with $2 \gamma$ products described by
induction. These products are kinds of shuffle of trees, generalizing the
shuffle of trees introduced by Loday~\cite{Lod01} intervening in the
construction of free dendriform algebras.
\medskip

Section~\ref{sec:as_gamma} extends a part of the operadic
butterfly~\cite{Lod01,Lod06}, a diagram of operads gathering the most
classical ones together, including the diassociative, dendriform, and
associative operads. To extends this diagram into our context, we
introduce a one-parameter nonnegative integer generalization $\As_\gamma$
of the associative operad. This operad, called $\gamma$-multiassociative
operad, has $\gamma$ associative generating operations, subjected to
precise relations. We prove that this operad can be seen as a vector
space of corollas labeled on $\{1, \dots, \gamma\}$ and that is Koszul
(Proposition~\ref{prop:realisation_koszulite_as_gamma}). Unlike the
associative operad which is self-dual for Koszul duality, $\As_\gamma$
is not when $\gamma \geq 2$. The Koszul dual of $\As_\gamma$, denoted
by $\DAs_\gamma$, is described by its presentation
(Proposition~\ref{prop:presentation_as_gamma_duale}) and is realized by
means of $\gamma$-alternating Schröder trees, that are Schröder trees
where internal nodes are labeled on $\{1, \dots, \gamma\}$ with an alternating
condition (Proposition~\ref{prop:realisation_das_gamma}). In passing, we
provide an alternative and simpler basis for the space of relations of
$\DAs_\gamma$ than the one obtained directly by considering the Koszul
dual of $\As_\gamma$ (Proposition~\ref{prop:autre_presentation_das_gamma}).
We end this section by establishing a new version of the diagram gathering
the diassociative, dendriform, and associative operads for the operads
$\Dias_\gamma$, $\As_\gamma$, $\DAs_\gamma$, and $\Dendr_\gamma$
(Theorem~\ref{thm:diagramme_dias_as_das_dendr_gamma}) by defining
appropriate morphisms between these.
\medskip

Finally, in Section~\ref{sec:generalisation_supplementaires}, we sustain
our previous ideas to propose one-parameter nonnegative integer
generalizations of some more operads. We start by proposing a new operad
$\Dup_\gamma$ generalizing the duplicial operad~\cite{Lod08}, called
$\gamma$-multiplicial operad. We prove that $\Dup_\gamma$ is Koszul and,
like the bases of $\Dendr_\gamma$, that the bases of $\Dup_\gamma$ are
indexed by $\gamma$-edge valued binary trees
(Proposition~\ref{prop:proprietes_dup_gamma}). The operads $\Dendr_\gamma$
and $\Dup_\gamma$ are nevertheless not isomorphic because there are
$2\gamma$ associative elements in $\Dup_\gamma$
(Proposition~\ref{prop:description_operateurs_associatifs_dup_gamma})
against only $\gamma$ in $\Dendr_\gamma$. Then, the free
$\gamma$-multiplicial algebra over one generator is constructed
(Theorem~\ref{thm:algebre_dup_gamma_libre}). Its underlying vector space
is the vector space of the $\gamma$-edge valued binary trees and is
endowed with $2 \gamma$ products, similar to the over and under products
on binary trees of Loday and Ronco~\cite{LR02}. Next, by using almost
the same tools as the ones used in Sections~\ref{sec:dias_gamma}
and~\ref{sec:dendr_gamma}, we propose a one-parameter nonnegative
integer generalization $\Trias_\gamma$ of the triassociative operad
$\Trias$~\cite{LR04} and of its Koszul dual, the tridendriform operad
$\TDendr$. This follows a very simple idea: like $\Dias_\gamma$,
$\Trias_\gamma$ is defined as a suboperad of $\T \Mca_\gamma$ generated
by the same generators as those of $\Dias_\gamma$, plus the word $00$.
In a previous work~\cite{Gir12,Gir15}, we showed that $\Trias_1$ is the
triassociative operad. We provide here a presentation
(Theorem~\ref{thm:presentation_trias_gamma}) of $\Trias_\gamma$ and
deduce a presentation for its Koszul dual, denoted by $\TDendr_\gamma$
(Theorem~\ref{thm:presentation_tdendr_gamma}). Since $\TDendr$ is the
Koszul dual of $\Trias$, the operads $\TDendr_\gamma$ are generalizations
of $\TDendr$. The knowledge of the Hilbert series of $\TDendr_\gamma$
(Proposition~\ref{prop:serie_hilbert_tdendr_gamma}) leads to establish
the fact that the bases of $\TDendr_\gamma$ are indexed by $\gamma$-edge
valued Schröder trees, that are Schröder trees where internal edges are
labelled on $\{1, \dots, \gamma\}$. We end this work by providing a one-parameter
nonnegative integer generalization of all the operads intervening in the
operadic butterfly. We then define the operads $\Com_\gamma$, $\Lie_\gamma$,
$\Zin_\gamma$, and $\Leib_\gamma$, that are respective generalizations
of the commutative operad, the Lie operad, the Zinbiel operad~\cite{Lod95}
and the Leibniz operad~\cite{Lod93}. We provide analogous versions for
our context of the arrows between the commutative operad and the Zinbiel
operad (Proposition~\ref{prop:morphism_com_zin_gamma}), and between the
dendriform operad and the Zinbiel operad
(Proposition~\ref{prop:morphism_dendr_zin_gamma}).
\medskip

{\it Acknowledgements.} The author would like to thank, for interesting
discussions, Jean-Christophe Novelli about Koszul duality for operads
and Vincent Vong about strategies for constructing free objects in the
categories encoded by operads. The author thanks also Matthieu
Josuat-Vergès and Jean-Yves-Thibon for their pertinent remarks and
questions about this work when it was in progress. Thanks are
addressed to Frederic Chapoton and Eric Hoffbeck for answering some
questions of the author respectively about the dendriform and
diassociative operads, and Koszulity of operads. The author thanks also
Vladimir Dotsenko and Bruno Vallette for pertinent bibliographic
suggestions. Finally, the author warmly thanks the referee for his very
careful reading and his suggestions, improving the quality of the paper.
\medskip

{\it Notations and general conventions.}
All the algebraic structures of this article have a field of characteristic
zero $\K$ as ground field. If $S$ is a set, $\Vect(S)$ denotes the linear
span of the elements of $S$. For any integers $a$ and $c$, $[a, c]$ denotes
the set $\{b \in \EnsNat : a \leq b \leq c\}$ and $[n]$, the set $[1, n]$.
The cardinality of a finite set $S$ is denoted by $\# S$. If $u$ is a
word, its letters are indexed from left to right from $1$ to its length
$|u|$. For any $i \in [|u|]$, $u_i$ is the letter of $u$ at position $i$.
If $a$ is a letter and $n$ is a nonnegative integer, $a^n$ denotes the
word consisting in $n$ occurrences of $a$. Notice that $a^0$ is the
empty word $\epsilon$.
\medskip

\section{Algebraic structures and main tools} \label{sec:outils}
This preliminary section sets our conventions and notations about operads
and algebras over an operad, and describes the main tools we will use.
The definitions of the diassociative and the dendriform operads are also
recalled. This section does not contains new results but it is a
self-contained set of definitions about operads intended to readers
familiar with algebra or combinatorics but not necessarily with operadic
theory.
\medskip

\subsection{Operads and algebras over an operad}
We list here several staple definitions about operads and algebras over
an operad. We present also an important tool for this work: the
construction $\T$ producing operads from monoids.
\medskip

\subsubsection{Operads}
A {\em nonsymmetric operad in the category of vector spaces}, or a
{\em nonsymmetric operad} for short, is a graded vector space
$\Oca := \bigoplus_{n \geq 1} \Oca(n)$ together with linear maps
\begin{equation}
    \circ_i : \Oca(n) \otimes \Oca(m) \to \Oca(n + m - 1),
    \qquad n, m \geq 1, i \in [n],
\end{equation}
called {\em partial compositions}, and a distinguished element
$\Unite \in \Oca(1)$, the {\em unit} of $\Oca$. This data has to satisfy
the three relations
\begin{subequations}
\begin{equation}
    (x \circ_i y) \circ_{i + j - 1} z = x \circ_i (y \circ_j z),
    \qquad x \in \Oca(n), y \in \Oca(m),
    z \in \Oca(k), i \in [n], j \in [m],
\end{equation}
\begin{equation}
    (x \circ_i y) \circ_{j + m - 1} z = (x \circ_j z) \circ_i y,
    \qquad x \in \Oca(n), y \in \Oca(m),
    z \in \Oca(k), i < j \in [n],
\end{equation}
\begin{equation}
    \Unite \circ_1 x = x = x \circ_i \Unite,
    \qquad x \in \Oca(n), i \in [n].
\end{equation}
\end{subequations}
Since we shall consider in this paper mainly nonsymmetric operads, we
shall call these simply {\em operads}. Moreover, all considered operads
are such that $\Oca(1)$ has dimension~$1$.
\medskip

If $x$ is an element of $\Oca$ such that $x \in \Oca(n)$ for a $n \geq 1$,
we say that $n$ is the {\em arity} of $x$ and we denote it by $|x|$. An
element $x$ of $\Oca$ of arity $2$ is {\em associative} if
$x \circ_1 x = x \circ_2 x$. If $\Oca_1$ and $\Oca_2$ are operads, a
linear map $\phi : \Oca_1 \to \Oca_2$ is an {\em operad morphism} if it
respects arities, sends the unit of $\Oca_1$ to the unit of $\Oca_2$,
and commutes with partial composition maps. We say that $\Oca_2$ is a
{\em suboperad} of $\Oca_1$ if $\Oca_2$ is a graded subspace of $\Oca_1$,
and $\Oca_1$ and $\Oca_2$ have the same unit and the same partial
compositions. For any set $G \subseteq \Oca$, the {\em operad generated
by} $G$ is the smallest suboperad of $\Oca$ containing $G$. When the
operad generated by $G$ is $\Oca$ itself and $G$ is minimal with respect
to inclusion among the subsets of $\Oca$ satisfying this property, $G$
is a {\em generating set} of $\Oca$ and its elements are {\em generators}
of $\Oca$. An {\em operad ideal} of $\Oca$ is a graded subspace $I$ of
$\Oca$ such that, for any $x \in \Oca$ and $y \in I$, $x \circ_i y$ and
$y \circ_j x$ are in $I$ for all valid integers $i$ and $j$. Given an
operad ideal $I$ of $\Oca$, one can define the {\em quotient operad}
$\Oca/_I$ of $\Oca$ by $I$ in the usual way. When $\Oca$ is such that
all $\Oca(n)$ are finite for all $n \geq 1$, the {\em Hilbert series}
of $\Oca$ is the series $\Hca_\Oca(t)$ defined by
\begin{equation}
    \Hca_\Oca(t) := \sum_{n \geq 1} \dim \Oca(n) t^n.
\end{equation}
\medskip

Instead of working with the partial composition maps of $\Oca$, it is
something useful to work with the maps
\begin{equation}
    \circ : \Oca(n) \otimes \Oca(m_1) \otimes \dots \otimes \Oca(m_n)
    \to \Oca(m_1 + \dots + m_n),
    \qquad n, m_1, \dots, m_n \geq 1,
\end{equation}
linearly defined for any $x \in \Oca$ of arity $n$ and
$y_1, \dots, y_{n - 1}, y_n \in \Oca$ by
\begin{equation}
    x \circ (y_1, \dots, y_{n - 1}, y_n) :=
    (\dots ((x \circ_n y_n) \circ_{n - 1} y_{n - 1}) \dots) \circ_1 y_1.
\end{equation}
These maps are called {\em composition maps} of $\Oca$.
\medskip

\subsubsection{Set-operads}
Instead of being a direct sum of vector spaces $\Oca(n)$, $n \geq 1$,
$\Oca$ can be a graded disjoint union of sets. In this context, $\Oca$
is a {\em set-operad}. All previous definitions remain valid by replacing
direct sums $\oplus$ by disjoint unions $\sqcup$, tensor products
$\otimes$ by Cartesian products $\times$, and vector space dimensions
$\dim$ by set cardinalities $\#$. Moreover, in the context of set-operads,
we work with {\em operad congruences} instead of operad ideals. An operad
congruence on a set-operad $\Oca$ is an equivalence relation $\Congr$ on
$\Oca$ such that all elements of a same $\Congr$-equivalence class have
the same arity and for all elements $x$, $x'$, $y$, and $y'$ of $\Oca$,
$x \Congr x'$ and $y \Congr y'$ imply $x \circ_i y \Congr x' \circ_i y'$
for all valid integers $i$. The {\em quotient operad} $\Oca/_\Congr$ of
$\Oca$ by $\Congr$ is the set-operad defined in the usual way.
\medskip

Any set-operad $\Oca$ gives naturally rise to an operad on $\Vect(\Oca)$
by extending the partial compositions of $\Oca$ by linearity. Besides
this, any equivalence relation $\Rel$ of $\Oca$ such that all elements
of a same $\Rel$-equivalence class have the same arity induces a subspace
of $\Vect(\Oca)$ generated by all $x - x'$ such that $x \Rel x'$, called
{\em space induced} by $\Rel$. In particular, any operad congruence
$\Congr$ on $\Oca$ induces an operad ideal of $\Vect(\Oca)$.
\medskip

\subsubsection{From monoids to operads}%
\label{subsec:monoides_vers_operades}
In a previous work \cite{Gir12,Gir15}, the author introduced a
construction which, from any monoid, produces an operad. This
construction is described as follows. Let $\Mca$ be a monoid with an
associative product $\bullet$ admitting a unit $1$. We denote by
$\T \Mca$ the operad $\T \Mca := \bigoplus_{n \geq 1} \T \Mca(n)$ where
for all $n \geq 1$,
\begin{equation}
    \T \Mca(n) := \Vect\left(\left\{u_1 \dots u_n : u_i \in \Mca
    \mbox{ for all } i \in [n] \right\}\right).
\end{equation}
The partial composition of two words $u \in \T \Mca(n)$ and
$v \in \T \Mca(m)$ is linearly defined by
\begin{equation}
    u \circ_i v := u_1 \dots u_{i - 1}
    \, (u_i \bullet v_1) \, \dots \, (u_i \bullet v_m) \,
    u_{i + 1} \dots u_n,
    \qquad i \in [n].
\end{equation}
The unit of $\T \Mca$ is $\Unite := 1$. In other words, $\T \Mca$ is
the vector space of words on $\Mca$ seen as an alphabet and the partial
composition returns to insert a word $v$ onto the $i$th letter $u_i$ of
a word $u$ together with a left multiplication by $u_i$.
\medskip

\subsubsection{Algebras over an operad}
Any operad $\Oca$ encodes a category of algebras whose objects are
called {\em $\Oca$-algebras}. An $\Oca$-algebra $\Alg_\Oca$ is a vector
space endowed with a right action
\begin{equation}
    \Action : \Alg_\Oca^{\otimes n} \otimes \Oca(n) \to \Alg_\Oca,
    \qquad n \geq 1,
\end{equation}
satisfying the relations imposed by the structure of $\Oca$, that are
\begin{multline} \label{equ:relation_algebre_sur_operade}
    (e_1 \otimes \dots \otimes e_{n + m - 1}) \Action (x \circ_i y)
    = \\
    (e_1 \otimes \dots
        \otimes e_{i - 1} \otimes
        (e_i \otimes \dots \otimes e_{i + m - 1}) \Action y
        \otimes e_{i + m} \otimes
        \dots \otimes e_{n + m - 1}) \Action x,
\end{multline}
for all
$e_1 \otimes \dots \otimes e_{n + m - 1} \in \Alg_\Oca^{\otimes {n + m - 1}}$,
$x \in \Oca(n)$, $y \in \Oca(m)$, and $i \in [n]$. Notice that,
by~\eqref{equ:relation_algebre_sur_operade}, if $G$ is a generating set
of $\Oca$, it is enough to define the action of each $x \in G$ on
$\Alg_\Oca^{\otimes |x|}$ to wholly define~$\Action$.
\medskip

In other words, any element $x$ of $\Oca$ of arity $n$ plays the role of
a linear operation
\begin{equation}
    x : \Alg_\Oca^{\otimes n}  \to \Alg_\Oca,
\end{equation}
taking $n$ elements of $\Alg_\Oca$ as inputs and computing an element of
$\Alg_\Oca$. By a slight but convenient abuse of notation, for any
$x \in \Oca(n)$, we shall denote by $x(e_1, \dots, e_n)$, or by
$e_1 \, x \, e_2$ if $x$ has arity $2$, the element
$(e_1 \otimes \dots \otimes e_n) \Action x$ of $\Alg_\Oca$, for any
$e_1 \otimes \dots \otimes e_n \in \Alg_\Oca^{\otimes n}$.
Observe that by~\eqref{equ:relation_algebre_sur_operade}, any associative
element of $\Oca$ gives rise to an associative operation on $\Alg_\Oca$.
\medskip

Arrows in the category of $\Oca$-algebras are
{\em $\Oca$-algebra morphisms}, that are linear maps
$\phi : \Alg_1 \to \Alg_2$ between two $\Oca$-algebras $\Alg_1$ and
$\Alg_2$ such that
\begin{equation}
    \phi(x(e_1, \dots, e_n)) = x(\phi(e_1), \dots, \phi(e_n)),
\end{equation}
for all $e_1,  \dots, e_n \in \Alg_1$ and $x \in \Oca(n)$. We say that
$\Alg_2$ is an {\em $\Oca$-subalgebra} of $\Alg_1$ if $\Alg_2$ is a
subspace of $\Alg_1$ and $\Alg_1$ and $\Alg_2$ are endowed with the same
right action of $\Oca$. If $G$ is a set of elements of an $\Oca$-algebra
$\Alg$, the {\em $\Oca$-algebra generated by} $G$ is the smallest
$\Oca$-subalgebra of $\Alg$ containing $G$. When the $\Oca$-algebra
generated by $G$ is $\Alg$ itself and $G$ is minimal with respect to
inclusion among the subsets of $\Alg$ satisfying this property, $G$ is
a {\em generating set} of $\Alg$ and its elements are {\em generators}
of $\Alg$. An {\em $\Oca$-algebra ideal} of $\Alg$ is a subspace $I$ of
$\Alg$ such that for all operation $x$ of $\Oca$ of arity $n$ and elements
$e_1$, \dots, $e_n$ of $\Oca$, $x(e_1, \dots, e_n)$ is in $I$ whenever
there is a $i \in [n]$ such that $e_i$ is in $I$.
\medskip

The {\em free $\Oca$-algebra over one generator} is the $\Oca$-algebra
$\AlgLibre_\Oca$ defined in the following way. We set
$\AlgLibre_\Oca := \oplus_{n \geq 1} \AlgLibre_\Oca(n)
:= \oplus_{n \geq 1} \Oca(n)$,
and for any
$e_1, \dots, e_n \in \AlgLibre_\Oca$ and $x \in \Oca(n)$, the right
action of $x$ on $e_1 \otimes \dots \otimes e_n$ is defined by
\begin{equation}
    x(e_1, \dots, e_n) := x \circ (e_1, \dots, e_n).
\end{equation}
Then, any element $x$ of $\Oca(n)$ endows $\AlgLibre_\Oca$ with an
operation
\begin{equation}
    x : \AlgLibre_\Oca(m_1) \otimes \dots \otimes \AlgLibre_\Oca(m_n)
    \to \AlgLibre_\Oca(m_1 + \dots + m_n)
\end{equation}
respecting the graduation of $\AlgLibre_\Oca$.
\medskip

\subsection{Free operads, rewrite rules, and Koszul duality}
We recall here a description of free operads through syntax trees and
presentations of operads by generators and relations. The Koszul duality
and the Koszul property for operads are very important tools and notions
in this paper. We recall these and describe an already known criterion
to prove that a set-operad is Koszul by passing by rewrite rules on syntax
trees.
\medskip

\subsubsection{Syntax trees}
Unless otherwise specified, we use in the sequel the standard
terminology ({\em i.e.}, {\em node}, {\em edge}, {\em root},
{\em parent}, {\em child}, {\em path}, {\em ancestor}, {\em etc.}) about
planar rooted trees~\cite{Knu97}. Let $\Tfr$ be a planar rooted tree.
The {\em arity} of a node of $\Tfr$ is its number of children. An
{\em internal node} (resp. a {\em leaf}) of $\Tfr$ is a node with a
nonzero (resp. null) arity. Given an internal node $x$ of $\Tfr$, due to
the planarity of $\Tfr$, the children of $x$ are totally ordered from
left to right and are thus indexed from $1$ to the arity of $x$. If $y$
is a child of $x$, $y$ defines a {\em subtree} of $\Tfr$, that is the
planar rooted tree with root $y$ and consisting in the nodes of $\Tfr$
that have $y$ as ancestor. We shall call {\em $i$th subtree} of $x$ the
subtree of $\Tfr$ rooted at the $i$th child of $x$. A {\em partial subtree}
of $\Tfr$ is a subtree of $\Tfr$ in which some internal nodes have been
replaced by leaves and its descendants has been forgotten. Besides, due
to the planarity of $\Tfr$, its leaves are totally ordered from left to
right and thus are indexed from $1$ to the arity of $\Tfr$. In our
graphical representations, each tree is depicted so that its root is the
uppermost node.
\medskip

Let $S := \sqcup_{n \geq 1} S(n)$ be a graded set. By extension, we say
that the {\em arity} of an element $x$ of $S$ is $n$ provided that
$x \in S(n)$. A {\em syntax tree on $S$} is a planar rooted tree such
that its internal nodes of arity $n$ are labeled on elements of arity
$n$ of $S$. The {\em degree} (resp. {\em arity}) of a syntax tree is its
number of internal nodes (resp. leaves). For instance, if
$S := S(2) \sqcup S(3)$ with $S(2) := \{\La, \Lc\}$ and $S(3) := \{\Lb\}$,
\begin{equation}
    \begin{split}
    \begin{tikzpicture}[xscale=.35,yscale=.17]
        \node(0)at(0.00,-6.50){};
        \node(10)at(8.00,-9.75){};
        \node(12)at(10.00,-9.75){};
        \node(2)at(2.00,-6.50){};
        \node(4)at(3.00,-3.25){};
        \node(5)at(4.00,-9.75){};
        \node(7)at(5.00,-9.75){};
        \node(8)at(6.00,-9.75){};
        \node(1)at(1.00,-3.25){\begin{math}\Lc\end{math}};
        \node(11)at(9.00,-6.50){\begin{math}\La\end{math}};
        \node(3)at(3.00,0.00){\begin{math}\Lb\end{math}};
        \node(6)at(5.00,-6.50){\begin{math}\Lb\end{math}};
        \node(9)at(7.00,-3.25){\begin{math}\La\end{math}};
        \node(r)at(3.00,2.75){};
        \draw(0)--(1); \draw(1)--(3); \draw(10)--(11); \draw(11)--(9);
        \draw(12)--(11); \draw(2)--(1); \draw(4)--(3); \draw(5)--(6);
        \draw(6)--(9); \draw(7)--(6); \draw(8)--(6); \draw(9)--(3);
        \draw(r)--(3);
    \end{tikzpicture}
    \end{split}
\end{equation}
is a syntax tree on $S$ of degree $5$ and arity $8$. Its root is labeled
by $\Lb$ and has arity $3$.
\medskip

\subsubsection{Free operads}
Let $S$ be a graded set. The {\em free operad $\OpLibre(S)$ over $S$} is
the operad wherein for any $n \geq 1$, $\OpLibre(S)(n)$ is the vector
space of syntax trees on $S$ of arity $n$, the partial composition
$\Sfr \circ_i \Tfr$ of two syntax trees $\Sfr$ and $\Tfr$ on $S$
consists in grafting the root of $\Tfr$ on the $i$th leaf of $\Sfr$, and
its unit is the tree consisting in one leaf. For instance, if
$S := S(2) \sqcup S(3)$ with $S(2) := \{\La, \Lc\}$ and
$S(3) := \{\Lb\}$, one has in $\OpLibre(S)$,
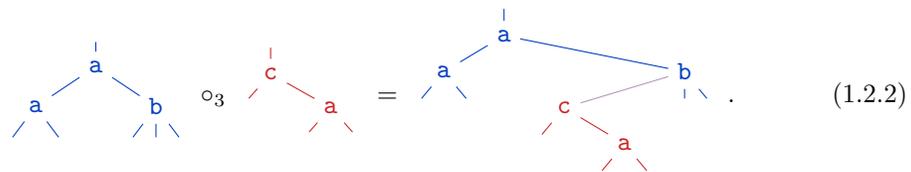
\begin{equation}\begin{split}\end{split}
    \begin{split}
    \begin{tikzpicture}[xscale=.4,yscale=.2]
        \node(0)at(0.00,-5.33){};
        \node(2)at(2.00,-5.33){};
        \node(4)at(4.00,-5.33){};
        \node(6)at(5.00,-5.33){};
        \node(7)at(6.00,-5.33){};
        \node[text=Bleu](1)at(1.00,-2.67){\begin{math}\La\end{math}};
        \node[text=Bleu](3)at(3.00,0.00){\begin{math}\La\end{math}};
        \node[text=Bleu](5)at(5.00,-2.67){\begin{math}\Lb\end{math}};
        \node(r)at(3.00,2.25){};
        \draw[draw=Bleu](0)--(1); \draw[draw=Bleu](1)--(3);
        \draw[draw=Bleu](2)--(1); \draw[draw=Bleu](4)--(5);
        \draw[draw=Bleu](5)--(3); \draw[draw=Bleu](6)--(5);
        \draw[draw=Bleu](7)--(5); \draw[draw=Bleu](r)--(3);
    \end{tikzpicture}
    \end{split}
    \circ_3
    \begin{split}
    \begin{tikzpicture}[xscale=.4,yscale=.27]
        \node(0)at(0.00,-1.67){};
        \node(2)at(2.00,-3.33){};
        \node(4)at(4.00,-3.33){};
        \node[text=Rouge](1)at(1.00,0.00){\begin{math}\Lc\end{math}};
        \node[text=Rouge](3)at(3.00,-1.67){\begin{math}\La\end{math}};
        \node(r)at(1.00,1.75){};
        \draw[draw=Rouge](0)--(1); \draw[draw=Rouge](2)--(3);
        \draw[draw=Rouge](3)--(1); \draw[draw=Rouge](4)--(3);
        \draw[draw=Rouge](r)--(1);
    \end{tikzpicture}
    \end{split}
    =
    \begin{split}
    \begin{tikzpicture}[xscale=.4,yscale=.2]
        \node(0)at(0.00,-4.80){};
        \node(10)at(9.00,-4.80){};
        \node(11)at(10.00,-4.80){};
        \node(2)at(2.00,-4.80){};
        \node(4)at(4.00,-7.20){};
        \node(6)at(6.00,-9.60){};
        \node(8)at(8.00,-9.60){};
        \node[text=Bleu](1)at(1.00,-2.40){\begin{math}\La\end{math}};
        \node[text=Bleu](3)at(3.00,0.00){\begin{math}\La\end{math}};
        \node[text=Rouge](5)at(5.00,-4.80){\begin{math}\Lc\end{math}};
        \node[text=Rouge](7)at(7.00,-7.20){\begin{math}\La\end{math}};
        \node[text=Bleu](9)at(9.00,-2.40){\begin{math}\Lb\end{math}};
        \node(r)at(3.00,2.40){};
        \draw[draw=Bleu](0)--(1); \draw[draw=Bleu](1)--(3);
        \draw[draw=Bleu](10)--(9); \draw[draw=Bleu](11)--(9);
        \draw[draw=Bleu](2)--(1); \draw[draw=Rouge](4)--(5);
        \draw[draw=Bleu!50!Rouge!50](5)--(9); \draw[draw=Rouge](6)--(7);
        \draw[draw=Rouge](7)--(5); \draw[draw=Rouge](8)--(7);
        \draw[draw=Bleu](9)--(3); \draw[draw=Bleu](r)--(3);
    \end{tikzpicture}
    \end{split}\,.
\end{equation}
\medskip

We denote by $\Corolle : S \to \OpLibre(S)$ the inclusion map, sending
any $x$ of $S$ to the {\em corolla} labeled by $x$, that is the syntax
tree consisting in one internal node labeled by $x$ attached to a
required number of leaves. In the sequel, if required by the context, we
shall implicitly see any element $x$ of $S$ as the corolla $\Corolle(x)$
of $\OpLibre(S)$. For instance, when $x$ and $y$ are two elements of $S$,
we shall simply denote by $x \circ_i y$ the syntax tree
$\Corolle(x) \circ_i \Corolle(y)$ for all valid integers $i$.
\medskip

For any operad $\Oca$, by seeing $\Oca$ as a graded set, $\OpLibre(\Oca)$
is the free operad of the syntax trees linearly labeled by elements of
$\Oca$. The {\em evaluation map} of $\Oca$ is the map
\begin{equation}
    \Eval_\Oca : \OpLibre(\Oca) \to \Oca,
\end{equation}
recursively defined by
\begin{equation}
    \Eval_\Oca(\Tfr) :=
    \begin{cases}
        \Unite & \mbox{if } \Tfr \mbox{ is the leaf}, \\
        x \circ (\Eval_\Oca(\Sfr_1), \dots, \Eval_\Oca(\Sfr_n)) &
        \mbox{otherwise},
    \end{cases}
\end{equation}
where $\Unite$ is the unit of $\Oca$, $x$ is the label of the root of
$\Tfr$, and $\Sfr_1$, \dots, $\Sfr_n$ are, from left to right, the
subtrees of the root of $\Tfr$. In other words, any tree $\Tfr$ of
$\OpLibre(\Oca)$ can be seen as a tree-like expression for an element
$\Eval_\Oca(\Tfr)$ of $\Oca$. Moreover, by induction on the degree of
$\Tfr$, it appears that $\Eval_\Oca$ is a well-defined surjective operad
morphism.
\medskip

\subsubsection{Presentations by generators and relations}
A {\em presentation} of an operad $\Oca$ consists in a pair
$(\GenLibre, \RelLibre)$ such that
$\GenLibre := \sqcup_{n \geq 1} \GenLibre(n)$ is a graded set,
$\RelLibre$ is a subspace of $\OpLibre(\GenLibre)$, and  $\Oca$ is
isomorphic to $\OpLibre(\GenLibre)/_{\langle \RelLibre \rangle}$,  where
$\langle \RelLibre \rangle$ is the operad ideal of $\OpLibre(\GenLibre)$
generated by $\RelLibre$. We call $\GenLibre$ the {\em set of generators}
and $\RelLibre$ the {\em space of relations} of $\Oca$. We say that
$\Oca$ is {\em quadratic} if one can exhibit a presentation
$(\GenLibre, \RelLibre)$ of $\Oca$ such that $\RelLibre$ is a homogeneous
subspace of $\OpLibre(\GenLibre)$ consisting in syntax trees of degree
$2$. Besides, we say that $\Oca$ is {\em binary} if one can exhibit a
presentation $(\GenLibre, \RelLibre)$ of $\Oca$ such that $\GenLibre$ is
concentrated in arity~$2$.
\medskip

With knowledge of a presentation $(\GenLibre, \RelLibre)$ of $\Oca$, it
is easy to describe the category of the $\Oca$-algebras. Indeed, by
denoting by
$\pi : \OpLibre(\GenLibre) \to \OpLibre(\GenLibre)/_{\langle \RelLibre \rangle}$
the canonical surjection map, the category of $\Oca$-algebras is the
category of vector spaces $\Alg_\Oca$ endowed with maps $\pi(g)$,
$g \in \GenLibre$, satisfying for all $r \in \RelLibre$ the relations
\begin{equation}
    r(e_1,\dots, e_n) = 0,
\end{equation}
for all $e_1, \dots, e_n \in \Alg_\Oca$, where $n$ is the arity of $r$.
\medskip

\subsubsection{Rewrite rules}
Let $S$ be a graded set. A {\em rewrite rule} on syntax trees on $S$ is
a binary relation $\Recr$ on $\OpLibre(S)$ whenever for all trees $\Sfr$
and $\Tfr$ of $\OpLibre(S)$, $\Sfr \Recr \Tfr$ only if $\Sfr$ and $\Tfr$
have the same arity. When $\Recr$ involves only syntax trees of degree
two, $\Recr$ is {\em quadratic}. We say that a syntax tree $\Sfr'$ can
be {\em rewritten} by $\Recr$ into $\Tfr'$ if there exist two syntax
trees $\Sfr$ and $\Tfr$ satisfying $\Sfr \Recr \Tfr$ and $\Sfr'$ has a
partial subtree equal to $\Sfr$ such that, by replacing it by $\Tfr$ in
$\Sfr'$, we obtain $\Tfr'$. By a slight but convenient abuse of notation,
we denote by $\Sfr' \Recr \Tfr'$ this property. When a syntax tree $\Tfr$
can be obtained by performing a sequence of $\Recr$-rewritings from a
syntax tree $\Sfr$, we say that $\Sfr$ is {\em rewritable} by $\Recr$
into $\Tfr$ and we denote this property by $\Sfr \overset{*}{\Recr} \Tfr$.
For instance, for $S := S(2) \sqcup S(3)$ with $S(2) := \{\La, \Lc\}$
and $S(3) := \{\Lb\}$, consider the rewrite rule $\Recr$ on $\OpLibre(S)$
satisfying
\begin{equation}
    \begin{split}
    \begin{tikzpicture}[xscale=.4,yscale=.25]
        \node(0)at(0.00,-2.00){};
        \node(2)at(1.00,-2.00){};
        \node(3)at(2.00,-2.00){};
        \node(1)at(1.00,0.00){\begin{math}\Lb\end{math}};
        \draw(0)--(1);
        \draw(2)--(1);
        \draw(3)--(1);
        \node(r)at(1.00,1.75){};
        \draw(r)--(1);
    \end{tikzpicture}
    \end{split}
    \begin{split} \enspace \Recr \enspace \end{split}
    \begin{split}
    \begin{tikzpicture}[xscale=.3,yscale=.3]
        \node(0)at(0.00,-3.33){};
        \node(2)at(2.00,-3.33){};
        \node(4)at(4.00,-1.67){};
        \node(1)at(1.00,-1.67){\begin{math}\La\end{math}};
        \node(3)at(3.00,0.00){\begin{math}\La\end{math}};
        \draw(0)--(1);
        \draw(1)--(3);
        \draw(2)--(1);
        \draw(4)--(3);
        \node(r)at(3.00,1.5){};
        \draw(r)--(3);
    \end{tikzpicture}
    \end{split}
    \qquad \mbox{and} \qquad
    \begin{split}
    \begin{tikzpicture}[xscale=.3,yscale=.3]
        \node(0)at(0.00,-3.33){};
        \node(2)at(2.00,-3.33){};
        \node(4)at(4.00,-1.67){};
        \node(1)at(1.00,-1.67){\begin{math}\La\end{math}};
        \node(3)at(3.00,0.00){\begin{math}\Lc\end{math}};
        \draw(0)--(1);
        \draw(1)--(3);
        \draw(2)--(1);
        \draw(4)--(3);
        \node(r)at(3.00,1.5){};
        \draw(r)--(3);
    \end{tikzpicture}
    \end{split}
    \begin{split} \enspace \Recr \enspace \end{split}
    \begin{split}
    \begin{tikzpicture}[xscale=.3,yscale=.3]
        \node(0)at(0.00,-1.67){};
        \node(2)at(2.00,-3.33){};
        \node(4)at(4.00,-3.33){};
        \node(1)at(1.00,0.00){\begin{math}\La\end{math}};
        \node(3)at(3.00,-1.67){\begin{math}\Lc\end{math}};
        \draw(0)--(1);
        \draw(2)--(3);
        \draw(3)--(1);
        \draw(4)--(3);
        \node(r)at(1.00,1.5){};
        \draw(r)--(1);
    \end{tikzpicture}
    \end{split}\,.
\end{equation}
We then have the following sequence of rewritings
\begin{equation}
    \begin{split}
    \begin{tikzpicture}[xscale=.22,yscale=.17]
        \node(0)at(0.00,-6.50){};
        \node(10)at(8.00,-6.50){};
        \node(12)at(10.00,-6.50){};
        \node(2)at(1.00,-9.75){};
        \node(4)at(3.00,-9.75){};
        \node(5)at(4.00,-9.75){};
        \node(7)at(5.00,-9.75){};
        \node(8)at(6.00,-9.75){};
        \node[text=Rouge](1)at(2.00,-3.25){\begin{math}\Lb\end{math}};
        \node(11)at(9.00,-3.25){\begin{math}\La\end{math}};
        \node(3)at(2.00,-6.50){\begin{math}\Lc\end{math}};
        \node(6)at(5.00,-6.50){\begin{math}\Lb\end{math}};
        \node(9)at(7.00,0.00){\begin{math}\Lc\end{math}};
        \draw[draw=Rouge](0)--(1);
        \draw[draw=Rouge](1)--(9);
        \draw(10)--(11);
        \draw(11)--(9);
        \draw(12)--(11);
        \draw(2)--(3);
        \draw[draw=Rouge](3)--(1);
        \draw(4)--(3);
        \draw(5)--(6);
        \draw[draw=Rouge](6)--(1);
        \draw(7)--(6);
        \draw(8)--(6);
        \node(r)at(7.00,2.5){};
        \draw(r)--(9);
    \end{tikzpicture}
    \end{split}
    \begin{split} \enspace \Recr \enspace \end{split}
    \begin{split}
    \begin{tikzpicture}[xscale=.22,yscale=.17]
        \node(0)at(0.00,-8.40){};
        \node(11)at(10.00,-5.60){};
        \node(13)at(12.00,-5.60){};
        \node(2)at(2.00,-11.20){};
        \node(4)at(4.00,-11.20){};
        \node(6)at(6.00,-8.40){};
        \node(8)at(7.00,-8.40){};
        \node(9)at(8.00,-8.40){};
        \node(1)at(1.00,-5.60){\begin{math}\La\end{math}};
        \node[text=Rouge](10)at(9.00,0.00){\begin{math}\Lc\end{math}};
        \node(12)at(11.00,-2.80){\begin{math}\La\end{math}};
        \node(3)at(3.00,-8.40){\begin{math}\Lc\end{math}};
        \node[text=Rouge](5)at(5.00,-2.80){\begin{math}\La\end{math}};
        \node(7)at(7.00,-5.60){\begin{math}\Lb\end{math}};
        \draw(0)--(1);
        \draw[draw=Rouge](1)--(5);
        \draw(11)--(12);
        \draw[draw=Rouge](12)--(10);
        \draw(13)--(12);
        \draw(2)--(3);
        \draw(3)--(1);
        \draw(4)--(3);
        \draw[draw=Rouge](5)--(10);
        \draw(6)--(7);
        \draw[draw=Rouge](7)--(5);
        \draw(8)--(7);
        \draw(9)--(7);
        \node(r)at(9.00,2.25){};
        \draw[draw=Rouge](r)--(10);
    \end{tikzpicture}
    \end{split}
    \begin{split} \enspace \Recr \enspace \end{split}
    \begin{split}
    \begin{tikzpicture}[xscale=.22,yscale=.17]
        \node(0)at(0.00,-7.00){};
        \node(11)at(10.00,-10.50){};
        \node(13)at(12.00,-10.50){};
        \node(2)at(2.00,-10.50){};
        \node(4)at(4.00,-10.50){};
        \node(6)at(6.00,-10.50){};
        \node(8)at(7.00,-10.50){};
        \node(9)at(8.00,-10.50){};
        \node(1)at(1.00,-3.50){\begin{math}\La\end{math}};
        \node(10)at(9.00,-3.50){\begin{math}\Lc\end{math}};
        \node(12)at(11.00,-7.00){\begin{math}\La\end{math}};
        \node(3)at(3.00,-7.00){\begin{math}\Lc\end{math}};
        \node(5)at(5.00,0.00){\begin{math}\La\end{math}};
        \node[text=Rouge](7)at(7.00,-7.00){\begin{math}\Lb\end{math}};
        \draw(0)--(1);
        \draw(1)--(5);
        \draw(10)--(5);
        \draw(11)--(12);
        \draw(12)--(10);
        \draw(13)--(12);
        \draw(2)--(3);
        \draw(3)--(1);
        \draw(4)--(3);
        \draw[draw=Rouge](6)--(7);
        \draw[draw=Rouge](7)--(10);
        \draw[draw=Rouge](8)--(7);
        \draw[draw=Rouge](9)--(7);
        \node(r)at(5.00,2.5){};
        \draw(r)--(5);
    \end{tikzpicture}
    \end{split}
    \begin{split} \enspace \Recr \enspace \end{split}
    \begin{split}
    \begin{tikzpicture}[xscale=.22,yscale=.17]
        \node(0)at(0.00,-6.00){};
        \node(10)at(10.00,-9.00){};
        \node(12)at(12.00,-9.00){};
        \node(14)at(14.00,-9.00){};
        \node(2)at(2.00,-9.00){};
        \node(4)at(4.00,-9.00){};
        \node(6)at(6.00,-12.00){};
        \node(8)at(8.00,-12.00){};
        \node(1)at(1.00,-3.00){\begin{math}\La\end{math}};
        \node(11)at(11.00,-3.00){\begin{math}\Lc\end{math}};
        \node(13)at(13.00,-6.00){\begin{math}\La\end{math}};
        \node(3)at(3.00,-6.00){\begin{math}\Lc\end{math}};
        \node(5)at(5.00,0.00){\begin{math}\La\end{math}};
        \node(7)at(7.00,-9.00){\begin{math}\La\end{math}};
        \node(9)at(9.00,-6.00){\begin{math}\La\end{math}};
        \draw(0)--(1);
        \draw(1)--(5);
        \draw(10)--(9);
        \draw(11)--(5);
        \draw(12)--(13);
        \draw(13)--(11);
        \draw(14)--(13);
        \draw(2)--(3);
        \draw(3)--(1);
        \draw(4)--(3);
        \draw(6)--(7);
        \draw(7)--(9);
        \draw(8)--(7);
        \draw(9)--(11);
        \node(r)at(5.00,2.25){};
        \draw(r)--(5);
    \end{tikzpicture}
    \end{split}\,.
\end{equation}
We shall use the standard terminology ({\em confluent}, {\em terminating},
{\em convergent}, {\em normal form}, {\em critical pair}, {\em etc.})
about rewrite rules (see~\cite{BN98}).
\medskip

Any rewrite rule $\Recr$ on $\OpLibre(S)$ defines an operad congruence
$\Congr_\Recr$ on $\OpLibre(S)$ seen as a set-operad, the
{\em operad congruence induced} by $\Recr$, as the finest operad
congruence on $\OpLibre(S)$ containing the reflexive, symmetric, and
transitive closure of $\Recr$.
\medskip

\subsubsection{Koszul duality and Koszulity}%
\label{subsubsec:dual_de_Koszul}
In~\cite{GK94}, Ginzburg and Kapranov extended the notion of Koszul
duality of quadratic associative algebras to quadratic operads. Starting
with a binary and quadratic operad $\Oca$ admitting a presentation
$(\GenLibre, \RelLibre)$, the {\em Koszul dual} of $\Oca$ is the operad
$\Oca^!$, isomorphic to the operad admitting the presentation
$\left(\GenLibre, \RelLibre^\perp\right)$ where $\RelLibre^\perp$ is the
annihilator of $\RelLibre$ in $\OpLibre(\GenLibre)$ with respect to the
scalar product
\begin{equation}
    \langle -, - \rangle :
    \OpLibre(\GenLibre)(3) \otimes \OpLibre(\GenLibre)(3) \to \K
\end{equation}
linearly defined, for all $x, x', y, y' \in \GenLibre(2)$, by
\begin{equation}
    \left\langle x \circ_i y, x' \circ_{i'} y' \right\rangle :=
    \begin{cases}
        1 & \mbox{if }
            x = x', y = y', \mbox{ and } i = i' = 1, \\
        -1 & \mbox{if }
            x = x', y = y', \mbox{ and } i = i' = 2, \\
        0 & \mbox{otherwise}.
    \end{cases}
\end{equation}
Then, knowing a presentation of $\Oca$, one can compute a presentation
of~$\Oca^!$.
\medskip

Besides, we say a quadratic operad $\Oca$ is {\em Koszul} if its Koszul
complex is acyclic~\cite{GK94,LV12}. In this work, to prove the
Koszulity of an operad $\Oca$, we shall make use of a combinatorial tool
introduced by Hoffbeck~\cite{Hof10} (see also~\cite{LV12}) consisting
in exhibiting a particular basis of $\Oca$, a so-called
{\em Poincaré-Birkhoff-Witt basis}.
\medskip

In this paper, we shall use this tool only in the context of set-operads,
which reformulates, thanks to the work of Dotsenko and
Khoroshkin~\cite{DK10}, as follows. A set-operad $\Oca$ is Kosuzl if
there is a graded set $S$ and a rewrite rule $\Recr$ on $\OpLibre(S)$
such that $\Oca$ is isomorphic to $\OpLibre(S)/_{\Congr_\Recr}$ and
$\Recr$ is a convergent quadratic rewrite rule. Moreover, the set of
normal forms of $\Recr$ forms a Poincaré-Birkhoff-Witt basis of $\Oca$.
\medskip

Furthermore, when $\Oca$ and $\Oca^!$ are two operads Koszul dual one of
the other, and moreover, when they are Koszul operads and admit Hilbert
series, their Hilbert series satisfy~\cite{GK94}
\begin{equation} \label{equ:relation_series_hilbert_operade_duale}
    \Hca_\Oca\left(-\Hca_{\Oca^!}(-t)\right) = t.
\end{equation}
We shall make use of~\eqref{equ:relation_series_hilbert_operade_duale}
to compute the dimensions of Koszul operads defined as Koszul duals of
known ones.
\medskip

\subsection{Diassociative and dendriform operads}%
\label{subsec:dias_et_dendr}
We recall here, by using the notions presented during the previous
sections, the definitions and some properties of the diassociative and
dendriform operads.
\medskip

\subsubsection{Diassociative operad and diassociative algebras}%
\label{subsubsec:dias}
The {\em diassociative operad} $\Dias$ was introduced by
Loday~\cite{Lod01} as the operad admitting the presentation
$\left(\GenLibre_{\Dias}, \RelLibre_{\Dias}\right)$ where
$\GenLibre_{\Dias} := \GenLibre_{\Dias}(2) := \{\GDias, \DDias\}$
and $\RelLibre_{\Dias}$ is the space induced by the equivalence
relation $\Congr$ satisfying
\begin{subequations}
\begin{equation} \label{equ:relation_dias_1}
    \GDias \circ_1 \DDias \; \Congr \;
    \DDias \circ_2 \GDias,
\end{equation}
\begin{equation} \label{equ:relation_dias_2}
    \GDias \circ_1 \GDias \; \Congr \;
    \GDias \circ_2 \GDias \; \Congr \;
    \GDias \circ_2 \DDias,
\end{equation}
\begin{equation} \label{equ:relation_dias_3}
    \DDias \circ_1 \GDias \; \Congr \;
    \DDias \circ_1 \DDias \; \Congr \;
    \DDias \circ_2 \DDias.
\end{equation}
\end{subequations}
Note that $\Dias$ is a binary and quadratic operad.
\medskip

This operad admits the following realization~\cite{Cha05}. For any
$n \geq 1$, $\Dias(n)$ is the linear span of the $\Efr_{n, k}$,
$k \in [n]$, and the partial compositions linearly satisfy, for all
$n, m \geq 1$, $k \in [n]$, $\ell \in [m]$, and $i \in [n]$,
\begin{equation}
    \Efr_{n, k} \circ_i \Efr_{m, \ell} =
    \begin{cases}
        \Efr_{n + m - 1, k + m - 1}
            & \mbox{if } i < k, \\
        \Efr_{n + m - 1, k + \ell - 1}
            & \mbox{if } i = k, \\
        \Efr_{n + m - 1, k}
            & \mbox{otherwise (} i > k \mbox{)}.
    \end{cases}
\end{equation}
Since the partial composition of two basis elements of $\Dias$ produces
exactly one basis element, $\Dias$ is well-defined as a set-operad.
Moreover, this realization shows that $\dim \Dias(n) = n$ and hence,
the Hilbert series of $\Dias$ satisfies
\begin{equation}
    \Hca_\Dias(t) = \frac{t}{(1 - t)^2}.
\end{equation}
\medskip

From the presentation of $\Dias$, we deduce that any $\Dias$-algebra,
also called {\em diassociative algebra}, is a vector space $\Alg_\Dias$
endowed with linear operations $\GDias$ and $\DDias$ satisfying the
relations encoded by~\eqref{equ:relation_dias_1}---%
\eqref{equ:relation_dias_3}.
\medskip

From the realization of $\Dias$, we deduce that the free diassociative
algebra $\AlgLibre_\Dias$ over one generator is the vector space $\Dias$
endowed with the linear operations
\begin{equation}
    \GDias, \DDias :
    \AlgLibre_\Dias \otimes \AlgLibre_\Dias \to \AlgLibre_\Dias,
\end{equation}
satisfying, for all $n, m \geq 1$, $k \in [n]$, $\ell \in [m]$,
\begin{equation}
    \Efr_{n, k} \GDias \Efr_{m, \ell}
        = (\Efr_{n, k} \otimes \Efr_{m, \ell}) \Action \Efr_{2, 1}
        = (\Efr_{2, 1} \circ_2 \Efr_{m, \ell}) \circ_1 \Efr_{n, k}
        = \Efr_{n + m, k},
\end{equation}
and
\begin{equation}
    \Efr_{n, k} \DDias \Efr_{m, \ell}
        = (\Efr_{n, k} \otimes \Efr_{m, \ell}) \Action \Efr_{2, 2}
        = (\Efr_{2, 2} \circ_2 \Efr_{m, \ell}) \circ_1 \Efr_{n, k}
        = \Efr_{n + m, n + \ell}.
\end{equation}
\medskip

As shown in~\cite{Gir12,Gir15}, the diassociative operad is isomorphic
to the suboperad of $\T \Mca$ generated by $01$ and $10$ where $\Mca$ is
the multiplicative monoid on $\{0, 1\}$. The concerned isomorphism sends
any $\Efr_{n, k}$ of $\Dias$ to the word $0^{k - 1} \, 1 \, 0^{n - k}$ of
$\T \Mca$.
\medskip

\subsubsection{Dendriform operad and dendriform algebras}
The {\em dendriform operad} $\Dendr$ was also introduced by
Loday~\cite{Lod01}. It is the operad admitting the presentation
$\left(\GenLibre_{\Dendr}, \RelLibre_{\Dendr}\right)$ where
$\GenLibre_{\Dendr} := \GenLibre_{\Dendr}(2) := \{\GDendr, \DDendr\}$
and $\RelLibre_{\Dendr}$ is the vector space generated by
\begin{subequations}
\begin{equation} \label{equ:relation_dendr_1}
    \GDendr \circ_1 \DDendr - \DDendr \circ_2 \GDendr,
\end{equation}
\begin{equation} \label{equ:relation_dendr_2}
    \GDendr \circ_1 \GDendr -
    \GDendr \circ_2 \GDendr -
    \GDendr \circ_2 \DDendr,
\end{equation}
\begin{equation} \label{equ:relation_dendr_3}
    \DDendr \circ_1 \GDendr +
    \DDendr \circ_1 \DDendr -
    \DDendr \circ_2 \DDendr.
\end{equation}
\end{subequations}
Note that $\Dendr$ is a binary and quadratic operad.
\medskip

This operad admits a quite complicated realization~\cite{Lod01}. For all
$n \geq 1$, the $\Dendr(n)$ are vector spaces of binary trees with $n$
internal nodes. The partial composition of two binary trees can be
described by means of intervals of the Tamari order~\cite{HT72}, a
partial order relation involving binary trees. This realization shows
that $\dim \Dendr(n) = \Cat(n)$ where
\begin{equation}
    \Cat(n) := \frac{1}{n + 1} \binom{2n}{n}
\end{equation}
is the $n$th {\em Catalan number}, counting the binary trees with
respect to their number of internal nodes. Therefore, the Hilbert series
of $\Dendr$ satisfies
\begin{equation}
    \Hca_\Dendr(t) = \frac{1 - \sqrt{1 - 4t} - 2t}{2t}.
\end{equation}
\medskip

Throughout this article, we shall graphically represent binary trees in
a slightly different manner than syntax trees. We represent the leaves
of binary trees by squares $\Feuille$, internal nodes by circles
\tikz{\node[Noeud]{};}, and edges by thick segments
\tikz{\draw[Arete](0,0)--(0,.27);}.
\medskip

From the presentation of $\Dendr$, we deduce that any $\Dendr$-algebra,
also called {\em dendriform algebra}, is a vector space $\Alg_\Dendr$
endowed with linear operations $\GDendr$ and $\DDendr$ satisfying the
relations encoded by~\eqref{equ:relation_dendr_1}---%
\eqref{equ:relation_dendr_3}. Classical examples of dendriform algebras
include Rota-Baxter algebras~\cite{Agu00} and shuffle algebras~\cite{Lod01}.
\medskip

The operation obtained by summing $\GDendr$ and $\DDendr$ is associative.
Therefore, we can see a dendriform algebra as an associative algebra in
which its associative product has been split into two parts satisfying
Relations~\eqref{equ:relation_dendr_1}, \eqref{equ:relation_dendr_2},
and~\eqref{equ:relation_dendr_3}. More precisely, we say that an
associative algebra $\Alg$ admits a {\em dendriform structure} if there
exist two nonzero binary operations $\GDendr$ and $\DDendr$ such that
the associative operation $\MAs$ of $\Alg$ satisfies
$\MAs = \GDendr + \DDendr$, and $\Alg$ endowed with the operations
$\GDendr$ and $\DDendr$, is a dendriform algebra
\medskip

The free dendriform algebra $\AlgLibre_\Dendr$ over one generator is the
vector space $\Dendr$ of binary trees with at least one internal node
endowed with the linear operations
\begin{equation}
    \GDendr, \DDendr :
    \AlgLibre_\Dendr \otimes \AlgLibre_\Dendr \to \AlgLibre_\Dendr,
\end{equation}
defined recursively, for any binary tree $\Sfr$ with at least one
internal node, and binary trees $\Tfr_1$ and $\Tfr_2$ by
\begin{equation}
    \Sfr \GDendr \Feuille
    := \Sfr =:
    \Feuille \DDendr \Sfr,
\end{equation}
\begin{equation}
    \Feuille \GDendr \Sfr := 0 =: \Sfr \DDendr \Feuille,
\end{equation}
\begin{equation}
    \ArbreBin{\Tfr_1}{\Tfr_2} \GDendr \Sfr :=
    \ArbreBin{\Tfr_1}{\Tfr_2 \GDendr \Sfr}
    + \ArbreBin{\Tfr_1}{\Tfr_2 \DDendr \Sfr}\,,
\end{equation}
\begin{equation}
    \begin{split} \Sfr \DDendr \end{split}
    \ArbreBin{\Tfr_1}{\Tfr_2} :=
    \ArbreBin{\Sfr \DDendr \Tfr_1}{\Tfr_2}
    + \ArbreBin{\Sfr \GDendr \Tfr_1}{\Tfr_2}.
\end{equation}
Note that neither $\Feuille \GDendr \Feuille$ nor
$\Feuille \DDendr \Feuille$ are defined.
\medskip

We have for instance,
\begin{equation}
    \begin{split}
    \begin{tikzpicture}[xscale=.2,yscale=.16]
        \node[Feuille](0)at(0.00,-4.50){};
        \node[Feuille](2)at(2.00,-4.50){};
        \node[Feuille](4)at(4.00,-6.75){};
        \node[Feuille](6)at(6.00,-6.75){};
        \node[Feuille](8)at(8.00,-4.50){};
        \node[Noeud](1)at(1.00,-2.25){};
        \node[Noeud](3)at(3.00,0.00){};
        \node[Noeud](5)at(5.00,-4.50){};
        \node[Noeud](7)at(7.00,-2.25){};
        \draw[Arete](0)--(1);
        \draw[Arete](1)--(3);
        \draw[Arete](2)--(1);
        \draw[Arete](4)--(5);
        \draw[Arete](5)--(7);
        \draw[Arete](6)--(5);
        \draw[Arete](7)--(3);
        \draw[Arete](8)--(7);
        \node(r)at(3.00,2.25){};
        \draw[Arete](r)--(3);
    \end{tikzpicture}
    \end{split}
    \GDendr
    \begin{split}
    \begin{tikzpicture}[xscale=.2,yscale=.2]
        \node[Feuille](0)at(0.00,-3.50){};
        \node[Feuille](2)at(2.00,-5.25){};
        \node[Feuille](4)at(4.00,-5.25){};
        \node[Feuille](6)at(6.00,-1.75){};
        \node[Noeud](1)at(1.00,-1.75){};
        \node[Noeud](3)at(3.00,-3.50){};
        \node[Noeud](5)at(5.00,0.00){};
        \draw[Arete](0)--(1);
        \draw[Arete](1)--(5);
        \draw[Arete](2)--(3);
        \draw[Arete](3)--(1);
        \draw[Arete](4)--(3);
        \draw[Arete](6)--(5);
        \node(r)at(5.00,1.75){};
        \draw[Arete](r)--(5);
    \end{tikzpicture}
    \end{split}
    =
    \begin{split}
    \begin{tikzpicture}[xscale=.18,yscale=.13]
        \node[Feuille](0)at(0.00,-5.00){};
        \node[Feuille](10)at(10.00,-12.50){};
        \node[Feuille](12)at(12.00,-12.50){};
        \node[Feuille](14)at(14.00,-7.50){};
        \node[Feuille](2)at(2.00,-5.00){};
        \node[Feuille](4)at(4.00,-7.50){};
        \node[Feuille](6)at(6.00,-7.50){};
        \node[Feuille](8)at(8.00,-10.00){};
        \node[Noeud](1)at(1.00,-2.50){};
        \node[Noeud](11)at(11.00,-10.00){};
        \node[Noeud](13)at(13.00,-5.00){};
        \node[Noeud](3)at(3.00,0.00){};
        \node[Noeud](5)at(5.00,-5.00){};
        \node[Noeud](7)at(7.00,-2.50){};
        \node[Noeud](9)at(9.00,-7.50){};
        \draw[Arete](0)--(1);
        \draw[Arete](1)--(3);
        \draw[Arete](10)--(11);
        \draw[Arete](11)--(9);
        \draw[Arete](12)--(11);
        \draw[Arete](13)--(7);
        \draw[Arete](14)--(13);
        \draw[Arete](2)--(1);
        \draw[Arete](4)--(5);
        \draw[Arete](5)--(7);
        \draw[Arete](6)--(5);
        \draw[Arete](7)--(3);
        \draw[Arete](8)--(9);
        \draw[Arete](9)--(13);
        \node(r)at(3.00,2.50){};
        \draw[Arete](r)--(3);
    \end{tikzpicture}
    \end{split}
    +
    \begin{split}
    \begin{tikzpicture}[xscale=.18,yscale=.13]
        \node[Feuille](0)at(0.00,-5.00){};
        \node[Feuille](10)at(10.00,-12.50){};
        \node[Feuille](12)at(12.00,-12.50){};
        \node[Feuille](14)at(14.00,-5.00){};
        \node[Feuille](2)at(2.00,-5.00){};
        \node[Feuille](4)at(4.00,-10.00){};
        \node[Feuille](6)at(6.00,-10.00){};
        \node[Feuille](8)at(8.00,-10.00){};
        \node[Noeud](1)at(1.00,-2.50){};
        \node[Noeud](11)at(11.00,-10.00){};
        \node[Noeud](13)at(13.00,-2.50){};
        \node[Noeud](3)at(3.00,0.00){};
        \node[Noeud](5)at(5.00,-7.50){};
        \node[Noeud](7)at(7.00,-5.00){};
        \node[Noeud](9)at(9.00,-7.50){};
        \draw[Arete](0)--(1);
        \draw[Arete](1)--(3);
        \draw[Arete](10)--(11);
        \draw[Arete](11)--(9);
        \draw[Arete](12)--(11);
        \draw[Arete](13)--(3);
        \draw[Arete](14)--(13);
        \draw[Arete](2)--(1);
        \draw[Arete](4)--(5);
        \draw[Arete](5)--(7);
        \draw[Arete](6)--(5);
        \draw[Arete](7)--(13);
        \draw[Arete](8)--(9);
        \draw[Arete](9)--(7);
        \node(r)at(3.00,2.50){};
        \draw[Arete](r)--(3);
    \end{tikzpicture}
    \end{split}
    +
    \begin{split}
    \begin{tikzpicture}[xscale=.18,yscale=.13]
        \node[Feuille](0)at(0.00,-5.00){};
        \node[Feuille](10)at(10.00,-10.00){};
        \node[Feuille](12)at(12.00,-10.00){};
        \node[Feuille](14)at(14.00,-5.00){};
        \node[Feuille](2)at(2.00,-5.00){};
        \node[Feuille](4)at(4.00,-12.50){};
        \node[Feuille](6)at(6.00,-12.50){};
        \node[Feuille](8)at(8.00,-10.00){};
        \node[Noeud](1)at(1.00,-2.50){};
        \node[Noeud](11)at(11.00,-7.50){};
        \node[Noeud](13)at(13.00,-2.50){};
        \node[Noeud](3)at(3.00,0.00){};
        \node[Noeud](5)at(5.00,-10.00){};
        \node[Noeud](7)at(7.00,-7.50){};
        \node[Noeud](9)at(9.00,-5.00){};
        \draw[Arete](0)--(1);
        \draw[Arete](1)--(3);
        \draw[Arete](10)--(11);
        \draw[Arete](11)--(9);
        \draw[Arete](12)--(11);
        \draw[Arete](13)--(3);
        \draw[Arete](14)--(13);
        \draw[Arete](2)--(1);
        \draw[Arete](4)--(5);
        \draw[Arete](5)--(7);
        \draw[Arete](6)--(5);
        \draw[Arete](7)--(9);
        \draw[Arete](8)--(7);
        \draw[Arete](9)--(13);
        \node(r)at(3.00,2.50){};
        \draw[Arete](r)--(3);
    \end{tikzpicture}
    \end{split}\,,
\end{equation}
and
\begin{equation}
    \begin{split}
    \begin{tikzpicture}[xscale=.2,yscale=.16]
        \node[Feuille](0)at(0.00,-4.50){};
        \node[Feuille](2)at(2.00,-4.50){};
        \node[Feuille](4)at(4.00,-6.75){};
        \node[Feuille](6)at(6.00,-6.75){};
        \node[Feuille](8)at(8.00,-4.50){};
        \node[Noeud](1)at(1.00,-2.25){};
        \node[Noeud](3)at(3.00,0.00){};
        \node[Noeud](5)at(5.00,-4.50){};
        \node[Noeud](7)at(7.00,-2.25){};
        \draw[Arete](0)--(1);
        \draw[Arete](1)--(3);
        \draw[Arete](2)--(1);
        \draw[Arete](4)--(5);
        \draw[Arete](5)--(7);
        \draw[Arete](6)--(5);
        \draw[Arete](7)--(3);
        \draw[Arete](8)--(7);
        \node(r)at(3.00,2.25){};
        \draw[Arete](r)--(3);
    \end{tikzpicture}
    \end{split}
    \DDendr
    \begin{split}
    \begin{tikzpicture}[xscale=.2,yscale=.2]
        \node[Feuille](0)at(0.00,-3.50){};
        \node[Feuille](2)at(2.00,-5.25){};
        \node[Feuille](4)at(4.00,-5.25){};
        \node[Feuille](6)at(6.00,-1.75){};
        \node[Noeud](1)at(1.00,-1.75){};
        \node[Noeud](3)at(3.00,-3.50){};
        \node[Noeud](5)at(5.00,0.00){};
        \draw[Arete](0)--(1);
        \draw[Arete](1)--(5);
        \draw[Arete](2)--(3);
        \draw[Arete](3)--(1);
        \draw[Arete](4)--(3);
        \draw[Arete](6)--(5);
        \node(r)at(5.00,1.75){};
        \draw[Arete](r)--(5);
    \end{tikzpicture}
    \end{split}
    =
    \begin{split}
    \begin{tikzpicture}[xscale=.18,yscale=.13]
        \node[Feuille](0)at(0.00,-7.50){};
        \node[Feuille](10)at(10.00,-12.50){};
        \node[Feuille](12)at(12.00,-12.50){};
        \node[Feuille](14)at(14.00,-2.50){};
        \node[Feuille](2)at(2.00,-7.50){};
        \node[Feuille](4)at(4.00,-10.00){};
        \node[Feuille](6)at(6.00,-10.00){};
        \node[Feuille](8)at(8.00,-10.00){};
        \node[Noeud](1)at(1.00,-5.00){};
        \node[Noeud](11)at(11.00,-10.00){};
        \node[Noeud](13)at(13.00,0.00){};
        \node[Noeud](3)at(3.00,-2.50){};
        \node[Noeud](5)at(5.00,-7.50){};
        \node[Noeud](7)at(7.00,-5.00){};
        \node[Noeud](9)at(9.00,-7.50){};
        \draw[Arete](0)--(1);
        \draw[Arete](1)--(3);
        \draw[Arete](10)--(11);
        \draw[Arete](11)--(9);
        \draw[Arete](12)--(11);
        \draw[Arete](14)--(13);
        \draw[Arete](2)--(1);
        \draw[Arete](3)--(13);
        \draw[Arete](4)--(5);
        \draw[Arete](5)--(7);
        \draw[Arete](6)--(5);
        \draw[Arete](7)--(3);
        \draw[Arete](8)--(9);
        \draw[Arete](9)--(7);
        \node(r)at(13.00,2.50){};
        \draw[Arete](r)--(13);
        \end{tikzpicture}
    \end{split}
    +
    \begin{split}
    \begin{tikzpicture}[xscale=.18,yscale=.13]
        \node[Feuille](0)at(0.00,-7.50){};
        \node[Feuille](10)at(10.00,-10.00){};
        \node[Feuille](12)at(12.00,-10.00){};
        \node[Feuille](14)at(14.00,-2.50){};
        \node[Feuille](2)at(2.00,-7.50){};
        \node[Feuille](4)at(4.00,-12.50){};
        \node[Feuille](6)at(6.00,-12.50){};
        \node[Feuille](8)at(8.00,-10.00){};
        \node[Noeud](1)at(1.00,-5.00){};
        \node[Noeud](11)at(11.00,-7.50){};
        \node[Noeud](13)at(13.00,0.00){};
        \node[Noeud](3)at(3.00,-2.50){};
        \node[Noeud](5)at(5.00,-10.00){};
        \node[Noeud](7)at(7.00,-7.50){};
        \node[Noeud](9)at(9.00,-5.00){};
        \draw[Arete](0)--(1);
        \draw[Arete](1)--(3);
        \draw[Arete](10)--(11);
        \draw[Arete](11)--(9);
        \draw[Arete](12)--(11);
        \draw[Arete](14)--(13);
        \draw[Arete](2)--(1);
        \draw[Arete](3)--(13);
        \draw[Arete](4)--(5);
        \draw[Arete](5)--(7);
        \draw[Arete](6)--(5);
        \draw[Arete](7)--(9);
        \draw[Arete](8)--(7);
        \draw[Arete](9)--(3);
        \node(r)at(13.00,2.50){};
        \draw[Arete](r)--(13);
        \end{tikzpicture}
    \end{split}
    +
    \begin{split}
    \begin{tikzpicture}[xscale=.18,yscale=.13]
        \node[Feuille](0)at(0.00,-10.00){};
        \node[Feuille](10)at(10.00,-7.50){};
        \node[Feuille](12)at(12.00,-7.50){};
        \node[Feuille](14)at(14.00,-2.50){};
        \node[Feuille](2)at(2.00,-10.00){};
        \node[Feuille](4)at(4.00,-12.50){};
        \node[Feuille](6)at(6.00,-12.50){};
        \node[Feuille](8)at(8.00,-10.00){};
        \node[Noeud](1)at(1.00,-7.50){};
        \node[Noeud](11)at(11.00,-5.00){};
        \node[Noeud](13)at(13.00,0.00){};
        \node[Noeud](3)at(3.00,-5.00){};
        \node[Noeud](5)at(5.00,-10.00){};
        \node[Noeud](7)at(7.00,-7.50){};
        \node[Noeud](9)at(9.00,-2.50){};
        \draw[Arete](0)--(1);
        \draw[Arete](1)--(3);
        \draw[Arete](10)--(11);
        \draw[Arete](11)--(9);
        \draw[Arete](12)--(11);
        \draw[Arete](14)--(13);
        \draw[Arete](2)--(1);
        \draw[Arete](3)--(9);
        \draw[Arete](4)--(5);
        \draw[Arete](5)--(7);
        \draw[Arete](6)--(5);
        \draw[Arete](7)--(3);
        \draw[Arete](8)--(7);
        \draw[Arete](9)--(13);
        \node(r)at(13.00,2.50){};
        \draw[Arete](r)--(13);
        \end{tikzpicture}
    \end{split}\,.
\end{equation}
\medskip

As shown in~\cite{Lod01}, the dendriform operad is the Koszul dual of
the diassociative operad. This can be checked by a simple computation
following what is explained in Section~\ref{subsubsec:dual_de_Koszul}.
Besides that, since theses two operads are Koszul operads, the alternating
versions of their Hilbert series are the inverses for each other for
series composition.
\medskip

We invite the reader to take a look
at~\cite{LR98,Agu00,Lod02,Foi07,EMP08,EM09,LV12} for a supplementary
review of properties of dendriform algebras and of the dendriform operad.
\medskip

\section{Pluriassociative operads}%
\label{sec:dias_gamma}
In this section, we define the main object of this work: a generalization
on a nonnegative integer parameter $\gamma$ of the diassociative operad.
We provide a complete study of this new operad.
\medskip

\subsection{Construction and first properties}
We define here our generalization of the diassociative operad using the
functor $\T$ (whose definition is recalled in
Section~\ref{subsec:monoides_vers_operades}). We then describe the
elements and establish the Hilbert series of our generalization.
\medskip

\subsubsection{Construction} \label{subsubsec:construction_dias_gamma}
For any integer $\gamma \geq 0$, let $\Mca_\gamma$ be the monoid
$\{0\} \cup [\gamma]$ with the binary operation $\max$ as product,
denoted by $\Max$. We define $\Dias_\gamma$ as the suboperad of
$\T \Mca_\gamma$ generated by
\begin{equation} \label{equ:generateurs_dias_gamma}
    \{0a, a0 : a \in [\gamma]\}.
\end{equation}
By definition, $\Dias_\gamma$ is the vector space of words that can be
obtained by partial compositions of words
of~\eqref{equ:generateurs_dias_gamma}. We have, for instance,
\begin{equation}
    \Dias_2(1)
    =\Vect(\{0\}),
\end{equation}
\begin{equation}
    \Dias_2(2)
    =\Vect(\{01, 02, 10, 20\}),
\end{equation}
\begin{equation}
    \Dias_2(3)
    =\Vect(\{011, 012, 021, 022, 101, 102, 201, 202, 110, 120, 210, 220\}),
\end{equation}
and
\begin{equation}
    \textcolor{Bleu}{211} {\bf 2} \textcolor{Bleu}{01}
        \circ_4 \textcolor{Rouge}{31103}
    = \textcolor{Bleu}{211} \textcolor{Rouge}{32223} \textcolor{Bleu}{01},
\end{equation}
\begin{equation}
    \textcolor{Bleu}{11} {\bf 1} \textcolor{Bleu}{101}
        \circ_3 \textcolor{Rouge}{20}
    = \textcolor{Bleu}{11} \textcolor{Rouge}{21} \textcolor{Bleu}{101},
\end{equation}
\begin{equation}
    \textcolor{Bleu}{1} {\bf 0} \textcolor{Bleu}{13}
        \circ_2 \textcolor{Rouge}{210}
    = \textcolor{Bleu}{1} \textcolor{Rouge}{210} \textcolor{Bleu}{13}.
\end{equation}
\medskip

It follows immediately from the definition of $\Dias_\gamma$ as a
suboperad of $\T \Mca_\gamma$ that $\Dias_\gamma$ is a set-operad. Indeed,
any partial composition of two basis elements of $\Dias_\gamma$ gives
rises to exactly one basis element. We then shall see $\Dias_\gamma$ as
a set-operad over all Section~\ref{sec:dias_gamma}.
\medskip

Notice that $\Dias_\gamma(2)$ is the set~\eqref{equ:generateurs_dias_gamma}
of generators of $\Dias_\gamma$. Besides, observe that $\Dias_0$ is the
trivial operad and that $\Dias_\gamma$ is a suboperad of $\Dias_{\gamma + 1}$.
We call $\Dias_\gamma$ the {\em $\gamma$-pluriassociative operad}.
\medskip

\subsubsection{Elements and dimensions}
\begin{Proposition} \label{prop:elements_dias_gamma}
    For any integer $\gamma \geq 0$, as a set-operad, the underlying set
    of $\Dias_\gamma$ is the set of the words on the alphabet
    $\{0\} \cup [\gamma]$ containing exactly one occurrence of $0$.
\end{Proposition}
\begin{proof}
    Let us show that any word $x$ of $\Dias_\gamma$ satisfies the
    statement of the proposition by induction on the length $n$ of $x$.
    This is true when $n = 1$ because we necessarily have $x = 0$.
    Otherwise, when $n \geq 2$, there is a word $y$ of $\Dias_\gamma$
    of length $n - 1$ and a generator $g$ of $\Dias_\gamma$ such that
    $x = y \circ_i g$ for a $i \in [n - 1]$. Then, $x$ is obtained by
    replacing the $i$th letter $a$ of $y$ by the factor $u := u_1 u_2$
    where $u_1 := a \Max g_1$ and $u_2 := a \Max g_2$. Since $g$ contains
    exactly one $0$, this operation consists in inserting a nonzero
    letter of $[\gamma]$ into $y$. Since by induction hypothesis $y$
    contains exactly one $0$, it follows that $x$ satisfies the
    statement of the proposition.
    \smallskip

    Conversely, let us show that any word $x$ satisfying the statement
    of the proposition belongs to $\Dias_\gamma$ by induction on the
    length $n$ of $x$. This is true when $n = 1$ because we necessarily
    have $x = 0$ and $0$ belongs to $\Dias_\gamma$ since it is its unit.
    Otherwise, when $n \geq 2$, there is an integer $i \in [n - 1]$ such
    that $x_i x_{i + 1} \in \{0a, a0\}$ for an $a \in [\gamma]$. Let us
    suppose without loss of generality that $x_i x_{i + 1} = a0$. By
    setting $y$ as the word obtained by erasing the $i$th letter of $x$,
    we have $x = y \circ_i a0$. Thus, since by induction hypothesis $y$
    is an element of $\Dias_\gamma$, it follows that $x$ also is.
\end{proof}
\medskip

We deduce from Proposition~\ref{prop:elements_dias_gamma} that the
Hilbert series of $\Dias_\gamma$ satisfies
\begin{equation} \label{equ:serie_hilbert_dias_gamma}
    \Hca_{\Dias_\gamma}(t) = \frac{t}{(1 - \gamma t)^2}
\end{equation}
and that for all $n \geq 1$, $\dim \Dias_\gamma(n) = n \gamma^{n - 1}$.
For instance, the first dimensions of $\Dias_1$, $\Dias_2$, $\Dias_3$,
and $\Dias_4$ are respectively
\begin{equation}
    1, 2, 3, 4, 5, 6, 7, 8, 9, 10, 11,
\end{equation}
\begin{equation}
    1, 4, 12, 32, 80, 192, 448, 1024, 2304, 5120, 11264,
\end{equation}
\begin{equation}
    1, 6, 27, 108, 405, 1458, 5103, 17496, 59049, 196830, 649539,
\end{equation}
\begin{equation}
    1, 8, 48, 256, 1280, 6144, 28672, 131072, 589824, 2621440, 11534336.
\end{equation}
The second one is Sequence~\Sloane{A001787}, the third one is
Sequence~\Sloane{A027471}, and the last one is Sequence~\Sloane{A002697}
of~\cite{Slo}.
\medskip

\subsection{Presentation by generators and relations}%
\label{subsec:presentation_dias_gamma}
To establish a presentation of $\Dias_\gamma$, we shall start by defining
a morphism $\Mot_\gamma$ from a free operad to $\Dias_\gamma$. Then,
after showing that $\Mot_\gamma$ is a surjection, we will show that
$\Mot_\gamma$ induces an operad isomorphism between a quotient of
a free operad by a certain operad congruence $\Congr_\gamma$ and
$\Dias_\gamma$. The space of relations of $\Dias_\gamma$ of its
presentation will be induced by $\Congr_\gamma$.
\medskip

\subsubsection{From syntax trees to words}%
\label{subsubsec:arbres_vers_mots}
For any integer $\gamma \geq 0$, let $\GenDias := \GenDias(2)$ be the
graded set where
\begin{equation}
    \GenDias(2) := \{\GDias_a, \DDias_a : a \in [\gamma]\}.
\end{equation}
\medskip

Let $\Tfr$ be a syntax tree of $\OpLibre\left(\GenDias\right)$ and $x$
be a leaf of $\Tfr$. We say that an integer $a \in \{0\} \cup [\gamma]$
is {\em eligible} for $x$ if $a = 0$ or there is an ancestor $y$ of $x$
labeled by $\GDias_a$ (resp. $\DDias_a$) and $x$ is in the right (resp.
left) subtree of $y$. The {\em image} of $x$ is its greatest eligible
integer. Moreover, let
\begin{equation} \label{equ:application_mot_gamma}
    \Mot_\gamma : \OpLibre\left(\GenDias\right)(n) \to \Dias_\gamma(n),
    \qquad n \geq 1,
\end{equation}
the map where $\Mot_\gamma(\Tfr)$ is the word obtained by considering,
from left to right, the images of the leaves of $\Tfr$ (see
Figure~\ref{fig:exemple_mot_gamma}).
\begin{figure}[ht]
    \centering
     \begin{tikzpicture}[xscale=.28,yscale=.18]
        \node(0)at(0.00,-11.50){};
        \node(10)at(10.00,-7.67){};
        \node(12)at(12.00,-15.33){};
        \node(14)at(14.00,-15.33){};
        \node(16)at(16.00,-19.17){};
        \node(18)at(18.00,-19.17){};
        \node(2)at(2.00,-11.50){};
        \node(20)at(20.00,-19.17){};
        \node(22)at(22.00,-19.17){};
        \node(4)at(4.00,-11.50){};
        \node(6)at(6.00,-15.33){};
        \node(8)at(8.00,-15.33){};
        \node(1)at(1.00,-7.67){\begin{math}\GDias_4\end{math}};
        \node(11)at(11.00,-3.83){\begin{math}\DDias_2\end{math}};
        \node(13)at(13.00,-11.50){\begin{math}\DDias_1\end{math}};
        \node(15)at(15.00,-7.67){\begin{math}\DDias_3\end{math}};
        \node(17)at(17.00,-15.33){\begin{math}\DDias_2\end{math}};
        \node(19)at(19.00,-11.50){\begin{math}\GDias_1\end{math}};
        \node(21)at(21.00,-15.33){\begin{math}\DDias_4\end{math}};
        \node(3)at(3.00,-3.83){\begin{math}\DDias_3\end{math}};
        \node(5)at(5.00,-7.67){\begin{math}\GDias_1\end{math}};
        \node(7)at(7.00,-11.50){\begin{math}\GDias_2\end{math}};
        \node(9)at(9.00,0.00){\begin{math}\GDias_2\end{math}};
        \draw(0)--(1); \draw(1)--(3); \draw(10)--(11); \draw(11)--(9);
        \draw(12)--(13); \draw(13)--(15); \draw(14)--(13); \draw(15)--(11);
        \draw(16)--(17); \draw(17)--(19); \draw(18)--(17); \draw(19)--(15);
        \draw(2)--(1); \draw(20)--(21); \draw(21)--(19); \draw(22)--(21);
        \draw(3)--(9); \draw(4)--(5); \draw(5)--(3); \draw(6)--(7);
        \draw(7)--(5); \draw(8)--(7);
        \node(r)at(9,3){};
        \draw(9)--(r);
        \node[below of=0,node distance=3mm]
            {\small \begin{math}\textcolor{Bleu}{3}\end{math}};
        \node[below of=2,node distance=3mm]
            {\small \begin{math}\textcolor{Bleu}{4}\end{math}};
        \node[below of=4,node distance=3mm]
            {\small \begin{math}\textcolor{Bleu}{0}\end{math}};
        \node[below of=6,node distance=3mm]
            {\small \begin{math}\textcolor{Bleu}{1}\end{math}};
        \node[below of=8,node distance=3mm]
            {\small \begin{math}\textcolor{Bleu}{2}\end{math}};
        \node[below of=10,node distance=3mm]
            {\small \begin{math}\textcolor{Bleu}{2}\end{math}};
        \node[below of=12,node distance=3mm]
            {\small \begin{math}\textcolor{Bleu}{3}\end{math}};
        \node[below of=14,node distance=3mm]
            {\small \begin{math}\textcolor{Bleu}{3}\end{math}};
        \node[below of=16,node distance=3mm]
            {\small \begin{math}\textcolor{Bleu}{2}\end{math}};
        \node[below of=18,node distance=3mm]
            {\small \begin{math}\textcolor{Bleu}{2}\end{math}};
        \node[below of=20,node distance=3mm]
            {\small \begin{math}\textcolor{Bleu}{4}\end{math}};
        \node[below of=22,node distance=3mm]
            {\small \begin{math}\textcolor{Bleu}{2}\end{math}};
    \end{tikzpicture}
    \caption{A syntax tree $\Tfr$ of $\OpLibre\left(\GenDias\right)$
    where images of its leaves are shown. This tree satisfies
    $\Mot_\gamma(\Tfr) = \textcolor{Bleu}{340122332242}$.}
    \label{fig:exemple_mot_gamma}
\end{figure}
\medskip

\begin{Lemme} \label{lem:mot_gamma_morphisme}
    For any integer $\gamma \geq 0$, the map $\Mot_\gamma$ is an
    operad morphism from $\OpLibre\left(\GenDias\right)$ to
    $\Dias_\gamma$.
\end{Lemme}
\begin{proof}
    Let us first show that $\Mot_\gamma$ is a well-defined map. Let
    $\Tfr$ be a syntax tree of $\OpLibre\left(\GenDias\right)$ of arity
    $n$. Observe that by starting from the root of $\Tfr$, there is a
    unique maximal path obtained by following the directions specified
    by its internal nodes (a $\GDias_a$ means to go the left child while
    a $\DDias_a$ means to go to the right child). Then, the leaf at the
    end of this path is the only leaf with $0$ as image. Others $n - 1$
    leaves have integers of $[\gamma]$ as images. By
    Proposition~\ref{prop:elements_dias_gamma}, this implies that
    $\Mot_\gamma(\Tfr)$ is an element of $\Dias_\gamma(n)$.
    \smallskip

    To prove that $\Mot_\gamma$ is an operad morphism, we consider its
    following alternative description. If $\Tfr$ is a syntax tree of
    $\OpLibre\left(\GenDias\right)$, we can consider the tree $\Tfr'$
    obtained by replacing in $\Tfr$ each label $\GDias_a$ (resp. $\DDias_a$)
    by the word $0a$ (resp. $a0$), where $a \in [\gamma]$. Then, by a
    straightforward induction on the number of internal nodes of $\Tfr$,
    we obtain that $\Eval_{\Dias_\gamma}(\Tfr')$, where $\Tfr'$ is seen
    as a syntax tree of $\OpLibre\left(\Dias_\gamma(2)\right)$, is
    $\Mot_\gamma(\Tfr)$. It then follows that $\Mot_\gamma$ is an operad
    morphism.
\end{proof}
\medskip

\subsubsection{Hook syntax trees}
Let us now consider the map
\begin{equation} \label{equ:application_equerre_gamma}
    \Equerre_\gamma : \Dias_\gamma(n) \to \OpLibre\left(\GenDias\right)(n),
    \qquad n \geq 1,
\end{equation}
defined for any word $x$ of $\Dias_\gamma$ by
\begin{equation} \label{equ:definition_application_equerre_gamma}
    \begin{split}\Equerre_\gamma(x)\end{split} :=
    \begin{split}
    \begin{tikzpicture}[xscale=.5,yscale=.45]
        \node(0)at(0.00,-5.40){};
        \node(2)at(3.00,-7.20){};
        \node(4)at(5.00,-7.20){};
        \node(6)at(6.00,-3.60){};
        \node(8)at(9.00,-1.80){};
        \node(1)at(1.00,-3.60){\begin{math}\DDias_{u_1}\end{math}};
        \node(3)at(4.00,-5.40){\begin{math}\DDias_{u_{|u|}}\end{math}};
        \node(5)at(5.00,-1.80){\begin{math}\GDias_{v_1}\end{math}};
        \node(7)at(8.00,0.00){\begin{math}\GDias_{v_{|v|}}\end{math}};
        \node(r)at(8,1.5){};
        \draw(0)--(1);
        \draw(1)--(5);
        \draw(2)--(3);
        \draw[densely dashed](3)--(1);
        \draw(4)--(3);
        \draw[densely dashed](5)--(7);
        \draw(6)--(5);
        \draw(8)--(7);
        \draw(7)--(r);
    \end{tikzpicture}
    \end{split}\,,
\end{equation}
where $x$ decomposes, by Proposition~\ref{prop:elements_dias_gamma},
uniquely in $x = u0v$ where $u$ and $v$ are words on the alphabet
$[\gamma]$. The dashed edges denote, depending on their orientation,
a right comb (wherein internal nodes are labeled, from top to bottom
by $\DDias_{u_1}$, \dots, $\DDias_{u_{|u|}}$) or a left comb
(wherein internal nodes are labeled, from bottom to top, by
$\GDias_{v_1}$, \dots, $\GDias_{v_{|v|}}$). We shall call any syntax
tree of the form \eqref{equ:definition_application_equerre_gamma} a
{\em hook syntax tree}.
\medskip

\begin{Lemme} \label{lem:mot_gamma_surjection}
    For any integer $\gamma \geq 0$, the map $\Mot_\gamma$ is a
    surjective operad morphism from $\OpLibre\left(\GenDias\right)$
    onto $\Dias_\gamma$. Moreover, for any element $x$ of $\Dias_\gamma$,
    $\Equerre_\gamma(x)$ belongs to the fiber of $x$ under $\Mot_\gamma$.
\end{Lemme}
\begin{proof}
    The fact that $x$ belongs to the fiber of $x$ under $\Mot_\gamma$ is
    an immediate consequence of the definitions of $\Mot_\gamma$ and
    $\Equerre_\gamma$, and the fact that by
    Proposition~\ref{prop:elements_dias_gamma}, any word $x$ of
    $\Dias_\gamma$ decomposes uniquely in $x = u0v$ where $u$ and $v$
    are words on the alphabet~$[\gamma]$. Then, $\Mot_\gamma$ is
    surjective as a map. Moreover, since by
    Lemma~\ref{lem:mot_gamma_morphisme}, $\Mot_\gamma$ is an operad
    morphism, it is a surjective operad morphism.
\end{proof}
\medskip

\subsubsection{A rewrite rule on syntax trees}
Let $\Recr_\gamma$ be the quadratic rewrite rule on
$\OpLibre\left(\GenDias\right)$ satisfying
\begin{subequations}
\begin{equation} \label{equ:reecriture_dias_gamma_1}
    \DDias_{a'} \circ_2 \GDias_a
    \enspace \Recr_\gamma \enspace
    \GDias_a \circ_1 \DDias_{a'},
    \qquad a, a' \in [\gamma],
\end{equation}
\begin{equation} \label{equ:reecriture_dias_gamma_2}
    \GDias_a \circ_2 \DDias_b
    \enspace \Recr_\gamma \enspace
    \GDias_a \circ_1 \GDias_b,
    \qquad a < b \in [\gamma],
\end{equation}
\begin{equation} \label{equ:reecriture_dias_gamma_3}
    \DDias_a \circ_1 \GDias_b
    \enspace \Recr_\gamma \enspace
    \DDias_a \circ_2 \DDias_b,
    \qquad a < b \in [\gamma],
\end{equation}
\begin{equation} \label{equ:reecriture_dias_gamma_4}
    \GDias_a \circ_2 \GDias_b
    \enspace \Recr_\gamma \enspace
    \GDias_b \circ_1 \GDias_a,
    \qquad a < b \in [\gamma],
\end{equation}
\begin{equation} \label{equ:reecriture_dias_gamma_5}
    \DDias_a \circ_1 \DDias_b
    \enspace \Recr_\gamma \enspace
    \DDias_b \circ_2 \DDias_a,
    \qquad a < b \in [\gamma],
\end{equation}
\begin{equation} \label{equ:reecriture_dias_gamma_6}
    \GDias_d \circ_2 \GDias_c
    \enspace \Recr_\gamma \enspace
    \GDias_d \circ_1 \GDias_d,
    \qquad c \leq d \in [\gamma],
\end{equation}
\begin{equation} \label{equ:reecriture_dias_gamma_7}
    \GDias_d \circ_2 \DDias_c
    \enspace \Recr_\gamma \enspace
    \GDias_d \circ_1 \GDias_d,
    \qquad c \leq d \in [\gamma],
\end{equation}
\begin{equation} \label{equ:reecriture_dias_gamma_8}
    \DDias_d \circ_1 \GDias_c
    \enspace \Recr_\gamma \enspace
    \DDias_d \circ_2 \DDias_d,
    \qquad c \leq d \in [\gamma],
\end{equation}
\begin{equation} \label{equ:reecriture_dias_gamma_9}
    \DDias_d \circ_1 \DDias_c
    \enspace \Recr_\gamma \enspace
    \DDias_d \circ_2 \DDias_d,
    \qquad c \leq d \in [\gamma],
\end{equation}
\end{subequations}
and denote by $\Congr_\gamma$ the operadic congruence on
$\OpLibre\left(\GenDias\right)$ induced by $\Recr_\gamma$.
\medskip

\begin{Lemme} \label{lem:mot_gamma_stable_classes_equivalence}
    For any integer $\gamma \geq 0$ and any syntax trees $\Tfr_1$ and
    $\Tfr_2$ of $\OpLibre\left(\GenDias\right)$,
    $\Tfr_1 \Congr_\gamma \Tfr_2$ implies
    $\Mot_\gamma(\Tfr_1) = \Mot_\gamma(\Tfr_2)$.
\end{Lemme}
\begin{proof}
    Let us denote by $\Rel_\gamma$ the symmetric closure of $\Recr_\gamma$.
    In the first place, observe that for any relation
    $\Sfr_1 \Rel_\gamma \Sfr_2$ where $\Sfr_1$ and $\Sfr_2$ are syntax
    trees of $\OpLibre\left(\GenDias\right)(3)$, for any $i \in [3]$,
    the eligible integers for the $i$th leaves of $\Sfr_1$ and $\Sfr_2$
    are the same. Besides, by definition of $\Congr_\gamma$, since
    $\Tfr_1 \Congr_\gamma \Tfr_2$, one can obtain $\Tfr_2$ from $\Tfr_1$
    by performing a sequence of $\Rel_\gamma$-rewritings. According to
    the previous observation, a $\Rel_\gamma$-rewriting preserve the
    eligible integers of all leaves of the tree on which they are
    performed. Therefore, the images of the leaves of $\Tfr_2$ are, from
    left to right, the same as the images of the leaves of $\Tfr_1$ and
    hence, $\Mot_\gamma(\Tfr_1) = \Mot_\gamma(\Tfr_2)$.
\end{proof}
\medskip

Lemma~\ref{lem:mot_gamma_stable_classes_equivalence} implies that the map
\begin{equation}
    \bar\Mot_\gamma :
    \OpLibre\left(\GenDias\right)(n)/_{\Congr_\gamma}
    \to \Dias_\gamma(n),
    \qquad n \geq 1,
\end{equation}
satisfying, for any $\Congr_\gamma$-equivalence class
$[\Tfr]_{\Congr_\gamma}$,
\begin{equation}
    \bar\Mot_\gamma\left([\Tfr]_\gamma\right) = \Mot_\gamma(\Tfr),
\end{equation}
where $\Tfr$ is any tree of $[\Tfr]_{\Congr_\gamma}$ is well-defined.
\medskip

\begin{Lemme} \label{lem:reecriture_dias_gamma}
    For any integer $\gamma \geq 0$, any syntax tree $\Tfr$ of
    $\OpLibre\left(\GenDias\right)$ can be rewritten, by a sequence of
    $\Recr_\gamma$-rewritings, into a hook syntax tree. Moreover, this
    hook syntax tree is $\Equerre_\gamma(\Mot_\gamma(\Tfr))$.
\end{Lemme}
\begin{proof}
    In the following, to gain readability, we shall denote by $\GDias_*$
    (resp. $\DDias_*$) any element $\GDias_a$ (resp. $\DDias_a$) of
    $\GenDias$ when taking into account the value of $a \in [\gamma]$
    is not necessary. Using this notation,
    from~\eqref{equ:reecriture_dias_gamma_1}---%
    \eqref{equ:reecriture_dias_gamma_9}, we observe that
    $\Recr_\gamma$
    expresses as
    \begin{subequations}
    \begin{equation} \label{equ:reecriture_sans_etiq_dias_gamma_1}
        \DDias_* \circ_2 \GDias_*
        \enspace \Recr_\gamma \enspace
        \GDias_* \circ_1 \DDias_*,
    \end{equation}
    \begin{equation} \label{equ:reecriture_sans_etiq_dias_gamma_2}
        \GDias_* \circ_2 \DDias_*
        \enspace \Recr_\gamma \enspace
        \GDias_* \circ_1 \GDias_*,
    \end{equation}
    \begin{equation} \label{equ:reecriture_sans_etiq_dias_gamma_3}
        \DDias_* \circ_1 \GDias_*
        \enspace \Recr_\gamma \enspace
        \DDias_* \circ_2 \DDias_*,
    \end{equation}
    \begin{equation} \label{equ:reecriture_sans_etiq_dias_gamma_4}
        \GDias_* \circ_2 \GDias_*
        \enspace \Recr_\gamma \enspace
        \GDias_* \circ_1 \GDias_*,
    \end{equation}
    \begin{equation} \label{equ:reecriture_sans_etiq_dias_gamma_5}
        \DDias_* \circ_1 \DDias_*
        \enspace \Recr_\gamma \enspace
        \DDias_* \circ_2 \DDias_*.
    \end{equation}
    \end{subequations}
    \smallskip

    Let us first focus on the first part of the statement of the lemma
    to show that $\Tfr$ is rewritable by $\Recr_\gamma$ into a hook
    syntax tree. We reason by induction on the arity $n$ of $\Tfr$. When
    $n \leq 2$, $\Tfr$ is immediately a hook syntax tree. Otherwise,
    $\Tfr$ has at least two internal nodes. Then, $\Tfr$ is made of a
    root connected to a first subtree $\Tfr_1$ and a second subtree
    $\Tfr_2$. By induction hypothesis, $\Tfr$ is rewritable by
    $\Recr_\gamma$ into a tree made of a root $r$ of the same label as
    the one of the root of $\Tfr$, connected to a first subtree $\Sfr_1$
    such that $\Tfr_1 \overset{*}{\Recr_\gamma} \Sfr_1$ and a second
    subtree $\Sfr_2$ such that $\Tfr_2 \overset{*}{\Recr_\gamma} \Sfr_2$,
    both being hook syntax trees. We have to deal two cases following
    the number of internal nodes of $\Tfr_1$.
    \smallskip

    \begin{enumerate}[label={\it Case \arabic*.},fullwidth]
        \item \label{item:reecriture_dias_gamma_cas_1}
        If $\Tfr_1$ has at least one internal node, we have the two
        $\overset{*}{\Recr_\gamma}$-relations
        \begin{equation} \label{equ:reecriture_dias_gamma_cas_1}
            \begin{split}
                \Tfr \enspace \overset{*}{\Recr_\gamma} \enspace
            \end{split}
            \begin{split}
            \begin{tikzpicture}[xscale=.4,yscale=.3]
                \node(0)at(0.00,-7.43){};
                \node(10)at(10.00,-3.71){};
                \node(12)at(12.00,-1.86){\begin{math}\Sfr_2\end{math}};
                \node(2)at(2.00,-9.29){};
                \node(4)at(4.00,-11.14){};
                \node(6)at(6.00,-11.14){};
                \node(8)at(8.00,-5.57){};
                \node(1)at(1.00,-5.57){\begin{math}x\end{math}};
                \node(11)at(11.00,0.00){\begin{math}r\end{math}};
                \node(3)at(3.00,-7.43){\begin{math}\DDias_*\end{math}};
                \node(5)at(5.00,-9.29){\begin{math}\DDias_*\end{math}};
                \node(7)at(7.00,-3.71){\begin{math}\GDias_*\end{math}};
                \node(9)at(9.00,-1.86){\begin{math}\GDias_*\end{math}};
                \draw(0)--(1); \draw(1)--(7); \draw(10)--(9);
                \draw(12)--(11); \draw(2)--(3); \draw(3)--(1);
                \draw(4)--(5); \draw[densely dashed](5)--(3);
                \draw(6)--(5); \draw[densely dashed](7)--(9);
                \draw(8)--(7); \draw(9)--(11);
                \node(r)at(11,1.5){}; \draw[](r)--(11);
            \end{tikzpicture}
            \end{split}
            \begin{split}
                \enspace \overset{*}{\Recr_\gamma} \enspace
            \end{split}
            \begin{split}
            \begin{tikzpicture}[xscale=.4,yscale=.3]
                \node(0)at(0.00,-8.57){};
                \node(10)at(10.00,-8.57){};
                \node(12)at(12.00,-4.29){};
                \node(14)at(14.00,-2.14){};
                \node(2)at(2.00,-10.71){};
                \node(4)at(4.00,-12.86){};
                \node(6)at(6.00,-12.86){};
                \node(8)at(8.00,-8.57){};
                \node(1)at(1.00,-6.43){\begin{math}x\end{math}};
                \node(11)at(11.00,-2.14){\begin{math}\GDias_*\end{math}};
                \node(13)at(13.00,0.00){\begin{math}\GDias_*\end{math}};
                \node(3)at(3.00,-8.57){\begin{math}\DDias_*\end{math}};
                \node(5)at(5.00,-10.71){\begin{math}\DDias_*\end{math}};
                \node(7)at(7.00,-4.29){\begin{math}\DDias_*\end{math}};
                \node(9)at(9.00,-6.43){\begin{math}\DDias_*\end{math}};
                \draw(0)--(1); \draw(1)--(7); \draw(10)--(9);
                \draw[densely dashed](11)--(13); \draw(12)--(11);
                \draw(14)--(13); \draw(2)--(3); \draw(3)--(1);
                \draw(4)--(5); \draw[densely dashed](5)--(3);
                \draw(6)--(5); \draw(7)--(11); \draw(8)--(9);
                \draw[densely dashed](9)--(7); \node(r)at(13,1.5){};
                \draw(r)--(13);
            \end{tikzpicture}
            \end{split}\,.
        \end{equation}
        The first $\overset{*}{\Recr_\gamma}$-relation
        of~\eqref{equ:reecriture_dias_gamma_cas_1} has just been
        explained. The second one comes from the application of the
        induction hypothesis on the upper part of the tree of the middle
        of~\eqref{equ:reecriture_dias_gamma_cas_1} obtained by cutting
        the edge connecting the node $x$ to its father. When the
        rightmost tree of~\eqref{equ:reecriture_dias_gamma_cas_1} is not
        already a hook syntax tree, one has two cases following the
        label of $x$.
        \smallskip

        \begin{enumerate}[label={\it Case \arabic{enumi}.\arabic*.},fullwidth]
            \item If $x$ is labeled by $\DDias_*$,
            by~\eqref{equ:reecriture_sans_etiq_dias_gamma_5}, the bottom
            part of the rightmost tree
            of~\eqref{equ:reecriture_dias_gamma_cas_1} consisting
            in internal nodes labeled by $\DDias_*$ is rewritable by
            $\Recr_\gamma$ into a right comb tree wherein internal nodes
            are labeled by $\DDias_*$. Then, the rightmost tree
            of~\eqref{equ:reecriture_dias_gamma_cas_1} is rewritable
            by $\Recr_\gamma$ into a hook syntax tree, and then $\Tfr$
            also is.
            \smallskip

            \item Otherwise, $x$ is labeled by $\GDias_*$. By definition
            of $\Equerre_\gamma$, the second subtree of $x$ is a leaf.
            By~\eqref{equ:reecriture_sans_etiq_dias_gamma_3}, the bottom
            part of the rightmost tree
            of~\eqref{equ:reecriture_dias_gamma_cas_1} consisting
            in $x$ and internal nodes labeled by $\DDias_*$ can be
            rewritten by $\Recr_\gamma$ into a right comb tree wherein
            internal nodes are labeled by $\DDias_*$. Then, the rightmost
            tree of \eqref{equ:reecriture_dias_gamma_cas_1} is
            rewritable by $\Recr_\gamma$ into a hook syntax tree, and
            then $\Tfr$ also is.
        \end{enumerate}
        \smallskip

        \item Otherwise, $\Tfr_1$ is the leaf. We then have the
        $\overset{*}{\Recr_\gamma}$-relation
        \begin{equation} \label{equ:reecriture_dias_gamma_cas_2}
            \begin{split}
                \Tfr \enspace \overset{*}{\Recr_\gamma} \enspace
            \end{split}
            \begin{split}
            \begin{tikzpicture}[xscale=.45,yscale=.4]
                \node(0)at(0.00,-1.67){};
                \node(2)at(2.00,-3.33){\begin{math}\Sfr_{21}\end{math}};
                \node(4)at(4.00,-3.33){\begin{math}\Sfr_{22}\end{math}};
                \node(1)at(1.00,0.00){\begin{math}r\end{math}};
                \node(3)at(3.00,-1.67){\begin{math}r'\end{math}};
                \draw(0)--(1); \draw(2)--(3); \draw(3)--(1); \draw(4)--(3);
                \node(r)at(1,1.25){};
                \draw(r)--(1);
            \end{tikzpicture}
            \end{split}\,,
        \end{equation}
        where $\Sfr_{21}$ is the first subtree of the root of $\Sfr_2$,
        $\Sfr_{22}$ is the second subtree of the root of $\Sfr_2$, and
        $r'$ is a node with the same label as the root of $\Sfr_2$.
        \smallskip

        \begin{enumerate}[label={\it Case \arabic{enumi}.\arabic*.},fullwidth]
            \item If $r \circ_2 r'$ is equal to
            $\DDias_* \circ_2 \GDias_*$, $\GDias_* \circ_2 \DDias_*$, or
            $\GDias_* \circ_2 \GDias_*$, respectively
            by~\eqref{equ:reecriture_sans_etiq_dias_gamma_1},
            \eqref{equ:reecriture_sans_etiq_dias_gamma_2},
            and~\eqref{equ:reecriture_sans_etiq_dias_gamma_4},
            the rightmost tree
            of~\eqref{equ:reecriture_dias_gamma_cas_2} can be rewritten
            by $\Recr_\gamma$ into a tree $\Rfr$ having a first subtree
            with at least one internal node. Hence, $\Rfr$ is of the
            form required to be treated
            by~\ref{item:reecriture_dias_gamma_cas_1}, implying that
            $\Tfr$ is rewritable by $\Recr_\gamma$ into a hook syntax tree.
            \smallskip

            \item Otherwise, $r \circ_2 r'$ is equal to
            $\DDias_* \circ_2 \DDias_*$. Since $\Sfr_2$ is by hypothesis
            a hook syntax tree, it is necessarily a right comb tree
            whose internal nodes are labeled by $\DDias_*$. Hence, the
            rightmost tree of~\eqref{equ:reecriture_dias_gamma_cas_2} is
            already a hook syntax tree, showing that $\Tfr$ is rewritable
            by $\Recr_\gamma$ into a hook syntax tree.
        \end{enumerate}
    \end{enumerate}
    \smallskip

    Let us finally show the last part of the statement of the lemma.
    Observe that, by definition of $\Equerre_\gamma$ and $\Mot_\gamma$,
    if $\Sfr_1$ and $\Sfr_2$ are two different hook syntax trees,
    $\Mot_\gamma(\Sfr_1) \ne \Mot_\gamma(\Sfr_2)$. We have just shown
    that $\Tfr$ is rewritable by $\Recr_\gamma$ into a hook syntax tree
    $\Sfr$. Besides, by Lemma~\ref{lem:mot_gamma_stable_classes_equivalence},
    one has $\Mot_\gamma(\Tfr) = \Mot_\gamma(\Sfr)$. Then, $\Sfr$ is
    necessarily the hook syntax tree $\Equerre_\gamma(\Mot_\gamma(\Tfr))$.
\end{proof}
\medskip

\subsubsection{Presentation by generators and relations}

\begin{Lemme} \label{lem:mot_gamma_quotient_bijection}
    For any integers $\gamma \geq 0$ and $n \geq 1$, the map
    $\bar\Mot_\gamma$ defines a bijection between
    $\OpLibre\left(\GenDias\right)(n)/_{\Congr_\gamma}$ and
    $\Dias_\gamma(n)$.
\end{Lemme}
\begin{proof}
    Let us show that $\bar\Mot_\gamma$ is injective. Let $\Tfr_1$ and
    $\Tfr_2$ be two syntax trees of $\OpLibre\left(\GenDias\right)$ such
    that $\Mot_\gamma(\Tfr_1) = \Mot_\gamma(\Tfr_2)$ and let
    $\Sfr := \Equerre_\gamma(\Mot_\gamma(\Tfr_1)) =
    \Equerre_\gamma(\Mot_\gamma(\Tfr_2))$.
    By Lemma~\ref{lem:reecriture_dias_gamma}, one has
    $\Tfr_1 \overset{*}{\Recr_\gamma} \Sfr$ and $
    \Tfr_2 \overset{*}{\Recr_\gamma} \Sfr$, and hence,
    $\Tfr_1 \Congr_\gamma \Tfr_2$. By the definition of the map
    $\bar \Mot_\gamma$ from the map $\Mot_\gamma$, this show that
    $\bar\Mot_\gamma$ is injective. Besides, by
    Lemma~\ref{lem:mot_gamma_surjection}, $\bar\Mot_\gamma$ is surjective,
    whence the statement of the lemma.
\end{proof}
\medskip

\begin{Theoreme} \label{thm:presentation_dias_gamma}
    For any integer $\gamma \geq 0$, the operad $\Dias_\gamma$
    admits the following presentation. It is generated by $\GenDias$ and
    its space of relations $\RelDias$ is the space induced by the
    equivalence relation $\Rel_\gamma$ satisfying
    \begin{subequations}
    \begin{equation} \label{equ:relation_dias_gamma_1}
        \GDias_a \circ_1 \DDias_{a'}
        \enspace \Rel_\gamma \enspace
        \DDias_{a'} \circ_2 \GDias_a,
        \qquad a, a' \in [\gamma],
    \end{equation}
    \begin{equation} \label{equ:relation_dias_gamma_2}
        \GDias_a \circ_1 \GDias_b
        \enspace \Rel_\gamma \enspace
        \GDias_a \circ_2 \DDias_b,
        \qquad a < b \in [\gamma],
    \end{equation}
    \begin{equation} \label{equ:relation_dias_gamma_3}
        \DDias_a \circ_1 \GDias_b
        \enspace \Rel_\gamma \enspace
        \DDias_a \circ_2 \DDias_b,
        \qquad a < b \in [\gamma],
    \end{equation}
    \begin{equation} \label{equ:relation_dias_gamma_4}
        \GDias_b \circ_1 \GDias_a
        \enspace \Rel_\gamma \enspace
        \GDias_a \circ_2 \GDias_b,
        \qquad a < b \in [\gamma],
    \end{equation}
    \begin{equation} \label{equ:relation_dias_gamma_5}
        \DDias_a \circ_1 \DDias_b
        \enspace \Rel_\gamma \enspace
        \DDias_b \circ_2 \DDias_a,
        \qquad a < b \in [\gamma],
    \end{equation}
    \begin{equation} \label{equ:relation_dias_gamma_6}
        \GDias_d \circ_1 \GDias_d
        \enspace \Rel_\gamma \enspace
        \GDias_d \circ_2 \GDias_c
        \enspace \Rel_\gamma \enspace
        \GDias_d \circ_2 \DDias_c,
        \qquad c \leq d \in [\gamma],
    \end{equation}
    \begin{equation} \label{equ:relation_dias_gamma_7}
        \DDias_d \circ_1 \GDias_c
        \enspace \Rel_\gamma \enspace
        \DDias_d \circ_1 \DDias_c
        \enspace \Rel_\gamma \enspace
        \DDias_d \circ_2 \DDias_d,
        \qquad c \leq d \in [\gamma].
    \end{equation}
    \end{subequations}
\end{Theoreme}
\begin{proof}
    By Lemma~\ref{lem:mot_gamma_quotient_bijection}, the map
    $\bar\Mot_\gamma$ is, for any $n \geq 1$, a bijection between the
    sets $\OpLibre\left(\GenDias\right)(n)/_{\Congr_\gamma}$ and
    $\Dias_\gamma(n)$. Moreover, by Lemma~\ref{lem:mot_gamma_morphisme},
    $\Mot_\gamma$ is an operad morphism, and then $\bar\Mot_\gamma$ also
    is. Hence, $\bar\Mot_\gamma$ is an operad isomorphism between
    $\OpLibre\left(\GenDias\right)/_{\Congr_\gamma}$ and $\Dias_\gamma$.
    Therefore, since $\RelLibre_{\Dias_\gamma}$ is the space induced by
    $\Congr_\gamma$, $\Dias_\gamma$ admits the stated presentation.
\end{proof}
\medskip

The space of relations $\RelDias$ of $\Dias_\gamma$ exhibited by
Theorem~\ref{thm:presentation_dias_gamma} can be rephrased in a more
compact way as the space generated by
\begin{subequations}
\begin{equation} \label{equ:relation_dias_gamma_1_concise}
    \GDias_a \circ_1 \DDias_{a'} - \DDias_{a'} \circ_2 \GDias_a,
    \qquad a, a' \in [\gamma],
\end{equation}
\begin{equation} \label{equ:relation_dias_gamma_2_concise}
    \GDias_a \circ_1 \GDias_{a \Max a'} - \GDias_a \circ_2 \DDias_{a'},
    \qquad a, a' \in [\gamma],
\end{equation}
\begin{equation} \label{equ:relation_dias_gamma_3_concise}
    \DDias_a \circ_1 \GDias_{a'} - \DDias_a \circ_2 \DDias_{a \Max a'},
    \qquad a, a' \in [\gamma],
\end{equation}
\begin{equation} \label{equ:relation_dias_gamma_4_concise}
    \GDias_{a \Max a'} \circ_1 \GDias_a - \GDias_a \circ_2 \GDias_{a'},
    \qquad a, a' \in [\gamma],
\end{equation}
\begin{equation} \label{equ:relation_dias_gamma_5_concise}
    \DDias_a \circ_1 \DDias_{a'} - \DDias_{a \Max a'} \circ_2 \DDias_a,
    \qquad a, a' \in [\gamma].
\end{equation}
\end{subequations}
\medskip

Observe that, by Theorem \ref{thm:presentation_dias_gamma}, $\Dias_1$
and the diassociative operad (see~\cite{Lod01} or
Section~\ref{subsubsec:dias}) admit the same presentation.
Then, for all integers $\gamma \geq 0$, the operads $\Dias_\gamma$
are generalizations of the diassociative operad.
\medskip

\subsection{Miscellaneous properties}
From the description of the elements of $\Dias_\gamma$ and its structure
revealed by its presentation, we develop here some of its properties.
Unless otherwise specified, $\Dias_\gamma$ is still considered in this
section as a set-operad.
\medskip

\subsubsection{Koszulity}

\begin{Theoreme} \label{thm:koszulite_dias_gamma}
    For any integer $\gamma \geq 0$, $\Dias_\gamma$ is a Koszul operad.
    Moreover, the set of hook syntax trees of $\OpLibre\left(\GenDias\right)$
    forms a Poincaré-Birkhoff-Witt basis of $\Dias_\gamma$.
\end{Theoreme}
\begin{proof}
    From the definition of hook syntax trees, it appears that no hook
    syntax tree can be rewritten by $\Recr_\gamma$ into another syntax
    tree. Hence, and by Lemma~\ref{lem:reecriture_dias_gamma},
    $\Recr_\gamma$ is a terminating rewrite rule and its normal forms
    are hook syntax trees. Moreover, again by
    Lemma~\ref{lem:reecriture_dias_gamma}, since any syntax tree is
    rewritable by $\Recr_\gamma$ into a unique hook syntax tree,
    $\Recr_\gamma$ is a confluent rewrite rule, and hence, $\Recr_\gamma$
    is convergent. Now, since by Theorem~\ref{thm:presentation_dias_gamma},
    the space of relations of $\Dias_\gamma$ is the space induced by the
    operad congruence induced by $\Recr_\gamma$, by the Koszulity
    criterion~\cite{Hof10,DK10,LV12} we have reformulated in
    Section~\ref{subsubsec:dual_de_Koszul}, $\Dias_\gamma$ is a Koszul
    operad and the set of of hook syntax trees of
    $\OpLibre\left(\GenDias\right)$ forms a Poincaré-Birkhoff-Witt basis
    of $\Dias_\gamma$.
\end{proof}
\medskip

\subsubsection{Symmetries}
If $\Oca_1$ and $\Oca_2$ are two operads, a linear map
$\phi : \Oca_1 \to \Oca_2$ is an {\em operad antimorphism} if it
respects arities and anticommutes with partial composition maps, that is,
\begin{equation}
    \phi(x \circ_i y) = \phi(x) \circ_{n - i + 1} \phi(y),
    \qquad x \in \Oca(n), y \in \Oca, i \in [n].
\end{equation}
A {\em symmetry} of an operad $\Oca$ is either an automorphism or an
antiautomorphism. The set of all symmetries of $\Oca$ form a group for
the composition, called the {\em group of symmetries} of $\Oca$.
\medskip

\begin{Proposition} \label{prop:symetries_dias_gamma}
    For any integer $\gamma \geq 0$, the group of symmetries of
    $\Dias_\gamma$ as a set-operad contains two elements: the identity
    map and the linear map sending any word of $\Dias_\gamma$ to its
    mirror image.
\end{Proposition}
\begin{proof}
    Let us denote by $\Gen_\gamma$ the set $\{0a, a0 : a \in [\gamma]\}$.
    Since $\Dias_\gamma$ is generated by $\Gen_\gamma$, any automorphism
    or antiautomorphism $\phi$ of $\Dias_\gamma$ is wholly determined by
    the images of the elements of $\Gen_\gamma$. Besides let us observe
    that $\phi$ is in particular a permutation of $\Gen_\gamma$.
    \smallskip

    By contradiction, assume that $\phi$ is an automorphism of
    $\Dias_\gamma$ different from the identity map. We have two cases to
    explore.
    \smallskip

    \begin{enumerate}[label={\it Case \arabic*.},fullwidth]
        \item If there are $a, a' \in [\gamma]$ satisfying
        $\phi(0a) = a'0$, since $\phi$ is a permutation of $\Gen_\gamma$,
        there are $b, b' \in [\gamma]$ satisfying $\phi(b0) = 0b'$. Then,
        we have at the same time $b0 \circ_2 0a = b0a = 0a \circ_1 b0$,
        \begin{equation}
            \phi(b0 \circ_2 0a) = \phi(b0) \circ_2 \phi(0a)
            = 0b' \circ_2 a'0 = 0 \, (b'\Max a') \, b',
        \end{equation}
        and
        \begin{equation}
            \phi(0a \circ_1 b0) = \phi(0a) \circ_1 \phi(b0)
            = a'0 \circ_1 0b' = a' \, (a' \Max b') \, 0.
        \end{equation}
        This shows that $\phi(b0 \circ_2 0a) \ne \phi(0a \circ_1 b0)$
        and hence, $\phi$ is not an operad morphism. By a similar
        argument, one can show that there are no $a, a' \in [\gamma]$
        such that $\phi(a0) = 0a'$.
        \smallskip

        \item Otherwise, for all $a \in [\gamma]$, we have $\phi(0a) = 0a'$
        and $\phi(a0) = a''0$ for some $a', a'' \in [\gamma]$. Since, by
        hypothesis, $\phi$ is not the identity map, there exist
        $a \ne a' \in [\gamma]$ such that $\phi(0a) = 0a'$ or
        $\phi(a0) = a'0$. Let us assume, without loss of generality,
        that $\phi(0a) = 0a'$. Since $\phi$ is a permutation of
        $\Gen_\gamma$, there exist $b \ne b' \in [\gamma]$ such that
        $\phi(0b) = 0b'$. One can assume, without loss of generality,
        that $a < b$ and $b' < a'$. Then, we have at the same time
        $0a \circ_2 0b = 0ab = 0b \circ_1 0a$,
        \begin{equation}
            \phi(0a \circ_2 0b) = \phi(0a) \circ_2 \phi(0b)
            = 0a' \circ_2 0b' = 0a'a',
        \end{equation}
        and
        \begin{equation}
            \phi(0b \circ_1 0a) = \phi(0b) \circ_1 \phi(0a)
            = 0b' \circ_1 0a' = 0a'b'.
        \end{equation}
        This shows that $\phi(0a \circ_2 0b) \ne \phi(0b \circ_1 0a)$
        and hence, that $\phi$ is not an operad morphism. By a similar
        argument, one can show that there are no
        $a \ne a' \in [\gamma]$ such that $\phi(a0) = \phi(a'0)$.
    \end{enumerate}
    \smallskip

    We then have shown that if $\phi$ is an automorphism of $\Dias_\gamma$,
    it is necessarily the identity map.
    \smallskip

    Finally, by Proposition~\ref{prop:elements_dias_gamma}, if $x$ is
    an element of $\Dias_\gamma$, its mirror image also is in
    $\Dias_\gamma$. Moreover, it is immediate to see that the map sending
    a word to its mirror image is an antiautomorphism of $\Dias_\gamma$.
    Similar arguments as the ones developed previously show that it is
    the only.
\end{proof}
\medskip

\subsubsection{Basic operad}
A set-operad $\Oca$ is {\em basic} if for all $y_1, \dots, y_n \in \Oca$,
all the maps
\begin{equation}
    \circ^{y_1, \dots, y_n} :
    \Oca(n) \to \Oca(|y_1| + \dots + |y_n|)
\end{equation}
defined by
\begin{equation}
    \circ^{y_1, \dots, y_n}(x) := x \circ (y_1, \dots, y_n),
    \qquad x \in \Oca(n),
\end{equation}
are injective. This property for set-operads introduced by
Vallette~\cite{Val07} is a very relevant one since there is a general
construction producing a family of posets (see~\cite{MY91} and~\cite{CL07})
from a basic set-operad. This family of posets leads to the definition
of an incidence Hopf algebra by a construction of Schmitt~\cite{Sch94}.
\medskip

\begin{Proposition} \label{prop:dias_gamma_basique}
    For any integer $\gamma \geq 0$, $\Dias_\gamma$ is a basic operad.
\end{Proposition}
\begin{proof}
    Let $n \geq 1$, $y_1, \dots, y_n$ be words of $\Dias_\gamma$,
    and $x$ and $x'$ be two words of $\Dias_\gamma(n)$ such that
    $\circ^{y_1, \dots, y_n}(x) = \circ^{y_1, \dots, y_n}(x')$. Then,
    for all $i \in [n]$ and $j \in [|y_i|]$, we have
    $x_i \Max y_{i, j} = x'_i \Max y_{i, j}$ where $y_{i, j}$ is the
    $j$th letter of $y_i$. Since by
    Proposition~\ref{prop:elements_dias_gamma}, any word $y_i$ contains
    a $0$, we have in particular $x_i \Max 0 = x'_i \Max 0$ for all
    $i \in [n]$. This implies $x = x'$ and thus, that
    $\circ^{y_1, \dots, y_n}$ is injective.
\end{proof}
\medskip

\subsubsection{Rooted operad}
We restate here a property on operads introduced by
Chapoton~\cite{Cha14}. An operad $\Oca$ is {\em rooted} if there is a map
\begin{equation}
    \Racine : \Oca(n) \to [n],
    \qquad n \geq 1,
\end{equation}
satisfying, for all $x \in \Oca(n)$, $y \in \Oca(m)$, and $i \in [n]$,
\begin{equation} \label{equ:operade_enracinee}
    \Racine(x \circ_i y) =
    \begin{cases}
        \Racine(x) + m - 1 & \mbox{if } i \leq \Racine(x) - 1, \\
        \Racine(x) + \Racine(y) - 1 & \mbox{if } i = \Racine(x), \\
        \Racine(x) & \mbox{otherwise (} i \geq \Racine(x) + 1 \mbox{)}.
    \end{cases}
\end{equation}
We call such a map a {\em root map}. More intuitively, the root map of a
rooted operad associates a particular input with any of its elements and
this input is preserved by partial compositions.
\medskip

It is immediate that any operad $\Oca$ is a rooted operad for the root
maps $\Racine_{\mathrm{L}}$ and $\Racine_{\mathrm{R}}$, which send
respectively all elements $x$ of arity $n$ to $1$ or to $n$. For this
reason, we say that an operad~$\Oca$ is {\em nontrivially rooted} if
it can be endowed with a root map different from $\Racine_{\mathrm{L}}$
and~$\Racine_{\mathrm{R}}$.
\medskip

\begin{Proposition} \label{prop:dias_gamma_enracinee}
    For any integer $\gamma \geq 0$, $\Dias_\gamma$ is a nontrivially
    rooted operad for the root map sending any word of $\Dias_\gamma$ to
    the position of its $0$.
\end{Proposition}
\begin{proof}
    Thanks to Proposition~\ref{prop:elements_dias_gamma}, the map of the
    statement of the proposition is well-defined. The fact that $0$ is
    the neutral element for the $\Max$ operation and the fact that any
    word of $\Dias_\gamma$ contains exactly one $0$ imply that this map
    satisfies~\eqref{equ:operade_enracinee}. Finally, this map is
    obviously different from $\Racine_{\mathrm{L}}$ and
    $\Racine_{\mathrm{R}}$, whence the statement of the proposition.
\end{proof}
\medskip

\subsubsection{Alternative basis} \label{subsubsec:base_K}
In this section, $\Dias_\gamma$ is considered as an operad in the
category of vector spaces.
\medskip

Let $\OrdDias_\gamma$ be the order relation on the underlying set of
$\Dias_\gamma(n)$, $n \geq 1$, where for all words $x$ and $y$ of
$\Dias_\gamma$ of a same arity $n$, we have
\begin{equation}
    x \OrdDias_\gamma y
    \qquad \mbox{ if } x_i \leq y_i \mbox{ for all } i \in [n].
\end{equation}
This order relation allows to define
for all word $x$ of $\Dias_\gamma$ the elements
\begin{equation} \label{equ:base_K_vers_mots}
    \Ksf^{(\gamma)}_x :=
    \sum_{x \OrdDias_\gamma x'} \mu_\gamma(x, x') \, x',
\end{equation}
where $\mu_\gamma$ is the Möbius function of the poset defined by
$\OrdDias_\gamma$. For instance,
\begin{equation}
    \Ksf^{(2)}_{102} = 102 - 202,
\end{equation}
\begin{equation}
    \Ksf^{(3)}_{102} = \Ksf^{(4)}_{102} = 102 - 103 - 202 + 203,
\end{equation}
\begin{equation}
    \Ksf^{(3)}_{23102}
    = 23102 - 23103 - 23202 + 23203 - 33102 + 33103 + 33202 - 33203.
\end{equation}
\medskip

Since, by Möbius inversion, for any word $x$ of $\Dias_\gamma$ one has
\begin{equation} \label{equ:mots_vers_base_K}
    x = \sum_{x \OrdDias_\gamma x'} \Ksf^{(\gamma)}_{x'},
\end{equation}
the family of all $\Ksf^{(\gamma)}_x$, where the $x$ are words of
$\Dias_\gamma$, forms by triangularity a basis of $\Dias_\gamma$, called
the {\em $\Ksf$-basis}.
\medskip

If $u$ and $v$ are two words of a same length $n$, we denote by
$\Hamming(u, v)$ the {\em Hamming distance} between $u$ and $v$ that is
the number of positions $i \in [n]$ such that $u_i \ne v_i$. Moreover,
for any word $x$ of $\Dias_\gamma$ of length $n$ and any subset $J$ of
$[n]$, we denote by $\Incr_\gamma(x, J)$ the set of words obtained by
incrementing by one some letters of $x$ smaller than $\gamma$ and
greater than $0$ whose positions are in $J$. We shall simply denote by
$\Incr_\gamma(x)$ the set $\Incr_\gamma(x, [n])$.
Proposition~\ref{prop:elements_dias_gamma} ensures that all
$\Incr_\gamma(x, J)$ are sets of words of $\Dias_\gamma$.
\medskip

\begin{Lemme} \label{lem:expression_directe_base_K}
    For any integer $\gamma \geq 0$ and any word $x$ of $\Dias_\gamma$,
    \begin{equation} \label{equ:expression_directe_base_K}
        \Ksf^{(\gamma)}_x =
        \sum_{x' \in \Incr_\gamma(x)} (-1)^{\Hamming(x, x')} \, x'.
    \end{equation}
\end{Lemme}
\begin{proof}
    Let $n$ be the arity of $x$. To compute $\Ksf^{(\gamma)}_x$ from its
    definition~\eqref{equ:base_K_vers_mots}, it is enough to know the
    Möbius function $\mu_\gamma$ of the poset $\Pbb^{(\gamma)}_x$
    consisting in the words $x'$ of $\Dias_\gamma$ satisfying
    $x \OrdDias_\gamma x'$. Immediately from the definition of
    $\OrdDias_\gamma$, it appears that $\Pbb^{(\gamma)}_x$ is isomorphic
    to the Cartesian product poset
    \begin{equation} \label{equ:expression_directe_base_K_poset}
        \Tbb^{(\gamma)}_x :=
        \Tbb\left(\gamma - x_1\right)
        \times \dots \times
        \Tbb\left(\gamma - x_{r - 1}\right)
        \times
        \Tbb(0)
        \times
        \Tbb\left(\gamma - x_{r + 1}\right)
        \times \dots \times
        \Tbb\left(\gamma - x_n\right),
    \end{equation}
    where for any nonnegative integer $k$, $\Tbb(k)$ denotes the poset
    over $\{0\} \cup [k]$ with the natural total order relation, and $r$
    is the position of, by Proposition~\ref{prop:elements_dias_gamma},
    the only $0$ of $x$. The map
    $\phi^{(\gamma)}_x : \Pbb^{(\gamma)}_x \to \Tbb^{(\gamma)}_x$
    defined for all words $x'$ of $\Pbb^{(\gamma)}_x$ by
    \begin{equation}
        \phi^{(\gamma)}_x(x') := \left(x'_1 - x_1, \dots,
        x'_{r - 1} - x_{r - 1}, 0, x'_{r + 1} - x_{r + 1},
        \dots, x'_n - x_n\right)
    \end{equation}
    is an isomorphism of posets.
    \smallskip

    Recall that the Möbius function $\mu$ of $\Tbb(k)$ satisfies, for
    all $a, a' \in \Tbb(k)$,
    \begin{equation}
        \mu(a, a') =
        \begin{cases}
            1 & \mbox{if } a' = a, \\
            -1 & \mbox{if } a' = a + 1, \\
            0 & \mbox{otherwise}.
        \end{cases}
    \end{equation}
    Moreover, since by~\cite{Sta11}, the Möbius function of a Cartesian
    product poset is the product of the Möbius functions of the posets
    involved in the product, through the isomorphism $\phi^{(\gamma)}_x$,
    we obtain that when $x'$ is in $\Incr_\gamma(x)$,
    $\mu_\gamma(x, x') = (-1)^{\Hamming(x, x')}$ and that when $x'$ is
    not in $\Incr_\gamma$, $\mu_\gamma(x, x') = 0$. Therefore,
    \eqref{equ:expression_directe_base_K} is established.
\end{proof}
\medskip

\begin{Lemme} \label{lem:somme_alternee_incr}
    For any integer $\gamma \geq 0$, any word $x$ of $\Dias_\gamma$, and
    any nonempty set $J$ of positions of letters of $x$ that are greater
    than $0$ and smaller than $\gamma$,
    \begin{equation}
        \sum_{x' \in \Incr_\gamma(x, J)} (-1)^{\Hamming(x, x')} = 0.
    \end{equation}
\end{Lemme}
\begin{proof}
    The statement of the lemma follows by induction on the nonzero
    cardinality of $J$.
\end{proof}
\medskip

To compute a direct expression for the partial composition of
$\Dias_\gamma$ over the $\Ksf$-basis, we have to introduce two notations.
If $x$ is a word of $\Dias_\gamma$ of length nonsmaller than $2$, we
denote by $\min(x)$ the smallest letter of $x$ among its letters
different from $0$. Proposition~\ref{prop:elements_dias_gamma} ensures
that $\min(x)$ is well-defined. Moreover, for all words $x$ and $y$ of
$\Dias_\gamma$, a position $i$ such that $x_i \ne 0$, and $a \in [\gamma]$,
we denote by $x \circ_{a, i} y$ the word $x \circ_i y$ in which the $0$
coming from $y$ is replaced by $a$ instead of $x_i$.
\medskip

\begin{Theoreme} \label{thm:composition_base_K}
    For any integer $\gamma \geq 0$, the partial composition of
    $\Dias_\gamma$ over the $\Ksf$-basis satisfies, for all words $x$
    and $y$ of $\Dias_\gamma$ of arities nonsmaller than $2$,
    \begin{equation}
        \Ksf^{(\gamma)}_x \circ_i \Ksf^{(\gamma)}_y =
        \begin{cases}
            \Ksf^{(\gamma)}_{x \circ_i y} & \mbox{if } \min(y) > x_i, \\
            \sum_{a \in [x_i, \gamma]} \Ksf^{(\gamma)}_{x \circ_{a, i} y}
                & \mbox{if } \min(y) = x_i, \\
            0 & \mbox{otherwise (} \min(y) < x_i \mbox{)}.
        \end{cases}
    \end{equation}
\end{Theoreme}
\begin{proof}
    First of all, by Lemma~\ref{lem:expression_directe_base_K} together
    with~\eqref{equ:mots_vers_base_K}, we obtain
    \begin{equation} \label{equ:demo_composition_base_K} \begin{split}
        \Ksf^{(\gamma)}_x \circ_i \Ksf^{(\gamma)}_y
            & = \sum_{\substack{x' \in \Incr_\gamma(x) \\
                                y' \in \Incr_\gamma(y)}}
            (-1)^{\Hamming(x, x') + \Hamming(y, y')}
            \left(\sum_{x' \circ_i y' \OrdDias_\gamma z}
            \Ksf^{(\gamma)}_z\right) \\
            & = \sum_{x \circ_i y \OrdDias_\gamma z} \;
            \sum_{\substack{x' \in \Incr_\gamma(x) \\
                                  y' \in \Incr_\gamma(y) \\
                                  x' \circ_i y' \OrdDias_\gamma z}}
            (-1)^{\Hamming(x, x') + \Hamming(y, y')}
            \, \Ksf^{(\gamma)}_z.
    \end{split} \end{equation}
    \smallskip

    Let us denote by $n$ (resp. $m$) the arity of $x$ (resp. $y$) and
    let $z$ be a word of $\Dias_\gamma$ such that
    $x \circ_i y \OrdDias_\gamma z$. Let $x' \in \Incr_\gamma(x)$ and
    $y' \in \Incr_\gamma(y)$. We have, by definition of the partial
    composition of $\Dias_\gamma$,
    \begin{equation}
        x \circ_i y =
        x_1 \dots x_{i - 1}
            \, t_1 \dots t_{r - 1} \, x_i \, t_{r + 1} \dots t_m \,
        x_{i + 1} \dots x_n,
    \end{equation}
    and
    \begin{equation}
        x' \circ_i y' =
        x'_1 \dots x'_{i - 1}
            \, t'_1 \dots t'_{r - 1} \, x'_i \, t'_{r + 1} \dots t'_m \,
        x'_{i + 1} \dots x'_n,
    \end{equation}
    where $r$ denotes the position of the only, by
    Proposition~\ref{prop:elements_dias_gamma}, $0$ of $y$ and for all
    $j \in [m] \setminus \{r\}$, $t_j := x_i \Max y_j$ and
    $t'_j := x'_i \Max y'_j$. By~\eqref{equ:demo_composition_base_K},
    the pair $(x', y')$ contributes to the coefficient of $\Ksf^{(\gamma)}_z$
    in~\eqref{equ:demo_composition_base_K} if and only if
    $x \circ_i y \OrdDias_\gamma x' \circ_i y' \OrdDias z$. To compute
    this coefficient, we have three cases to consider following the value
    of $\min(y)$ compared to the value of~$x_i$.
    \smallskip

    \begin{enumerate}[label={\it Case \arabic*.},fullwidth]
        \item Assume first that $\min(y) < x_i$. Then, there is at least
        a $s \in [m] \setminus \{r\}$ such that $y_s < x_i$. This implies
        that $t_s = x_i$ and that $y'_s$ has no influence on $t'_s$ and
        then, on $x' \circ_i y'$. Thus, the word
        $y'' := y'_1 \dots y'_{s - 1} a y'_{s + 1} \dots y'_m$ where $a$
        is the only possible letter such that $y'' \in \Incr_\gamma(y)$
        and $a \ne y'_s$ satisfies $x' \circ_i y'' = x' \circ_i y'$.
        Therefore, since $\Hamming(y', y'') = 1$, the contribution of
        the pair $(x', y')$ for the coefficient of $\Ksf^{(\gamma)}_z$
        in~\eqref{equ:demo_composition_base_K} is compensated by the
        contribution of the pair $(x', y'')$. This shows that this
        coefficient is $0$ and hence,
        $\Ksf^{(\gamma)}_x \circ_i \Ksf^{(\gamma)}_y = 0$.
        \smallskip

        \item \label{item:composition_base_K_cas_2}
        Assume now that $\min(y) > x_i$. Then, for all
        $j \in [m] \setminus \{r\}$, we have $y_j > x_i$ and thus,
        $t_j = y_j$. When $z = x \circ_i y$, we necessarily have $x' = x$
        and $y' = y$. Hence, the coefficient of $\Ksf^{(\gamma)}_{x \circ_i y}$
        in~\eqref{equ:demo_composition_base_K} is $1$. Else, when
        $z \ne x \circ_i y$, we have
        $x' \circ_i y' \in \Incr_\gamma(x \circ_i y, J)$, where $J$ is
        the nonempty set of the positions of letters of $z$ different
        from letters of $x \circ_i y$. Now,
        from~\eqref{equ:demo_composition_base_K}, the coefficient of
        $\Ksf^{(\gamma)}_z$ in~\eqref{equ:demo_composition_base_K} is
        \begin{equation}
            \sum_{x' \circ_i y' \in \Incr_\gamma(x \circ_i y, J)}
            (-1)^{\Hamming(x, x') + \Hamming(y, y')}.
        \end{equation}
        Lemma~\ref{lem:somme_alternee_incr} implies that this coefficient
        is $0$. This shows that
        $\Ksf^{(\gamma)}_x \circ_i \Ksf^{(\gamma)}_y =
        \Ksf^{(\gamma)}_{x \circ_i y}$.
        \smallskip

        \item The last case occurs when $\min(y) = x_i$. Then, for all
        $j \in [m] \setminus \{r\}$, we have $y_j \geq x_i$ and thus,
        $t_j = y_j$. Moreover, there is at least a $s \in [m] \setminus \{r\}$
        such that $y_s = x_i$. When $z = x \circ_{a, i} y$ with
        $a \in [x_i, \gamma]$, we necessarily have $x' = x$ and $y' = y$.
        Therefore, for all $a \in [x_i, \gamma]$, the
        $\Ksf^{(\gamma)}_{x \circ_{a, i}}$ have coefficient~$1$
        in~\eqref{equ:demo_composition_base_K}. The same argument as the
        one exposed for \ref{item:composition_base_K_cas_2} shows that
        when $z \ne x \circ_{a, i} y$ for all $a \in [x_i, \gamma]$, the
        coefficient of $\Ksf^{(\gamma)}_z$ is zero. Hence,
        $\Ksf^{(\gamma)}_x \circ_i \Ksf^{(\gamma)}_y =
        \sum_{a \in [x_i, \gamma]} \Ksf^{(\gamma)}_{x \circ_{a, i} y}$.
    \end{enumerate}
\end{proof}
\medskip

We have for instance
\begin{equation}
    \Ksf^{(5)}_{20413} \circ_1 \Ksf^{(5)}_{304} = \Ksf^{(5)}_{3240413},
\end{equation}
\begin{equation}
    \Ksf^{(5)}_{20413} \circ_2 \Ksf^{(5)}_{304} = \Ksf^{(5)}_{2304413},
\end{equation}
\begin{equation}
    \Ksf^{(5)}_{20413} \circ_3 \Ksf^{(5)}_{304} = 0,
\end{equation}
\begin{equation}
    \Ksf^{(5)}_{20413} \circ_4 \Ksf^{(5)}_{304} = \Ksf^{(5)}_{2043143},
\end{equation}
\begin{equation}
    \Ksf^{(5)}_{20413} \circ_5 \Ksf^{(5)}_{304}
        = \Ksf^{(5)}_{2041334} + \Ksf^{(5)}_{2041344}
        + \Ksf^{(5)}_{2041354}.
\end{equation}
\medskip

Theorem~\ref{thm:composition_base_K} implies in particular that the
structure coefficients of the partial composition of $\Dias_\gamma$ over
the $\Ksf$-basis are $0$ or $1$. It is possible to define another bases
of $\Dias_\gamma$ by reversing in~\eqref{equ:base_K_vers_mots} the
relation $\OrdDias_\gamma$ and by suppressing or keeping the Möbius
function $\mu_\gamma$. This gives obviously rise to three other bases.
It worth to note that, as small computations reveal, over all these
additional bases, the structure coefficients of the partial composition
of $\Dias_\gamma$ can be negative or different from $1$. This observation
makes the $\Ksf$-basis even more particular and interesting. It has some
other properties, as next section will show.
\medskip

\subsubsection{Alternative presentation}%
\label{subsubsec:presentation_dias_gamma_alternative}
The $\Ksf$-basis introduced in the previous section leads to state a new
presentation for $\Dias_\gamma$ in the following way.
\medskip

For any integer $\gamma \geq 0$, let $\GDiasA_a$ and $\DDiasA_a$,
$a \in [\gamma]$, be the elements of $\OpLibre\left(\GenDias\right)(2)$
defined by
\begin{subequations}
\begin{equation}
    \GDiasA_a :=
    \begin{cases}
        \GDias_\gamma & \mbox{if } a = \gamma, \\
        \GDias_a - \GDias_{a + 1} & \mbox{otherwise},
    \end{cases}
\end{equation}
and
\begin{equation}
    \DDiasA_a :=
    \begin{cases}
        \DDias_\gamma & \mbox{if } a = \gamma, \\
        \DDias_a - \DDias_{a + 1} & \mbox{otherwise}.
    \end{cases}
\end{equation}
\end{subequations}
Then, since for all $a \in [\gamma]$ we have
\begin{subequations}
\begin{equation}
    \GDias_a = \sum_{a \leq b \in [\gamma]} \GDiasA_b
\end{equation}
and
\begin{equation}
    \DDias_a = \sum_{a \leq b \in [\gamma]} \DDiasA_b,
\end{equation}
\end{subequations}
by triangularity, the family
$\GenDias' := \{\GDiasA_a, \DDiasA_a : a \in [\gamma]\}$
forms a basis of $\OpLibre\left(\GenDias\right)(2)$ and then, generates
$\OpLibre\left(\GenDias\right)$ as an operad. This change of basis
from $\OpLibre\left(\GenDias\right)$ to $\OpLibre(\GenDias')$
comes from the change of basis from the usual basis of $\Dias_\gamma$
to the $\Ksf$-basis. Let us now express a presentation of $\Dias_\gamma$
through the family $\GenDias'$.
\medskip

\begin{Proposition} \label{prop:presentation_alternative_dias_gamma}
    For any integer $\gamma \geq 0$, the operad $\Dias_\gamma$ admits
    the following presentation. It is generated by $\GenDias'$ and its
    space of relations is $\RelLibre'_{\Dias_\gamma}$ is generated by
    \begin{subequations}
    \begin{equation} \label{equ:relation_dias_gamma_1_alternative}
        \GDiasA_a \circ_1 \DDiasA_{a'} -
        \DDiasA_{a'} \circ_2 \GDiasA_a,
        \qquad a, a' \in [\gamma],
    \end{equation}
    \begin{equation} \label{equ:relation_dias_gamma_2_alternative}
        \DDiasA_b \circ_1 \DDiasA_a, \qquad a < b \in [\gamma],
    \end{equation}
    \begin{equation} \label{equ:relation_dias_gamma_3_alternative}
        \GDiasA_b \circ_2 \GDiasA_a, \qquad a < b \in [\gamma],
    \end{equation}
    \begin{equation} \label{equ:relation_dias_gamma_4_alternative}
        \DDiasA_b \circ_1 \GDiasA_a, \qquad a < b \in [\gamma],
    \end{equation}
    \begin{equation} \label{equ:relation_dias_gamma_5_alternative}
        \GDiasA_b \circ_2 \DDiasA_a, \qquad a < b \in [\gamma],
    \end{equation}
    \begin{equation} \label{equ:relation_dias_gamma_6_alternative}
        \DDiasA_a \circ_1 \DDiasA_b
        - \DDiasA_b \circ_2 \DDiasA_a,
        \qquad a < b \in [\gamma],
    \end{equation}
    \begin{equation} \label{equ:relation_dias_gamma_7_alternative}
        \GDiasA_b \circ_1 \GDiasA_a
        - \GDiasA_a \circ_2 \GDiasA_b,
        \qquad a < b \in [\gamma],
    \end{equation}
    \begin{equation} \label{equ:relation_dias_gamma_8_alternative}
        \DDiasA_a \circ_1 \GDiasA_b
        - \DDiasA_a \circ_2 \DDiasA_b,
        \qquad a < b \in [\gamma],
    \end{equation}
    \begin{equation} \label{equ:relation_dias_gamma_9_alternative}
        \GDiasA_a \circ_1 \GDiasA_b
        - \GDiasA_a \circ_2 \DDiasA_b,
        \qquad a < b \in [\gamma],
    \end{equation}
    \begin{equation} \label{equ:relation_dias_gamma_10_alternative}
        \DDiasA_a \circ_1 \DDiasA_a -
        \left(\sum_{a \leq b \in [\gamma]} \DDiasA_a \circ_2 \DDiasA_b\right),
        \qquad a \in [\gamma],
    \end{equation}
    \begin{equation} \label{equ:relation_dias_gamma_11_alternative}
        \left(\sum_{a \leq b \in [\gamma]} \GDiasA_a \circ_1 \GDiasA_b\right)
        - \GDiasA_a \circ_2 \GDiasA_a,
        \qquad a \in [\gamma],
    \end{equation}
    \begin{equation} \label{equ:relation_dias_gamma_12_alternative}
        \DDiasA_a \circ_1 \GDiasA_a
        - \left(\sum_{a \leq b \in [\gamma]} \DDiasA_b \circ_2 \DDiasA_a\right),
        \qquad a \in [\gamma],
    \end{equation}
    \begin{equation} \label{equ:relation_dias_gamma_13_alternative}
        \left(\sum_{a \leq b \in [\gamma]} \GDiasA_b \circ_1 \GDiasA_a\right)
        - \GDiasA_a \circ_2 \DDiasA_a,
        \qquad a \in [\gamma].
    \end{equation}
    \end{subequations}
\end{Proposition}
\begin{proof}
    Let us show that $\RelLibre'_{\Dias_\gamma}$ is equal to the space
    of relations $\RelLibre_{\Dias_\gamma}$ of $\Dias_\gamma$ defined in
    the statement of Theorem~\ref{thm:presentation_dias_gamma}. First of
    all, recall that the map
    $\Mot_\gamma : \OpLibre\left(\GenDias\right) \to \Dias_\gamma$
    defined in Section~\ref{subsubsec:arbres_vers_mots} satisfies
    $\Mot_\gamma(\GDias_a) = 0a$ and $\Mot_\gamma(\DDias_a) = a0$ for
    all $a \in [\gamma]$. By Theorem~\ref{thm:presentation_dias_gamma},
    for any $x \in  \OpLibre\left(\GenDias\right)(3)$, $x$ is in
    $\RelLibre_{\Dias_\gamma}$ if and only if $\Mot_\gamma(x) = 0$.
    \smallskip

    Besides, by definition of $\GDiasA_a$, $\DDiasA_a$, $a \in [\gamma]$,
    and by making use of the $\Ksf$-basis of $\Dias_\gamma$, we have
    $\Mot_\gamma(\GDiasA_a) = \Ksf^{(\gamma)}_{0a}$ and
    $\Mot_\gamma(\DDiasA_a) = \Ksf^{(\gamma)}_{a0}$. By using the partial
    composition rules for $\Dias_\gamma$ over the $\Ksf$-basis of
    Theorem~\ref{thm:composition_base_K}, straightforward computations
    show that $\Mot_\gamma(x) = 0$ for all elements $x$
    among~\eqref{equ:relation_dias_gamma_1_alternative}---%
    \eqref{equ:relation_dias_gamma_13_alternative}.
    This implies that $\RelLibre'_{\Dias_\gamma}$ is a subspace
    of $\RelLibre_{\Dias_\gamma}$.
    \smallskip

    Now, one can observe that
    elements~\eqref{equ:relation_dias_gamma_1_alternative}---%
    \eqref{equ:relation_dias_gamma_13_alternative} are
    linearly independent. Then, $\RelLibre'_{\Dias_\gamma}$ has
    dimension $5 \gamma^2$ which is also, by
    Theorem~\ref{thm:presentation_dias_gamma}, the dimension of
    $\RelLibre_{\Dias_\gamma}$. Hence, $\RelLibre'_{\Dias_\gamma}$ and
    $\RelLibre_{\Dias_\gamma}$ are equal. The statement of the
    proposition follows.
\end{proof}
\medskip

Despite the apparent complexity of the presentation of $\Dias_\gamma$
exhibited by Proposition~\ref{prop:presentation_alternative_dias_gamma},
as we will see in Section~\ref{sec:dendr_gamma}, the Koszul dual
of $\Dias_\gamma$ computed from this presentation has a very simple
and manageable expression.
\medskip

\section{Pluriassociative algebras} \label{sec:algebres_dias_gamma}
We now focus on algebras over $\gamma$-pluriassociative operads.
For this purpose, we construct free $\Dias_\gamma$-algebras over one
generator, and define and study two notions of units for
$\Dias_\gamma$-algebras. We end this section by introducing a convenient
way to define $\Dias_\gamma$-algebras and give several examples of such
algebras.
\medskip

\subsection{Category of pluriassociative algebras and free objects}
Let us study the category of $\Dias_\gamma$-algebras and the units for
algebras in this category.
\medskip

\subsubsection{Pluriassociative algebras}
We call {\em $\gamma$-pluriassociative algebra} any
$\Dias_\gamma$-algebra. From the presentation of $\Dias_\gamma$ provided
by Theorem~\ref{thm:presentation_dias_gamma}, any
$\gamma$-pluriassociative algebra is a vector space endowed with linear
operations $\GDias_a, \DDias_a$, $a \in [\gamma]$, satisfying
the relations encoded by~\eqref{equ:relation_dias_gamma_1_concise}---%
\eqref{equ:relation_dias_gamma_5_concise}.
\medskip

\subsubsection{General definitions}
Let $\Pca$ be a $\gamma$-pluriassociative algebra. We say that $\Pca$
is {\em commutative} if for all $x, y \in \Pca$ and $a \in [\gamma]$,
$x \GDias_a y = y \DDias_a x$. Besides, $\Pca$ is {\em pure} for all
$a, a' \in [\gamma]$, $a \ne a'$ implies $\GDias_a \ne \GDias_{a'}$ and
$\DDias_a \ne \DDias_{a'}$.
\medskip

Given a subset $C$ of $[\gamma]$, one can keep on the vector space $\Pca$
only the operations $\GDias_a$ and $\DDias_a$ such that $a \in C$.
By renumbering the indexes of these operations from $1$ to $\# C$ by
respecting their former relative numbering, we obtain a
$\# C$-pluriassociative algebra. We call it the
{\em $\# C$-pluriassociative subalgebra induced by} $C$ of $\Pca$.
\medskip

\subsubsection{Free pluriassociative algebras}
\label{subsubsec:algebre_dias_gamma_libre}
Recall that $\AlgLibre_{\Dias_\gamma}$ denotes the free
$\Dias_\gamma$-algebra over one generator. By definition,
$\AlgLibre_{\Dias_\gamma}$ is the linear span of the set of the words on
$\{0\} \cup [\gamma]$ with exactly one occurrence of~$0$. Let us endow
this space with the linear operations
\begin{equation}
    \GDias_a, \DDias_a :
    \AlgLibre_{\Dias_\gamma} \otimes \AlgLibre_{\Dias_\gamma}
    \to \AlgLibre_{\Dias_\gamma},
    \qquad
    a \in [\gamma],
\end{equation}
satisfying, for any such words $u$ and $v$,
\begin{subequations}
\begin{equation}
    u \GDias_a v := u \: \Augm_a(v)
\end{equation}
and
\begin{equation}
    u \DDias_a v := \Augm_a(u) \: v,
\end{equation}
\end{subequations}
where $\Augm_a(u)$ (resp. $\Augm_a(v)$) is the word obtained by replacing
in $u$ (resp. $v$) any occurrence of a letter smaller than $a$ by $a$.
\medskip

\begin{Proposition} \label{prop:algebre_dias_gamma_libre}
    For any integer $\gamma \geq 0$, the vector space
    $\AlgLibre_{\Dias_\gamma}$ of nonempty words on $\{0\} \cup [\gamma]$
    containing exactly one occurrence of $0$ endowed with the operations
    $\GDias_a$, $\DDias_a$, $a \in [\gamma]$, is the free
    $\gamma$-pluriassociative algebra over one generator.
\end{Proposition}
\begin{proof}
    The fact that $\AlgLibre_{\Dias_\gamma}$ is the stated vector space
    is a consequence of the description of the elements of $\Dias_\gamma$
    provided by Proposition~\ref{prop:elements_dias_gamma}. Since
    $\Dias_\gamma$ is by definition the suboperad of $\T \Mca_\gamma$
    generated by $\{0a, a0 : a \in [\gamma]\}$, $\AlgLibre_{\Dias_\gamma}$
    is endowed with $2\gamma$ binary operations where any generator
    $0a$ (resp. $a0$) gives rise to the operation $\GDias_a$ (resp.
    $\DDias_a$) of $\AlgLibre_{\Dias_\gamma}$. Moreover, by making use
    of the realization of $\Dias_\gamma$, we have for all
    $u, v \in \AlgLibre_{\Dias_\gamma}$ and $a \in [\gamma]$,
    \begin{subequations}
    \begin{equation}
        u \GDias_a v
        = (u \otimes v) \Action 0a
        = (0a \circ_2 v) \circ_1 u = u \: \Augm_a(v)
    \end{equation}
    and
    \begin{equation}
        u \DDias_a v
        = (u \otimes v) \Action a0
        = (a0 \circ_2 v) \circ_1 u = \Augm_a(u) \: v.
    \end{equation}
    \end{subequations}
\end{proof}
\medskip

One has for instance in $\AlgLibre_{\Dias_4}$,
\begin{equation}
    \textcolor{Bleu}{101241} \GDias_2 \textcolor{Rouge}{203} =
    \textcolor{Bleu}{101241}\textcolor{Rouge}{2{\bf 2}3}
\end{equation}
and
\begin{equation}
    \textcolor{Bleu}{101241} \DDias_3 \textcolor{Rouge}{203} =
    \textcolor{Bleu}{{\bf 3333}4{\bf 3}}\textcolor{Rouge}{203}.
\end{equation}
\medskip

\subsection{Bar and wire-units}
Loday has defined in~\cite{Lod01} some notions of units in diassociative
algebras. We generalize here these definitions to the context of
$\gamma$-pluriassociative algebras.
\medskip

\subsubsection{Bar-units}
Let $\Pca$ be a $\gamma$-pluriassociative algebra and $a \in [\gamma]$.
We say that an element $e$ of $\Pca$ is an {\em $a$-bar-unit}, or
simply a {\em bar-unit} when taking into account the value of $a$
is not necessary, of $\Pca$ if for all $x \in \Pca$,
\begin{equation}
    x \GDias_a e = x = e \DDias_a x.
\end{equation}
As we shall see below, a $\gamma$-pluriassociative algebra can have,
for a given $a \in [\gamma]$, several $a$-bar-units. The {\em $a$-halo}
of $\Pca$, denoted by $\Halo_a(\Pca)$, is the set of the $a$-bar-units
of $\Pca$.
\medskip

\subsubsection{Wire-units}
Let $\Pca$ be a $\gamma$-pluriassociative algebra and $a \in [\gamma]$.
We say that an element $e$ of $\Pca$ is an {\em $a$-wire-unit}, or
simply a {\em wire-unit} when taking into account the value of $a$
is not necessary, of $\Pca$ if for all $x \in \Pca$,
\begin{equation}
    e \GDias_a x = x = x \DDias_a e.
\end{equation}
As shows the following proposition, the presence of a wire-unit in
$\Pca$ has some implications.
\medskip

\begin{Proposition} \label{prop:unite_fil}
    Let $\gamma \geq 0$ be an integer and $\Pca$ be a
    $\gamma$-pluriassociative algebra admitting a $b$-wire-unit $e$
    for a $b \in [\gamma]$. Then
    \begin{enumerate}[label={\it (\roman*)}]
        \item \label{prop:unite_fil_1}
        for all $a \in [b]$, the operations $\GDias_a$,
        $\GDias_b$, $\DDias_a$, and $\DDias_b$ of $\Pca$ are equal;
        \item \label{prop:unite_fil_2}
        $e$ is also an $a$-wire-unit for all $a \in [b]$;
        \item \label{prop:unite_fil_3}
        $e$ is the only wire-unit of $\Pca$;
        \item \label{prop:unite_fil_4}
        if $e'$ is an $a$-bar unit for a $a \in [b]$, then $e' = e$.
    \end{enumerate}
\end{Proposition}
\begin{proof}
    Let us show part~\ref{prop:unite_fil_1}. By
    Relation~\eqref{equ:relation_dias_gamma_4_concise} of
    $\gamma$-pluriassociative algebras and by the fact that $e$ is a
    $b$-wire-unit of $\Pca$, we have for all elements $y$ and $z$
    of $\Pca$ and all $a \in [b]$,
    \begin{equation}
        y \GDias_a z = e \GDias_b (y \GDias_a z) =
        e \GDias_b (y \DDias_a z) = y \DDias_a z.
    \end{equation}
    Thus, the operations $\GDias_a$ and $\DDias_a$ of $\Pca$ are equal.
    Moreover, for the same reasons, we have
    \begin{equation}
        y \GDias_a z = e \GDias_b (y \GDias_a z) =
        (e \GDias_b y) \GDias_b z = y \GDias_b z.
    \end{equation}
    Then, the operations $\GDias_a$ and $\GDias_b$ of $\Pca$ are equal,
    whence~\ref{prop:unite_fil_1}.
    \smallskip

    Now, by~\ref{prop:unite_fil_1} and by the fact that $e$ is a
    $b$-wire-unit, we have for all elements $x$ of $\Pca$ and all
    $a \in [b]$,
    \begin{equation}
        e \GDias_a x = e \GDias_b x = x = x \DDias_b e = x \DDias_a e,
    \end{equation}
    showing~\ref{prop:unite_fil_2}.
    \smallskip

    To prove~\ref{prop:unite_fil_3}, assume that $e'$ is a $b'$-wire-unit
    of $\Pca$ for a $b' \in [\gamma]$. By~\ref{prop:unite_fil_1} and by
    the fact that $e$ is a $b$-wire-unit, one has
    \begin{equation}
        e = e \DDias_{b'} e' = e \GDias_b e' = e',
    \end{equation}
    showing~\ref{prop:unite_fil_3}.
    \smallskip

    To establish~\ref{prop:unite_fil_4}, let us first prove that $e$ is
    a $b$-bar-unit. By~\ref{prop:unite_fil_1} and by the fact that $e$
    is a $b$-wire-unit, we have for all elements $x$ of $\Pca$,
    \begin{equation}
        e \DDias_b x = e \GDias_b x = x = x \DDias_b e = x \GDias_b e.
    \end{equation}
    Now, since $e'$ is an $a$-bar-unit for an $a \in [b]$,
    by~\ref{prop:unite_fil_1} and by the fact that $e$ is a $b$-wire-unit,
    \begin{equation}
        e = e' \DDias_a e = e' \DDias_b e = e'.
    \end{equation}
    This shows~\ref{prop:unite_fil_4}.
\end{proof}
\medskip

Relying on Proposition~\ref{prop:unite_fil}, we define the {\em height}
of a $\gamma$-pluriassociative algebra $\Pca$ as zero if $\Pca$ has no
wire-unit, otherwise as the greatest integer $h \in [\gamma]$ such that
the unique wire-unit $e$ of $\Pca$ is a $h$-wire-unit. Observe that any
pure $\gamma$-pluriassociative algebra has height $0$ or $1$.
\medskip

\subsection{Construction of pluriassociative algebras}
We now present a general way to construct $\gamma$-pluriassociative
algebras. Our construction is a natural generalization of some
constructions introduced by Loday~\cite{Lod01} in the context of
diassociative algebras. We introduce in this section new algebraic
structures, the so-called $\gamma$-multiprojection algebras, which are
inputs of our construction.
\medskip

\subsubsection{Multiassociative algebras}%
\label{subsubsec:algebres_multiassociatives}
For any integer $\gamma \geq 0$, a {\em $\gamma$-multiassociative algebra}
is a vector space $\Mca$ endowed with linear operations
\begin{equation}
    \MAs_a : \Mca \otimes \Mca \to \Mca,
    \qquad a \in [\gamma],
\end{equation}
satisfying, for all $x, y, z \in \Mca$, the relations
\begin{equation} \label{equ:relation_algebre_multiassoc}
    (x \MAs_a y) \MAs_b z =
    (x \MAs_b y) \MAs_{a'} z =
    x \MAs_{a''} (y \MAs_b z) =
    x \MAs_b (y \MAs_{a'''} z),
    \qquad a, a', a'', a''' \leq b \in [\gamma].
\end{equation}
These algebras are obvious generalizations of associative algebras since
all of its operations are associative. Observe that
by~\eqref{equ:relation_algebre_multiassoc}, all bracketings of
an expression involving elements of a $\gamma$-multiassociative algebra
and some of its operations are equal. Then, since the bracketings of such
expressions are not significant, we shall denote these without parenthesis.
In Section~\ref{sec:as_gamma} we will study the operads governing
these for a very specific purpose.
\medskip

If $\Mca_1$ and $\Mca_2$ are two $\gamma$-multiassociative algebras,
a linear map $\phi : \Mca_1 \to \Mca_2$ is a
{\em $\gamma$-multiassociative algebra morphism} if it commutes with
the operations of $\Mca_1$ and $\Mca_2$. We say that $\Mca$ is
{\em commutative} when all operations of $\Mca$ are commutative.
Besides, for an $a \in [\gamma]$, an element $\Unite$ of $\Mca$ is an
{\em $a$-unit}, or simply a {\em unit} when taking into account the
value of $a$ is not necessary, of $\Mca$ if for all $x \in \Mca$,
$\Unite \MAs_a x = x = x \MAs_a \Unite$. When $\Mca$ admits a unit,
we say that $\Mca$ is {\em unital}. As shows the following proposition,
the presence of a unit in $\Mca$ has some implications.
\medskip

\begin{Proposition} \label{prop:unites_algebre_multiassociative}
    Let $\gamma \geq 0$ be an integer and $\Mca$ be a
    $\gamma$-multiassociative algebra admitting a $b$-unit $\Unite$
    for a $b \in [\gamma]$. Then
    \begin{enumerate}[label={\it (\roman*)}]
        \item \label{prop:unites_algebre_multiassociative_1}
        for all $a \in [b]$, the operations $\MAs_a$ and
        $\MAs_b$ of $\Mca$ are equal;
        \item \label{prop:unites_algebre_multiassociative_2}
        $\Unite$ is also an $a$-unit for all $a \in [b]$;
        \item \label{prop:unites_algebre_multiassociative_3}
        $\Unite$ is the only unit of $\Mca$.
    \end{enumerate}
\end{Proposition}
\begin{proof}
    By Relation~\eqref{equ:relation_algebre_multiassoc} of
    $\gamma$-multiassociative algebras and by the fact that $\Unite$ is
    a $b$-unit of $\Mca$, we have for all elements $y$ and $z$ of $\Mca$
    and all $a \in [b]$,
    \begin{equation}
        y \MAs_a z
            = y \MAs_a z \MAs_b \Unite
            = y \MAs_b z \MAs_b \Unite
            = y \MAs_b z.
    \end{equation}
    Therefore, $\MAs_a = \MAs_b$,
    showing~\ref{prop:unites_algebre_multiassociative_1}.
    \smallskip

    Now, by~\ref{prop:unites_algebre_multiassociative_1} and by the fact
    that $\Unite$ is a $b$-unit, we have for all elements $x$ of $\Mca$
    and all $a \in [b]$,
    \begin{equation}
        \Unite \MAs_a x
            = \Unite \MAs_b x
            = x
            = x \MAs_b \Unite = x \MAs_a \Unite,
    \end{equation}
    showing~\ref{prop:unites_algebre_multiassociative_2}.
    \smallskip

    To prove~\ref{prop:unites_algebre_multiassociative_3}, assume
    that $\Unite'$ is a $b'$-unit of $\Mca$ for a $b' \in [\gamma]$.
    By~\ref{prop:unites_algebre_multiassociative_1} and by the fact that
    $\Unite$ is a $b$-unit, one has
    \begin{equation}
        \Unite
            = \Unite \MAs_{b'} \Unite'
            = \Unite \MAs_b \Unite'
            = \Unite',
    \end{equation}
    establishing~\ref{prop:unites_algebre_multiassociative_3}.
\end{proof}
\medskip

Relying on Proposition~\ref{prop:unites_algebre_multiassociative},
similarly to the case of $\gamma$-pluriassociative algebras, we define
the {\em height} of a $\gamma$-multiassociative algebra $\Mca$ as zero
if $\Mca$ has no unit, otherwise as the greatest integer $h \in [\gamma]$
such that the unit $\Unite$ of $\Mca$ is an $h$-unit.
\medskip

\subsubsection{Multiprojection algebras}
We call {\em $\gamma$-multiprojection algebra} any $\gamma$-multiassociative
algebra $\Mca$ endowed with endomorphisms
\begin{equation}
    \pi_a : \Mca \to \Mca,
    \qquad a \in [\gamma],
\end{equation}
satisfying
\begin{equation} \label{equ:relation_algebre_multiproj}
    \pi_a \circ \pi_{a'} = \pi_{a \Max a'},
    \qquad a, a' \in [\gamma].
\end{equation}
\medskip

By extension, the {\em height} of $\Mca$ is its height as a
$\gamma$-multiassociative algebra. We say that $\Mca$ is {\em unital} as
a $\gamma$-multiprojection algebra if $\Mca$ is unital as a
$\gamma$-multiassociative algebra and its only, by
Proposition~\ref{prop:unites_algebre_multiassociative}, unit $\Unite$
satisfies $\pi_a(\Unite) = \Unite$ for all $a \in [h]$ where $h$ is the
height of~$\Mca$.
\medskip

\subsubsection{From multiprojection algebras to pluriassociative algebras}
Next result describes how to construct $\gamma$-pluriassociative
algebras from $\gamma$-multiprojection algebras.
\medskip

\begin{Theoreme} \label{thm:algebre_multiprojections_vers_dias_gamma}
    For any integer $\gamma \geq 0$ and any $\gamma$-multiprojection
    algebra $\Mca$, the vector space $\Mca$ endowed with binary linear
    operations $\GDias_a$, $\DDias_a$, $a \in [\gamma]$, defined for all
    $x, y \in \Mca$ by
    \begin{subequations}
    \begin{equation}
        x \GDias_a y := x \MAs_a \pi_a(y)
    \end{equation}
    and
    \begin{equation}
        x \DDias_a y := \pi_a(x) \MAs_a y,
    \end{equation}
    \end{subequations}
    where the $\MAs_a$, $a \in [\gamma]$, are the operations of $\Mca$ and
    the $\pi_a$, $a \in [\gamma]$, are its endomorphisms, is a
    $\gamma$-pluriassociative algebra, denoted by $\MProjVersPluri(\Mca)$.
\end{Theoreme}
\begin{proof}
    This is a verification of the relations of $\gamma$-pluriassociative
    algebras in $\MProjVersPluri(\Mca)$. Let $x$, $y$, and $z$ be three
    elements of $\MProjVersPluri(\Mca)$ and  $a, a' \in [\gamma]$.
    \smallskip

    By~\eqref{equ:relation_algebre_multiassoc}, we have
    \begin{equation}
        (x \DDias_{a'} y) \GDias_a z =
        \pi_{a'}(x) \MAs_{a'} y \MAs_a \pi_a(z) =
        x \DDias_{a'} (y \GDias_a z),
    \end{equation}
    showing that~\eqref{equ:relation_dias_gamma_1_concise} is
    satisfied in $\MProjVersPluri(\Mca)$.
    \smallskip

    Moreover, by~\eqref{equ:relation_algebre_multiassoc}
    and~\eqref{equ:relation_algebre_multiproj}, we have
    \begin{equation} \begin{split}
        x \GDias_a (y \DDias_{a'} z)
            & = x \MAs_a \pi_a(\pi_{a'}(y) \MAs_{a'} z) \\
            & = x \MAs_a \pi_{a \Max a'}(y) \MAs_{a'} \pi_a(z) \\
            & = x \MAs_{a \Max a'} \pi_{a \Max a'}(y) \MAs_a \pi_a(z) \\
            & = (x \GDias_{a \Max a'} y) \GDias_a z,
    \end{split} \end{equation}
    so that~\eqref{equ:relation_dias_gamma_2_concise},
    and for the same
    reasons~\eqref{equ:relation_dias_gamma_3_concise},
    check out in $\MProjVersPluri(\Mca)$.
    \smallskip

    Finally, again by~\eqref{equ:relation_algebre_multiassoc}
    and~\eqref{equ:relation_algebre_multiproj}, we have
    \begin{equation} \begin{split}
        x \GDias_a (y \GDias_{a'} z)
            & = x \MAs_a \pi_a(y \MAs_{a'} \pi_{a'}(z)) \\
            & = x \MAs_a \pi_a(y) \MAs_{a'} \pi_{a \Max a'}(z) \\
            & = x \MAs_a \pi_a(y) \MAs_{a \Max a'} \pi_{a \Max a'}(z) \\
            & = (x \GDias_a y) \GDias_{a \Max a'} z,
    \end{split} \end{equation}
    showing that~\eqref{equ:relation_dias_gamma_4_concise},
    and for the same
    reasons~\eqref{equ:relation_dias_gamma_5_concise}, are
    satisfied in $\MProjVersPluri(\Mca)$.
\end{proof}
\medskip

When $\Mca$ is commutative, since for all $x, y \in \MProjVersPluri(\Mca)$
and $a \in [\gamma]$,
\begin{equation}
    x \GDias_a y = x \MAs_a \pi_a(y) =
    \pi_a(y) \MAs_a x = y \DDias_a x,
\end{equation}
it appears that $\MProjVersPluri(\Mca)$ is a commutative
$\gamma$-pluriassociative algebra.
\medskip

When $\Mca$ is unital, $\MProjVersPluri(\Mca)$ has several properties,
summarized in the next proposition.
\medskip

\begin{Proposition}%
\label{prop:algebre_multiprojections_vers_dias_gamma_proprietes}
    Let $\gamma \geq 0$ be an integer, $\Mca$ be a unital
    $\gamma$-multiprojection algebra of height~$h$.
    Then, by denoting by $\Unite$ the unit of $\Mca$ and by $\pi_a$,
    $a \in [\gamma]$, its endomorphisms,
    \begin{enumerate}[label={\it (\roman*)}]
        \item
        \label{item:algebre_multiprojections_vers_dias_gamma_proprietes_1}
        for any $a \in [h]$, $\Unite$ is an $a$-bar-unit of
        $\MProjVersPluri(\Mca)$;
        \item
        \label{item:algebre_multiprojections_vers_dias_gamma_proprietes_2}
        for any $a \leq b \in [h]$,
        $\Halo_a(\MProjVersPluri(\Mca))$ is a subset of
        $\Halo_b(\MProjVersPluri(\Mca))$;
        \item
        \label{item:algebre_multiprojections_vers_dias_gamma_proprietes_3}
        for any $a \in [h]$, the linear span of $\Halo_a(\MProjVersPluri(\Mca))$
        forms an $h\!-\!a\!+\!1$-pluriassociative subalgebra of the
        $h\!-\!a\!+\!1$-pluriassociative subalgebra of
        $\MProjVersPluri(\Mca)$ induced by $[a, h]$;
        \item
        \label{item:algebre_multiprojections_vers_dias_gamma_proprietes_4}
        for any $a \in [h]$, $\pi_a$
        is the identity map if and only if $\Unite$ is an $a$-wire-unit
        of $\MProjVersPluri(\Mca)$.
    \end{enumerate}
\end{Proposition}
\begin{proof}
    Let us denote by $\MAs_a$, $a \in [\gamma]$, the operations of
    $\Mca$.
    \smallskip

    Since $\Unite$ is an $h$-unit of $\Mca$, for all elements $x$ of
    $\MProjVersPluri(\Mca)$ and all $a \in [h]$,
    \begin{equation}
        x \GDias_a \Unite
            = x \MAs_a \pi_a(\Unite) = x \MAs_a \Unite
            = x =
            \Unite \MAs_a x = \pi_a(\Unite) \MAs_a x
        = \Unite \DDias_a x,
    \end{equation}
    showing~\ref{item:algebre_multiprojections_vers_dias_gamma_proprietes_1}.
    \smallskip

    Assume that $e$ is an element of $\Halo_a(\MProjVersPluri(\Mca))$
    for an $a \in [h]$, that is, $e$ is an $a$-bar-unit of
    $\MProjVersPluri(\Mca)$. Then, for all elements $x$ of
    $\MProjVersPluri(\Mca)$,
    \begin{equation}
        x \GDias_a e
            = x \MAs_a \pi_a(e) = x = \pi_a(e) \MAs_a x
        = e \DDias_a x,
    \end{equation}
    showing that $\pi_a(e)$ is the unit for the operation $\MAs_a$ on
    $\MProjVersPluri(\Mca)$
    and therefore, $\pi_a(e) = \Unite$. Since $\Mca$ is unital,
    we have $\pi_b(\Unite) = \Unite$ for all $b \in [h]$. Hence,
    and by~\eqref{equ:relation_algebre_multiproj},
    for all $a \leq b \in [h]$,
    \begin{equation}
        \pi_b(e) = \pi_b(\pi_a(e)) = \pi_b(\Unite) = \Unite.
    \end{equation}
    Then, for all elements $x$ of $\MProjVersPluri(\Mca)$ and all
    $a \leq b \in [h]$,
    \begin{equation}
        x \GDias_b e
            = x \MAs_b \pi_b(e) = x \MAs_b \Unite = x
            = \Unite \MAs_b x = \pi_b(e) \MAs_b x
        = e \DDias_b x,
    \end{equation}
    showing that $e$ is also a $b$-bar-unit of $\MProjVersPluri(\Mca)$,
    whence~\ref{item:algebre_multiprojections_vers_dias_gamma_proprietes_2}.
    \smallskip

    Let $a \in [\gamma]$ and $e$ and $e'$ be elements of
    $\Halo_a(\MProjVersPluri(\Mca))$.
    By~\ref{item:algebre_multiprojections_vers_dias_gamma_proprietes_2},
    $e$ and $e'$ are $b$-bar-units of $\MProjVersPluri(\Mca)$ for all
    $a \leq b \in [h]$ and hence,
    \begin{equation}
        e \GDias_b e' = e = e' \DDias_b e.
    \end{equation}
    Therefore, the linear span of $\Halo_a(\MProjVersPluri(\Mca))$ is
    stable for the operations $\GDias_b$ and $\DDias_b$. This
    implies~\ref{item:algebre_multiprojections_vers_dias_gamma_proprietes_3}.
    \smallskip

    Finally, assume that $\pi_a$ is the identity map for an $a \in [h]$.
    Then, for all elements $x$ of $\MProjVersPluri(\Mca)$,
    \begin{equation}
        \Unite \GDias_a x
            = \Unite \MAs_a \pi_a(x) = \Unite \MAs_a x
            = x
            = x \MAs_a \Unite = \pi_a(x) \MAs_a \Unite
        = x \DDias_a \Unite,
    \end{equation}
    showing that $\Unite$ is an $a$-wire unit of $\MProjVersPluri(\Mca)$.
    Conversely, if $\Unite$ is an $a$-wire unit of $\MProjVersPluri(\Mca)$,
    for all elements $x$ of $\MProjVersPluri(\Mca)$, the relations
    $\Unite \GDias_a x = x = x \DDias_a \Unite$ imply
    $\Unite \MAs_a \pi_a(x) = x = \pi_a(x) \MAs_a \Unite$ and hence,
    $\pi_a(x) = x$. This
    shows~\ref{item:algebre_multiprojections_vers_dias_gamma_proprietes_4}.
\end{proof}
\medskip

\subsubsection{Examples of constructions of pluriassociative algebras}
The construction $\MProjVersPluri$ of
Theorem~\ref{thm:algebre_multiprojections_vers_dias_gamma} allows
to build several $\gamma$-pluriassociative algebras. Here follows few
examples.
\medskip

\paragraph{\bf The $\gamma$-pluriassociative algebra of positive integers}
Let $\gamma \geq 1$ be an integer and consider the vector space $\AlgPos$
of positive integers, endowed with the operations $\MAs_a$, $a \in [\gamma]$,
all equal to the operation $\Max$ extended by linearity and with the
endomorphisms $\pi_a$, $a \in [\gamma]$, linearly defined for any
positive integer $x$ by $\pi_a(x) := a \Max x$. Then, $\AlgPos$ is a
non-unital $\gamma$-multiprojection algebra. By
Theorem~\ref{thm:algebre_multiprojections_vers_dias_gamma},
$\MProjVersPluri(\AlgPos)$ is a $\gamma$-pluriassociative algebra. We
have for instance
\begin{equation}
    \textcolor{Bleu}{2} \GDias_3 \textcolor{Rouge}{5} =
    \textcolor{Rouge}{5},
\end{equation}
and
\begin{equation}
    \textcolor{Bleu}{1} \DDias_3 \textcolor{Rouge}{2} = 3.
\end{equation}
We can observe that $\MProjVersPluri(\AlgPos)$ is commutative, pure, and
its $1$-halo is $\{1\}$. Moreover, when
$\gamma \geq 2$, $\MProjVersPluri(\AlgPos)$ has no wire-unit and no
$a$-bar-unit for $a \geq 2 \in [\gamma]$. This example is important
because it provides a counterexample
for~\ref{item:algebre_multiprojections_vers_dias_gamma_proprietes_2} of
Proposition~\ref{prop:algebre_multiprojections_vers_dias_gamma_proprietes}
in the case when the construction $\MProjVersPluri$ is applied
to a non-unital $\gamma$-multiprojection algebra.
\medskip

\paragraph{\bf The $\gamma$-pluriassociative algebra of finite sets}
Let $\gamma \geq 1$ be an integer and consider the vector space $\AlgEns$
of finite sets of positive integers, endowed with the operations $\MAs_a$,
$a \in [\gamma]$, all equal to the union operation $\cup$ extended by
linearity and with the endomorphisms $\pi_a$, $a \in [\gamma]$, linearly
defined for any finite set of positive integers $x$ by
$\pi_a(x) := x \cap [a, \gamma]$. Then, $\AlgEns$ is a $\gamma$-multiprojection
algebra. By Theorem~\ref{thm:algebre_multiprojections_vers_dias_gamma},
$\MProjVersPluri(\AlgEns)$ is a $\gamma$-pluriassociative algebra. We
have for instance
\begin{equation}
    \{\textcolor{Bleu}{2}, \textcolor{Bleu}{4}\}
    \GDias_3
    \{\textcolor{Rouge}{1}, \textcolor{Rouge}{3}, \textcolor{Rouge}{5}\}
    = \{\textcolor{Bleu}{2}, \textcolor{Rouge}{3},
    \textcolor{Bleu}{4}, \textcolor{Rouge}{5}\},
\end{equation}
and
\begin{equation}
    \{\textcolor{Bleu}{1}, \textcolor{Bleu}{2},
    \textcolor{Bleu}{4}\}
    \DDias_3
    \{\textcolor{Rouge}{1}, \textcolor{Rouge}{3}, \textcolor{Rouge}{5}\}
    = \{\textcolor{Rouge}{1}, \textcolor{Rouge}{3},
    \textcolor{Bleu}{4}, \textcolor{Rouge}{5}\}.
\end{equation}
We can observe that $\MProjVersPluri(\AlgEns)$ is commutative and pure.
Moreover, $\emptyset$ is a $1$-wire-unit of $\MProjVersPluri(\AlgEns)$ and,
by Proposition~\ref{prop:unite_fil}, it is its only wire-unit. Therefore,
$\MProjVersPluri(\AlgEns)$ has height $1$. Observe that for any
$a \in [\gamma]$, the $a$-halo of $\MProjVersPluri(\AlgEns)$ consists
in the subsets of $[a - 1]$. Besides, since $\AlgEns$ is a unital
$\gamma$-multiprojection algebra, $\MProjVersPluri(\AlgEns)$ satisfies
all properties exhibited by
Proposition~\ref{prop:algebre_multiprojections_vers_dias_gamma_proprietes}.
\medskip

\paragraph{\bf The $\gamma$-pluriassociative algebra of words}
Let $\gamma \geq 1$ be an integer and consider the vector space $\AlgMots$
of the words of positive integers. Let us endow $\AlgMots$ with the
operations $\MAs_a$, $a \in [\gamma]$, all equal to the concatenation
operation extended by linearity and with the endomorphisms $\pi_a$,
$a \in [\gamma]$, where for any word $x$ of positive integers, $\pi_a(x)$
is the longest subword of $x$ consisting in letters greater than or equal
to $a$. Then, $\AlgMots$ is a $\gamma$-multiprojection algebra. By
Theorem~\ref{thm:algebre_multiprojections_vers_dias_gamma},
$\MProjVersPluri(\AlgMots)$ is a $\gamma$-pluriassociative algebra. We
have for instance
\begin{equation}
    \textcolor{Bleu}{412} \GDias_3
    \textcolor{Rouge}{14231} =
    \textcolor{Bleu}{412}\textcolor{Rouge}{43},
\end{equation}
and
\begin{equation}
    \textcolor{Bleu}{11} \DDias_2 \textcolor{Rouge}{323} =
    \textcolor{Rouge}{323}.
\end{equation}
We can observe that $\MProjVersPluri(\AlgMots)$ is not commutative and is pure.
Moreover, $\epsilon$ is a $1$-wire-unit of $\MProjVersPluri(\AlgMots)$ and
by Proposition~\ref{prop:unite_fil}, it is its only wire-unit. Therefore,
$\MProjVersPluri(\AlgMots)$ has height $1$. Observe that for any
$a \in [\gamma]$, the $a$-halo of $\MProjVersPluri(\AlgMots)$ consists in
the words on the alphabet $[a - 1]$. Besides, since $\AlgMots$ is a unital
$\gamma$-multiprojection algebra, $\MProjVersPluri(\AlgMots)$ satisfies
all properties exhibited by
Proposition~\ref{prop:algebre_multiprojections_vers_dias_gamma_proprietes}.
\medskip

The $\gamma$-pluriassociative algebras $\MProjVersPluri(\AlgEns)$ and
$\MProjVersPluri(\AlgMots)$ are related in the following way. Let
$I_{\rm com}$ be the subspace of $\MProjVersPluri(\AlgMots)$ generated
by the $x - x'$ where $x$ and $x'$ are words of positive integers and
have the same commutative image. Since $I_{\rm com}$ is a
$\gamma$-pluriassociative algebra ideal of $\MProjVersPluri(\AlgMots)$,
one can consider the quotient $\gamma$-pluriassociative algebra
$\AlgMotsCom := \MProjVersPluri(\AlgMots)/_{I_{\rm com}}$. Its elements
can be seen as commutative words of positive integers.
\medskip

Moreover, let $I_{\rm occ}$ be the subspace of
$\MProjVersPluri(\AlgMotsCom)$ generated by the $x - x'$ where $x$ and $x'$
are commutative words of positive integers and for any letter
$a \in [\gamma]$, $a$ appears in $x$ if and only if $a$ appears in $x'$.
Since $I_{\rm occ}$ is a $\gamma$-pluriassociative algebra ideal of
$\MProjVersPluri(\AlgMotsCom)$, one can consider the quotient
$\gamma$-pluriassociative algebra
$\MProjVersPluri(\AlgMotsCom)/_{I_{\rm occ}}$. Its elements can be seen
as finite subsets of positive integers and we observe that
$\MProjVersPluri(\AlgMotsCom)/_{I_{\rm occ}} = \MProjVersPluri(\AlgEns)$.
\medskip

\paragraph{\bf The $\gamma$-pluriassociative algebra of marked words}
Let $\gamma \geq 1$ be an integer and consider the vector space
$\AlgMotsMarq$ of the words of positive integers where letters can be
marked or not, with at least one occurrence of a marked letter. We
denote by $\bar a$ any {\em marked letter} $a$ and we say that the
{\em value} of $\bar a$ is $a$. Let us endow $\AlgMotsMarq$ with the
linear operations $\MAs_a$, $a \in [\gamma]$, where for all words $u$
and $v$ of $\AlgMotsMarq$, $u \MAs_a v$ is obtained by concatenating $u$
and $v$, and by replacing therein all marked letters by $\bar c$ where
$c := \max(u) \Max a \Max \max(v)$ where $\max(u)$ (resp. $\max(v)$)
denotes the greatest value among the marked letters of $u$ (resp. $v$).
For instance,
\begin{equation}
    \textcolor{Bleu}{2 \bar 1 3 1 \bar 3}
    \MAs_2
    \textcolor{Rouge}{3 \bar 4 \bar 1 2 1}
    =
    \textcolor{Bleu}{2 {\bf \bar 4} 3 1 {\bf \bar 4}}
    \textcolor{Rouge}{3 \bar 4 {\bf \bar 4} 2 1},
\end{equation}
and
\begin{equation}
    \textcolor{Bleu}{\bar 2 1 1 \bar 1}
    \MAs_3
    \textcolor{Rouge}{3 4 \bar 2} =
    \textcolor{Bleu}{{\bf \bar 3} 1 1 {\bf \bar 3}}
    \textcolor{Rouge}{3 4 {\bf \bar 3}}.
\end{equation}
We also endow $\AlgMotsMarq$ with the endomorphisms $\pi_a$,
$a \in [\gamma]$, where for any word $u$ of $\AlgMotsMarq$, $\pi_a(u)$
is obtained by replacing in $u$ any occurrence of a nonmarked letter
smaller than $a$ by $a$. For instance,
\begin{equation}
    \pi_3\left(\textcolor{Rouge}{2} \textcolor{Bleu}{\bar 2}
        \textcolor{Rouge}{14} \textcolor{Bleu}{\bar 4}
        \textcolor{Rouge}{3} \textcolor{Bleu}{\bar 5}\right)
    = \textcolor{Rouge}{\bf 3} \textcolor{Bleu}{\bar 2}
        \textcolor{Rouge}{{\bf 3} 4} \textcolor{Bleu}{\bar 4}
        \textcolor{Rouge}{\bf 3} \textcolor{Bleu}{\bar 5}.
\end{equation}
One can show without difficulty that $\AlgMotsMarq$ is a
$\gamma$-multiprojection algebra. By
Theorem~\ref{thm:algebre_multiprojections_vers_dias_gamma},
$\MProjVersPluri(\AlgMotsMarq)$ is a $\gamma$-pluriassociative algebra.
We have for instance
\begin{equation}
    \textcolor{Bleu}{3 \bar 2 5}
    \GDias_3
    \textcolor{Rouge}{4 \bar 4 1}
    = \textcolor{Bleu}{3 {\bf \bar 4} 5}
        \textcolor{Rouge}{4 \bar 4 {\bf 3}},
\end{equation}
and
\begin{equation}
    \textcolor{Bleu}{1 \bar 3 4 \bar 1 3}
    \DDias_2
    \textcolor{Rouge}{3 1 \bar 2 3 \bar 1 1}
    = \textcolor{Bleu}{{\bf 2} \bar 3 4 {\bf \bar 3} 3}
      \textcolor{Rouge}{3 1 {\bf \bar 3} 3 {\bf \bar 3} 1}.
\end{equation}
We can observe that $\MProjVersPluri(\AlgMotsMarq)$ is not
commutative, pure, and has no wire-units neither bar-units.
\medskip

\paragraph{\bf The free $\gamma$-pluriassociative algebra over one generator}
Let $\gamma \geq 0$ be an integer. We give here a construction of the
free $\gamma$-pluriassociative algebra $\AlgLibre_{\Dias_\gamma}$ over one
generator described in Section~\ref{subsubsec:algebre_dias_gamma_libre}
passing through the following $\gamma$-multiprojection algebra
and the construction $\MProjVersPluri$. Consider the vector space of
nonempty words on the alphabet $\{0\} \cup [\gamma]$ with exactly one
occurrence of $0$, endowed with the operations $\MAs_a$, $a \in [\gamma]$,
all equal to the concatenation operation extended by linearity and with
the endomorphisms $\Augm_a$, $a \in [\gamma]$, defined in
Section~\ref{subsubsec:algebre_dias_gamma_libre}. This vector
space is a $\gamma$-multiprojection algebra. Therefore, by
Theorem~\ref{thm:algebre_multiprojections_vers_dias_gamma}, it gives
rise by the construction $\MProjVersPluri$ to a $\gamma$-pluriassociative
algebra and it appears that it is $\AlgLibre_{\Dias_\gamma}$. Besides,
we can now observe that $\AlgLibre_{\Dias_\gamma}$ is not commutative,
pure, and has no wire-units neither bar-units.
\medskip

\section{Polydendriform operads} \label{sec:dendr_gamma}
At this point, the situation is ripe enough to introduce our
generalization on a nonnegative integer $\gamma$ of the dendriform operad
and dendriform algebras. We first construct this operad, compute its
dimensions, and give then two presentations by generators and relations.
This section ends by a description of free algebras over one generator
in the category encoded by our generalization.
\medskip

\subsection{Construction and properties}%
\label{subsec:construcion_dendr_gamma}
Theorem \ref{thm:presentation_dias_gamma}, by exhibiting a presentation
of $\Dias_\gamma$, shows that this operad is binary and quadratic.
It then admits a Koszul dual, denoted by $\Dendr_\gamma$ and called
{\em $\gamma$-polydendriform operad}.
\medskip

\subsubsection{Definition and presentation}%
\label{subsubsec:presentation_dendr_gamma_alternative}
A description of $\Dendr_\gamma$ is provided by the following presentation
by generators and relations.
\medskip

\begin{Theoreme} \label{thm:presentation_dendr_gamma}
    For any integer $\gamma \geq 0$, the operad $\Dendr_\gamma$
    admits the following presentation. It is generated by
    $\GenDendr := \GenDendr(2):=
    \{\GDendrA_a, \DDendrA_a : a \in [\gamma]\}$ and its space of
    relations $\RelDendr$ is generated by
    \begin{subequations}
    \begin{equation} \label{equ:relation_dendr_gamma_1_alternative}
        \GDendrA_a \circ_1 \DDendrA_{a'} - \DDendrA_{a'} \circ_2 \GDendrA_a,
        \qquad a, a' \in [\gamma],
    \end{equation}
    \begin{equation} \label{equ:relation_dendr_gamma_2_alternative}
        \GDendrA_a \circ_1 \GDendrA_b - \GDendrA_a \circ_2 \DDendrA_b,
        \qquad a < b \in [\gamma],
    \end{equation}
    \begin{equation} \label{equ:relation_dendr_gamma_3_alternative}
        \DDendrA_a \circ_1 \GDendrA_b - \DDendrA_a \circ_2 \DDendrA_b,
        \qquad a < b \in [\gamma],
    \end{equation}
    \begin{equation} \label{equ:relation_dendr_gamma_4_alternative}
        \GDendrA_a \circ_1 \GDendrA_b - \GDendrA_a \circ_2 \GDendrA_b,
        \qquad a < b \in [\gamma],
    \end{equation}
    \begin{equation} \label{equ:relation_dendr_gamma_5_alternative}
        \DDendrA_a \circ_1 \DDendrA_b - \DDendrA_a \circ_2 \DDendrA_b,
        \qquad a < b \in [\gamma],
    \end{equation}
    \begin{equation} \label{equ:relation_dendr_gamma_6_alternative}
        \GDendrA_d \circ_1 \GDendrA_d -
        \left(\sum_{c \in [d]} \GDendrA_d \circ_2 \GDendrA_c
                + \GDendrA_d \circ_2 \DDendrA_c\right),
        \qquad d \in [\gamma],
    \end{equation}
    \begin{equation} \label{equ:relation_dendr_gamma_7_alternative}
        \left(\sum_{c \in [d]} \DDendrA_d \circ_1 \DDendrA_c
            + \DDendrA_d \circ_1 \GDendrA_c\right)
            - \DDendrA_d \circ_2 \DDendrA_d,
        \qquad d \in [\gamma].
    \end{equation}
    \end{subequations}
\end{Theoreme}
\begin{proof}
    By Theorem~\ref{thm:presentation_dias_gamma}, we know that
    $\Dias_\gamma$ is a binary and quadratic operad, and that its space
    of relations $\RelDias$ is the space induced by the equivalence
    relation $\Rel_\gamma$ defined
    by~\eqref{equ:relation_dias_gamma_1}--\eqref{equ:relation_dias_gamma_7}.
    Now, by a straightforward computation, and by identifying $\GDendrA_a$
    (resp. $\DDendrA_a$) with $\GDias_a$ (resp. $\DDias_a$) for any
    $a \in [\gamma]$, we obtain that the space $\RelDendr$ of the
    statement of the theorem satisfies $\RelDias^\perp = \RelDendr$.
    Hence, $\Dendr_\gamma$ admits the claimed presentation.
\end{proof}
\medskip

Theorem \ref{thm:presentation_dendr_gamma} provides a quite complicated
presentation of $\Dendr_\gamma$. We shall below define a more convenient
basis for the space of relations of $\Dendr_\gamma$.
\medskip

\subsubsection{Elements and dimensions}
\begin{Proposition} \label{prop:serie_hilbert_dendr_gamma}
    For any integer $\gamma \geq 0$, the Hilbert series
    $\Hca_{\Dendr_\gamma}(t)$ of the operad $\Dendr_\gamma$ satisfies
    \begin{equation} \label{equ:serie_hilbert_dendr_gamma}
        \Hca_{\Dendr_\gamma}(t)
            = t + 2\gamma t \, \Hca_{\Dendr_\gamma}(t)
            + \gamma^2 t \, \Hca_{\Dendr_\gamma}(t)^2.
    \end{equation}
\end{Proposition}
\begin{proof}
    By setting $\bar \Hca_{\Dendr_\gamma}(t) := \Hca_{\Dendr_\gamma}(-t)$,
    from~\eqref{equ:serie_hilbert_dendr_gamma}, we obtain
    \begin{equation}
        t = \frac{-\bar \Hca_{\Dendr_\gamma}(t)}
                 {\left(1 + \gamma \, \bar \Hca_{\Dendr_\gamma}(t)\right)^2}.
    \end{equation}
    Moreover, by setting
    $\bar \Hca_{\Dias_\gamma}(t) := \Hca_{\Dias_\gamma}(-t)$, where
    $\Hca_{\Dias_\gamma}(t)$ is defined by~\eqref{equ:serie_hilbert_dias_gamma},
    we have
    \begin{equation} \label{equ:serie_hilbert_dendr_gamma_demo}
        \bar \Hca_{\Dias_\gamma}\left(\bar \Hca_{\Dendr_\gamma}(t)\right)
            = \frac{-\bar \Hca_{\Dendr_\gamma}(t)}
                  {\left(1 + \gamma \, \bar \Hca_{\Dendr_\gamma}(t)\right)^2}
            = t,
    \end{equation}
    showing that $\bar \Hca_{\Dias_\gamma}(t)$ and
    $\bar \Hca_{\Dendr_\gamma}(t)$ are the inverses for each other for
    series composition.
    \smallskip

    Now, since  by Theorem~\ref{thm:koszulite_dias_gamma} and
    Proposition~\ref{prop:elements_dias_gamma}, $\Dias_\gamma$ is a Koszul
    operad and its Hilbert series is $\Hca_{\Dias_\gamma}(t)$, and since
    $\Dendr_\gamma$ is by definition the Koszul dual of $\Dias_\gamma$,
    the Hilbert series of these two operads
    satisfy~\eqref{equ:relation_series_hilbert_operade_duale}.
    Therefore, \eqref{equ:serie_hilbert_dendr_gamma_demo} implies that
    the Hilbert series of $\Dendr_\gamma$ is $\Hca_{\Dendr_\gamma}(t)$.
\end{proof}
\medskip

By examining the expression for $\Hca_{\Dendr_\gamma}(t)$ of the
statement of Proposition~\ref{prop:serie_hilbert_dendr_gamma}, we
observe that for any $n \geq 1$, $\Dendr_\gamma(n)$ can be seen as the
vector space $\AlgLibre_{\Dendr_\gamma}(n)$ of binary trees with $n$
internal nodes wherein its $n - 1$ edges connecting two internal nodes
are labeled on $[\gamma]$. We call these trees {\em $\gamma$-edge valued
binary trees}. In our graphical representations of $\gamma$-edge valued
binary trees, any edge label is drawn into a hexagon located half the
edge. For instance,
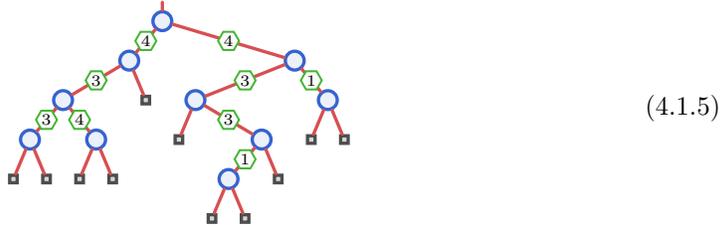
\begin{equation}
    \begin{split}
    \begin{tikzpicture}[xscale=.22,yscale=.15]
        \node[Feuille](0)at(0.00,-14.00){};
        \node[Feuille](10)at(10.00,-10.50){};
        \node[Feuille](12)at(12.00,-17.50){};
        \node[Feuille](14)at(14.00,-17.50){};
        \node[Feuille](16)at(16.00,-14.00){};
        \node[Feuille](18)at(18.00,-10.50){};
        \node[Feuille](2)at(2.00,-14.00){};
        \node[Feuille](20)at(20.00,-10.50){};
        \node[Feuille](4)at(4.00,-14.00){};
        \node[Feuille](6)at(6.00,-14.00){};
        \node[Feuille](8)at(8.00,-7.00){};
        \node[Noeud](1)at(1.00,-10.50){};
        \node[Noeud](11)at(11.00,-7.00){};
        \node[Noeud](13)at(13.00,-14.00){};
        \node[Noeud](15)at(15.00,-10.50){};
        \node[Noeud](17)at(17.00,-3.50){};
        \node[Noeud](19)at(19.00,-7.00){};
        \node[Noeud](3)at(3.00,-7.00){};
        \node[Noeud](5)at(5.00,-10.50){};
        \node[Noeud](7)at(7.00,-3.50){};
        \node[Noeud](9)at(9.00,0.00){};
        \draw[Arete](0)--(1);
        \draw[Arete](1)edge[]node[EtiqArete]{\begin{math}3\end{math}}(3);
        \draw[Arete](10)--(11);
        \draw[Arete](11)edge[]node[EtiqArete]{\begin{math}3\end{math}}(17);
        \draw[Arete](12)--(13);
        \draw[Arete](13)edge[]node[EtiqArete]{\begin{math}1\end{math}}(15);
        \draw[Arete](14)--(13);
        \draw[Arete](15)edge[]node[EtiqArete]{\begin{math}3\end{math}}(11);
        \draw[Arete](16)--(15);
        \draw[Arete](17)edge[]node[EtiqArete]{\begin{math}4\end{math}}(9);
        \draw[Arete](18)--(19);
        \draw[Arete](19)edge[]node[EtiqArete]{\begin{math}1\end{math}}(17);
        \draw[Arete](2)--(1);
        \draw[Arete](20)--(19);
        \draw[Arete](3)edge[]node[EtiqArete]{\begin{math}3\end{math}}(7);
        \draw[Arete](4)--(5);
        \draw[Arete](5)edge[]node[EtiqArete]{\begin{math}4\end{math}}(3);
        \draw[Arete](6)--(5);
        \draw[Arete](7)edge[]node[EtiqArete]{\begin{math}4\end{math}}(9);
        \draw[Arete](8)--(7);
        \node(r)at(9.00,2.50){};
        \draw[Arete](r)--(9);
    \end{tikzpicture}
    \end{split}
\end{equation}
is a $4$-edge valued binary tree and a basis element of $\Dendr_4(10)$.
\medskip

We deduce from Proposition~\ref{prop:serie_hilbert_dendr_gamma} that the
Hilbert series of $\Dendr_\gamma$ satisfies
\begin{equation}
    \Hca_{\Dendr_\gamma}(t) =
    \frac{1 - \sqrt{1 - 4 \gamma t} - 2 \gamma t}{2\gamma^2 t},
\end{equation}
and we also obtain that for all $n \geq 1$,
$\dim \Dendr_\gamma(n) = \gamma^{n - 1} \Cat(n)$.
For instance, the first dimensions of $\Dendr_1$, $\Dendr_2$, $\Dendr_3$,
and $\Dendr_4$ are respectively
\begin{equation}
    1, 2, 5, 14, 42, 132, 429, 1430, 4862, 16796, 58786,
\end{equation}
\begin{equation}
    1, 4, 20, 112, 672, 4224, 27456, 183040, 1244672, 8599552, 60196864,
\end{equation}
\begin{equation}
    1, 6, 45, 378, 3402, 32076, 312741, 3127410, 31899582, 330595668, 3471254514,
\end{equation}
\begin{equation}
    1, 8, 80, 896, 10752, 135168, 1757184, 23429120, 318636032,
    4402970624, 61641588736.
\end{equation}
The first one is Sequence~\Sloane{A000108}, the second one is
Sequence~\Sloane{A003645}, and the third one is Sequence~\Sloane{A101600}
of~\cite{Slo}. Last sequence is not listed in~\cite{Slo} at this time.
\medskip

\subsubsection{Associative operations}
In the same manner as in the dendriform operad the sum of its two
operations produces an associative operation, in the $\gamma$-dendriform
operad there is a way to build associative operations, as shows next
statement.
\medskip

\begin{Proposition} \label{prop:operateur_associatif_dendr_gamma_autre}
    For any integers $\gamma \geq 0$ and $b \in [\gamma]$, the element
    \begin{equation}
        \OpAsDendrA_b :=
        \pi\left(\sum_{a \in [b]} \GDendrA_a + \DDendrA_a\right)
    \end{equation}
    of $\Dendr_\gamma$, where
    $\pi : \OpLibre\left(\GenDendr\right) \to \Dendr_\gamma$ is the
    canonical surjection map, is associative.
\end{Proposition}
\begin{proof}
    By setting
    \begin{equation}
        x := \sum_{a \in [b]} \GDendrA_a + \DDendrA_a,
    \end{equation}
    we have
    \begin{multline} \label{equ:operateur_associatif_dendr_gamma_autre_demo}
        x \circ_1 x - x \circ_2 x =
        \GDendrA_a \circ_1 \GDendrA_{a'} +
        \GDendrA_a \circ_1 \DDendrA_{a'}
            + \DDendrA_a \circ_1 \GDendrA_{a'} +
        \DDendrA_a \circ_1 \DDendrA_{a'} \\
            -
            \GDendrA_a \circ_2 \GDendrA_{a'} -
        \GDendrA_a \circ_2 \DDendrA_{a'}
            -
        \DDendrA_a \circ_2 \GDendrA_{a'} -
        \DDendrA_a \circ_2 \DDendrA_{a'}.
    \end{multline}
    We the observe that~\eqref{equ:operateur_associatif_dendr_gamma_autre_demo}
    is the sum of
    elements~\eqref{equ:relation_dendr_gamma_1_alternative}---%
    \eqref{equ:relation_dendr_gamma_7_alternative}
    which generate, by Theorem~\ref{thm:presentation_dendr_gamma}, the
    space of relations of $\Dendr_\gamma$. Therefore, we have
    $\pi(x \circ_1 x - x \circ_2 x) = 0$, implying
    $\OpAsDendrA_b \circ_1 \OpAsDendrA_b -
    \OpAsDendrA_b \circ_2 \OpAsDendrA_b = 0$ and
    the associativity of $\OpAsDendrA_b$.
\end{proof}
\medskip

\subsubsection{Alternative presentation}
For any integer $\gamma \geq 0$, let $\GDendr_b$ and $\DDendr_b$,
$b \in [\gamma]$, the elements of $\OpLibre\left(\GenDendr\right)$
defined by
\begin{subequations}
\begin{equation} \label{equ:definition_operateur_dendr_gauche}
    \GDendr_b := \sum_{a \in [b]} \GDendrA_a,
\end{equation}
and
\begin{equation} \label{equ:definition_operateur_dendr_droite}
    \DDendr_b := \sum_{a \in [b]} \DDendrA_a.
\end{equation}
\end{subequations}
Then, since for all $b \in [\gamma]$ we have
\begin{subequations}
\begin{equation}
    \GDendrA_b =
    \begin{cases}
        \GDendr_1 & \mbox{if } b = 1, \\
        \GDendr_b - \GDendr_{b - 1} & \mbox{otherwise},
    \end{cases}
\end{equation}
and
\begin{equation}
    \DDendrA_b =
    \begin{cases}
        \DDendr_1 & \mbox{if } b = 1, \\
        \DDendr_b - \DDendr_{b - 1} & \mbox{otherwise},
    \end{cases}
\end{equation}
\end{subequations}
by triangularity, the family
$\GenDendr' := \{\GDendr_b, \DDendr_b : b \in [\gamma]\}$ forms a
basis of $\OpLibre\left(\GenDendr\right)(2)$ and then, generates
$\OpLibre\left(\GenDendr\right)$ as an operad. This change of basis
from $\OpLibre\left(\GenDendr\right)$ to $\OpLibre(\GenDendr')$
is similar to the change of basis from $\OpLibre(\GenDias')$
to $\OpLibre\left(\GenDias\right)$ introduced in
Section~\ref{subsubsec:presentation_dias_gamma_alternative}.
Let us now express a presentation of $\Dendr_\gamma$ through the family
$\GenDendr'$.
\medskip

\begin{Theoreme} \label{thm:autre_presentation_dendr_gamma}
    For any integer $\gamma \geq 0$, the operad $\Dendr_\gamma$
    admits the following presentation. It is generated by $\GenDendr'$
    and its space of relations $\RelDendr'$ is generated by
    \begin{subequations}
    \begin{equation} \label{equ:relation_dendr_gamma_1}
        \GDendr_a \circ_1 \DDendr_{a'} - \DDendr_{a'} \circ_2 \GDendr_a,
        \qquad a, a' \in [\gamma],
    \end{equation}
    \begin{equation} \label{equ:relation_dendr_gamma_2}
        \GDendr_a \circ_1 \GDendr_b
            - \GDendr_a \circ_2 \DDendr_b
            - \GDendr_a \circ_2 \GDendr_a,
        \qquad a < b \in [\gamma],
    \end{equation}
    \begin{equation} \label{equ:relation_dendr_gamma_3}
            \DDendr_a \circ_1 \DDendr_a
            + \DDendr_a \circ_1 \GDendr_b
            - \DDendr_a \circ_2 \DDendr_b,
            \qquad a < b \in [\gamma],
    \end{equation}
    \begin{equation} \label{equ:relation_dendr_gamma_4}
        \GDendr_b \circ_1 \GDendr_a
            - \GDendr_a \circ_2 \GDendr_b
            - \GDendr_a \circ_2 \DDendr_a,
        \qquad a < b \in [\gamma],
    \end{equation}
    \begin{equation} \label{equ:relation_dendr_gamma_5}
        \DDendr_a \circ_1 \GDendr_a
            + \DDendr_a \circ_1 \DDendr_b
            - \DDendr_b \circ_2 \DDendr_a,
        \qquad a < b \in [\gamma],
    \end{equation}
    \begin{equation} \label{equ:relation_dendr_gamma_6}
        \GDendr_a \circ_1 \GDendr_a
            - \GDendr_a \circ_2 \DDendr_a
            - \GDendr_a \circ_2 \GDendr_a,
        \qquad a \in [\gamma],
    \end{equation}
    \begin{equation} \label{equ:relation_dendr_gamma_7}
        \DDendr_a \circ_1 \DDendr_a
            + \DDendr_a \circ_1 \GDendr_a
            - \DDendr_a \circ_2 \DDendr_a,
        \qquad a \in [\gamma].
    \end{equation}
    \end{subequations}
\end{Theoreme}
\begin{proof}
    Let us show that $\RelDendr'$ is equal to the space of relations
    $\RelDendr$ of $\Dendr_\gamma$ defined in the statement of
    Theorem~\ref{thm:presentation_dendr_gamma}. By this last theorem,
    for any $x \in \OpLibre\left(\GenDendr\right)(3)$, $x$ is in
    $\RelDendr$ if and only if $\pi(x) = 0$ where
    $\pi : \OpLibre\left(\GenDendr\right) \to \Dendr_\gamma$ is the
    canonical surjection map. By straightforward computations, by
    expanding any element $x$ of~\eqref{equ:relation_dendr_gamma_1}---%
    \eqref{equ:relation_dendr_gamma_7} over the elements
    $\GDendrA_a$, $\DDendrA_a$, $a \in [\gamma]$, by
    using~\eqref{equ:definition_operateur_dendr_gauche}
    and~\eqref{equ:definition_operateur_dendr_droite} we obtain that
    $x$ can be expressed as a sum of elements of $\RelDendr$. This
    implies that $\pi(x) = 0$ and hence that $\RelDendr'$ is a subspace
    of $\RelDendr$.
    \smallskip

    Now, one can observe that
    elements~\eqref{equ:relation_dendr_gamma_1}---%
    \eqref{equ:relation_dendr_gamma_6} are linearly independent.
    Then, $\RelDendr'$ has dimension $3\gamma^2$ which is also, by
    Theorem~\ref{thm:presentation_dendr_gamma}, the dimension of
    $\RelDendr$. The statement of the theorem follows.
\end{proof}
\medskip

The presentation of $\Dendr_\gamma$ provided by
Theorem~\ref{thm:autre_presentation_dendr_gamma} is easier to
handle than the one provided by Theorem \ref{thm:presentation_dendr_gamma}.
The main reason is that Relations~\eqref{equ:relation_dendr_gamma_6_alternative}
and~\eqref{equ:relation_dendr_gamma_7_alternative} of the first
presentation involve a nonconstant number of terms, while all relations
of this second presentation always involve only two or three terms. As a
very remarkable fact, it is worthwhile to note that the presentation of
$\Dendr_\gamma$ provided by Theorem~\ref{thm:autre_presentation_dendr_gamma}
can be directly obtained by considering the Koszul dual of $\Dias_\gamma$
over the $\Ksf$-basis (see Sections~\ref{subsubsec:base_K}
and~\ref{subsubsec:presentation_dias_gamma_alternative}). Therefore,
an alternative way to establish this presentation consists in computing
the Koszul dual of $\Dias_\gamma$ seen through the presentation having
$\RelDendr'$ as space of relations, which is made of the relations of
$\Dias_\gamma$ expressed over the $\Ksf$-basis
(see Proposition~\ref{prop:presentation_alternative_dias_gamma}).
\medskip

From now on, $\Min$ denotes the operation $\min$ on integers. Using this
notation, the space of relations $\RelDendr'$ of $\Dendr_\gamma$
exhibited by Theorem~\ref{thm:autre_presentation_dendr_gamma}
can be rephrased in a more compact way as the space generated by
\begin{subequations}
\begin{equation} \label{equ:relation_dendr_gamma_1_concise}
    \GDendr_a \circ_1 \DDendr_{a'} - \DDendr_{a'} \circ_2 \GDendr_a,
    \qquad a, a' \in [\gamma],
\end{equation}
\begin{equation} \label{equ:relation_dendr_gamma_2_concise}
    \GDendr_a \circ_1 \GDendr_{a'}
        - \GDendr_{a \Min a'} \circ_2 \GDendr_a
        - \GDendr_{a \Min a'} \circ_2 \DDendr_{a'},
    \qquad a, a' \in [\gamma],
\end{equation}
\begin{equation} \label{equ:relation_dendr_gamma_3_concise}
    \DDendr_{a \Min a'} \circ_1 \GDendr_{a'}
        + \DDendr_{a \Min a'} \circ_1 \DDendr_a
        - \DDendr_a \circ_2 \DDendr_{a'},
    \qquad a, a' \in [\gamma].
\end{equation}
\end{subequations}
\medskip

Over the family $\GenDendr'$, one can build associative operations in
$\Dendr_\gamma$ in the following way.
\medskip

\begin{Proposition} \label{prop:operateur_associatif_dendr_gamma}
    For any integers $\gamma \geq 0$ and $b \in [\gamma]$, the element
    \begin{equation}
        \OpAsDendr_b := \pi(\GDendr_b + \DDendr_b)
    \end{equation}
    of $\Dendr_\gamma$, where
    $\pi : \OpLibre(\GenDendr') \to \Dendr_\gamma$ is the
    canonical surjection map, is associative.
\end{Proposition}
\begin{proof}
    By definition of the $\GDendr_b$ and $\DDendr_b$, $b \in [\gamma]$,
    we have
    \begin{equation}
        \GDendr_b + \DDendr_b =
        \sum_{a \in [b]} \GDendrA_a + \DDendrA_a.
    \end{equation}
    We hence observe that $\OpAsDendr_b = \OpAsDendrA_b$, where
    $\OpAsDendrA_b$ is the element of $\Dendr_\gamma$ defined in the
    statement of Proposition~\ref{prop:operateur_associatif_dendr_gamma_autre}.
    Hence, by this latter proposition, $\OpAsDendr_b$ is associative.
\end{proof}
\medskip

\begin{Proposition}
\label{prop:description_operateurs_associatifs_dendr_gamma}
    For any integer $\gamma \geq 0$, any associative element of
    $\Dendr_\gamma$ is proportional to $\OpAsDendr_b$ for a
    $b \in [\gamma]$.
\end{Proposition}
\begin{proof}
    Let $\pi : \OpLibre(\GenDendr') \to \Dendr_\gamma$ be the canonical
    surjection map. Consider the element
    \begin{equation}
        x :=
        \sum_{a \in [\gamma]} \alpha_a \GDendr_a + \beta_a \DDendr_a
    \end{equation}
    of $\OpLibre(\GenDendr')$, where $\alpha_a, \beta_a \in \K$ for all
    $a \in [\gamma]$, such that $\pi(x)$ is associative in $\Dendr_\gamma$.
    Since we have $\pi(r) = 0$ for all elements $r$ of $\RelDendr'$
    (see~\eqref{equ:relation_dendr_gamma_1_concise},
    \eqref{equ:relation_dendr_gamma_2_concise},
    and~\eqref{equ:relation_dendr_gamma_3_concise}),
    the fact that $\pi(x \circ_1 x - x \circ_2 x) = 0$ implies the
    constraints
    \begin{equation}\begin{split}
        \alpha_a \, \beta_{a'} = \beta_{a'} \, \alpha_a,
        & \qquad a, a' \in [\gamma], \\
        \alpha_a \, \alpha_{a'} = \alpha_{a \Min a'} \, \alpha_a =
        \alpha_{a \Min a'} \, \beta_{a'}, & \qquad a, a' \in [\gamma], \\
        \beta_{a \Min a'} \, \alpha_{a'} = \beta_{a \Min a'} \, \beta_a =
        \beta_a \, \beta_{a'}, & \qquad a, a' \in [\gamma],
    \end{split}\end{equation}
    on the coefficients intervening in $x$. Moreover, since the syntax
    trees $\DDendr_b \circ_1 \DDendr_a$, $\DDendr_b \circ_1 \GDendr_a$,
    $\GDendr_b \circ_2 \GDendr_a$, and $\GDendr_b \circ_2 \DDendr_a$ do
    not appear in $\RelDendr'$ for all $a < b \in [\gamma]$ , we have the
    further constraints
    \begin{equation}\begin{split}
        \beta_b \, \beta_a & = 0, \qquad a < b \in [\gamma], \\
        \beta_b \, \alpha_a & = 0, \qquad a < b \in [\gamma], \\
        \alpha_b \, \alpha_a & = 0, \qquad a < b \in [\gamma], \\
        \alpha_b \, \beta_a & = 0, \qquad a < b \in [\gamma].
    \end{split}\end{equation}
    These relations imply that there are at most one $c \in [\gamma]$ and
    one $d \in [\gamma]$ such that $\alpha_c \ne 0$ and $\beta_d \ne 0$.
    In this case, these relations imply also that $c = d$, and
    $\alpha_c = \beta_c$. Therefore, $x$ is of the form
    $x = \alpha_a \GDendr_a + \alpha_a \DDendr_a$ for an $a \in [\gamma]$,
    whence the statement of the proposition.
\end{proof}
\medskip

\subsection{Category of polydendriform algebras and free objects}
The aim of this section is to describe the category of
$\Dendr_\gamma$-algebras and more particularly the free
$\Dendr_\gamma$-algebra over one generator.
\medskip

\subsubsection{Polydendriform algebras}
We call {\em $\gamma$-polydendriform algebra} any
$\Dendr_\gamma$-algebra. From the presentation of $\Dendr_\gamma$
provided by Theorem~\ref{thm:presentation_dendr_gamma}, any
$\gamma$-polydendriform algebra is a vector space endowed with linear
operations $\GDendrA_a, \DDendrA_a$, $a \in [\gamma]$, satisfying the
relations encoded by~\eqref{equ:relation_dendr_gamma_1_alternative}---%
\eqref{equ:relation_dendr_gamma_7_alternative}. By considering the
presentation of $\Dendr_\gamma$ exhibited by
Theorem~\ref{thm:autre_presentation_dendr_gamma}, any
$\gamma$-polydendriform algebra is a vector space endowed with linear
operations $\GDendr_a, \DDendr_a$, $a \in [\gamma]$, satisfying the
relations encoded by~\eqref{equ:relation_dendr_gamma_1_concise}---%
\eqref{equ:relation_dendr_gamma_3_concise}.
\medskip

\subsubsection{Two ways to split associativity}
Like dendriform algebras, which offer a way to split an associative
operation into two parts, $\gamma$-polydendriform algebras propose
two ways to split associativity depending on its chosen presentation.
\medskip

On the one hand, in a $\gamma$-polydendriform algebra $\Dca$ over the
operations $\GDendrA_a$, $\DDendrA_a$, $a \in [\gamma]$,
by Proposition~\ref{prop:operateur_associatif_dendr_gamma_autre}, an
associative operation $\OpAsDendrA$ is split into the $2\gamma$ operations
$\GDendrA_a$, $\DDendrA_a$, $a \in [\gamma]$, so that for all $x, y \in \Dca$,
\begin{equation}
    x \OpAsDendrA y =
    \sum_{a \in [\gamma]}  x \GDendrA_a y + x \DDendrA_a y,
\end{equation}
and all partial sums operations $\OpAsDendrA_b$, $b \in [\gamma]$,
satisfying
\begin{equation}
    x \OpAsDendrA_b y = \sum_{a \in [b]} x \GDendrA_a y + x \DDendrA_a x,
\end{equation}
also are associative.
\medskip

On the other hand, in a $\gamma$-polydendriform algebra over the operations
$\GDendr_a$, $\DDendr_a$, $a \in [\gamma]$, by
Proposition~\ref{prop:operateur_associatif_dendr_gamma}, several
associative operations $\OpAsDendr_a$, $a \in [\gamma]$, are each split
into two operations $\GDendr_a$, $\DDendr_a$, $a \in [\gamma]$, so
that for all $x, y \in \Dca$,
\begin{equation}
    x \OpAsDendr_a y = x \GDendr_a y + x \DDendr_a y.
\end{equation}
\medskip

Therefore, we can observe that $\gamma$-polydendriform algebras over the
operations $\GDendrA_a$, $\DDendrA_a$, $a \in [\gamma]$, are adapted to
study associative algebras (by splitting its single product in the way
we have described above) while $\gamma$-polydendriform algebras over the
operations $\GDendr_a$, $\DDendr_a$, $a \in [\gamma]$, are adapted to
study vectors spaces endowed with several associative products (by
splitting each one in the way we have described above). Algebras with
several associative products will be studied in Section~\ref{sec:as_gamma}.
\medskip

\subsubsection{Free polydendriform algebras}
From now, in order to simplify and make uniform next definitions, we
consider that in any $\gamma$-edge valued binary tree $\Tfr$, all edges
connecting internal nodes of $\Tfr$ with leaves are labeled by
$\infty$. By convention, for all $a \in [\gamma]$, we have
$a \Min \infty = a = \infty \Min a$.
\medskip

Let us endow the vector space $\AlgLibre_{\Dendr_\gamma}$
of $\gamma$-edge valued binary trees
with linear operations
\begin{equation}
    \GDendr_a, \DDendr_a :
    \AlgLibre_{\Dendr_\gamma} \otimes \AlgLibre_{\Dendr_\gamma}
    \to \AlgLibre_{\Dendr_\gamma},
    \qquad
    a \in [\gamma],
\end{equation}
recursively defined, for any $\gamma$-edge valued binary tree $\Sfr$
and any $\gamma$-edge valued binary trees or leaves $\Tfr_1$ and $\Tfr_2$
by
\begin{equation}
    \Sfr \GDendr_a \Feuille
    := \Sfr =:
    \Feuille \DDendr_a \Sfr,
\end{equation}
\begin{equation}
    \Feuille \GDendr_a \Sfr := 0 =: \Sfr \DDendr_a \Feuille,
\end{equation}
\begin{equation}
    \ArbreBinValue{x}{y}{\Tfr_1}{\Tfr_2}
    \GDendr_a \Sfr :=
    \ArbreBinValue{x}{z}{\Tfr_1}{\Tfr_2 \GDendr_a \Sfr}
    +
    \ArbreBinValue{x}{z}{\Tfr_1}{\Tfr_2 \DDendr_y \Sfr}\,,
    \qquad
    z := a \Min y,
\end{equation}
\begin{equation}
    \begin{split}\Sfr \DDendr_a\end{split}
    \ArbreBinValue{x}{y}{\Tfr_1}{\Tfr_2}
    :=
    \ArbreBinValue{z}{y}{\Sfr \DDendr_a \Tfr_1}{\Tfr_2}
    +
    \ArbreBinValue{z}{y}{\Sfr \GDendr_x \Tfr_1}{\Tfr_2}\,,
    \qquad
    z := a \Min x.
\end{equation}
Note that neither $\Feuille \GDendr_a \Feuille$ nor
$\Feuille \DDendr_a \Feuille$ are defined.
\medskip

For example, we have
\begin{multline}
    \begin{split}
    \begin{tikzpicture}[xscale=.25,yscale=.2]
        \node[Feuille](0)at(0.00,-4.50){};
        \node[Feuille](2)at(2.00,-4.50){};
        \node[Feuille](4)at(4.00,-4.50){};
        \node[Feuille](6)at(6.00,-6.75){};
        \node[Feuille](8)at(8.00,-6.75){};
        \node[Noeud](1)at(1.00,-2.25){};
        \node[Noeud](3)at(3.00,0.00){};
        \node[Noeud](5)at(5.00,-2.25){};
        \node[Noeud](7)at(7.00,-4.50){};
        \draw[Arete](0)--(1);
        \draw[Arete](1)edge[]node[EtiqArete]{\begin{math}1\end{math}}(3);
        \draw[Arete](2)--(1);
        \draw[Arete](4)--(5);
        \draw[Arete](5)edge[]node[EtiqArete]{\begin{math}3\end{math}}(3);
        \draw[Arete](6)--(7);
        \draw[Arete](7)edge[]node[EtiqArete]{\begin{math}1\end{math}}(5);
        \draw[Arete](8)--(7);
        \node(r)at(3.00,1.7){};
        \draw[Arete](r)--(3);
    \end{tikzpicture}
    \end{split}
    \GDendr_2
    \begin{split}
    \begin{tikzpicture}[xscale=.25,yscale=.2]
        \node[Feuille](0)at(0.00,-4.67){};
        \node[Feuille](2)at(2.00,-4.67){};
        \node[Feuille](4)at(4.00,-4.67){};
        \node[Feuille](6)at(6.00,-4.67){};
        \node[Noeud](1)at(1.00,-2.33){};
        \node[Noeud](3)at(3.00,0.00){};
        \node[Noeud](5)at(5.00,-2.33){};
        \draw[Arete](0)--(1);
        \draw[Arete](1)edge[]node[EtiqArete]{\begin{math}1\end{math}}(3);
        \draw[Arete](2)--(1);
        \draw[Arete](4)--(5);
        \draw[Arete](5)edge[]node[EtiqArete]{\begin{math}2\end{math}}(3);
        \draw[Arete](6)--(5);
        \node(r)at(3.00,1.7){};
        \draw[Arete](r)--(3);
    \end{tikzpicture}
    \end{split}
    =
    \begin{split}
    \begin{tikzpicture}[xscale=.22,yscale=.15]
        \node[Feuille](0)at(0.00,-5.00){};
        \node[Feuille](10)at(10.00,-12.50){};
        \node[Feuille](12)at(12.00,-12.50){};
        \node[Feuille](14)at(14.00,-12.50){};
        \node[Feuille](2)at(2.00,-5.00){};
        \node[Feuille](4)at(4.00,-5.00){};
        \node[Feuille](6)at(6.00,-7.50){};
        \node[Feuille](8)at(8.00,-12.50){};
        \node[Noeud](1)at(1.00,-2.50){};
        \node[Noeud](11)at(11.00,-7.50){};
        \node[Noeud](13)at(13.00,-10.00){};
        \node[Noeud](3)at(3.00,0.00){};
        \node[Noeud](5)at(5.00,-2.50){};
        \node[Noeud](7)at(7.00,-5.00){};
        \node[Noeud](9)at(9.00,-10.00){};
        \draw[Arete](0)--(1);
        \draw[Arete](1)edge[]node[EtiqArete]{\begin{math}1\end{math}}(3);
        \draw[Arete](10)--(9);
        \draw[Arete](11)edge[]node[EtiqArete]{\begin{math}2\end{math}}(7);
        \draw[Arete](12)--(13);
        \draw[Arete](13)edge[]node[EtiqArete]{\begin{math}2\end{math}}(11);
        \draw[Arete](14)--(13);
        \draw[Arete](2)--(1);
        \draw[Arete](4)--(5);
        \draw[Arete](5)edge[]node[EtiqArete]{\begin{math}2\end{math}}(3);
        \draw[Arete](6)--(7);
        \draw[Arete](7)edge[]node[EtiqArete]{\begin{math}1\end{math}}(5);
        \draw[Arete](8)--(9);
        \draw[Arete](9)edge[]node[EtiqArete]{\begin{math}1\end{math}}(11);
        \node(r)at(3.00,2.50){};
        \draw[Arete](r)--(3);
    \end{tikzpicture}
    \end{split}
    +
    \begin{split}
    \begin{tikzpicture}[xscale=.22,yscale=.15]
        \node[Feuille](0)at(0.00,-5.00){};
        \node[Feuille](10)at(10.00,-12.50){};
        \node[Feuille](12)at(12.00,-10.00){};
        \node[Feuille](14)at(14.00,-10.00){};
        \node[Feuille](2)at(2.00,-5.00){};
        \node[Feuille](4)at(4.00,-5.00){};
        \node[Feuille](6)at(6.00,-10.00){};
        \node[Feuille](8)at(8.00,-12.50){};
        \node[Noeud](1)at(1.00,-2.50){};
        \node[Noeud](11)at(11.00,-5.00){};
        \node[Noeud](13)at(13.00,-7.50){};
        \node[Noeud](3)at(3.00,0.00){};
        \node[Noeud](5)at(5.00,-2.50){};
        \node[Noeud](7)at(7.00,-7.50){};
        \node[Noeud](9)at(9.00,-10.00){};
        \draw[Arete](0)--(1);
        \draw[Arete](1)edge[]node[EtiqArete]{\begin{math}1\end{math}}(3);
        \draw[Arete](10)--(9);
        \draw[Arete](11)edge[]node[EtiqArete]{\begin{math}1\end{math}}(5);
        \draw[Arete](12)--(13);
        \draw[Arete](13)edge[]node[EtiqArete]{\begin{math}2\end{math}}(11);
        \draw[Arete](14)--(13);
        \draw[Arete](2)--(1);
        \draw[Arete](4)--(5);
        \draw[Arete](5)edge[]node[EtiqArete]{\begin{math}2\end{math}}(3);
        \draw[Arete](6)--(7);
        \draw[Arete](7)edge[]node[EtiqArete]{\begin{math}1\end{math}}(11);
        \draw[Arete](8)--(9);
        \draw[Arete](9)edge[]node[EtiqArete]{\begin{math}1\end{math}}(7);
        \node(r)at(3.00,2.50){};
        \draw[Arete](r)--(3);
    \end{tikzpicture}
    \end{split} \\
    +
    \begin{split}
    \begin{tikzpicture}[xscale=.22,yscale=.15]
        \node[Feuille](0)at(0.00,-5.00){};
        \node[Feuille](10)at(10.00,-10.00){};
        \node[Feuille](12)at(12.00,-10.00){};
        \node[Feuille](14)at(14.00,-10.00){};
        \node[Feuille](2)at(2.00,-5.00){};
        \node[Feuille](4)at(4.00,-5.00){};
        \node[Feuille](6)at(6.00,-12.50){};
        \node[Feuille](8)at(8.00,-12.50){};
        \node[Noeud](1)at(1.00,-2.50){};
        \node[Noeud](11)at(11.00,-5.00){};
        \node[Noeud](13)at(13.00,-7.50){};
        \node[Noeud](3)at(3.00,0.00){};
        \node[Noeud](5)at(5.00,-2.50){};
        \node[Noeud](7)at(7.00,-10.00){};
        \node[Noeud](9)at(9.00,-7.50){};
        \draw[Arete](0)--(1);
        \draw[Arete](1)edge[]node[EtiqArete]{\begin{math}1\end{math}}(3);
        \draw[Arete](10)--(9);
        \draw[Arete](11)edge[]node[EtiqArete]{\begin{math}1\end{math}}(5);
        \draw[Arete](12)--(13);
        \draw[Arete](13)edge[]node[EtiqArete]{\begin{math}2\end{math}}(11);
        \draw[Arete](14)--(13);
        \draw[Arete](2)--(1);
        \draw[Arete](4)--(5);
        \draw[Arete](5)edge[]node[EtiqArete]{\begin{math}2\end{math}}(3);
        \draw[Arete](6)--(7);
        \draw[Arete](7)edge[]node[EtiqArete]{\begin{math}1\end{math}}(9);
        \draw[Arete](8)--(7);
        \draw[Arete](9)edge[]node[EtiqArete]{\begin{math}1\end{math}}(11);
        \node(r)at(3.00,2.50){};
        \draw[Arete](r)--(3);
    \end{tikzpicture}
    \end{split}
    +
    \begin{split}
    \begin{tikzpicture}[xscale=.22,yscale=.15]
        \node[Feuille](0)at(0.00,-5.00){};
        \node[Feuille](10)at(10.00,-12.50){};
        \node[Feuille](12)at(12.00,-7.50){};
        \node[Feuille](14)at(14.00,-7.50){};
        \node[Feuille](2)at(2.00,-5.00){};
        \node[Feuille](4)at(4.00,-7.50){};
        \node[Feuille](6)at(6.00,-10.00){};
        \node[Feuille](8)at(8.00,-12.50){};
        \node[Noeud](1)at(1.00,-2.50){};
        \node[Noeud](11)at(11.00,-2.50){};
        \node[Noeud](13)at(13.00,-5.00){};
        \node[Noeud](3)at(3.00,0.00){};
        \node[Noeud](5)at(5.00,-5.00){};
        \node[Noeud](7)at(7.00,-7.50){};
        \node[Noeud](9)at(9.00,-10.00){};
        \draw[Arete](0)--(1);
        \draw[Arete](1)edge[]node[EtiqArete]{\begin{math}1\end{math}}(3);
        \draw[Arete](10)--(9);
        \draw[Arete](11)edge[]node[EtiqArete]{\begin{math}2\end{math}}(3);
        \draw[Arete](12)--(13);
        \draw[Arete](13)edge[]node[EtiqArete]{\begin{math}2\end{math}}(11);
        \draw[Arete](14)--(13);
        \draw[Arete](2)--(1);
        \draw[Arete](4)--(5);
        \draw[Arete](5)edge[]node[EtiqArete]{\begin{math}1\end{math}}(11);
        \draw[Arete](6)--(7);
        \draw[Arete](7)edge[]node[EtiqArete]{\begin{math}1\end{math}}(5);
        \draw[Arete](8)--(9);
        \draw[Arete](9)edge[]node[EtiqArete]{\begin{math}1\end{math}}(7);
        \node(r)at(3.00,2.50){};
        \draw[Arete](r)--(3);
    \end{tikzpicture}
    \end{split}
    +
    \begin{split}
    \begin{tikzpicture}[xscale=.22,yscale=.15]
        \node[Feuille](0)at(0.00,-5.00){};
        \node[Feuille](10)at(10.00,-10.00){};
        \node[Feuille](12)at(12.00,-7.50){};
        \node[Feuille](14)at(14.00,-7.50){};
        \node[Feuille](2)at(2.00,-5.00){};
        \node[Feuille](4)at(4.00,-7.50){};
        \node[Feuille](6)at(6.00,-12.50){};
        \node[Feuille](8)at(8.00,-12.50){};
        \node[Noeud](1)at(1.00,-2.50){};
        \node[Noeud](11)at(11.00,-2.50){};
        \node[Noeud](13)at(13.00,-5.00){};
        \node[Noeud](3)at(3.00,0.00){};
        \node[Noeud](5)at(5.00,-5.00){};
        \node[Noeud](7)at(7.00,-10.00){};
        \node[Noeud](9)at(9.00,-7.50){};
        \draw[Arete](0)--(1);
        \draw[Arete](1)edge[]node[EtiqArete]{\begin{math}1\end{math}}(3);
        \draw[Arete](10)--(9);
        \draw[Arete](11)edge[]node[EtiqArete]{\begin{math}2\end{math}}(3);
        \draw[Arete](12)--(13);
        \draw[Arete](13)edge[]node[EtiqArete]{\begin{math}2\end{math}}(11);
        \draw[Arete](14)--(13);
        \draw[Arete](2)--(1);
        \draw[Arete](4)--(5);
        \draw[Arete](5)edge[]node[EtiqArete]{\begin{math}1\end{math}}(11);
        \draw[Arete](6)--(7);
        \draw[Arete](7)edge[]node[EtiqArete]{\begin{math}1\end{math}}(9);
        \draw[Arete](8)--(7);
        \draw[Arete](9)edge[]node[EtiqArete]{\begin{math}1\end{math}}(5);
        \node(r)at(3.00,2.50){};
        \draw[Arete](r)--(3);
    \end{tikzpicture}
    \end{split}
    +
    \begin{split}
    \begin{tikzpicture}[xscale=.22,yscale=.15]
        \node[Feuille](0)at(0.00,-5.00){};
        \node[Feuille](10)at(10.00,-7.50){};
        \node[Feuille](12)at(12.00,-7.50){};
        \node[Feuille](14)at(14.00,-7.50){};
        \node[Feuille](2)at(2.00,-5.00){};
        \node[Feuille](4)at(4.00,-10.00){};
        \node[Feuille](6)at(6.00,-12.50){};
        \node[Feuille](8)at(8.00,-12.50){};
        \node[Noeud](1)at(1.00,-2.50){};
        \node[Noeud](11)at(11.00,-2.50){};
        \node[Noeud](13)at(13.00,-5.00){};
        \node[Noeud](3)at(3.00,0.00){};
        \node[Noeud](5)at(5.00,-7.50){};
        \node[Noeud](7)at(7.00,-10.00){};
        \node[Noeud](9)at(9.00,-5.00){};
        \draw[Arete](0)--(1);
        \draw[Arete](1)edge[]node[EtiqArete]{\begin{math}1\end{math}}(3);
        \draw[Arete](10)--(9);
        \draw[Arete](11)edge[]node[EtiqArete]{\begin{math}2\end{math}}(3);
        \draw[Arete](12)--(13);
        \draw[Arete](13)edge[]node[EtiqArete]{\begin{math}2\end{math}}(11);
        \draw[Arete](14)--(13);
        \draw[Arete](2)--(1);
        \draw[Arete](4)--(5);
        \draw[Arete](5)edge[]node[EtiqArete]{\begin{math}3\end{math}}(9);
        \draw[Arete](6)--(7);
        \draw[Arete](7)edge[]node[EtiqArete]{\begin{math}1\end{math}}(5);
        \draw[Arete](8)--(7);
        \draw[Arete](9)edge[]node[EtiqArete]{\begin{math}1\end{math}}(11);
        \node(r)at(3.00,2.50){};
        \draw[Arete](r)--(3);
    \end{tikzpicture}
    \end{split}\,,
\end{multline}
and
\begin{multline}
    \begin{split}
    \begin{tikzpicture}[xscale=.25,yscale=.2]
        \node[Feuille](0)at(0.00,-4.50){};
        \node[Feuille](2)at(2.00,-4.50){};
        \node[Feuille](4)at(4.00,-4.50){};
        \node[Feuille](6)at(6.00,-6.75){};
        \node[Feuille](8)at(8.00,-6.75){};
        \node[Noeud](1)at(1.00,-2.25){};
        \node[Noeud](3)at(3.00,0.00){};
        \node[Noeud](5)at(5.00,-2.25){};
        \node[Noeud](7)at(7.00,-4.50){};
        \draw[Arete](0)--(1);
        \draw[Arete](1)edge[]node[EtiqArete]{\begin{math}1\end{math}}(3);
        \draw[Arete](2)--(1);
        \draw[Arete](4)--(5);
        \draw[Arete](5)edge[]node[EtiqArete]{\begin{math}3\end{math}}(3);
        \draw[Arete](6)--(7);
        \draw[Arete](7)edge[]node[EtiqArete]{\begin{math}1\end{math}}(5);
        \draw[Arete](8)--(7);
        \node(r)at(3.00,1.7){};
        \draw[Arete](r)--(3);
    \end{tikzpicture}
    \end{split}
    \DDendr_2
    \begin{split}
    \begin{tikzpicture}[xscale=.25,yscale=.2]
        \node[Feuille](0)at(0.00,-4.67){};
        \node[Feuille](2)at(2.00,-4.67){};
        \node[Feuille](4)at(4.00,-4.67){};
        \node[Feuille](6)at(6.00,-4.67){};
        \node[Noeud](1)at(1.00,-2.33){};
        \node[Noeud](3)at(3.00,0.00){};
        \node[Noeud](5)at(5.00,-2.33){};
        \draw[Arete](0)--(1);
        \draw[Arete](1)edge[]node[EtiqArete]{\begin{math}1\end{math}}(3);
        \draw[Arete](2)--(1);
        \draw[Arete](4)--(5);
        \draw[Arete](5)edge[]node[EtiqArete]{\begin{math}2\end{math}}(3);
        \draw[Arete](6)--(5);
        \node(r)at(3.00,1.7){};
        \draw[Arete](r)--(3);
    \end{tikzpicture}
    \end{split}
    =
    \begin{split}
    \begin{tikzpicture}[xscale=.22,yscale=.15]
        \node[Feuille](0)at(0.00,-7.50){};
        \node[Feuille](10)at(10.00,-12.50){};
        \node[Feuille](12)at(12.00,-5.00){};
        \node[Feuille](14)at(14.00,-5.00){};
        \node[Feuille](2)at(2.00,-7.50){};
        \node[Feuille](4)at(4.00,-7.50){};
        \node[Feuille](6)at(6.00,-10.00){};
        \node[Feuille](8)at(8.00,-12.50){};
        \node[Noeud](1)at(1.00,-5.00){};
        \node[Noeud](11)at(11.00,0.00){};
        \node[Noeud](13)at(13.00,-2.50){};
        \node[Noeud](3)at(3.00,-2.50){};
        \node[Noeud](5)at(5.00,-5.00){};
        \node[Noeud](7)at(7.00,-7.50){};
        \node[Noeud](9)at(9.00,-10.00){};
        \draw[Arete](0)--(1);
        \draw[Arete](1)edge[]node[EtiqArete]{\begin{math}1\end{math}}(3);
        \draw[Arete](10)--(9);
        \draw[Arete](12)--(13);
        \draw[Arete](13)edge[]node[EtiqArete]{\begin{math}2\end{math}}(11);
        \draw[Arete](14)--(13);
        \draw[Arete](2)--(1);
        \draw[Arete](3)edge[]node[EtiqArete]{\begin{math}1\end{math}}(11);
        \draw[Arete](4)--(5);
        \draw[Arete](5)edge[]node[EtiqArete]{\begin{math}1\end{math}}(3);
        \draw[Arete](6)--(7);
        \draw[Arete](7)edge[]node[EtiqArete]{\begin{math}1\end{math}}(5);
        \draw[Arete](8)--(9);
        \draw[Arete](9)edge[]node[EtiqArete]{\begin{math}1\end{math}}(7);
        \node(r)at(11.00,2.50){};
        \draw[Arete](r)--(11);
    \end{tikzpicture}
    \end{split}
    +
    \begin{split}
    \begin{tikzpicture}[xscale=.22,yscale=.15]
        \node[Feuille](0)at(0.00,-7.50){};
        \node[Feuille](10)at(10.00,-10.00){};
        \node[Feuille](12)at(12.00,-5.00){};
        \node[Feuille](14)at(14.00,-5.00){};
        \node[Feuille](2)at(2.00,-7.50){};
        \node[Feuille](4)at(4.00,-7.50){};
        \node[Feuille](6)at(6.00,-12.50){};
        \node[Feuille](8)at(8.00,-12.50){};
        \node[Noeud](1)at(1.00,-5.00){};
        \node[Noeud](11)at(11.00,0.00){};
        \node[Noeud](13)at(13.00,-2.50){};
        \node[Noeud](3)at(3.00,-2.50){};
        \node[Noeud](5)at(5.00,-5.00){};
        \node[Noeud](7)at(7.00,-10.00){};
        \node[Noeud](9)at(9.00,-7.50){};
        \draw[Arete](0)--(1);
        \draw[Arete](1)edge[]node[EtiqArete]{\begin{math}1\end{math}}(3);
        \draw[Arete](10)--(9);
        \draw[Arete](12)--(13);
        \draw[Arete](13)edge[]node[EtiqArete]{\begin{math}2\end{math}}(11);
        \draw[Arete](14)--(13);
        \draw[Arete](2)--(1);
        \draw[Arete](3)edge[]node[EtiqArete]{\begin{math}1\end{math}}(11);
        \draw[Arete](4)--(5);
        \draw[Arete](5)edge[]node[EtiqArete]{\begin{math}1\end{math}}(3);
        \draw[Arete](6)--(7);
        \draw[Arete](7)edge[]node[EtiqArete]{\begin{math}1\end{math}}(9);
        \draw[Arete](8)--(7);
        \draw[Arete](9)edge[]node[EtiqArete]{\begin{math}1\end{math}}(5);
        \node(r)at(11.00,2.50){};
        \draw[Arete](r)--(11);
    \end{tikzpicture}
    \end{split} \\
    +
    \begin{split}
    \begin{tikzpicture}[xscale=.22,yscale=.15]
        \node[Feuille](0)at(0.00,-7.50){};
        \node[Feuille](10)at(10.00,-7.50){};
        \node[Feuille](12)at(12.00,-5.00){};
        \node[Feuille](14)at(14.00,-5.00){};
        \node[Feuille](2)at(2.00,-7.50){};
        \node[Feuille](4)at(4.00,-10.00){};
        \node[Feuille](6)at(6.00,-12.50){};
        \node[Feuille](8)at(8.00,-12.50){};
        \node[Noeud](1)at(1.00,-5.00){};
        \node[Noeud](11)at(11.00,0.00){};
        \node[Noeud](13)at(13.00,-2.50){};
        \node[Noeud](3)at(3.00,-2.50){};
        \node[Noeud](5)at(5.00,-7.50){};
        \node[Noeud](7)at(7.00,-10.00){};
        \node[Noeud](9)at(9.00,-5.00){};
        \draw[Arete](0)--(1);
        \draw[Arete](1)edge[]node[EtiqArete]{\begin{math}1\end{math}}(3);
        \draw[Arete](10)--(9);
        \draw[Arete](12)--(13);
        \draw[Arete](13)edge[]node[EtiqArete]{\begin{math}2\end{math}}(11);
        \draw[Arete](14)--(13);
        \draw[Arete](2)--(1);
        \draw[Arete](3)edge[]node[EtiqArete]{\begin{math}1\end{math}}(11);
        \draw[Arete](4)--(5);
        \draw[Arete](5)edge[]node[EtiqArete]{\begin{math}3\end{math}}(9);
        \draw[Arete](6)--(7);
        \draw[Arete](7)edge[]node[EtiqArete]{\begin{math}1\end{math}}(5);
        \draw[Arete](8)--(7);
        \draw[Arete](9)edge[]node[EtiqArete]{\begin{math}1\end{math}}(3);
        \node(r)at(11.00,2.50){};
        \draw[Arete](r)--(11);
    \end{tikzpicture}
    \end{split}
    +
    \begin{split}
    \begin{tikzpicture}[xscale=.22,yscale=.15]
        \node[Feuille](0)at(0.00,-10.00){};
        \node[Feuille](10)at(10.00,-5.00){};
        \node[Feuille](12)at(12.00,-5.00){};
        \node[Feuille](14)at(14.00,-5.00){};
        \node[Feuille](2)at(2.00,-10.00){};
        \node[Feuille](4)at(4.00,-10.00){};
        \node[Feuille](6)at(6.00,-12.50){};
        \node[Feuille](8)at(8.00,-12.50){};
        \node[Noeud](1)at(1.00,-7.50){};
        \node[Noeud](11)at(11.00,0.00){};
        \node[Noeud](13)at(13.00,-2.50){};
        \node[Noeud](3)at(3.00,-5.00){};
        \node[Noeud](5)at(5.00,-7.50){};
        \node[Noeud](7)at(7.00,-10.00){};
        \node[Noeud](9)at(9.00,-2.50){};
        \draw[Arete](0)--(1);
        \draw[Arete](1)edge[]node[EtiqArete]{\begin{math}1\end{math}}(3);
        \draw[Arete](10)--(9);
        \draw[Arete](12)--(13);
        \draw[Arete](13)edge[]node[EtiqArete]{\begin{math}2\end{math}}(11);
        \draw[Arete](14)--(13);
        \draw[Arete](2)--(1);
        \draw[Arete](3)edge[]node[EtiqArete]{\begin{math}2\end{math}}(9);
        \draw[Arete](4)--(5);
        \draw[Arete](5)edge[]node[EtiqArete]{\begin{math}3\end{math}}(3);
        \draw[Arete](6)--(7);
        \draw[Arete](7)edge[]node[EtiqArete]{\begin{math}1\end{math}}(5);
        \draw[Arete](8)--(7);
        \draw[Arete](9)edge[]node[EtiqArete]{\begin{math}1\end{math}}(11);
        \node(r)at(11.00,2.50){};
        \draw[Arete](r)--(11);
    \end{tikzpicture}
    \end{split}\,.
\end{multline}
\medskip

\begin{Lemme} \label{lem:produit_gamma_dendriforme}
    For any integer $\gamma \geq 0$, the vector space
    $\AlgLibre_{\Dendr_\gamma}$ of $\gamma$-edge valued binary trees
    endowed with the operations $\GDendr_a$, $\DDendr_a$, $a \in [\gamma]$,
    is a $\gamma$-polydendriform algebra.
\end{Lemme}
\begin{proof}
    We have to check that the operations $\GDendr_a$, $\DDendr_a$,
    $a \in [\gamma]$, of $\AlgLibre_{\Dendr_\gamma}$ satisfy
    Relations~\eqref{equ:relation_dendr_gamma_1_concise},
    \eqref{equ:relation_dendr_gamma_2_concise},
    and~\eqref{equ:relation_dendr_gamma_3_concise} of
    $\gamma$-polydendriform algebras. Let $\Rfr$, $\Sfr$, and $\Tfr$ be
    three $\gamma$-edge valued binary trees and $a, a' \in [\gamma]$.
    \smallskip

    Denote by $\Sfr_1$ (resp. $\Sfr_2$) the left subtree (resp. right
    subtree) of $\Sfr$ and by $x$ (resp. $y$) the label of the left
    (resp. right) edge incident to the root of $\Sfr$. We have
    \begin{multline}
        (\Rfr \DDendr_{a'} \Sfr) \GDendr_a \Tfr  =
        \left(\Rfr \DDendr_{a'} \ArbreBinValue{x}{y}{\Sfr_1}{\Sfr_2}\right)
        \GDendr_a \Tfr
        = \left(\ArbreBinValue{z}{y}{\Rfr \DDendr_{a'} \Sfr_1}{\Sfr_2}
        +
        \ArbreBinValue{z}{y}{\Rfr \GDendr_x \Sfr_1}{\Sfr_2}\right)
        \GDendr_a \Tfr
        \displaybreak[0]
        \\
        = \ArbreBinValue{z}{t}{\Rfr \DDendr_{a'} \Sfr_1}
                {\Sfr_2 \GDendr_a \Tfr}
        +
        \ArbreBinValue{z}{t}{\Rfr \DDendr_{a'} \Sfr_1}
                {\Sfr_2 \DDendr_y \Tfr}
        +
        \ArbreBinValue{z}{t}{\Rfr \GDendr_x \Sfr_1}
                {\Sfr_2 \GDendr_a \Tfr}
        +
        \ArbreBinValue{z}{t}{\Rfr \GDendr_x \Sfr_1}
                {\Sfr_2 \DDendr_y \Tfr}
        \displaybreak[0]
        \\
        = \Rfr \DDendr_{a'}
        \left(\ArbreBinValue{x}{t}{\Sfr_1}{\Sfr_2 \GDendr_a \Tfr}
        +
        \ArbreBinValue{x}{t}{\Sfr_1}{\Sfr_2 \DDendr_y \Tfr}\right)
        = \Rfr \DDendr_{a'}
        \left(\ArbreBinValue{x}{y}{\Sfr_1}{\Sfr_2}
        \GDendr_a \Tfr \right) =
        \Rfr \DDendr_{a'} (\Sfr \GDendr_a \Tfr),
    \end{multline}
    where $z := a' \Min x$ and $t := a \Min y$.
    This shows that~\eqref{equ:relation_dendr_gamma_1_concise}
    is satisfied in $\AlgLibre_{\Dendr_\gamma}$.
    \medskip

    We now prove that
    Relations~\eqref{equ:relation_dendr_gamma_2_concise}
    and~\eqref{equ:relation_dendr_gamma_3_concise} hold
    by induction on the sum of the number of internal nodes of $\Rfr$, $\Sfr$,
    and $\Tfr$. Base case holds when all these trees have exactly one
    internal node, and since
    \begin{multline}
        \left(\Noeud \GDendr_{a'} \Noeud\right) \GDendr_a \Noeud
        - \Noeud \GDendr_{a \Min a'} \left( \Noeud \GDendr_a \Noeud \right)
        - \Noeud \GDendr_{a \Min a'} \left( \Noeud \DDendr_{a'} \Noeud \right)
        \displaybreak[0]
        \\
        =
        \begin{split}
        \begin{tikzpicture}[xscale=.3,yscale=.25]
            \node[Feuille](0)at(0.00,-1.67){};
            \node[Feuille](2)at(2.00,-3.33){};
            \node[Feuille](4)at(4.00,-3.33){};
            \node[Noeud](1)at(1.00,0.00){};
            \node[Noeud](3)at(3.00,-1.67){};
            \draw[Arete](0)--(1);
            \draw[Arete](2)--(3);
            \draw[Arete](3)edge[]node[EtiqArete]{\begin{math}a'\end{math}}(1);
            \draw[Arete](4)--(3);
            \node(r)at(1.00,1.67){};
            \draw[Arete](r)--(1);
        \end{tikzpicture}
        \end{split}
        \GDendr_a \Noeud -
        \Noeud \GDendr_{a \Min a'}
        \begin{split}
        \begin{tikzpicture}[xscale=.3,yscale=.25]
            \node[Feuille](0)at(0.00,-1.67){};
            \node[Feuille](2)at(2.00,-3.33){};
            \node[Feuille](4)at(4.00,-3.33){};
            \node[Noeud](1)at(1.00,0.00){};
            \node[Noeud](3)at(3.00,-1.67){};
            \draw[Arete](0)--(1);
            \draw[Arete](2)--(3);
            \draw[Arete](3)edge[]node[EtiqArete]{\begin{math}a\end{math}}(1);
            \draw[Arete](4)--(3);
            \node(r)at(1.00,1.67){};
            \draw[Arete](r)--(1);
        \end{tikzpicture}
        \end{split}
        - \Noeud \GDendr_{a \Min a'}
        \begin{split}
        \begin{tikzpicture}[xscale=.3,yscale=.25]
            \node[Feuille](0)at(0.00,-3.33){};
            \node[Feuille](2)at(2.00,-3.33){};
            \node[Feuille](4)at(4.00,-1.67){};
            \node[Noeud](1)at(1.00,-1.67){};
            \node[Noeud](3)at(3.00,0.00){};
            \draw[Arete](0)--(1);
            \draw[Arete](1)edge[]node[EtiqArete]{\begin{math}a'\end{math}}(3);
            \draw[Arete](2)--(1);
            \draw[Arete](4)--(3);
            \node(r)at(3.00,1.67){};
            \draw[Arete](r)--(3);
        \end{tikzpicture}
        \end{split}
        \\
        =
        \begin{split}
        \begin{tikzpicture}[xscale=.3,yscale=.25]
            \node[Feuille](0)at(0.00,-1.75){};
            \node[Feuille](2)at(2.00,-3.50){};
            \node[Feuille](4)at(4.00,-5.25){};
            \node[Feuille](6)at(6.00,-5.25){};
            \node[Noeud](1)at(1.00,0.00){};
            \node[Noeud](3)at(3.00,-1.75){};
            \node[Noeud](5)at(5.00,-3.50){};
            \draw[Arete](0)--(1);
            \draw[Arete](2)--(3);
            \draw[Arete](3)edge[]node[EtiqArete]{\begin{math}z\end{math}}(1);
            \draw[Arete](4)--(5);
            \draw[Arete](5)edge[]node[EtiqArete]{\begin{math}a\end{math}}(3);
            \draw[Arete](6)--(5);
            \node(r)at(1.00,1.75){};
            \draw[Arete](r)--(1);
        \end{tikzpicture}
        \end{split}
        +
        \begin{split}
        \begin{tikzpicture}[xscale=.3,yscale=.25]
            \node[Feuille](0)at(0.00,-1.75){};
            \node[Feuille](2)at(2.00,-5.25){};
            \node[Feuille](4)at(4.00,-5.25){};
            \node[Feuille](6)at(6.00,-3.50){};
            \node[Noeud](1)at(1.00,0.00){};
            \node[Noeud](3)at(3.00,-3.50){};
            \node[Noeud](5)at(5.00,-1.75){};
            \draw[Arete](0)--(1);
            \draw[Arete](2)--(3);
            \draw[Arete](3)edge[]node[EtiqArete]{\begin{math}a'\end{math}}(5);
            \draw[Arete](4)--(3);
            \draw[Arete](5)edge[]node[EtiqArete]{\begin{math}z\end{math}}(1);
            \draw[Arete](6)--(5);
            \node(r)at(1.00,1.75){};
            \draw[Arete](r)--(1);
        \end{tikzpicture}
        \end{split}
        -
        \begin{split}
        \begin{tikzpicture}[xscale=.3,yscale=.25]
            \node[Feuille](0)at(0.00,-1.75){};
            \node[Feuille](2)at(2.00,-3.50){};
            \node[Feuille](4)at(4.00,-5.25){};
            \node[Feuille](6)at(6.00,-5.25){};
            \node[Noeud](1)at(1.00,0.00){};
            \node[Noeud](3)at(3.00,-1.75){};
            \node[Noeud](5)at(5.00,-3.50){};
            \draw[Arete](0)--(1);
            \draw[Arete](2)--(3);
            \draw[Arete](3)edge[]node[EtiqArete]{\begin{math}z\end{math}}(1);
            \draw[Arete](4)--(5);
            \draw[Arete](5)edge[]node[EtiqArete]{\begin{math}a\end{math}}(3);
            \draw[Arete](6)--(5);
            \node(r)at(1.00,1.75){};
            \draw[Arete](r)--(1);
        \end{tikzpicture}
        \end{split}
        -
        \begin{split}
        \begin{tikzpicture}[xscale=.3,yscale=.25]
            \node[Feuille](0)at(0.00,-1.75){};
            \node[Feuille](2)at(2.00,-5.25){};
            \node[Feuille](4)at(4.00,-5.25){};
            \node[Feuille](6)at(6.00,-3.50){};
            \node[Noeud](1)at(1.00,0.00){};
            \node[Noeud](3)at(3.00,-3.50){};
            \node[Noeud](5)at(5.00,-1.75){};
            \draw[Arete](0)--(1);
            \draw[Arete](2)--(3);
            \draw[Arete](3)edge[]node[EtiqArete]{\begin{math}a'\end{math}}(5);
            \draw[Arete](4)--(3);
            \draw[Arete](5)edge[]node[EtiqArete]{\begin{math}z\end{math}}(1);
            \draw[Arete](6)--(5);
            \node(r)at(1.00,1.75){};
            \draw[Arete](r)--(1);
        \end{tikzpicture}
        \end{split}
        \\ = 0,
    \end{multline}
    where $z := a \Min a'$,
    \eqref{equ:relation_dendr_gamma_2_concise} holds on trees
    with exactly one internal node. For the same arguments, we can show
    that~\eqref{equ:relation_dendr_gamma_3_concise} holds on
    trees with exactly one internal node. Denote now by $\Rfr_1$ (resp.
    $\Rfr_2$) the left subtree (resp. right subtree) of $\Rfr$ and by $x$
    (resp. $y$) the label of the left (resp. right) edge incident to the
    root of $\Rfr$. We have
    \begin{multline} \label{equ:produit_gamma_dendriforme_expr}
        (\Rfr \GDendr_{a'} \Sfr) \GDendr_a \Tfr
        - \Rfr \GDendr_{a \Min a'} (\Sfr \GDendr_a \Tfr)
        - \Rfr \GDendr_{a \Min a'} (\Sfr \DDendr_{a'} \Tfr)
        \\[1em]
        =
        \left(\ArbreBinValue{x}{y}{\Rfr_1}{\Rfr_2}
        \GDendr_{a'} \Sfr \right) \GDendr_a \Tfr
        -
        \ArbreBinValue{x}{y}{\Rfr_1}{\Rfr_2}
        \GDendr_{a \Min a'} (\Sfr \GDendr_a \Tfr)
        -
        \ArbreBinValue{x}{y}{\Rfr_1}{\Rfr_2}
        \GDendr_{a \Min a'} (\Sfr \DDendr_{a'} \Tfr)
        \displaybreak[0]
        \\
        =
        \left(\ArbreBinValue{x}{z}{\Rfr_1}{\Rfr_2 \GDendr_{a'} \Sfr}
        +
        \ArbreBinValue{x}{z}{\Rfr_1}{\Rfr_2 \DDendr_y \Sfr}
        \right) \GDendr_a \Tfr \\
        -
        \ArbreBinValue{x}{y}{\Rfr_1}{\Rfr_2}
        \GDendr_{a \Min a'} (\Sfr \GDendr_a \Tfr)
        -
        \ArbreBinValue{x}{y}{\Rfr_1}{\Rfr_2}
        \GDendr_{a \Min a'} (\Sfr \DDendr_{a'} \Tfr)
        \displaybreak[0]
        \\
        =
        \ArbreBinValue{x}{t}{\Rfr_1}{(\Rfr_2 \GDendr_{a'} \Sfr) \GDendr_a \Tfr}
        +
        \ArbreBinValue{x}{t}{\Rfr_1}{(\Rfr_2 \GDendr_{a'} \Sfr) \DDendr_z \Tfr}
        +
        \ArbreBinValue{x}{t}{\Rfr_1}{(\Rfr_2 \DDendr_y \Sfr) \GDendr_a \Tfr}
        +
        \ArbreBinValue{x}{t}{\Rfr_1}{(\Rfr_2 \DDendr_y \Sfr) \DDendr_z \Tfr} \\
        -
        \ArbreBinValue{x}{t}{\Rfr_1}{\Rfr_2 \GDendr_u (\Sfr \GDendr_a \Tfr)}
        -
        \ArbreBinValue{x}{t}{\Rfr_1}{\Rfr_2 \DDendr_y (\Sfr \GDendr_a \Tfr)}
        -
        \ArbreBinValue{x}{t}{\Rfr_1}{\Rfr_2 \GDendr_u (\Sfr \DDendr_{a'} \Tfr)}
        -
        \ArbreBinValue{x}{t}{\Rfr_1}{\Rfr_2 \DDendr_y (\Sfr \DDendr_{a'} \Tfr)}\,,
    \end{multline}
    where $z := y \Min a'$, $t := z \Min a = y \Min a' \Min a$, and
    $u := a \Min a'$. Now, by induction hypothesis,
    Relation~\eqref{equ:relation_dendr_gamma_2_concise} holds
    on $\Rfr_2$, $\Sfr$, and $\Tfr$. Hence, the sum of the first, fifth,
    and seventh terms of~\eqref{equ:produit_gamma_dendriforme_expr} is
    zero. Again by induction hypothesis,
    Relation~\eqref{equ:relation_dendr_gamma_3_concise} holds
    on $\Rfr_2$, $\Sfr$, and $\Tfr$. Thus, the sum of the second, fourth,
    and last terms of~\eqref{equ:produit_gamma_dendriforme_expr} is zero.
    Finally, by what we just have proven in the first part of this proof,
    the sum of the third and sixth terms
    of~\eqref{equ:relation_dendr_gamma_3_concise} is zero.
    Therefore, \eqref{equ:produit_gamma_dendriforme_expr} is zero
    and~\eqref{equ:relation_dendr_gamma_2_concise}
    is satisfied in $\AlgLibre_{\Dendr_\gamma}$.
    \smallskip

    Finally, for the same arguments, we can show
    that~\eqref{equ:relation_dendr_gamma_3_concise}
    is satisfied in $\AlgLibre_{\Dendr_\gamma}$, implying the statement
    of the lemma.
\end{proof}
\medskip

\begin{Lemme} \label{lem:produit_gamma_engendre_dendriforme}
    For any integer $\gamma \geq 0$, the $\gamma$-pluriassociative
    algebra $\AlgLibre_{\Dendr_\gamma}$ of $\gamma$-edge valued binary
    trees endowed with the operations $\GDendr_a$, $\DDendr_a$,
    $a \in [\gamma]$, is generated by
    \begin{equation}
        \Noeud\,.
    \end{equation}
\end{Lemme}
\begin{proof}
    First, Lemma~\ref{lem:produit_gamma_dendriforme} shows that
    $\AlgLibre_{\Dendr_\gamma}$ is a $\gamma$-polydendriform algebra.
    Let $\Dca$ be the $\gamma$-polydendriform subalgebra of
    $\AlgLibre_{\Dendr_\gamma}$ generated by $\NoeudTexte$. Let us show
    that any $\gamma$-edge valued binary tree $\Tfr$ is in $\Dca$ by
    induction on the number $n$ of its internal nodes. When $n = 1$,
    $\Tfr = \NoeudTexte$ and hence the property is satisfied. Otherwise,
    let $\Tfr_1$ (resp. $\Tfr_2$) be the left (resp. right) subtree of
    the root of $\Tfr$ and denote by $x$ (resp. $y$) the label of the
    left (resp. right) edge incident to the root of $\Tfr$. Since
    $\Tfr_1$ and $\Tfr_2$ have less internal nodes than $\Tfr$, by
    induction hypothesis, $\Tfr_1$ and~$\Tfr_2$ are in $\Dca$. Moreover,
    by definition of the operations $\GDendr_a$, $\DDendr_a$,
    $a \in [\gamma]$, of $\AlgLibre_{\Dendr_\gamma}$, one has
    \begin{equation}
        \begin{split}\end{split}
        \left(\Tfr_1 \DDendr_x \Noeud\right) \GDendr_y \Tfr_2
        =
        \begin{split}
        \begin{tikzpicture}[xscale=.5,yscale=.4]
            \node(0)at(-.50,-1.50){\begin{math}\Tfr_1\end{math}};
            \node[Feuille](2)at(2.0,-1.00){};
            \node[Noeud](1)at(1.00,.50){};
            \draw[Arete](0)edge[]node[EtiqArete]
                {\begin{math}x\end{math}}(1);
            \draw[Arete](2)--(1);
            \node(r)at(1.00,1.5){};
            \draw[Arete](r)--(1);
        \end{tikzpicture}
        \end{split}
        \GDendr_y \Tfr_2
        =
        \ArbreBinValue{x}{y}{\Tfr_1}{\Tfr_2}
        = \Tfr,
    \end{equation}
    showing that $\Tfr$ also is in $\Dca$. Therefore,
    $\Dca$ is $\AlgLibre_{\Dendr_\gamma}$, showing that
    $\AlgLibre_{\Dendr_\gamma}$
    is generated by~$\NoeudTexte$.
\end{proof}
\medskip

\begin{Theoreme} \label{thm:algebre_dendr_gamma_libre}
    For any integer $\gamma \geq 0$, the vector space
    $\AlgLibre_{\Dendr_\gamma}$ of $\gamma$-edge valued binary trees
    endowed with the operations $\GDendr_a$, $\DDendr_a$, $a \in [\gamma]$,
    is the free $\gamma$-polydendriform algebra over one generator.
\end{Theoreme}
\begin{proof}
    By Lemmas~\ref{lem:produit_gamma_dendriforme}
    and~\ref{lem:produit_gamma_engendre_dendriforme},
    $\AlgLibre_{\Dendr_\gamma}$ is a $\gamma$-polydendriform algebra
    over one generator.
    \smallskip

    Moreover, since by Proposition~\ref{prop:serie_hilbert_dendr_gamma},
    for any $n \geq 1$, the dimension of $\AlgLibre_{\Dendr_\gamma}(n)$
    is the same as the dimension of $\Dendr_\gamma(n)$, there cannot be
    relations in $\AlgLibre_{\Dendr_\gamma}(n)$ involving $\Gfr$ that are
    not $\gamma$-polydendriform
    relations (see~\eqref{equ:relation_dendr_gamma_1_concise},
    \eqref{equ:relation_dendr_gamma_2_concise},
    and~\eqref{equ:relation_dendr_gamma_3_concise}). Hence,
    $\AlgLibre_{\Dendr_\gamma}$ is free as a $\gamma$-polydendriform
    algebra over one generator.
\end{proof}
\medskip

\section{Multiassociative operads} \label{sec:as_gamma}
There is a well-known diagram, whose definition is recalled below,
gathering the diassociative, associative, and  dendriform operads.
The main goal of this section is to define a one-parameter nonnegative
integer generalization of the associative operad to obtain a new version
of this diagram, suited to the context of pluriassociative and
polydendriform operads.
\medskip

\subsection{Two generalizations of the associative operad}
The associative operad is generated by one binary element. This operad
admits two different generalizations generated by $\gamma$ binary
elements with the particularity that one is the Koszul dual of the other.
We introduce and study in this section these two operads.
\medskip

\subsubsection{Nonsymmetric associative operad}
Recall that the {\em nonsymmetric associative operad}, or the
{\em associative operad} for short, is the operad $\As$ admitting the
presentation $\left(\GenLibre_\As, \RelLibre_\As\right)$, where
$\GenLibre_\As := \GenLibre_\As(2) := \{\MAs\}$ and $\RelLibre_\As$ is
generated by $\MAs \circ_1 \MAs - \MAs \circ_2 \MAs$. It admits the
following realization. For any $n \geq 1$, $\As(n)$ is the vector space
of dimension one generated by the corolla of arity $n$ and the partial
composition $\Cfr_1 \circ_i \Cfr_2$ where $\Cfr_1$ is the corolla of
arity $n$ and $\Cfr_2$ is the corolla of arity $m$ is the corolla of
arity $n + m - 1$ for all valid $i$.
\medskip

\subsubsection{Multiassociative operads} \label{subsubsec:as_gamma}
For any integer $\gamma \geq 0$, we define $\As_\gamma$ as the operad
admitting the presentation $\left(\GenAs, \RelAs\right)$, where
$\GenLibre_{\As_\gamma} := \GenLibre_{\As_\gamma}(2)
:= \{\MAs_a : a \in [\gamma]\}$ and $\RelLibre_{\As_\gamma}$ is generated
by
\begin{subequations}
\begin{equation} \label{equ:relation_as_gamma_1}
    \MAs_a \circ_1 \MAs_b - \MAs_b \circ_2 \MAs_b,
    \qquad a \leq b \in [\gamma],
\end{equation}
\begin{equation} \label{equ:relation_as_gamma_2}
    \MAs_b \circ_1 \MAs_a - \MAs_b \circ_2 \MAs_b,
    \qquad a < b \in [\gamma],
\end{equation}
\begin{equation} \label{equ:relation_as_gamma_3}
    \MAs_a \circ_2 \MAs_b - \MAs_b \circ_2 \MAs_b,
    \qquad a < b \in [\gamma],
\end{equation}
\begin{equation} \label{equ:relation_as_gamma_4}
    \MAs_b \circ_2 \MAs_a - \MAs_b \circ_2 \MAs_b,
    \qquad a < b \in [\gamma].
\end{equation}
\end{subequations}
This space of relations can be rephrased in a more compact way as the
space generated by
\begin{subequations}
\begin{equation} \label{equ:relation_as_gamma_1_concise}
    \MAs_a \circ_1 \MAs_{a'} - \MAs_{a \Max a'} \circ_2 \MAs_{a \Max a'},
    \qquad a, a' \in [\gamma],
\end{equation}
\begin{equation} \label{equ:relation_as_gamma_1_concise}
    \MAs_a \circ_2 \MAs_{a'} - \MAs_{a \Max a'} \circ_2 \MAs_{a \Max a'},
    \qquad a, a' \in [\gamma].
\end{equation}
\end{subequations}
We call $\As_\gamma$ the {\em $\gamma$-multiassociative operad}.
\medskip

It follows immediately that $\As_\gamma$ is a set-operad and that it
provides a generalization of the associative operad. The algebras over
$\As_\gamma$ are the $\gamma$-multiassociative algebras introduced in
Section~\ref{subsubsec:algebres_multiassociatives}.
\medskip

Let us now provide a realization of $\As_\gamma$. A {\em $\gamma$-corolla}
is a rooted tree with at most one internal node labeled on $[\gamma]$.
Denote by $\AlgLibre_{\As_\gamma}(n)$ the vector space of $\gamma$-corollas
of arity $n \geq 1$, by $\AlgLibre_{\As_\gamma}$ the graded vector space
of all $\gamma$-corollas, and let
\begin{equation}
    \MAs : \AlgLibre_{\As_\gamma} \otimes \AlgLibre_{\As_\gamma}
    \to \AlgLibre_{\As_\gamma}
\end{equation}
be the linear operation where, for any $\gamma$-corollas $\Cfr_1$ and
$\Cfr_2$, $\Cfr_1 \MAs \Cfr_2$ is the $\gamma$-corolla
with $n + m - 1$ leaves and labeled by $a \Max a'$ where $n$ (resp. $m$)
is the number of leaves of $\Cfr_1$ (resp. $\Cfr_2$) and $a$ (resp. $a'$)
is the label of $\Cfr_1$ (resp. $\Cfr_2$).
\medskip

\begin{Proposition} \label{prop:realisation_koszulite_as_gamma}
    For any integer $\gamma \geq 0$, the operad $\As_\gamma$ is the
    vector space $\AlgLibre_{\As_\gamma}$ of $\gamma$-corollas and its
    partial compositions satisfy, for any $\gamma$-corollas $\Cfr_1$ and
    $\Cfr_2$, $\Cfr_1 \circ_i \Cfr_2 = \Cfr_1 \MAs \Cfr_2$ for all valid
    integer $i$. Besides, $\As_\gamma$ is a Koszul operad and the set of
    right comb syntax trees of $\OpLibre\left(\GenAs\right)$ where all
    internal nodes have a same label forms a Poincaré-Birkhoff-Witt basis
    of~$\As_\gamma$.
\end{Proposition}
\begin{proof}
    In this proof, we consider that $\GenAs$ is totally ordered by the
    relation $\leq$ satisfying $\MAs_a \leq \MAs_b$ whenever
    $a \leq b \in [\gamma]$. It is immediate that the vector space
    $\AlgLibre_{\As_\gamma}$ endowed with the partial compositions
    described in the statement of the proposition is an operad. Let us
    prove that this operad  admits the presentation
    $\left(\GenAs, \RelAs\right)$.
    \smallskip

    For this purpose, consider the quadratic rewrite rule $\Recr_\gamma$ on
    $\OpLibre\left(\GenAs\right)$ satisfying
    \begin{subequations}
    \begin{equation} \label{equ:reecriture_as_gamma_1}
        \MAs_a \circ_1 \MAs_b
        \enspace \Recr_\gamma \enspace
        \MAs_b \circ_2 \MAs_b,
        \qquad a \leq b \in [\gamma],
    \end{equation}
    \begin{equation} \label{equ:reecriture_as_gamma_2}
        \MAs_b \circ_1 \MAs_a
        \enspace \Recr_\gamma \enspace
        \MAs_b \circ_2 \MAs_b,
        \qquad a < b \in [\gamma],
    \end{equation}
    \begin{equation} \label{equ:reecriture_as_gamma_3}
        \MAs_a \circ_2 \MAs_b
        \enspace \Recr_\gamma \enspace
        \MAs_b \circ_2 \MAs_b,
        \qquad a < b \in [\gamma],
    \end{equation}
    \begin{equation} \label{equ:reecriture_as_gamma_4}
        \MAs_b \circ_2 \MAs_a
        \enspace \Recr_\gamma \enspace
        \MAs_b \circ_2 \MAs_b,
        \qquad a < b \in [\gamma].
    \end{equation}
    \end{subequations}
    Observe first that the space induced by the operad congruence induced
    by $\Recr_\gamma$ is $\RelAs$
    (see~\eqref{equ:relation_as_gamma_1}---\eqref{equ:relation_as_gamma_4}).
    Moreover, $\Recr_\gamma$ is a terminating rewrite rule and its normal
    forms are right comb syntax trees of $\OpLibre\left(\GenAs\right)$
    where all internal nodes have a same label. Besides, one can show
    that for any syntax tree $\Tfr$ of $\OpLibre\left(\GenAs\right)$, we
    have $\Tfr \overset{*}{\Recr_\gamma} \Sfr$ with $\Sfr$ is a right
    comb syntax tree where all internal nodes labeled by the greatest
    label of $\Tfr$. Therefore, $\Recr_\gamma$ is a convergent rewrite
    rule and the operad $\As$, admitting by definition the presentation
    $\left(\GenAs, \RelAs\right)$, has bases indexed by such trees.
    \smallskip

    Now, let
    \begin{equation}
        \phi : \As_\gamma \simeq
        \OpLibre\left(\GenAs\right)/_{\left\langle \RelAs \right\rangle}
        \to \AlgLibre_{\As_\gamma}
    \end{equation}
    be the map satisfying $\phi(\pi(\MAs_a)) = \Cfr_a$ where $\Cfr_a$ is
    the $\gamma$-corolla of arity $2$ with internal node labeled by
    $a \in [\gamma]$ and $\pi : \OpLibre\left(\GenAs\right) \to
    \As_\gamma$ is the canonical surjection map. Since we have
    $\phi(\pi(x)) \circ_i \phi(\pi(y)) = \phi(\pi(x')) \circ_{i'} \phi(\pi(y'))$
    for all relations $x \circ_i y \Recr_\gamma x' \circ_{i'} y'$ of
    \eqref{equ:reecriture_as_gamma_1}---\eqref{equ:reecriture_as_gamma_4},
    $\phi$ extends in a unique way into an operad morphism. First, since
    the set $G_\gamma$ of all $\gamma$-corollas of arity two is a
    generating set of $\AlgLibre_{\As_\gamma}$ and the image of $\phi$
    contains $G_\gamma$, $\phi$ is surjective. Second, since by
    definition of $\AlgLibre_{\As_\gamma}$, the bases of
    $\AlgLibre_{\As_\gamma}$ are indexed by $\gamma$-corollas, in
    accordance with what we have shown in the previous paragraph of this
    proof, $\AlgLibre_{\As_\gamma}$ and $\As_\gamma$ are isomorphic
    as graded vector spaces. Hence, $\phi$ is an operad isomorphism,
    showing that $\As_\gamma$ admits the claimed realization.
    \smallskip

    Finally, the existence of the convergent rewrite rule $\Recr_\gamma$
    implies, by the Koszulity criterion~\cite{Hof10,DK10,LV12} we have
    reformulated in Section~\ref{subsubsec:dual_de_Koszul}, that
    $\As_\gamma$ is Koszul and that its Poincaré-Birkhoff-Witt basis is
    the one described in the statement of the proposition.
\end{proof}
\medskip

We have for instance in $\As_3$,
\begin{equation}
    \begin{split}
    \begin{tikzpicture}[xscale=.27,yscale=.23]
        \node[Feuille](0)at(0.00,-1.50){};
        \node[Feuille](2)at(1.00,-1.50){};
        \node[Feuille](3)at(2.00,-1.50){};
        \node[NoeudCor](1)at(1.00,0.00){\begin{math}2\end{math}};
        \draw[Arete](0)--(1);
        \draw[Arete](2)--(1);
        \draw[Arete](3)--(1);
        \node(r)at(1.00,1.7){};
        \draw[Arete](r)--(1);
    \end{tikzpicture}
    \end{split}
    \circ_1
    \begin{split}
    \begin{tikzpicture}[xscale=.27,yscale=.23]
        \node[Feuille](0)at(0.00,-1.50){};
        \node[Feuille](2)at(2.00,-1.50){};
        \node[NoeudCor](1)at(1.00,0.00){\begin{math}1\end{math}};
        \draw[Arete](0)--(1);
        \draw[Arete](2)--(1);
        \node(r)at(1.00,1.7){};
        \draw[Arete](r)--(1);
    \end{tikzpicture}
    \end{split}
    =
    \begin{split}
    \begin{tikzpicture}[xscale=.27,yscale=.23]
        \node[Feuille](0)at(0.00,-1.50){};
        \node[Feuille](1)at(1.00,-1.50){};
        \node[Feuille](3)at(3.00,-1.50){};
        \node[Feuille](4)at(4.00,-1.50){};
        \node[NoeudCor](2)at(2.00,0.00){\begin{math}2\end{math}};
        \draw[Arete](0)--(2);
        \draw[Arete](1)--(2);
        \draw[Arete](3)--(2);
        \draw[Arete](4)--(2);
        \node(r)at(2.00,1.7){};
        \draw[Arete](r)--(2);
    \end{tikzpicture}
    \end{split}\,,
\end{equation}
and
\begin{equation}
    \begin{split}
    \begin{tikzpicture}[xscale=.27,yscale=.23]
        \node[Feuille](0)at(0.00,-1.50){};
        \node[Feuille](2)at(2.00,-1.50){};
        \node[NoeudCor](1)at(1.00,0.00){\begin{math}2\end{math}};
        \draw[Arete](0)--(1);
        \draw[Arete](2)--(1);
        \node(r)at(1.00,1.7){};
        \draw[Arete](r)--(1);
    \end{tikzpicture}
    \end{split}
    \circ_2
    \begin{split}
    \begin{tikzpicture}[xscale=.27,yscale=.23]
        \node[Feuille](0)at(0.00,-1.50){};
        \node[Feuille](2)at(1.00,-1.50){};
        \node[Feuille](3)at(2.00,-1.50){};
        \node[NoeudCor](1)at(1.00,0.00){\begin{math}3\end{math}};
        \draw[Arete](0)--(1);
        \draw[Arete](2)--(1);
        \draw[Arete](3)--(1);
        \node(r)at(1.00,1.7){};
        \draw[Arete](r)--(1);
    \end{tikzpicture}
    \end{split}
    =
    \begin{split}
    \begin{tikzpicture}[xscale=.27,yscale=.23]
        \node[Feuille](0)at(0.00,-1.50){};
        \node[Feuille](1)at(1.00,-1.50){};
        \node[Feuille](3)at(3.00,-1.50){};
        \node[Feuille](4)at(4.00,-1.50){};
        \node[NoeudCor](2)at(2.00,0.00){\begin{math}3\end{math}};
        \draw[Arete](0)--(2);
        \draw[Arete](1)--(2);
        \draw[Arete](3)--(2);
        \draw[Arete](4)--(2);
        \node(r)at(2.00,1.7){};
        \draw[Arete](r)--(2);
    \end{tikzpicture}
    \end{split}\,.
\end{equation}
\medskip

We deduce from Proposition~\ref{prop:realisation_koszulite_as_gamma}
that the Hilbert series of $\As_\gamma$ satisfies
\begin{equation} \label{equ:serie_hilbert_as_gamma}
    \Hca_{\As_\gamma}(t) = \frac{t + (\gamma - 1)t^2}{1 - t}.
\end{equation}
and that for all $n \geq 2$, $\dim \As_\gamma(n) = \gamma$.
\medskip

\subsubsection{Dual multiassociative operads} \label{subsubsec:das_gamma}
Since $\As_\gamma$ is a binary and quadratic operad, its admits a Koszul
dual, denoted by $\DAs_\gamma$ and called
{\em $\gamma$-dual multiassociative operad}. The presentation of this
operad is provided by next result.
\medskip

\begin{Proposition} \label{prop:presentation_as_gamma_duale}
    For any integer $\gamma \geq 0$, the operad $\DAs_\gamma$ admits the
    following presentation. It is generated by
    $\GenDAs := \GenDAs(2) := \{\MDAsA_a : a \in [\gamma]\}$ and its space
    of relations $\RelDAs$ is generated by
    \begin{equation} \label{equ:relation_das_gamma_alternative}
        \MDAsA_b \circ_1 \MDAsA_b - \MDAsA_b \circ_2 \MDAsA_b
        + \left(\sum_{a < b} \MDAsA_a \circ_1 \MDAsA_b
        + \MDAsA_b \circ_1 \MDAsA_a - \MDAsA_a \circ_2 \MDAsA_b
        - \MDAsA_b \circ_2 \MDAsA_a\right),
        \qquad b \in [\gamma].
    \end{equation}
\end{Proposition}
\begin{proof}
    By a straightforward computation, and by identifying $\MDAsA_a$ with
    $\MAs_a$ for any $a \in [\gamma]$, we obtain that the space $\RelDAs$
    of the statement of the proposition satisfies $\RelDAs^\perp = \RelAs$.
    Hence, $\DAs$ admits the claimed presentation.
\end{proof}
\medskip

For any integer $\gamma \geq 0$, let $\MDAs_b$, $b \in [\gamma]$, the
elements of $\OpLibre\left(\GenDAs\right)$ defined by
\begin{equation} \label{equ:definition_operateur_das_gamma}
    \MDAs_b := \sum_{a \in [b]} \MDAsA_a.
\end{equation}
Then, since for all $b \in [\gamma]$ we have
\begin{equation}
    \MDAsA_b =
    \begin{cases}
        \MDAs_1 & \mbox{if } b = 1, \\
        \MDAs_b - \MDAs_{b - 1} & \mbox{otherwise},
    \end{cases}
\end{equation}
by triangularity, the family $\GenDAs' := \{\MDAs_b : b \in [\gamma]\}$
forms a basis of $\OpLibre\left(\GenDAs\right)(2)$ and then, generates
$\OpLibre\left(\GenDAs\right)$ as an operad. Let us now express a
presentation of $\DAs_\gamma$ through the family $\GenDAs'$.

\begin{Proposition} \label{prop:autre_presentation_das_gamma}
    For any integer $\gamma \geq 0$, the operad $\DAs_\gamma$ admits the
    following presentation. It is generated by $\GenDAs'$ and its space
    of relations $\RelDAs'$ is generated by
    \begin{equation} \label{equ:relation_das_gamma}
        \MDAs_a \circ_1 \MDAs_a - \MDAs_a \circ_2 \MDAs_a,
        \qquad a \in [\gamma].
    \end{equation}
\end{Proposition}
\begin{proof}
    Let us show that $\RelDAs'$ is equal to the space of relations
    $\RelDAs$ of $\DAs_\gamma$ defined in the statement of
    Proposition~\ref{prop:presentation_as_gamma_duale}. By this last
    proposition, for any $x \in \OpLibre\left(\GenDAs\right)(3)$,
    $x$ is in $\RelDAs$ if and only if $\pi(x) = 0$ where
    $\pi : \OpLibre\left(\GenDAs\right) \to \DAs$ is the canonical
    surjection map. By a straightforward computation, by
    expanding~\eqref{equ:relation_das_gamma} over the elements $\MDAsA_a$,
    $a \in [\gamma]$, by
    using~\eqref{equ:definition_operateur_das_gamma} we obtain
    that~\eqref{equ:relation_das_gamma} can be expressed as a sum
    of elements of $\RelDAs$. This implies that $\pi(x) = 0$ and hence
    that $\RelDAs'$ is a subspace of $\RelDAs$.
    \smallskip

    Now, one can observe that for all $a \in [\gamma]$, the
    elements~\eqref{equ:relation_das_gamma} are linearly independent.
    Then, $\RelDAs'$ has dimension $\gamma$ which is also, by
    Proposition~\ref{prop:presentation_as_gamma_duale}, the dimension of
    $\RelDAs$. The statement of the proposition follows.
\end{proof}
\medskip

Observe, from the presentation provided by
Proposition~\ref{prop:autre_presentation_das_gamma} of $\DAs_\gamma$,
that $\DAs_2$ is the operad denoted by $\DeuxAs$ in \cite{LR06}.
\medskip

Notice that the Koszul dual of $\DAs_\gamma$ through its presentation
$\left(\GenDAs', \RelDAs'\right)$ of
Proposition~\ref{prop:autre_presentation_das_gamma} gives rise to
the following presentation for $\As_\gamma$. This last operad
admits the presentation $\left(\GenAs', \RelAs'\right)$ where
$\GenAs' := \GenAs'(2) := \{\MAsA_a : a \in [\gamma]\}$
and $\RelAs'$ is generated by
\begin{subequations}
\begin{equation}
    \MAsA_a \circ_1 \MAsA_{a'}, \qquad a \ne a' \in [\gamma],
\end{equation}
\begin{equation}
    \MAsA_a \circ_2 \MAsA_{a'}, \qquad a \ne a' \in [\gamma],
\end{equation}
\begin{equation}
    \MAsA_a \circ_1 \MAsA_a -
    \MAsA_a \circ_2 \MAsA_a, \qquad a \in [\gamma].
\end{equation}
\end{subequations}
Indeed, $\RelAs'$ is the space $\RelAs$ through the identification
\begin{equation}
    \MAsA_a =
    \begin{cases}
        \MAs_\gamma & \mbox{if } a = \gamma, \\
        \MAs_a - \MAs_{a + 1} & \mbox{otherwise}.
    \end{cases}
\end{equation}
\medskip

\begin{Proposition} \label{prop:serie_hilbert_das_gamma}
    For any integer $\gamma \geq 0$, the Hilbert series
    $\Hca_{\DAs_\gamma}(t)$ of the operad $\DAs_\gamma$ satisfies
    \begin{equation} \label{equ:serie_hilbert_das_gamma}
        \Hca_{\DAs_\gamma}(t) =
        t + t\, \Hca_{\DAs_\gamma}(t) +
        (\gamma - 1) \, \Hca_{\DAs_\gamma}(t)^2.
    \end{equation}
\end{Proposition}
\begin{proof}
    By setting $\bar \Hca_{\DAs_\gamma}(t) := \Hca_{\DAs_\gamma}(-t)$,
    from~\eqref{equ:serie_hilbert_das_gamma}, we obtain
    \begin{equation}
        t =
        \frac{-\bar \Hca_{\DAs_\gamma}(t) +
            (\gamma - 1)\bar \Hca_{\DAs_\gamma}(t)^2}
             {1 + \bar \Hca_{\DAs_\gamma}(t)}.
    \end{equation}
    Moreover, by setting
    $\bar \Hca_{\As_\gamma}(t) := \Hca_{\As_\gamma}(-t)$, where
    $\Hca_{\As_\gamma}(t)$ is defined by~\eqref{equ:serie_hilbert_as_gamma},
    we have
    \begin{equation} \label{equ:serie_hilbert_das_gamma_demo}
        \bar \Hca_{\As_\gamma}\left(\bar \Hca_{\DAs_\gamma}(t)\right) =
        \frac{-\bar \Hca_{\DAs_\gamma}(t) +
            (\gamma - 1)\bar \Hca_{\DAs_\gamma}(t)^2}
             {1 + \bar \Hca_{\DAs_\gamma}(t)}
        = t,
    \end{equation}
    showing that $\bar \Hca_{\As_\gamma}(t)$ and $\bar \Hca_{\DAs_\gamma}(t)$
    are the inverses for each other for series composition.
    \smallskip

    Now, since by Proposition~\ref{prop:realisation_koszulite_as_gamma},
    $\As_\gamma$ is a Koszul operad and its Hilbert series is
    $\Hca_{\As_\gamma}(t)$, and since $\DAs_\gamma$ is by definition
    the Koszul dual of $\As_\gamma$, the Hilbert series of these two
    operads satisfy~\eqref{equ:relation_series_hilbert_operade_duale}.
    Therefore, \eqref{equ:serie_hilbert_das_gamma_demo} implies that
    the Hilbert series of $\DAs_\gamma$ is $\Hca_{\DAs_\gamma}(t)$.
\end{proof}
\medskip

A {\em Schröder tree}~\cite{Sta01,Sta11} is a planar rooted tree such
that internal nodes have two of more children. By examining the expression
for $\Hca_{\DAs_\gamma}(t)$ of the statement of
Proposition~\ref{prop:serie_hilbert_das_gamma}, we observe that for any
$n \geq 1$, $\DAs_\gamma(n)$ can be seen as the vector space
$\AlgLibre_{\DAs_\gamma}(n)$ of Schröder trees with $n$ internal nodes,
all labeled on $[\gamma]$ such that the label of an internal node is
different from the labels of its children that are internal nodes. We
call these trees {\em $\gamma$-alternating Schröder trees}. Let us also
denote by $\AlgLibre_{\DAs_\gamma}$the graded vector space of all
$\gamma$-alternating Schröder trees. For instance,
\begin{equation}
    \begin{split}
    \begin{tikzpicture}[xscale=.27,yscale=.13]
        \node[Feuille](0)at(0.00,-14.40){};
        \node[Feuille](10)at(8.00,-19.20){};
        \node[Feuille](12)at(10.00,-19.20){};
        \node[Feuille](13)at(11.00,-19.20){};
        \node[Feuille](15)at(13.00,-19.20){};
        \node[Feuille](17)at(15.00,-14.40){};
        \node[Feuille](19)at(17.00,-19.20){};
        \node[Feuille](2)at(1.00,-14.40){};
        \node[Feuille](21)at(18.00,-19.20){};
        \node[Feuille](22)at(19.00,-19.20){};
        \node[Feuille](23)at(20.00,-4.80){};
        \node[Feuille](3)at(2.00,-14.40){};
        \node[Feuille](5)at(4.00,-9.60){};
        \node[Feuille](6)at(5.00,-4.80){};
        \node[Feuille](8)at(7.00,-14.40){};
        \node[NoeudSchr](1)at(1.00,-9.60){\begin{math}2\end{math}};
        \node[NoeudSchr](11)at(9.00,-14.40){\begin{math}1\end{math}};
        \node[NoeudSchr](14)at(12.00,-14.40){\begin{math}3\end{math}};
        \node[NoeudSchr](16)at(14.00,-4.80){\begin{math}3\end{math}};
        \node[NoeudSchr](18)at(16.00,-9.60){\begin{math}2\end{math}};
        \node[NoeudSchr](20)at(18.00,-14.40){\begin{math}1\end{math}};
        \node[NoeudSchr](4)at(3.00,-4.80){\begin{math}3\end{math}};
        \node[NoeudSchr](7)at(6.00,0.00){\begin{math}1\end{math}};
        \node[NoeudSchr](9)at(9.00,-9.60){\begin{math}2\end{math}};
        \draw[Arete](0)--(1);
        \draw[Arete](1)--(4);
        \draw[Arete](10)--(11);
        \draw[Arete](11)--(9);
        \draw[Arete](12)--(11);
        \draw[Arete](13)--(14);
        \draw[Arete](14)--(9);
        \draw[Arete](15)--(14);
        \draw[Arete](16)--(7);
        \draw[Arete](17)--(18);
        \draw[Arete](18)--(16);
        \draw[Arete](19)--(20);
        \draw[Arete](2)--(1);
        \draw[Arete](20)--(18);
        \draw[Arete](21)--(20);
        \draw[Arete](22)--(20);
        \draw[Arete](23)--(7);
        \draw[Arete](3)--(1);
        \draw[Arete](4)--(7);
        \draw[Arete](5)--(4);
        \draw[Arete](6)--(7);
        \draw[Arete](8)--(9);
        \draw[Arete](9)--(16);
        \node(r)at(6.00,3){};
        \draw[Arete](r)--(7);
    \end{tikzpicture}
    \end{split}
\end{equation}
is a $3$-alternating Schröder tree and a basis element of $\DAs_3(9)$.
\medskip

We deduce also from Proposition~\ref{prop:serie_hilbert_das_gamma} that
\begin{equation}
    \Hca_{\DAs_\gamma}(t) =
    \frac{1 - \sqrt{1 - (4\gamma - 2)t + t^2} - t}{2(\gamma - 1)}.
\end{equation}
By denoting by $\Nar(n, k)$ the {\em Narayana number}~\cite{Nar55}
defined by
\begin{equation} \label{equ:definition_narayana}
    \Nar(n, k) := \frac{1}{k + 1} \binom{n - 1}{k} \binom{n}{k},
\end{equation}
we obtain that for all $n \geq 1$,
\begin{equation}
    \dim \DAs_\gamma(n) =
    \sum_{k = 0}^{n - 2}
    \gamma^{k + 1} (\gamma - 1)^{n - k - 2} \, \Nar(n - 1, k).
\end{equation}
This formula is a consequence of the fact that $\Nar(n - 1, k)$ is the
number of binary trees with $n$ leaves and with exactly $k$ internal
nodes having a internal node as a left child, the fact that  the number
$\Schr(n)$ of Schröder trees with $n$ leaves expresses as
\begin{equation}
    \Schr(n) = \sum_{k = 0}^{n - 2} 2^k \, \Nar(n - 1, k),
\end{equation}
and the fact that any Schröder tree $\Sfr$ with $n$ leaves can be encoded
by a binary tree $\Tfr$ with $n$ leaves where any left oriented edge
connecting two internal nodes of $\Tfr$ is labeled on $[2]$ ($\Sfr$ is
obtained from $\Tfr$ by contracting all edges labeled by $2$).
\medskip

For instance, the first dimensions of  $\DAs_1$, $\DAs_2$, $\DAs_3$, and
$\DAs_4$ are respectively
\begin{equation}
    1, 1, 1, 1, 1, 1, 1, 1, 1, 1, 1,
\end{equation}
\begin{equation}
    1, 2, 6, 22, 90, 394, 1806, 8558, 41586, 206098, 1037718,
\end{equation}
\begin{equation}
    1, 3, 15, 93, 645, 4791, 37275, 299865, 2474025, 20819307, 178003815,
\end{equation}
\begin{equation}
    1, 4, 28, 244, 2380, 24868, 272188, 3080596, 35758828, 423373636, 5092965724.
\end{equation}
The second one is Sequence~\Sloane{A006318}, the third one is
Sequence~\Sloane{A103210}, and the last one is Sequence~\Sloane{A103211}
of~\cite{Slo}.
\medskip

Let us now establish a realization of $\DAs_\gamma$.
\medskip

\begin{Proposition} \label{prop:realisation_das_gamma}
    For any nonnegative integer $\gamma$, the operad $\DAs_\gamma$ is
    the vector space $\AlgLibre_{\DAs_\gamma}$ of $\gamma$-alternating
    Schröder trees. Moreover, for any $\gamma$-alternating Schröder
    trees $\Sfr$ and $\Tfr$, $\Sfr \circ_i \Tfr$ is the $\gamma$-alternating
    Schröder tree obtained by grafting the root of $\Tfr$ on the $i$th
    leaf $x$ of $\Sfr$ and then, if the father $y$ of $x$ and the root
    $z$ of $\Tfr$ have a same label, by contracting the edge connecting
    $y$ and $z$.
\end{Proposition}
\begin{proof}
    First, it is immediate that the vector space $\AlgLibre_{\DAs_\gamma}$
    endowed with the partial compositions described in the statement of
    the proposition is an operad.
    \smallskip

    Let
    \begin{equation}
        \phi : \DAs_\gamma \simeq \OpLibre\left(\GenDAs'\right)
        /_{\left\langle \RelDAs' \right\rangle} \to \AlgLibre_{\DAs_\gamma}
    \end{equation}
    be the map satisfying $\phi(\pi(\MDAs_a)) := \Cfr_a$ where $\Cfr_a$
    is the $\gamma$-alternating Schröder with two leaves and one internal
    node labeled by $a \in [\gamma]$ and
    $\pi : \OpLibre(\GenDAs') \to \DAs_\gamma$ is the canonical
    surjection map. Since we have
    $\phi(\pi(\MDAs_a)) \circ_1 \phi(\pi(\MDAs_a)) =
    \phi(\pi(\MDAs_a)) \circ_2 \phi(\pi(\MDAs_a))$
    for all $a \in [\gamma]$, $\phi$ extends in a unique way into an
    operad morphism. First, since the set $G_\gamma$ of all
    $\gamma$-alternating Schröder trees with two leaves and one internal
    node is a generating set of $\AlgLibre_{\DAs_\gamma}$ and the image
    of $\phi$ contains $G_\gamma$, $\phi$ is surjective. Second, since
    by definition of $\AlgLibre_{\DAs_\gamma}$, the bases of
    $\AlgLibre_{\DAs_\gamma}$ are indexed by $\gamma$-alternating
    Schröder trees, by Proposition~\ref{prop:serie_hilbert_das_gamma},
    $\AlgLibre_{\DAs_\gamma}$ and $\DAs_\gamma$ are isomorphic as graded
    vector spaces. Hence, $\phi$ is an operad isomorphism, showing that
    $\DAs_\gamma$ admits the claimed realization.
\end{proof}
\medskip

We have for instance in $\DAs_3$,
\begin{equation} \label{equ:exemple_composition_as_gamma_duale}
    \begin{split}
    \begin{tikzpicture}[xscale=.23,yscale=.13]
        \node[Feuille](0)at(0.00,-6.50){};
        \node[Feuille](10)at(8.00,-9.75){};
        \node[Feuille](12)at(10.00,-9.75){};
        \node[Feuille](2)at(2.00,-9.75){};
        \node[Feuille](4)at(4.00,-9.75){};
        \node[Feuille](6)at(5.00,-3.25){};
        \node[Feuille](7)at(6.00,-6.50){};
        \node[Feuille](9)at(7.00,-6.50){};
        \node[NoeudSchr](1)at(1.00,-3.25){\begin{math}1\end{math}};
        \node[NoeudSchr](11)at(9.00,-6.50){\begin{math}1\end{math}};
        \node[NoeudSchr](3)at(3.00,-6.50){\begin{math}2\end{math}};
        \node[NoeudSchr](5)at(5.00,0.00){\begin{math}2\end{math}};
        \node[NoeudSchr](8)at(7.00,-3.25){\begin{math}3\end{math}};
        \draw[Arete](0)--(1);
        \draw[Arete](1)--(5);
        \draw[Arete](10)--(11);
        \draw[Arete](11)--(8);
        \draw[Arete](12)--(11);
        \draw[Arete](2)--(3);
        \draw[Arete](3)--(1);
        \draw[Arete](4)--(3);
        \draw[Arete](6)--(5);
        \draw[Arete](7)--(8);
        \draw[Arete](8)--(5);
        \draw[Arete](9)--(8);
        \node(r)at(5.00,2.7){};
        \draw[Arete](r)--(5);
    \end{tikzpicture}
    \end{split}
    \circ_4
    \begin{split}
    \begin{tikzpicture}[xscale=.2,yscale=.17]
        \node[Feuille](0)at(0.00,-3.33){};
        \node[Feuille](2)at(2.00,-3.33){};
        \node[Feuille](4)at(4.00,-1.67){};
        \node[NoeudSchr](1)at(1.00,-1.67){\begin{math}2\end{math}};
        \node[NoeudSchr](3)at(3.00,0.00){\begin{math}3\end{math}};
        \draw[Arete](0)--(1);
        \draw[Arete](1)--(3);
        \draw[Arete](2)--(1);
        \draw[Arete](4)--(3);
        \node(r)at(3.00,2){};
        \draw[Arete](r)--(3);
    \end{tikzpicture}
    \end{split}
    =
    \begin{split}
    \begin{tikzpicture}[xscale=.2,yscale=.10]
        \node[Feuille](0)at(0.00,-8.50){};
        \node[Feuille](10)at(9.00,-8.50){};
        \node[Feuille](11)at(10.00,-8.50){};
        \node[Feuille](13)at(11.00,-8.50){};
        \node[Feuille](14)at(12.00,-12.75){};
        \node[Feuille](16)at(14.00,-12.75){};
        \node[Feuille](2)at(2.00,-12.75){};
        \node[Feuille](4)at(4.00,-12.75){};
        \node[Feuille](6)at(5.00,-12.75){};
        \node[Feuille](8)at(7.00,-12.75){};
        \node[NoeudSchr](1)at(1.00,-4.25){1};
        \node[NoeudSchr](12)at(11.00,-4.25){\begin{math}3\end{math}};
        \node[NoeudSchr](15)at(13.00,-8.50){\begin{math}1\end{math}};
        \node[NoeudSchr](3)at(3.00,-8.50){\begin{math}2\end{math}};
        \node[NoeudSchr](5)at(7.00,0.00){\begin{math}2\end{math}};
        \node[NoeudSchr](7)at(6.00,-8.50){\begin{math}2\end{math}};
        \node[NoeudSchr](9)at(8.00,-4.25){\begin{math}3\end{math}};
        \draw[Arete](0)--(1);
        \draw[Arete](1)--(5);
        \draw[Arete](10)--(9);
        \draw[Arete](11)--(12);
        \draw[Arete](12)--(5);
        \draw[Arete](13)--(12);
        \draw[Arete](14)--(15);
        \draw[Arete](15)--(12);
        \draw[Arete](16)--(15);
        \draw[Arete](2)--(3);
        \draw[Arete](3)--(1);
        \draw[Arete](4)--(3);
        \draw[Arete](6)--(7);
        \draw[Arete](7)--(9);
        \draw[Arete](8)--(7);
        \draw[Arete](9)--(5);
        \node(r)at(7.00,3.5){};
        \draw[Arete](r)--(5);
        \end{tikzpicture}
    \end{split}\,,
\end{equation}
and
\begin{equation}
    \begin{split}
    \begin{tikzpicture}[xscale=.23,yscale=.13]
        \node[Feuille](0)at(0.00,-6.50){};
        \node[Feuille](10)at(8.00,-9.75){};
        \node[Feuille](12)at(10.00,-9.75){};
        \node[Feuille](2)at(2.00,-9.75){};
        \node[Feuille](4)at(4.00,-9.75){};
        \node[Feuille](6)at(5.00,-3.25){};
        \node[Feuille](7)at(6.00,-6.50){};
        \node[Feuille](9)at(7.00,-6.50){};
        \node[NoeudSchr](1)at(1.00,-3.25){\begin{math}1\end{math}};
        \node[NoeudSchr](11)at(9.00,-6.50){\begin{math}1\end{math}};
        \node[NoeudSchr](3)at(3.00,-6.50){\begin{math}2\end{math}};
        \node[NoeudSchr](5)at(5.00,0.00){\begin{math}2\end{math}};
        \node[NoeudSchr](8)at(7.00,-3.25){\begin{math}3\end{math}};
        \draw[Arete](0)--(1);
        \draw[Arete](1)--(5);
        \draw[Arete](10)--(11);
        \draw[Arete](11)--(8);
        \draw[Arete](12)--(11);
        \draw[Arete](2)--(3);
        \draw[Arete](3)--(1);
        \draw[Arete](4)--(3);
        \draw[Arete](6)--(5);
        \draw[Arete](7)--(8);
        \draw[Arete](8)--(5);
        \draw[Arete](9)--(8);
        \node(r)at(5.00,2.7){};
        \draw[Arete](r)--(5);
    \end{tikzpicture}
    \end{split}
    \circ_5
    \begin{split}
    \begin{tikzpicture}[xscale=.2,yscale=.17]
        \node[Feuille](0)at(0.00,-3.33){};
        \node[Feuille](2)at(2.00,-3.33){};
        \node[Feuille](4)at(4.00,-1.67){};
        \node[NoeudSchr](1)at(1.00,-1.67){\begin{math}2\end{math}};
        \node[NoeudSchr](3)at(3.00,0.00){\begin{math}3\end{math}};
        \draw[Arete](0)--(1);
        \draw[Arete](1)--(3);
        \draw[Arete](2)--(1);
        \draw[Arete](4)--(3);
        \node(r)at(3.00,2){};
        \draw[Arete](r)--(3);
    \end{tikzpicture}
    \end{split}
    =
    \begin{split}
    \begin{tikzpicture}[xscale=.2,yscale=.10]
        \node[Feuille](0)at(0.00,-8.00){};
        \node[Feuille](10)at(9.00,-8.00){};
        \node[Feuille](12)at(11.00,-8.00){};
        \node[Feuille](13)at(12.00,-12.00){};
        \node[Feuille](15)at(14.00,-12.00){};
        \node[Feuille](2)at(2.00,-12.00){};
        \node[Feuille](4)at(4.00,-12.00){};
        \node[Feuille](6)at(5.00,-4.00){};
        \node[Feuille](7)at(6.00,-12.00){};
        \node[Feuille](9)at(8.00,-12.00){};
        \node[NoeudSchr](1)at(1.00,-4.00){1};
        \node[NoeudSchr](11)at(10.00,-4.00){3};
        \node[NoeudSchr](14)at(13.00,-8.00){1};
        \node[NoeudSchr](3)at(3.00,-8.00){2};
        \node[NoeudSchr](5)at(5.00,0.00){2};
        \node[NoeudSchr](8)at(7.00,-8.00){2};
        \draw[Arete](0)--(1);
        \draw[Arete](1)--(5);
        \draw[Arete](10)--(11);
        \draw[Arete](11)--(5);
        \draw[Arete](12)--(11);
        \draw[Arete](13)--(14);
        \draw[Arete](14)--(11);
        \draw[Arete](15)--(14);
        \draw[Arete](2)--(3);
        \draw[Arete](3)--(1);
        \draw[Arete](4)--(3);
        \draw[Arete](6)--(5);
        \draw[Arete](7)--(8);
        \draw[Arete](8)--(11);
        \draw[Arete](9)--(8);
        \node(r)at(5.00,3.5){};
        \draw[Arete](r)--(5);
    \end{tikzpicture}
    \end{split}\,.
\end{equation}
\medskip

\subsection{A diagram of operads}%
\label{subsec:diagramme_dias_as_dendr_gamma}
We now define morphisms between the operads $\Dias_\gamma$, $\As_\gamma$,
$\DAs_\gamma$, and $\Dendr_\gamma$ to obtain a generalization of a
classical diagram involving the diassociative, associative, and
dendriform operads.
\medskip

\subsubsection{Relating the diassociative and dendriform operads}
The diagram
\begin{equation} \label{equ:diagramme_dendr_as_dias}
    \begin{split}
    \begin{tikzpicture}[yscale=.65]
        \node(Dendr)at(0,0){\begin{math} \Dendr \end{math}};
        \node(As)at(2,0){\begin{math} \As \end{math}};
        \node(Dias)at(4,0){\begin{math} \Dias \end{math}};
        \draw[->](Dias)
            edge node[anchor=south]{\begin{math} \eta \end{math}}(As);
        \draw[->](As)
            edge node[anchor=south]{\begin{math} \zeta \end{math}}(Dendr);
        \draw[<->,dotted,loop above,looseness=13](As)
            edge node[anchor=south]{\begin{math} ! \end{math}}(As);
        \draw[<->,dotted,loop above,looseness=1.5](Dendr)
            edge node[anchor=south]{\begin{math} ! \end{math}}(Dias);
    \end{tikzpicture}
    \end{split}
\end{equation}
is a well-known diagram of operads, being a part of the so-called
{\em operadic butterfly} \cite{Lod01,Lod06} and summarizing in a nice way
the links between the dendriform, associative, and diassociative operads.
The operad $\As$, being at the center of the diagram, is it own Koszul
dual, while $\Dias$ and $\Dendr$ are Koszul dual one of the other.
\medskip

The operad morphisms $\eta : \Dias \to \As$ and $\zeta : \As \to \Dendr$
are linearly defined through the realizations of $\Dias$ and $\Dendr$
recalled in Section~\ref{subsec:dias_et_dendr} by
\begin{equation}\begin{split}\end{split}
    \eta(\Efr_{2, 1}) :=
    \begin{split}
    \begin{tikzpicture}[xscale=.2,yscale=.15]
        \node[Feuille](0)at(0.00,-1.50){};
        \node[Feuille](2)at(2.00,-1.50){};
        \node[NoeudCor](1)at(1.00,0.00){};
        \draw[Arete](0)--(1);
        \draw[Arete](2)--(1);
        \node(r)at(1.00,2){};
        \draw[Arete](r)--(1);
    \end{tikzpicture}
    \end{split}
     =: \eta(\Efr_{2, 2})\,,
\end{equation}
and
\begin{equation}\begin{split}\end{split}
    \zeta\left(
    \begin{split}
    \begin{tikzpicture}[xscale=.2,yscale=.15]
        \node[Feuille](0)at(0.00,-1.50){};
        \node[Feuille](2)at(2.00,-1.50){};
        \node[NoeudCor](1)at(1.00,0.00){};
        \draw[Arete](0)--(1);
        \draw[Arete](2)--(1);
        \node(r)at(1.00,2){};
        \draw[Arete](r)--(1);
    \end{tikzpicture}
    \end{split}
    \right) :=
    \begin{split}
    \begin{tikzpicture}[xscale=.22,yscale=.17]
        \node[Feuille](0)at(0.00,-3.33){};
        \node[Feuille](2)at(2.00,-3.33){};
        \node[Feuille](4)at(4.00,-1.67){};
        \node[Noeud](1)at(1.00,-1.67){};
        \node[Noeud](3)at(3.00,0.00){};
        \draw[Arete](0)--(1);
        \draw[Arete](1)--(3);
        \draw[Arete](2)--(1);
        \draw[Arete](4)--(3);
        \node(r)at(3.00,1.67){};
        \draw[Arete](r)--(3);
    \end{tikzpicture}
    \end{split}
    +
    \begin{split}
    \begin{tikzpicture}[xscale=.22,yscale=.17]
        \node[Feuille](0)at(0.00,-1.67){};
        \node[Feuille](2)at(2.00,-3.33){};
        \node[Feuille](4)at(4.00,-3.33){};
        \node[Noeud](1)at(1.00,0.00){};
        \node[Noeud](3)at(3.00,-1.67){};
        \draw[Arete](0)--(1);
        \draw[Arete](2)--(3);
        \draw[Arete](3)--(1);
        \draw[Arete](4)--(3);
        \node(r)at(1.00,1.67){};
        \draw[Arete](r)--(1);
    \end{tikzpicture}
    \end{split}\,.
\end{equation}
Since $\Dias$ is generated by $\Efr_{2, 1}$ and $\Efr_{2, 2}$, and since
$\As$ is generated by
$\raisebox{-.4em}{\begin{tikzpicture}[xscale=.2,yscale=.15]
    \node[Feuille](0)at(0.00,-1.50){};
    \node[Feuille](2)at(2.00,-1.50){};
    \node[NoeudCor](1)at(1.00,0.00){};
    \draw[Arete](0)--(1);
    \draw[Arete](2)--(1);
    \node(r)at(1.00,2){};
    \draw[Arete](r)--(1);
\end{tikzpicture}}$,
$\eta$ and $\zeta$ are wholly defined.
\medskip

\subsubsection{Relating the pluriassociative and polydendriform operads}
\begin{Proposition} \label{prop:morphisme_dias_gamma_vers_as_gamma}
    For any integer $\gamma \geq 0$, the map
    $\eta_\gamma : \Dias_\gamma \to \As_\gamma$ satisfying
    \begin{equation}
        \begin{split}\end{split}
        \eta_\gamma(0a) =
        \begin{split}
        \begin{tikzpicture}[xscale=.25,yscale=.2]
            \node[Feuille](0)at(0.00,-1.50){};
            \node[Feuille](2)at(2.00,-1.50){};
            \node[NoeudCor](1)at(1.00,0.00){\begin{math}a\end{math}};
            \draw[Arete](0)--(1);
            \draw[Arete](2)--(1);
            \node(r)at(1.00,1.7){};
            \draw[Arete](r)--(1);
        \end{tikzpicture}
        \end{split}
        = \eta_\gamma(a0),
        \qquad a \in [\gamma],
    \end{equation}
    extends in a unique way into an operad morphism. Moreover, this
    morphism is surjective.
\end{Proposition}
\begin{proof}
    Theorem~\ref{thm:presentation_dias_gamma} and
    Proposition~\ref{prop:realisation_das_gamma} allow to interpret
    the map $\eta_\gamma$ over the presentations of $\Dias_\gamma$
    and $\As_\gamma$. Then, via this interpretation, one has
    \begin{equation}
        \eta_\gamma(\pi(\GDias_a)) = \pi'(\MAs_a) = \eta_\gamma(\pi(\DDias_a)),
        \qquad a \in [\gamma],
    \end{equation}
    where $\pi : \OpLibre\left(\GenDias\right) \to \Dias_\gamma$ and
    $\pi' : \OpLibre\left(\GenAs\right) \to \As_\gamma$ are canonical
    surjection maps. Now, for any element $x$ of $\OpLibre\left(\GenDias\right)$
    generating the space of relations $\RelDias$ of $\Dias_\gamma$, we
    can check that $\eta_\gamma(\pi(x)) = 0$. This shows that $\eta_\gamma$
    extends in a unique way into an operad morphism. Finally, this
    morphism is a surjection since its image contains the set of all
    $\gamma$-corollas of arity $2$, which is a generating set of
    $\As_\gamma$.
\end{proof}
\medskip

By Proposition~\ref{prop:morphisme_dias_gamma_vers_as_gamma}, the map
$\eta_\gamma$, whose definition is only given in arity $2$, defines an
operad morphism. Nevertheless, by induction on the arity, one can prove
that for any word $x$ of $\Dias_\gamma$, $\eta_\gamma(x)$ is the
$\gamma$-corolla of arity $|x|$ labeled by the greatest letter of $x$.
\medskip

\begin{Proposition} \label{prop:morphisme_das_gamma_vers_dendr_gamma}
    For any integer $\gamma \geq 0$, the map
    $\zeta_\gamma : \DAs_\gamma \to \Dendr_\gamma$ satisfying
    \begin{equation}
        \begin{split}\end{split}
        \zeta_\gamma\left(
        \begin{split}
        \begin{tikzpicture}[xscale=.25,yscale=.2]
            \node[Feuille](0)at(0.00,-1.50){};
            \node[Feuille](2)at(2.00,-1.50){};
            \node[NoeudSchr](1)at(1.00,0.00){\begin{math}a\end{math}};
            \draw[Arete](0)--(1);
            \draw[Arete](2)--(1);
            \node(r)at(1.00,1.7){};
            \draw[Arete](r)--(1);
        \end{tikzpicture}
        \end{split}
        \right)
        = \ArbreBinGDeux{a} + \ArbreBinDDeux{a},
        \qquad a \in [\gamma],
    \end{equation}
    extends in a unique way into an operad morphism.
\end{Proposition}
\begin{proof}
    Propositions~\ref{prop:autre_presentation_das_gamma}
    and~\ref{prop:realisation_das_gamma}, and
    Theorem~\ref{thm:autre_presentation_dendr_gamma} allow to interpret
    the map $\zeta_\gamma$ over the presentations of $\DAs_\gamma$ and
    $\Dendr_\gamma$. Then, via this interpretation, one has
    \begin{equation}
        \zeta_\gamma(\pi(\MDAs_a)) = \pi'(\GDendr_a + \DDendr_a),
        \qquad a \in [\gamma],
    \end{equation}
    where $\pi : \OpLibre(\GenDAs') \to \DAs_\gamma$ and
    $\pi' : \OpLibre\left(\GenDendr'\right) \to \Dendr_\gamma$ are
    canonical surjection maps. We now observe that the image of
    $\pi(\MDAs_a)$ is $\OpAsDendr_a$, where $\OpAsDendr_a$ is the element
    of $\Dendr_\gamma$ defined in the statement of
    Proposition~\ref{prop:operateur_associatif_dendr_gamma}. Then, since
    by this last proposition this element is associative, for any element
    $x$ of $\OpLibre(\GenDAs')$ generating the space of relations of
    $\RelDAs'$ of $\DAs_\gamma$, $\zeta_\gamma(\pi(x)) = 0$. This shows
    that $\zeta_\gamma$ extends in a unique way into an operad morphism.
\end{proof}
\medskip

We have to observe that the morphism $\zeta_\gamma$ defined in the
statement of Proposition~\ref{prop:morphisme_das_gamma_vers_dendr_gamma}
is injective only for $\gamma \leq 1$. Indeed, when $\gamma \geq 2$,
we have the relation
\begin{equation}\begin{split}\end{split}
    \zeta_2\left(
    \begin{split}
    \begin{tikzpicture}[xscale=.2,yscale=.2]
        \node[Feuille](0)at(0.00,-5.25){};
        \node[Feuille](2)at(2.00,-5.25){};
        \node[Feuille](4)at(4.00,-3.50){};
        \node[Feuille](6)at(6.00,-1.75){};
        \node[NoeudSchr](1)at(1.00,-3.50){\begin{math}1\end{math}};
        \node[NoeudSchr](3)at(3.00,-1.75){\begin{math}2\end{math}};
        \node[NoeudSchr](5)at(5.00,0.00){\begin{math}1\end{math}};
        \draw[Arete](0)--(1);
        \draw[Arete](1)--(3);
        \draw[Arete](2)--(1);
        \draw[Arete](3)--(5);
        \draw[Arete](4)--(3);
        \draw[Arete](6)--(5);
        \node(r)at(5.00,1.75){};
        \draw[Arete](r)--(5);
    \end{tikzpicture}
    \end{split}\right)
    \begin{split} \; + \; \end{split}
    \zeta_2\left(
    \begin{split}
    \begin{tikzpicture}[xscale=.2,yscale=.2]
        \node[Feuille](0)at(0.00,-1.75){};
        \node[Feuille](2)at(2.00,-3.50){};
        \node[Feuille](4)at(4.00,-5.25){};
        \node[Feuille](6)at(6.00,-5.25){};
        \node[NoeudSchr](1)at(1.00,0.00){\begin{math}1\end{math}};
        \node[NoeudSchr](3)at(3.00,-1.75){\begin{math}2\end{math}};
        \node[NoeudSchr](5)at(5.00,-3.50){\begin{math}1\end{math}};
        \draw[Arete](0)--(1);
        \draw[Arete](2)--(3);
        \draw[Arete](3)--(1);
        \draw[Arete](4)--(5);
        \draw[Arete](5)--(3);
        \draw[Arete](6)--(5);
        \node(r)at(1.00,1.75){};
        \draw[Arete](r)--(1);
    \end{tikzpicture}
    \end{split}\right)
    \begin{split} \enspace = \enspace \end{split}
    \zeta_2\left(
    \begin{split}
    \begin{tikzpicture}[xscale=.25,yscale=.2]
        \node[Feuille](0)at(0.00,-2.00){};
        \node[Feuille](2)at(1.00,-4.00){};
        \node[Feuille](4)at(3.00,-4.00){};
        \node[Feuille](5)at(4.00,-2.00){};
        \node[NoeudSchr](1)at(2.00,0.00){\begin{math}1\end{math}};
        \node[NoeudSchr](3)at(2.00,-2.00){\begin{math}2\end{math}};
        \draw[Arete](0)--(1);
        \draw[Arete](2)--(3);
        \draw[Arete](3)--(1);
        \draw[Arete](4)--(3);
        \draw[Arete](5)--(1);
        \node(r)at(2.00,2.00){};
        \draw[Arete](r)--(1);
    \end{tikzpicture}
    \end{split}\right)
    \begin{split} \; + \; \end{split}
    \zeta_2\left(
    \begin{split}
    \begin{tikzpicture}[xscale=.22,yscale=.17]
        \node[Feuille](0)at(0.00,-4.67){};
        \node[Feuille](2)at(2.00,-4.67){};
        \node[Feuille](4)at(4.00,-4.67){};
        \node[Feuille](6)at(6.00,-4.67){};
        \node[NoeudSchr](1)at(1.00,-2.33){\begin{math}1\end{math}};
        \node[NoeudSchr](3)at(3.00,0.00){\begin{math}2\end{math}};
        \node[NoeudSchr](5)at(5.00,-2.33){\begin{math}1\end{math}};
        \draw[Arete](0)--(1);
        \draw[Arete](1)--(3);
        \draw[Arete](2)--(1);
        \draw[Arete](4)--(5);
        \draw[Arete](5)--(3);
        \draw[Arete](6)--(5);
        \node(r)at(3.00,2.33){};
        \draw[Arete](r)--(3);
    \end{tikzpicture}
    \end{split}\right)\,.
\end{equation}
\medskip

\begin{Theoreme} \label{thm:diagramme_dias_as_das_dendr_gamma}
    For any integer $\gamma \geq 0$, the operads $\Dias_\gamma$,
    $\Dendr_\gamma$, $\As_\gamma$, and $\DAs_\gamma$ fit into the
    diagram
    \begin{equation} \label{equ:diagramme_dendr_gamma_as_gamma_dias_gamma}
        \begin{split}
        \begin{tikzpicture}[yscale=.6]
            \node(Dendr)at(0,0){\begin{math} \Dendr_\gamma \end{math}};
            \node(DAs)at(2,0){\begin{math} \DAs_\gamma \end{math}};
            \node(As)at(4,0){\begin{math} \As_\gamma \end{math}};
            \node(Dias)at(6,0){\begin{math} \Dias_\gamma \end{math}};
            \draw[->](Dias)
                edge node[anchor=south]{\begin{math} \eta_\gamma \end{math}}(As);
            \draw[<->,dotted](As)
                edge node[anchor=south]{\begin{math} ! \end{math}}(DAs);
            \draw[->](DAs)
                edge node[anchor=south]{\begin{math} \zeta_\gamma \end{math}}(Dendr);
            \draw[<->,dotted,loop above,looseness=.9](Dendr)
                edge node[anchor=south]{\begin{math} ! \end{math}}(Dias);
        \end{tikzpicture}
        \end{split}\,,
    \end{equation}
    where $\eta_\gamma$ is the surjection defined in the statement
    of Proposition~\ref{prop:morphisme_dias_gamma_vers_as_gamma}
    and $\zeta_\gamma$ is the operad morphism defined in the statement
    of Proposition~\ref{prop:morphisme_das_gamma_vers_dendr_gamma}.
\end{Theoreme}
\begin{proof}
    This is a direct consequence of
    Propositions~\ref{prop:morphisme_dias_gamma_vers_as_gamma}
    and~\ref{prop:morphisme_das_gamma_vers_dendr_gamma}.
\end{proof}
\medskip

Diagram~\eqref{equ:diagramme_dendr_gamma_as_gamma_dias_gamma} is a
generalization of \eqref{equ:diagramme_dendr_as_dias} in which the
associative operad split into operads $\As_\gamma$ and $\DAs_\gamma$.
\medskip

\section{Further generalizations}%
\label{sec:generalisation_supplementaires}
In this last section, we propose some one-parameter nonnegative integer
generalizations of well-known operads. For this, we use similar tools as
the ones used in the first sections of this paper.
\medskip

\subsection{Duplicial operad}
We construct here a one-parameter nonnegative integer generalization of
the duplicial operad and describe the free algebras over one generator
in the category encoded by this generalization.
\medskip

\subsubsection{Multiplicial operads}%
\label{subsub:dup_gamma}
It is well-known~\cite{LV12} that the dendriform operad and the duplicial
operad $\Dup$~\cite{Lod08} are both specializations of a same operad
$\DupDendr_q$ with one parameter $q \in \K$. This operad  admits the
presentation
$\left(\GenLibre_{\DupDendr_q}, \RelLibre_{\DupDendr_q}\right)$, where
$\GenLibre_{\DupDendr_q} := \GenLibre_\Dendr$ and
$\RelLibre_{\DupDendr_q}$ is the vector space generated by
\begin{subequations}
\begin{equation}
    \GDendr \circ_1 \DDendr - \DDendr \circ_2 \GDendr,
\end{equation}
\begin{equation}
    \GDendr \circ_1 \GDendr - \GDendr \circ_2 \GDendr
    - q \GDendr \circ_2 \DDendr,
\end{equation}
\begin{equation}
    q \DDendr \circ_1 \GDendr + \DDendr \circ_1 \DDendr
     - \DDendr \circ_2 \DDendr.
\end{equation}
\end{subequations}
One can observe that $\DupDendr_1$ is the dendriform operad and that
$\DupDendr_0$ is the duplicial operad.
\medskip

On the basis of this observation, from the presentation of $\Dendr_\gamma$
provided by Theorem~\ref{thm:autre_presentation_dendr_gamma} and its
concise form provided by Relations~\eqref{equ:relation_dendr_gamma_1_concise},
\eqref{equ:relation_dendr_gamma_2_concise},
and~\eqref{equ:relation_dendr_gamma_3_concise} for its space of relations,
we define the operad $\DupDendr_{q, \gamma}$ with two parameters, an
integer $\gamma \geq 0$ and $q \in \K$, in the following way. We set
$\DupDendr_{q, \gamma}$ as the operad admitting the presentation
$\left(\GenLibre_{\DupDendr_{q, \gamma}},
\RelLibre_{\DupDendr_{q, \gamma}}\right)$,
where $\GenLibre_{\DupDendr_{q, \gamma}} := \GenDendr'$ and
$\RelLibre_{\DupDendr_{q, \gamma}}$ is the vector space generated by
\begin{subequations}
\begin{equation} \label{equ:dupdendr_gamma_1}
    \GDendr_a \circ_1 \DDendr_{a'} - \DDendr_{a'} \circ_2 \GDendr_a,
    \qquad a, a' \in [\gamma],
\end{equation}
\begin{equation} \label{equ:dupdendr_gamma_2}
    \GDendr_a \circ_1 \GDendr_{a'}
    - \GDendr_{a \Min a'} \circ_2 \GDendr_a
    - q\GDendr_{a \Min a'} \circ_2 \DDendr_{a'},
    \qquad a, a' \in [\gamma],
\end{equation}
\begin{equation} \label{equ:dupdendr_gamma_3}
    q\DDendr_{a \Min a'} \circ_1 \GDendr_{a'}
    + \DDendr_{a \Min a'} \circ_1 \DDendr_a
    - \DDendr_a \circ_2 \DDendr_{a'},
    \qquad a, a' \in [\gamma].
\end{equation}
\end{subequations}
One can observe that $\DupDendr_{1, \gamma}$ is the operad $\Dendr_\gamma$.
\medskip

Let us define the operad $\Dup_\gamma$, called
{\em $\gamma$-multiplicial operad}, as the operad $\DupDendr_{0, \gamma}$.
By using respectively the symbols $\GDup_a$ and $\DDup_a$ instead of
$\GDendr_a$ and $\DDendr_a$ for all $a \in [\gamma]$, we obtain that the
space of relations $\RelDup$ of $\Dup_\gamma$ is generated by
\begin{subequations}
\begin{equation} \label{equ:relation_dup_gamma_1}
    \GDup_a \circ_1 \DDup_{a'} - \DDup_{a'} \circ_2 \GDup_a,
    \qquad a, a' \in [\gamma],
\end{equation}
\begin{equation} \label{equ:relation_dup_gamma_2}
    \GDup_a \circ_1 \GDup_{a'} - \GDup_{a \Min a'} \circ_2 \GDup_a,
    \qquad a, a' \in [\gamma],
\end{equation}
\begin{equation} \label{equ:relation_dup_gamma_3}
    \DDup_{a \Min a'} \circ_1 \DDup_a - \DDup_a \circ_2 \DDup_{a'},
    \qquad a, a' \in [\gamma].
\end{equation}
\end{subequations}
We denote by $\GenDup$ the set of generators
$\{\GDup_a, \DDup_a : a \in [\gamma] \}$ of $\Dup_\gamma$.
\medskip

In order to establish some properties of $\Dup_\gamma$, let us consider
the quadratic rewrite rule $\Recr_\gamma$ on $\OpLibre(\GenDup)$
satisfying
\begin{subequations}
\begin{equation}
    \GDup_a \circ_1 \DDup_{a'}
    \enspace \Recr_\gamma \enspace
    \DDup_{a'} \circ_2 \GDup_a,
    \qquad a, a' \in [\gamma],
\end{equation}
\begin{equation}
    \GDup_a \circ_1 \GDup_{a'}
    \enspace \Recr_\gamma \enspace
    \GDup_{a \Min a'} \circ_2 \GDup_a,
    \qquad a, a' \in [\gamma],
\end{equation}
\begin{equation}
    \DDup_a \circ_2 \DDup_{a'}
    \enspace \Recr_\gamma \enspace
    \DDup_{a \Min a'} \circ_1 \DDup_a,
    \qquad a, a' \in [\gamma].
\end{equation}
\end{subequations}
Observe that the space induced by the operad congruence induced by
$\Recr_\gamma$ is $\RelDup$.
\medskip

\begin{Lemme} \label{lem:dup_gamma_reecriture}
    For any integer $\gamma \geq 0$, the rewrite rule $\Recr_\gamma$ is
    convergent and the generating series $\Gca_\gamma(t)$ of its normal
    forms counted by arity satisfies
    \begin{equation} \label{equ:dup_gamma_serie_gen_formes_normales}
        \Gca_\gamma(t) =
        t + 2\gamma t \, \Gca_\gamma(t) + \gamma^2 t \, \Gca_\gamma(t)^2.
    \end{equation}
\end{Lemme}
\begin{proof}
    Let us first prove that $\Recr_\gamma$ is terminating. Consider the
    map $\phi : \OpLibre(\GenDup) \to \EnsNat^2$ defined, for any syntax
    tree $\Tfr$ by $\phi(\Tfr) := (\alpha + \alpha', \beta)$, where
    $\alpha$ (resp. $\alpha'$, $\beta$) is the sum, for all internal
    nodes of $\Tfr$ labeled by $\GDup_a$ (resp. $\DDup_a$, $\DDup_a$),
    $a \in [\gamma]$, of the number of internal nodes in its right
    (resp. left, right) subtree. For the lexicographical order $\leq$ on
    $\EnsNat^2$, we can check that for all $\Recr_\gamma$-rewritings
    $\Sfr \Recr_\gamma \Tfr$ where $\Sfr$ and $\Tfr$ are syntax trees
    with two internal nodes, we have $\phi(\Sfr) \ne \phi(\Tfr)$ and
    $\phi(\Sfr) \leq \phi(\Tfr)$. This implies that any syntax tree
    $\Tfr$ obtained by a sequence of $\Recr_\gamma$-rewritings from a
    syntax tree $\Sfr$ satisfies $\phi(\Sfr) \ne \phi(\Tfr)$ and
    $\phi(\Sfr) \leq \phi(\Tfr)$. Then, since the set of syntax trees of
    $\OpLibre(\GenDup)$ of a fixed arity is finite, this shows that
    $\Recr_\gamma$ is a terminating rewrite rule.
    \smallskip

    Let us now prove that $\Recr_\gamma$ is convergent. We call
    {\em critical tree} any syntax tree $\Sfr$ with three internal nodes
    that can be rewritten by $\Recr_\gamma$ into two different trees $\Tfr$
    and $\Tfr'$. The pair $(\Tfr, \Tfr')$ is a {\em critical pair} for
    $\Recr_\gamma$. Critical trees for $\Recr_\gamma$ are, for all
    $a, b, c \in [\gamma]$,
    \begin{equation} \label{equ:dup_gamma_arbres_critiques}
        \begin{split}
        \begin{tikzpicture}[xscale=.4,yscale=.3]
            \node(0)at(0.00,-3.50){};
            \node(2)at(2.00,-5.25){};
            \node(4)at(4.00,-5.25){};
            \node(6)at(6.00,-1.75){};
            \node(1)at(1.00,-1.75){\begin{math}\DDup_a\end{math}};
            \node(3)at(3.00,-3.50){\begin{math}\DDup_b\end{math}};
            \node(5)at(5.00,0.00){\begin{math}\GDup_c\end{math}};
            \draw(0)--(1);\draw(1)--(5);\draw(2)--(3);\draw(3)--(1);
            \draw(4)--(3);\draw(6)--(5);
            \node(r)at(5.00,1.75){};
            \draw(r)--(5);
        \end{tikzpicture}
        \end{split}\,,
        \quad
        \begin{split}
        \begin{tikzpicture}[xscale=.4,yscale=.3]
            \node(0)at(0.00,-5.25){};
            \node(2)at(2.00,-5.25){};
            \node(4)at(4.00,-3.50){};
            \node(6)at(6.00,-1.75){};
            \node(1)at(1.00,-3.50){\begin{math}\DDup_a\end{math}};
            \node(3)at(3.00,-1.75){\begin{math}\GDup_b\end{math}};
            \node(5)at(5.00,0.00){\begin{math}\GDup_c\end{math}};
            \draw(0)--(1);\draw(1)--(3);\draw(2)--(1);\draw(3)--(5);
            \draw(4)--(3);\draw(6)--(5);
            \node(r)at(5.00,1.75){};
            \draw(r)--(5);
        \end{tikzpicture}
        \end{split}\,,
        \quad
        \begin{split}
        \begin{tikzpicture}[xscale=.4,yscale=.3]
            \node(0)at(0.00,-5.25){};
            \node(2)at(2.00,-5.25){};
            \node(4)at(4.00,-3.50){};
            \node(6)at(6.00,-1.75){};
            \node(1)at(1.00,-3.50){\begin{math}\GDup_a\end{math}};
            \node(3)at(3.00,-1.75){\begin{math}\GDup_b\end{math}};
            \node(5)at(5.00,0.00){\begin{math}\GDup_c\end{math}};
            \draw(0)--(1);\draw(1)--(3);\draw(2)--(1);\draw(3)--(5);
            \draw(4)--(3);\draw(6)--(5);
            \node(r)at(5.00,1.75){};
            \draw(r)--(5);
        \end{tikzpicture}
        \end{split}\,,
        \quad
        \begin{split}
        \begin{tikzpicture}[xscale=.4,yscale=.3]
            \node(0)at(0.00,-1.75){};
            \node(2)at(2.00,-3.50){};
            \node(4)at(4.00,-5.25){};
            \node(6)at(6.00,-5.25){};
            \node(1)at(1.00,0.00){\begin{math}\DDup_a\end{math}};
            \node(3)at(3.00,-1.75){\begin{math}\DDup_b\end{math}};
            \node(5)at(5.00,-3.50){\begin{math}\DDup_c\end{math}};
            \draw(0)--(1);\draw(2)--(3);\draw(3)--(1);\draw(4)--(5);
            \draw(5)--(3);\draw(6)--(5);
            \node(r)at(1.00,1.75){};
            \draw(r)--(1);
        \end{tikzpicture}
        \end{split}\,.
    \end{equation}
    Since $\Recr_\gamma$ is terminating, by the diamond lemma~\cite{New42}
    (see also~\cite{BN98}), to prove that $\Recr_\gamma$ is confluent,
    it is enough to check that for any critical tree $\Sfr$, there is a
    normal form $\Rfr$ of $\Recr_\gamma$ such that
    $\Sfr \Recr_\gamma \Tfr \overset{*}{\Recr_\gamma} \Rfr$ and
    $\Sfr \Recr_\gamma \Tfr' \overset{*}{\Recr_\gamma} \Rfr$, where
    $(\Tfr, \Tfr')$ is a critical pair. This can be done by hand for each
    of the critical trees depicted in~\eqref{equ:dup_gamma_arbres_critiques}.
    \smallskip

    Let us finally prove that the generating series of the normal forms
    of $\Recr_\gamma$ is~\eqref{equ:dup_gamma_serie_gen_formes_normales}.
    Since $\Recr_\gamma$ is terminating, its normal forms are the
    syntax trees that have no partial subtree equal to
    $\GDup_a \circ_1 \DDup_{a'}$, $\GDup_a \circ_1 \GDup_{a'}$, or
    $\DDup_a \circ_2 \DDup_{a'}$ for all $a, a' \in [\gamma]$. Then,
    the normal forms of $\Recr_\gamma$ are the syntax trees wherein any
    internal node labeled by $\GDup_a$, $a \in [\gamma]$, has a leaf as
    left child and any internal node labeled by $\DDup_a$, $a \in [\gamma]$,
    has a leaf or an internal node labeled by $\GDup_{a'}$,
    $a' \in [\gamma]$, as right child. Therefore, by denoting by
    $\Gca'_\gamma(t)$ the generating series of the normal forms of
    $\Recr_\gamma$ equal to the leaf or with a root labeled by $\GDup_a$,
    $a \in [\gamma]$, we obtain
    \begin{equation}
        \Gca'_\gamma(t) = t + \gamma t \, \Gca_\gamma(t)
    \end{equation}
    and
    \begin{equation}
        \Gca_\gamma(t) =
        \Gca'_\gamma(t) + \gamma \, \Gca_\gamma(t) \Gca'_\gamma(t).
    \end{equation}
    An elementary computation shows that $\Gca(t)$
    satisfies~\eqref{equ:dup_gamma_serie_gen_formes_normales}.
\end{proof}
\medskip

\begin{Proposition} \label{prop:proprietes_dup_gamma}
    For any integer $\gamma \geq 0$, the operad $\Dup_\gamma$ is Koszul
    and for any integer $n \geq 1$, $\Dup_\gamma(n)$ is the vector space
    of $\gamma$-edge valued binary trees with $n$ internal nodes.
\end{Proposition}
\begin{proof}
    Since the space induced by the operad congruence induced by
    $\Recr_\gamma$ is $\RelDup$, and since by
    Lemma~\ref{lem:dup_gamma_reecriture}, $\Recr_\gamma$ is convergent,
    by the Koszulity criterion~\cite{Hof10,DK10,LV12} we have reformulated
    in Section~\ref{subsubsec:dual_de_Koszul}, $\Dup_\gamma$ is a Koszul
    operad. Moreover, again because $\Recr_\gamma$ is convergent, as a
    vector space, $\Dup_\gamma(n)$ is isomorphic to the vector space of
    the normal forms of $\Recr_\gamma$ with $n \geq 1$ internal nodes.
    Since the generating series $\Gca_\gamma(t)$ of the normal forms
    of $\Recr_\gamma$ is also the generating series of $\gamma$-edge
    valued binary trees (see
    Proposition~\ref{prop:serie_hilbert_dendr_gamma}), the second part
    of the statement of the proposition follows.
\end{proof}
\medskip

Since Proposition~\ref{prop:proprietes_dup_gamma} shows that the operads
$\Dup_\gamma$ and $\Dendr_\gamma$ have the same underlying vector space,
asking if these two operads are isomorphic is natural. Next result implies
that this is not the case.
\medskip

\begin{Proposition}%
\label{prop:description_operateurs_associatifs_dup_gamma}
    For any integer $\gamma \geq 0$, any associative element of
    $\Dup_\gamma$ is proportional to $\pi(\GDup_a)$ or $\pi(\DDup_a)$
    for an $a \in [\gamma]$, where
    $\pi : \OpLibre\left(\GenDup\right) \to \Dup_\gamma$ is the canonical
    surjection map.
\end{Proposition}
\begin{proof}
    Let $\pi : \OpLibre\left(\GenDup\right) \to \Dup_\gamma$ be the
    canonical surjection map. Consider the element
    \begin{equation}
        x := \sum_{a \in [\gamma]} \alpha_a \GDup_a + \beta_a \DDup_a
    \end{equation}
    of $\OpLibre\left(\GenDup\right)$, where $\alpha_a, \beta_a \in \K$
    for all $a \in [\gamma]$, such that $\pi(x)$ is associative in
    $\Dup_\gamma$. Since we have $\pi(r) = 0$ for all elements $r$
    of $\RelLibre_{\Dup_\gamma}$ (see~\eqref{equ:relation_dup_gamma_1},
    \eqref{equ:relation_dup_gamma_2}, and~\eqref{equ:relation_dup_gamma_3}),
    the fact that $\pi(x \circ_1 x - x \circ_2 x) = 0$ implies the
    constraints
    \begin{equation}\begin{split}
        \alpha_a \, \beta_{a'} - \beta_{a'} \, \alpha_a & = 0,
        \qquad a, a' \in [\gamma], \\
        \alpha_a \, \alpha_{a'} - \alpha_{a \Min a'} \, \alpha_a & = 0,
        \qquad a, a' \in [\gamma], \\
        \beta_a \, \beta_{a'} - \beta_{a \Min a'} \, \beta_a & = 0,
        \qquad a, a' \in [\gamma],
    \end{split}\end{equation}
    on the coefficients intervening in $x$. Moreover, since the syntax
    trees $\DDup_b \circ_1 \DDup_a$, $\DDup_a \circ_1 \GDup_{a'}$,
    $\GDup_b \circ_2 \GDup_a$, and $\GDup_a \circ_2 \DDup_{a'}$ do not
    appear in $\RelLibre_{\Dup_\gamma}$ for all $a < b \in [\gamma]$ and
    $a, a' \in [\gamma]$, we have the further constraints
    \begin{equation}\begin{split}
        \beta_b \, \beta_a & = 0, \qquad a < b \in [\gamma], \\
        \beta_a \, \alpha_{a'} & = 0, \qquad a, a' \in [\gamma], \\
        \alpha_b \, \alpha_a & = 0, \qquad a < b \in [\gamma], \\
        \alpha_a \, \beta_{a'} & = 0, \qquad a, a' \in [\gamma].
    \end{split}\end{equation}
    These relations imply that there are at most one $c \in [\gamma]$ and
    one $d \in [\gamma]$ such that $\alpha_c \ne 0$ and $\beta_d \ne 0$.
    In this case, the relations imply also that $\alpha_c = 0$ or
    $\beta_d = 0$, or both. Therefore, $x$ is of the form
    $x = \alpha_a \GDup_a$ or $x = \beta_a \DDup_a$ for an $a \in [\gamma]$,
    whence the statement of the proposition.
\end{proof}
\medskip

By Proposition~\ref{prop:description_operateurs_associatifs_dup_gamma}
there are exactly $2\gamma$ nonproportional associative operations in
$\Dup_\gamma$ while, by
Proposition~\ref{prop:description_operateurs_associatifs_dendr_gamma}
there are exactly $\gamma$ such operations in $\Dendr_\gamma$.
Therefore, $\Dup_\gamma$ and $\Dendr_\gamma$ are not isomorphic.
\medskip

\subsubsection{Free multiplicial algebras}
We call {\em $\gamma$-multiplicial algebra} any $\Dup_\gamma$-algebra.
From the definition of $\Dup_\gamma$, any $\gamma$-multiplicial algebra
is a vector space endowed with linear operations $\GDup_a, \DDup_a$,
$a \in [\gamma]$, satisfying the relations encoded
by~\eqref{equ:relation_dup_gamma_1}---\eqref{equ:relation_dup_gamma_3}.
\medskip

In order the simplify and make uniform next definitions, we consider
that in any $\gamma$-edge valued binary tree $\Tfr$, all edges
connecting internal nodes of $\Tfr$ with leaves are labeled by $\infty$.
By convention, for all $a \in [\gamma]$, we have
$a \Min \infty = a = \infty \Min a$. Let us endow the vector space
$\AlgLibre_{\Dup_\gamma}$ of $\gamma$-edge valued binary trees with
linear operations
\begin{equation}
    \GDup_a, \DDup_a :
    \AlgLibre_{\Dup_\gamma} \otimes \AlgLibre_{\Dup_\gamma}
    \to \AlgLibre_{\Dup_\gamma},
    \qquad a \in [\gamma],
\end{equation}
recursively defined, for any $\gamma$-edge valued binary tree $\Sfr$
and any $\gamma$-edge valued binary trees or leaves $\Tfr_1$ and
$\Tfr_2$ by
\begin{equation}
    \Sfr \GDup_a \Feuille
    := \Sfr =:
    \Feuille \DDup_a \Sfr,
\end{equation}
\begin{equation}
    \Feuille \GDup_a \Sfr := 0 =: \Sfr \DDup_a \Feuille,
\end{equation}
\begin{equation}
    \ArbreBinValue{x}{y}{\Tfr_1}{\Tfr_2}
    \GDup_a \Sfr :=
    \ArbreBinValue{x}{z}{\Tfr_1}{\Tfr_2 \GDup_a \Sfr}\,,
    \qquad
    z := a \Min y,
\end{equation}
\begin{equation}
    \begin{split}\Sfr \DDup_a\end{split}
    \ArbreBinValue{x}{y}{\Tfr_1}{\Tfr_2}
    :=
    \ArbreBinValue{z}{y}{\Sfr \DDup_a \Tfr_1}{\Tfr_2}\,,
    \qquad
    z := a \Min x.
\end{equation}
Note that neither $\Feuille \GDendr_a \Feuille$ nor
$\Feuille \DDup_a \Feuille$ are defined.
\medskip

These recursive definitions for the operations $\GDup_a$, $\DDup_a$,
$a \in [\gamma]$, lead to the following direct reformulations. If $\Sfr$
and $\Tfr$ are two $\gamma$-edge valued binary trees, $\Tfr \GDup_a \Sfr$
(resp. $\Sfr \DDup_a \Tfr$) is obtained by replacing each label $y$
(resp. $x$) of any edge in the rightmost (resp. leftmost) path of $\Tfr$
by $a \Min y$ (resp. $a \Min x$) to obtain a tree $\Tfr'$, and by
grafting the root of $\Sfr$ on the rightmost (resp. leftmost) leaf of
$\Tfr'$. These two operations are respective generalizations of the
operations {\em under} and {\em over} on binary trees introduced by
Loday and Ronco~\cite{LR02}.
\medskip

For example, we have
\begin{equation}
    \begin{split}
    \begin{tikzpicture}[xscale=.25,yscale=.2]
        \node[Feuille](0)at(0.00,-4.50){};
        \node[Feuille](2)at(2.00,-4.50){};
        \node[Feuille](4)at(4.00,-4.50){};
        \node[Feuille](6)at(6.00,-6.75){};
        \node[Feuille](8)at(8.00,-6.75){};
        \node[Noeud](1)at(1.00,-2.25){};
        \node[Noeud](3)at(3.00,0.00){};
        \node[Noeud](5)at(5.00,-2.25){};
        \node[Noeud](7)at(7.00,-4.50){};
        \draw[Arete](0)--(1);
        \draw[Arete](1)edge[]node[EtiqArete]{\begin{math}1\end{math}}(3);
        \draw[Arete](2)--(1);
        \draw[Arete](4)--(5);
        \draw[Arete](5)edge[]node[EtiqArete]{\begin{math}3\end{math}}(3);
        \draw[Arete](6)--(7);
        \draw[Arete](7)edge[]node[EtiqArete]{\begin{math}1\end{math}}(5);
        \draw[Arete](8)--(7);
        \node(r)at(3.00,2){};
        \draw[Arete](r)--(3);
    \end{tikzpicture}
    \end{split}
    \GDup_2
    \begin{split}
    \begin{tikzpicture}[xscale=.25,yscale=.2]
        \node[Feuille](0)at(0.00,-4.67){};
        \node[Feuille](2)at(2.00,-4.67){};
        \node[Feuille](4)at(4.00,-4.67){};
        \node[Feuille](6)at(6.00,-4.67){};
        \node[Noeud](1)at(1.00,-2.33){};
        \node[Noeud](3)at(3.00,0.00){};
        \node[Noeud](5)at(5.00,-2.33){};
        \draw[Arete](0)--(1);
        \draw[Arete](1)edge[]node[EtiqArete]{\begin{math}1\end{math}}(3);
        \draw[Arete](2)--(1);
        \draw[Arete](4)--(5);
        \draw[Arete](5)edge[]node[EtiqArete]{\begin{math}2\end{math}}(3);
        \draw[Arete](6)--(5);
        \node(r)at(3.00,2){};
        \draw[Arete](r)--(3);
    \end{tikzpicture}
    \end{split}
    =
    \begin{split}
    \begin{tikzpicture}[xscale=.22,yscale=.15]
        \node[Feuille](0)at(0.00,-5.00){};
        \node[Feuille](10)at(10.00,-12.50){};
        \node[Feuille](12)at(12.00,-12.50){};
        \node[Feuille](14)at(14.00,-12.50){};
        \node[Feuille](2)at(2.00,-5.00){};
        \node[Feuille](4)at(4.00,-5.00){};
        \node[Feuille](6)at(6.00,-7.50){};
        \node[Feuille](8)at(8.00,-12.50){};
        \node[Noeud](1)at(1.00,-2.50){};
        \node[Noeud](11)at(11.00,-7.50){};
        \node[Noeud](13)at(13.00,-10.00){};
        \node[Noeud](3)at(3.00,0.00){};
        \node[Noeud](5)at(5.00,-2.50){};
        \node[Noeud](7)at(7.00,-5.00){};
        \node[Noeud](9)at(9.00,-10.00){};
        \draw[Arete](0)--(1);
        \draw[Arete](1)edge[]node[EtiqArete]{\begin{math}1\end{math}}(3);
        \draw[Arete](10)--(9);
        \draw[Arete](11)edge[]node[EtiqArete]{\begin{math}2\end{math}}(7);
        \draw[Arete](12)--(13);
        \draw[Arete](13)edge[]node[EtiqArete]{\begin{math}2\end{math}}(11);
        \draw[Arete](14)--(13);
        \draw[Arete](2)--(1);
        \draw[Arete](4)--(5);
        \draw[Arete](5)edge[]node[EtiqArete]{\begin{math}2\end{math}}(3);
        \draw[Arete](6)--(7);
        \draw[Arete](7)edge[]node[EtiqArete]{\begin{math}1\end{math}}(5);
        \draw[Arete](8)--(9);
        \draw[Arete](9)edge[]node[EtiqArete]{\begin{math}1\end{math}}(11);
        \node(r)at(3.00,2.50){};
        \draw[Arete](r)--(3);
    \end{tikzpicture}
    \end{split}\,,
\end{equation}
and
\begin{equation}
    \begin{split}
    \begin{tikzpicture}[xscale=.25,yscale=.2]
        \node[Feuille](0)at(0.00,-4.50){};
        \node[Feuille](2)at(2.00,-4.50){};
        \node[Feuille](4)at(4.00,-4.50){};
        \node[Feuille](6)at(6.00,-6.75){};
        \node[Feuille](8)at(8.00,-6.75){};
        \node[Noeud](1)at(1.00,-2.25){};
        \node[Noeud](3)at(3.00,0.00){};
        \node[Noeud](5)at(5.00,-2.25){};
        \node[Noeud](7)at(7.00,-4.50){};
        \draw[Arete](0)--(1);
        \draw[Arete](1)edge[]node[EtiqArete]{\begin{math}1\end{math}}(3);
        \draw[Arete](2)--(1);
        \draw[Arete](4)--(5);
        \draw[Arete](5)edge[]node[EtiqArete]{\begin{math}3\end{math}}(3);
        \draw[Arete](6)--(7);
        \draw[Arete](7)edge[]node[EtiqArete]{\begin{math}1\end{math}}(5);
        \draw[Arete](8)--(7);
        \node(r)at(3.00,2){};
        \draw[Arete](r)--(3);
    \end{tikzpicture}
    \end{split}
    \DDup_2
    \begin{split}
    \begin{tikzpicture}[xscale=.25,yscale=.2]
        \node[Feuille](0)at(0.00,-4.67){};
        \node[Feuille](2)at(2.00,-4.67){};
        \node[Feuille](4)at(4.00,-4.67){};
        \node[Feuille](6)at(6.00,-4.67){};
        \node[Noeud](1)at(1.00,-2.33){};
        \node[Noeud](3)at(3.00,0.00){};
        \node[Noeud](5)at(5.00,-2.33){};
        \draw[Arete](0)--(1);
        \draw[Arete](1)edge[]node[EtiqArete]{\begin{math}1\end{math}}(3);
        \draw[Arete](2)--(1);
        \draw[Arete](4)--(5);
        \draw[Arete](5)edge[]node[EtiqArete]{\begin{math}2\end{math}}(3);
        \draw[Arete](6)--(5);
        \node(r)at(3.00,2){};
        \draw[Arete](r)--(3);
    \end{tikzpicture}
    \end{split}
    =
    \begin{split}
    \begin{tikzpicture}[xscale=.22,yscale=.15]
        \node[Feuille](0)at(0.00,-10.00){};
        \node[Feuille](10)at(10.00,-5.00){};
        \node[Feuille](12)at(12.00,-5.00){};
        \node[Feuille](14)at(14.00,-5.00){};
        \node[Feuille](2)at(2.00,-10.00){};
        \node[Feuille](4)at(4.00,-10.00){};
        \node[Feuille](6)at(6.00,-12.50){};
        \node[Feuille](8)at(8.00,-12.50){};
        \node[Noeud](1)at(1.00,-7.50){};
        \node[Noeud](11)at(11.00,0.00){};
        \node[Noeud](13)at(13.00,-2.50){};
        \node[Noeud](3)at(3.00,-5.00){};
        \node[Noeud](5)at(5.00,-7.50){};
        \node[Noeud](7)at(7.00,-10.00){};
        \node[Noeud](9)at(9.00,-2.50){};
        \draw[Arete](0)--(1);
        \draw[Arete](1)edge[]node[EtiqArete]{\begin{math}1\end{math}}(3);
        \draw[Arete](10)--(9);
        \draw[Arete](12)--(13);
        \draw[Arete](13)edge[]node[EtiqArete]{\begin{math}2\end{math}}(11);
        \draw[Arete](14)--(13);
        \draw[Arete](2)--(1);
        \draw[Arete](3)edge[]node[EtiqArete]{\begin{math}2\end{math}}(9);
        \draw[Arete](4)--(5);
        \draw[Arete](5)edge[]node[EtiqArete]{\begin{math}3\end{math}}(3);
        \draw[Arete](6)--(7);
        \draw[Arete](7)edge[]node[EtiqArete]{\begin{math}1\end{math}}(5);
        \draw[Arete](8)--(7);
        \draw[Arete](9)edge[]node[EtiqArete]{\begin{math}1\end{math}}(11);
        \node(r)at(11.00,2.50){};
        \draw[Arete](r)--(11);
    \end{tikzpicture}
    \end{split}\,.
\end{equation}
\medskip

\begin{Lemme} \label{lem:produit_gamma_duplicial}
    For any integer $\gamma \geq 0$, the vector space
    $\AlgLibre_{\Dup_\gamma}$ of $\gamma$-edge valued binary trees
    endowed with the operations $\GDup_a$, $\DDup_a$, $a \in [\gamma]$,
    is a $\gamma$-multiplicial algebra.
\end{Lemme}
\begin{proof}
    We have to check that the operations $\GDup_a$, $\DDup_a$,
    $a \in [\gamma]$, of $\AlgLibre_{\Dup_\gamma}$ satisfy
    Relations~\eqref{equ:relation_dup_gamma_1},
    \eqref{equ:relation_dup_gamma_2},
    and~\eqref{equ:relation_dup_gamma_3} of
    $\gamma$-multiplicial algebras. Let $\Rfr$, $\Sfr$, and $\Tfr$ be
    three $\gamma$-edge valued binary trees and $a, a' \in [\gamma]$.
    \smallskip

    Denote by $\Sfr_1$ (resp. $\Sfr_2$) the left subtree (resp. right
    subtree) of $\Sfr$ and by $x$ (resp. $y$) the label of the left
    (resp. right) edge incident to the root of $\Sfr$. We have
    \begin{multline}
        (\Rfr \DDup_{a'} \Sfr) \GDup_a \Tfr  =
        \left(\Rfr \DDup_{a'} \ArbreBinValue{x}{y}{\Sfr_1}{\Sfr_2}\right)
        \GDup_a \Tfr
        = \left(\ArbreBinValue{z}{y}{\Rfr \DDup_{a'} \Sfr_1}{\Sfr_2}
        \right)
        \GDup_a \Tfr
        \displaybreak[0]
        \\
        = \ArbreBinValue{z}{t}{\Rfr \DDup_{a'} \Sfr_1}
                {\Sfr_2 \GDup_a \Tfr}
        \displaybreak[0]
        \\
        = \Rfr \DDup_{a'}
        \left(\ArbreBinValue{x}{t}{\Sfr_1}{\Sfr_2 \GDup_a \Tfr}\right)
        = \Rfr \DDup_{a'}
        \left(\ArbreBinValue{x}{y}{\Sfr_1}{\Sfr_2}
        \GDup_a \Tfr \right) =
        \Rfr \DDup_{a'} (\Sfr \GDup_a \Tfr),
    \end{multline}
    where $z := a' \Min x$ and $t := a \Min y$. This shows
    that~\eqref{equ:relation_dup_gamma_1} is satisfied in
    $\AlgLibre_{\Dup_\gamma}$.
    \medskip

    We now prove that Relations~\eqref{equ:relation_dup_gamma_2}
    and~\eqref{equ:relation_dup_gamma_3} hold by induction on the sum of
    the number of internal nodes of $\Rfr$, $\Sfr$, and $\Tfr$. Base
    case holds when all these trees have exactly one internal node, and
    since
    \begin{multline}
        \left(\Noeud \GDup_{a'} \Noeud\right) \GDup_a \Noeud
        - \Noeud \GDup_{a \Min a'} \left( \Noeud \GDup_a \Noeud \right)
        \displaybreak[0]
        \\
        =
        \begin{split}
        \begin{tikzpicture}[xscale=.3,yscale=.25]
            \node[Feuille](0)at(0.00,-1.67){};
            \node[Feuille](2)at(2.00,-3.33){};
            \node[Feuille](4)at(4.00,-3.33){};
            \node[Noeud](1)at(1.00,0.00){};
            \node[Noeud](3)at(3.00,-1.67){};
            \draw[Arete](0)--(1);
            \draw[Arete](2)--(3);
            \draw[Arete](3)edge[]node[EtiqArete]{\begin{math}a'\end{math}}(1);
            \draw[Arete](4)--(3);
            \node(r)at(1.00,1.67){};
            \draw[Arete](r)--(1);
        \end{tikzpicture}
        \end{split}
        \GDup_a \Noeud -
        \Noeud \GDup_{a \Min a'}
        \begin{split}
        \begin{tikzpicture}[xscale=.3,yscale=.25]
            \node[Feuille](0)at(0.00,-1.67){};
            \node[Feuille](2)at(2.00,-3.33){};
            \node[Feuille](4)at(4.00,-3.33){};
            \node[Noeud](1)at(1.00,0.00){};
            \node[Noeud](3)at(3.00,-1.67){};
            \draw[Arete](0)--(1);
            \draw[Arete](2)--(3);
            \draw[Arete](3)edge[]node[EtiqArete]{\begin{math}a\end{math}}(1);
            \draw[Arete](4)--(3);
            \node(r)at(1.00,1.67){};
            \draw[Arete](r)--(1);
        \end{tikzpicture}
        \end{split}
        \displaybreak[0]
        \\
        =
        \begin{split}
        \begin{tikzpicture}[xscale=.3,yscale=.25]
            \node[Feuille](0)at(0.00,-1.75){};
            \node[Feuille](2)at(2.00,-3.50){};
            \node[Feuille](4)at(4.00,-5.25){};
            \node[Feuille](6)at(6.00,-5.25){};
            \node[Noeud](1)at(1.00,0.00){};
            \node[Noeud](3)at(3.00,-1.75){};
            \node[Noeud](5)at(5.00,-3.50){};
            \draw[Arete](0)--(1);
            \draw[Arete](2)--(3);
            \draw[Arete](3)edge[]node[EtiqArete]{\begin{math}z\end{math}}(1);
            \draw[Arete](4)--(5);
            \draw[Arete](5)edge[]node[EtiqArete]{\begin{math}a\end{math}}(3);
            \draw[Arete](6)--(5);
            \node(r)at(1.00,1.75){};
            \draw[Arete](r)--(1);
        \end{tikzpicture}
        \end{split}
        -
        \begin{split}
        \begin{tikzpicture}[xscale=.3,yscale=.25]
            \node[Feuille](0)at(0.00,-1.75){};
            \node[Feuille](2)at(2.00,-3.50){};
            \node[Feuille](4)at(4.00,-5.25){};
            \node[Feuille](6)at(6.00,-5.25){};
            \node[Noeud](1)at(1.00,0.00){};
            \node[Noeud](3)at(3.00,-1.75){};
            \node[Noeud](5)at(5.00,-3.50){};
            \draw[Arete](0)--(1);
            \draw[Arete](2)--(3);
            \draw[Arete](3)edge[]node[EtiqArete]{\begin{math}z\end{math}}(1);
            \draw[Arete](4)--(5);
            \draw[Arete](5)edge[]node[EtiqArete]{\begin{math}a\end{math}}(3);
            \draw[Arete](6)--(5);
            \node(r)at(1.00,1.75){};
            \draw[Arete](r)--(1);
        \end{tikzpicture}
        \end{split}
        = 0,
    \end{multline}
    where $z := a \Min a'$, \eqref{equ:relation_dup_gamma_2}
    holds on trees with one internal node. For the same arguments,
    we can show that~\eqref{equ:relation_dup_gamma_3} holds
    on trees with exactly one internal node. Denote now by $\Rfr_1$
    (resp. $\Rfr_2$) the left subtree (resp. right subtree) of $\Rfr$
    and by $x$ (resp. $y$) the label of the left (resp. right) edge
    incident to the root of $\Rfr$. We have
    \begin{multline} \label{equ:produit_gamma_dup_expr}
        (\Rfr \GDup_{a'} \Sfr) \GDup_a \Tfr
        - \Rfr \GDup_{a \Min a'} (\Sfr \GDup_a \Tfr)
        \displaybreak[0]
        \\[1em]
        =
        \left(\ArbreBinValue{x}{y}{\Rfr_1}{\Rfr_2}
        \GDup_{a'} \Sfr \right) \GDup_a \Tfr
        -
        \ArbreBinValue{x}{y}{\Rfr_1}{\Rfr_2}
        \GDup_{a \Min a'} (\Sfr \GDup_a \Tfr)
        \displaybreak[0]
        \\
        =
        \left(\ArbreBinValue{x}{z}{\Rfr_1}{\Rfr_2 \GDup_{a'} \Sfr}
        \right) \GDup_a \Tfr
        -
        \ArbreBinValue{x}{y}{\Rfr_1}{\Rfr_2}
        \GDup_{a \Min a'} (\Sfr \GDup_a \Tfr)
        \displaybreak[0]
        \\
        =
        \ArbreBinValue{x}{t}{\Rfr_1}{(\Rfr_2 \GDup_{a'} \Sfr) \GDup_a \Tfr}
        -
        \ArbreBinValue{x}{t}{\Rfr_1}{\Rfr_2 \GDup_u (\Sfr \GDup_a \Tfr)}\,,
    \end{multline}
    where $z := y \Min a'$, $t := z \Min a = y \Min a' \Min a$,
    and $u := a \Min a'$. Now, since by induction hypothesis
    Relation~\eqref{equ:relation_dup_gamma_2} holds on $\Rfr_2$, $\Sfr$,
    and $\Tfr$, \eqref{equ:produit_gamma_dup_expr} is zero. Therefore,
    \eqref{equ:relation_dup_gamma_2} is satisfied
    in~$\AlgLibre_{\Dup_\gamma}$.
    \smallskip

    Finally, for the same arguments, we can show
    that~\eqref{equ:relation_dup_gamma_3} is satisfied in
    $\AlgLibre_{\Dup_\gamma}$, implying the statement of the lemma.
\end{proof}
\medskip

\begin{Lemme} \label{lem:produit_gamma_engendre_duplicial}
    For any integer $\gamma \geq 0$, the $\gamma$-multiplicial algebra
    $\AlgLibre_{\Dup_\gamma}$ of $\gamma$-edge valued binary trees
    endowed with the operations $\GDup_a$, $\DDup_a$, $a \in [\gamma]$,
    is generated by
    \begin{equation}
        \Noeud\,.
    \end{equation}
\end{Lemme}
\begin{proof}
    First, Lemma~\ref{lem:produit_gamma_duplicial} shows that
    $\AlgLibre_{\Dup_\gamma}$ is a $\gamma$-multiplicial algebra.
    Let $\Mca$ be the $\gamma$-multiplicial subalgebra of
    $\AlgLibre_{\Dup_\gamma}$ generated by $\NoeudTexte$. Let us show
    that any $\gamma$-edge valued binary tree $\Tfr$ is in $\Mca$
    by induction on the number $n$ of its internal nodes. When $n = 1$,
    $\Tfr = \NoeudTexte$ and hence the property is satisfied. Otherwise,
    let $\Tfr_1$ (resp. $\Tfr_2$) be the left (resp. right) subtree of
    the root of $\Tfr$ and denote by $x$ (resp. $y$) the label of the
    left (resp. right) edge incident to the root of $\Tfr$. Since $\Tfr_1$
    and~$\Tfr_2$ have less internal nodes than $\Tfr$, by induction
    hypothesis, $\Tfr_1$ and $\Tfr_2$ are in $\Mca$. Moreover, by
    definition of the operations $\GDup_a$, $\DDup_a$, $a \in [\gamma]$,
    of $\AlgLibre_{\Dup_\gamma}$, one has
    \begin{equation}
        \begin{split}\end{split}
        \left(\Tfr_1 \DDup_x \Noeud\right) \GDup_y \Tfr_2
        =
        \begin{split}
        \begin{tikzpicture}[xscale=.5,yscale=.4]
            \node(0)at(-.50,-1.50){\begin{math}\Tfr_1\end{math}};
            \node[Feuille](2)at(2.0,-1.00){};
            \node[Noeud](1)at(1.00,.50){};
            \draw[Arete](0)edge[]node[EtiqArete]
                {\begin{math}x\end{math}}(1);
            \draw[Arete](2)--(1);
            \node(r)at(1.00,1.5){};
            \draw[Arete](r)--(1);
        \end{tikzpicture}
        \end{split}
        \GDup_y \Tfr_2
        =
        \ArbreBinValue{x}{y}{\Tfr_1}{\Tfr_2} = \Tfr,
    \end{equation}
    showing that $\Tfr$ also is in $\Mca$. Therefore, $\Mca$ is
    $\AlgLibre_{\Dup_\gamma}$, showing that $\AlgLibre_{\Dup_\gamma}$ is
    generated by $\NoeudTexte$.
\end{proof}
\medskip

\begin{Theoreme} \label{thm:algebre_dup_gamma_libre}
    For any integer $\gamma \geq 0$, the vector space
    $\AlgLibre_{\Dup_\gamma}$ of $\gamma$-valued binary trees endowed
    with the operations $\GDup_a$, $\DDup_a$, $a \in [\gamma]$, is the
    free $\gamma$-multiplicial algebra over one generator.
\end{Theoreme}
\begin{proof}
    By Lemmas~\ref{lem:produit_gamma_duplicial}
    and~\ref{lem:produit_gamma_engendre_duplicial},
    $\AlgLibre_{\Dup_\gamma}$ is a $\gamma$-multiplicial algebra over
    one generator.
    \smallskip

    Moreover, since by Proposition~\ref{prop:proprietes_dup_gamma}, for
    any $n \geq 1$, the dimension of $\AlgLibre_{\Dup_\gamma}(n)$ is the
    same as the dimension of $\Dup_\gamma(n)$, there cannot be relations
    in $\AlgLibre_{\Dup_\gamma}(n)$ involving $\Gfr$ that are not
    $\gamma$-multiplicial relations (see~\eqref{equ:relation_dup_gamma_1},
    \eqref{equ:relation_dup_gamma_2}, and~\eqref{equ:relation_dup_gamma_3}).
    Hence, $\AlgLibre_{\Dup_\gamma}$ is free as a $\gamma$-multiplicial
    algebra over one generator.
\end{proof}
\medskip

\subsection{Triassociative and tridendriform operads}
Our original idea of using the $\T$ construction (see
Sections~\ref{subsec:monoides_vers_operades}
and~\ref{subsubsec:construction_dias_gamma}) to obtain a generalization
of the diassociative operad admits an analogue in the context of the
triassociative operad~\cite{LR04}. We describe in this section a
one-parameter nonnegative integer generalization of the triassociative
operad and of its Koszul dual, the tridendriform operad.
\medskip

Since the proofs of the results contained in this section are very
similar to the ones of Sections~\ref{sec:dias_gamma}
and~\ref{sec:dendr_gamma}, we omit proofs here.
\medskip

\subsubsection{Pluritriassociative operads} \label{subsubsec:trias_gamma}
For any integer $\gamma \geq 0$, we define $\Trias_\gamma$ as the
suboperad of $\Mca_\gamma$ generated by
\begin{equation} \label{equ:generateurs_trias_gamma}
   \{0a, 00, a0 : a \in [\gamma]\}.
\end{equation}
By definition, $\Trias_\gamma$ is the vector space of words that can be
obtained by partial compositions of words
of~\eqref{equ:generateurs_trias_gamma}. We have, for instance,
\begin{equation}
    \Trias_2(1) = \Vect(\{0\}),
\end{equation}
\begin{equation}
    \Trias_2(2) = \Vect(\{00, 01, 02, 10, 20\}),
\end{equation}
\begin{multline}
    \Trias_2(3)
    = \Vect(\{000, 001, 002, 010, 011, 012, 020, 021, \\
        022, 100, 101, 102, 110, 120, 200, 201, 202, 210, 220\}),
\end{multline}
\medskip

It follows immediately from the definition of $\Trias_\gamma$ as a
suboperad of $\T \Mca_\gamma$ that $\Trias_\gamma$ is a set-operad.
Moreover, one can observe that  $\Trias_\gamma$ is generated by the same
generators as the ones of $\Dias_\gamma$
(see~\eqref{equ:generateurs_dias_gamma}), plus the word $00$. Therefore,
$\Dias_\gamma$ is a suboperad of $\Trias_\gamma$. Besides, note that
$\Trias_0$ is the associative operad and that $\Trias_\gamma$ is a
suboperad of $\Trias_{\gamma + 1}$. We call $\Trias_\gamma$ the
{\em $\gamma$-pluritriassociative operad}.
\medskip

\subsubsection{Elements and dimensions}

\begin{Proposition} \label{prop:elements_trias_gamma}
    For any integer $\gamma \geq 0$, as a set-operad, the underlying set
    of $\Trias_\gamma$ is the set of the words on the alphabet
    $\{0\} \cup [\gamma]$ containing at least one occurrence of $0$.
\end{Proposition}
\medskip

We deduce from Proposition~\ref{prop:elements_trias_gamma} that the
Hilbert series of $\Trias_\gamma$ satisfies
\begin{equation}
    \Hca_{\Trias_\gamma}(t) =
    \frac{t}{(1 - \gamma t)(1 - \gamma t - t)}
\end{equation}
and that for all $n \geq 1$,
$\dim \Trias_\gamma(n) = (\gamma + 1)^n - \gamma^n$. For instance, the
first dimensions of  $\Trias_1$, $\Trias_2$, $\Trias_3$, and $\Trias_4$
are respectively
\begin{equation}
    1, 3, 7, 15, 31, 63, 127, 255, 511, 1023, 2047,
\end{equation}
\begin{equation}
    1, 5, 19, 65, 211, 665, 2059, 6305, 19171, 58025, 175099,
\end{equation}
\begin{equation}
    1, 7, 37, 175, 781, 3367, 14197, 58975, 242461, 989527, 4017157,
\end{equation}
\begin{equation}
    1, 9, 61, 369, 2101, 11529, 61741, 325089, 1690981, 8717049, 44633821.
\end{equation}
The first one is Sequence~\Sloane{A000225}, the second one is
Sequence~\Sloane{A001047}, the third one is Sequence~\Sloane{A005061},
and the last one is Sequence~\Sloane{A005060} of~\cite{Slo}.
\medskip

\subsubsection{Presentation and Koszulity}
We follow the same strategy as the one used in
Section~\ref{subsec:presentation_dias_gamma} to establish a presentation
by generators and relations of $\Trias_\gamma$ and prove that it is a
Koszul operad. As announced above, we omit complete proofs here but
we describe the analogue for $\Trias_\gamma$ of the maps $\Mot_\gamma$
and $\Equerre_\gamma$ defined in
Section~\ref{subsec:presentation_dias_gamma} for the operad $\Dias_\gamma$.
\medskip

For any integer $\gamma \geq 0$, let $\GenTrias := \GenTrias(2)$ be the
graded set where
\begin{equation}
    \GenTrias(2) := \{\GDias_a, \MTrias, \DDias_a : a \in [\gamma] \}.
\end{equation}
\medskip

Let $\Tfr$ be a syntax tree of $\OpLibre\left(\GenTrias\right)$ and $x$ be
a leaf of $\Tfr$. We say that an integer $a \in \{0\} \cup [\gamma]$ is
{\em eligible} for $x$ if $a = 0$ or there is an ancestor $y$ of $x$
labeled by $\GDias_a$ (resp. $\DDias_a$) and $x$ is in the right (resp.
left) subtree of $y$. The {\em image} of $x$ is its greatest eligible
integer. Moreover, let
\begin{equation}
    \MotT_\gamma : \OpLibre\left(\GenTrias\right)(n) \to \Trias_\gamma(n),
    \qquad n \geq 1,
\end{equation}
the map where $\MotT_\gamma(\Tfr)$ is the word obtained by considering,
from left to right, the images of the leaves of $\Tfr$
(see Figure~\ref{fig:exemple_mot_gamma_trias_gamma}).
\begin{figure}[ht]
    \centering
     \begin{tikzpicture}[xscale=.28,yscale=.18]
        \node(0)at(0.00,-15.33){};
        \node(10)at(10.00,-11.50){};
        \node(12)at(12.00,-15.33){};
        \node(14)at(14.00,-15.33){};
        \node(16)at(16.00,-19.17){};
        \node(18)at(18.00,-19.17){};
        \node(2)at(2.00,-15.33){};
        \node(20)at(20.00,-15.33){};
        \node(22)at(22.00,-11.50){};
        \node(4)at(4.00,-11.50){};
        \node(6)at(6.00,-11.50){};
        \node(8)at(8.00,-11.50){};
        \node(1)at(1.00,-11.50){\begin{math}\GDias_1\end{math}};
        \node(3)at(3.00,-7.67){\begin{math}\DDias_3\end{math}};
        \node(5)at(5.00,-3.83){\begin{math}\GDias_4\end{math}};
        \node(7)at(7.00,-7.67){\begin{math}\MTrias\end{math}};
        \node(9)at(9.00,0.00){\begin{math}\DDias_2\end{math}};
        \node(11)at(11.00,-7.67){\begin{math}\GDias_3\end{math}};
        \node(13)at(13.00,-11.50){\begin{math}\DDias_4\end{math}};
        \node(15)at(15.00,-3.83){\begin{math}\MTrias\end{math}};
        \node(17)at(17.00,-15.33){\begin{math}\DDias_3\end{math}};
        \node(19)at(19.00,-11.50){\begin{math}\DDias_2\end{math}};
        \node(21)at(21.00,-7.67){\begin{math}\GDias_1\end{math}};
        \draw(0)--(1);\draw(1)--(3);\draw(10)--(11);\draw(11)--(15);
        \draw(12)--(13);\draw(13)--(11);\draw(14)--(13);\draw(15)--(9);
        \draw(16)--(17);\draw(17)--(19);\draw(18)--(17);\draw(19)--(21);
        \draw(2)--(1);\draw(20)--(19);\draw(21)--(15);\draw(22)--(21);
        \draw(3)--(5);\draw(4)--(3);\draw(5)--(9);\draw(6)--(7);
        \draw(7)--(5);\draw(8)--(7);
        \node(r)at(9.00,3){};
        \draw(r)--(9);
        \node[below of=0,node distance=3mm]
            {\small \begin{math}\textcolor{Bleu}{3}\end{math}};
        \node[below of=2,node distance=3mm]
            {\small \begin{math}\textcolor{Bleu}{3}\end{math}};
        \node[below of=4,node distance=3mm]
            {\small \begin{math}\textcolor{Bleu}{2}\end{math}};
        \node[below of=6,node distance=3mm]
            {\small \begin{math}\textcolor{Bleu}{4}\end{math}};
        \node[below of=8,node distance=3mm]
            {\small \begin{math}\textcolor{Bleu}{4}\end{math}};
        \node[below of=10,node distance=3mm]
            {\small \begin{math}\textcolor{Bleu}{0}\end{math}};
        \node[below of=12,node distance=3mm]
            {\small \begin{math}\textcolor{Bleu}{4}\end{math}};
        \node[below of=14,node distance=3mm]
            {\small \begin{math}\textcolor{Bleu}{3}\end{math}};
        \node[below of=16,node distance=3mm]
            {\small \begin{math}\textcolor{Bleu}{3}\end{math}};
        \node[below of=18,node distance=3mm]
            {\small \begin{math}\textcolor{Bleu}{2}\end{math}};
        \node[below of=20,node distance=3mm]
            {\small \begin{math}\textcolor{Bleu}{0}\end{math}};
        \node[below of=22,node distance=3mm]
            {\small \begin{math}\textcolor{Bleu}{1}\end{math}};
    \end{tikzpicture}
    \caption{A syntax tree $\Tfr$ of $\OpLibre\left(\GenTrias\right)$
    where images of its leaves are shown. This tree satisfies
    $\MotT_\gamma(\Tfr) = \textcolor{Bleu}{332440433201}$.}
    \label{fig:exemple_mot_gamma_trias_gamma}
\end{figure}
Observe that $\MotT_\gamma$ is an extension of $\Mot_\gamma$
(see~\eqref{equ:application_mot_gamma}).
\medskip

Consider now the map
\begin{equation}
    \EquerreT_\gamma :
    \Trias_\gamma(n) \to \OpLibre\left(\GenTrias\right)(n),
    \qquad n \geq 1,
\end{equation}
defined for any word $x$ of $\Trias_\gamma$ by
\begin{equation} \label{equ:definition_equerre_trias_gamma}
    \begin{split}\EquerreT_\gamma(x)\end{split} :=
    \begin{split}
    \begin{tikzpicture}[xscale=.5,yscale=.45]
        \node(0)at(0.00,-5.40){};
        \node(2)at(3.00,-7.20){};
        \node(4)at(5.00,-7.20){};
        \node(6)at(6.00,-3.60){};
        \node(8)at(9.00,-1.80){};
        \node(1)at(3.00,-4){\begin{math}\Equerre_\gamma(u)\end{math}};
        \node(5)at(5.00,-2){\begin{math}\MTrias\end{math}};
        \node(7)at(12.00,2.00){\begin{math}\MTrias\end{math}};
        \node(9)at(7.5,-4){\begin{math}\GDias_{v^{(1)}_{k^{(1)}}}\end{math}};
        \node(10)at(5,-6.5){\begin{math}\GDias_{v^{(1)}_1}\end{math}};
        \node(11)at(14.5,0){\begin{math}\GDias_{v^{(\ell)}_{k^{(\ell)}}}\end{math}};
        \node(12)at(12,-2.5){\begin{math}\GDias_{v^{(\ell)}_1}\end{math}};
        \node(r)at(12,3.5){};
        \draw(1)--(5);
        \draw[densely dashed](5)--(7);
        \draw(9)--(5);
        \draw(11)--(7);
        \draw(7)--(r);
        \draw[densely dashed](9)--(10);
        \draw[densely dashed](11)--(12);
        \node(f1)at(4,-8){}; \draw(f1)--(10);
        \node(f2)at(6,-8){}; \draw(f2)--(10);
        \node(f3)at(8.5,-5.5){}; \draw(f3)--(9);
        \node(f4)at(11,-4){}; \draw(f4)--(12);
        \node(f5)at(13,-4){}; \draw(f5)--(12);
        \node(f6)at(15.5,-1.5){}; \draw(f6)--(11);
    \end{tikzpicture}
    \end{split}\,,
\end{equation}
where $x$ decomposes, by Proposition~\ref{prop:elements_trias_gamma},
uniquely in $x = u0v^{(1)}\dots 0v^{(\ell)}$ where $u$ is a word of
$\Dias_\gamma$ and for all $i \in [\ell]$, the $v^{(i)}$ are words on
the alphabet $[\gamma]$. The length $|v^{(i)}|$ of any $v_i$ is denoted
by $k^{(i)}$. The dashed edges denote left comb trees wherein internal
nodes are labeled as specified. Observe that $\EquerreT_\gamma$ is an
extension of $\Equerre_\gamma$ (see~\eqref{equ:application_equerre_gamma}).
We shall call any syntax tree of the
form~\eqref{equ:definition_equerre_trias_gamma} an
{\em extended hook syntax tree}.
\medskip

\begin{Theoreme} \label{thm:presentation_trias_gamma}
    For any integer $\gamma \geq 0$, the operad $\Trias_\gamma$ admits
    the following presentation. It is generated by $\GenTrias$ and its
    space of relations $\RelTrias$ is the space induced by the
    equivalence relation $\Rel_\gamma$ satisfying
    \begin{subequations}
    \begin{equation}\label{equ:relation_presentation_trias_gamma_1}
        \MTrias \circ_1 \MTrias
        \enspace \Rel_\gamma \enspace
        \MTrias \circ_2 \MTrias,
    \end{equation}
    \begin{equation}\label{equ:relation_presentation_trias_gamma_2}
        \GDias_a \circ_1 \MTrias
        \enspace \Rel_\gamma \enspace
        \MTrias \circ_2 \GDias_a,
        \qquad a \in [\gamma],
    \end{equation}
    \begin{equation}\label{equ:relation_presentation_trias_gamma_3}
        \MTrias \circ_1 \DDias_a
        \enspace \Rel_\gamma \enspace
        \DDias_a \circ_2 \MTrias,
        \qquad a \in [\gamma],
    \end{equation}
    \begin{equation}\label{equ:relation_presentation_trias_gamma_4}
        \MTrias \circ_1 \GDias_a
        \enspace \Rel_\gamma \enspace
        \MTrias \circ_2 \DDias_a,
        \qquad a \in [\gamma],
    \end{equation}
    \begin{equation}\label{equ:relation_presentation_trias_gamma_5}
        \GDias_a \circ_1 \DDias_{a'}
        \enspace \Rel_\gamma \enspace
        \DDias_{a'} \circ_2 \GDias_a,
        \qquad a, a' \in [\gamma],
    \end{equation}
    \begin{equation}\label{equ:relation_presentation_trias_gamma_6}
        \GDias_a \circ_1 \GDias_b
        \enspace \Rel_\gamma \enspace
        \GDias_a \circ_2 \DDias_b,
        \qquad a < b \in [\gamma],
    \end{equation}
    \begin{equation}\label{equ:relation_presentation_trias_gamma_7}
        \DDias_a \circ_1 \GDias_b
        \enspace \Rel_\gamma \enspace
        \DDias_a \circ_2 \DDias_b,
        \qquad a < b \in [\gamma],
    \end{equation}
    \begin{equation}\label{equ:relation_presentation_trias_gamma_8}
        \GDias_b \circ_1 \GDias_a
        \enspace \Rel_\gamma \enspace
        \GDias_a \circ_2 \GDias_b,
        \qquad a < b \in [\gamma],
    \end{equation}
    \begin{equation}\label{equ:relation_presentation_trias_gamma_9}
        \DDias_a \circ_1 \DDias_b
        \enspace \Rel_\gamma \enspace
        \DDias_b \circ_2 \DDias_a,
        \qquad a < b \in [\gamma],
    \end{equation}
    \begin{equation}\label{equ:relation_presentation_trias_gamma_10}
        \GDias_d \circ_1 \GDias_d
        \enspace \Rel_\gamma \enspace
        \GDias_d \circ_2 \MTrias
        \enspace \Rel_\gamma \enspace
        \GDias_d \circ_2 \GDias_c
        \enspace \Rel_\gamma \enspace
        \GDias_d \circ_2 \DDias_c,
        \qquad c \leq d \in [\gamma],
    \end{equation}
    \begin{equation}\label{equ:relation_presentation_trias_gamma_11}
        \DDias_d \circ_1 \GDias_c
        \enspace \Rel_\gamma \enspace
        \DDias_d \circ_1 \DDias_c
        \enspace \Rel_\gamma \enspace
        \DDias_d \circ_1 \MTrias
        \enspace \Rel_\gamma \enspace
        \DDias_d \circ_2 \DDias_d,
        \qquad c \leq d \in [\gamma].
    \end{equation}
    \end{subequations}
\end{Theoreme}
\medskip

Observe that, by Theorem~\ref{thm:presentation_trias_gamma}, $\Trias_1$
and the triassociative operad~\cite{LR04} admit the same presentation.
Then, for all integers $\gamma \geq 0$, the operads $\Trias_\gamma$ are
generalizations of the triassociative operad.
\medskip

\begin{Theoreme} \label{thm:koszulite_trias_gamma}
    For any integer $\gamma \geq 0$, $\Trias_\gamma$ is a Koszul operad.
    Moreover, the set of extended hook syntax trees of
    $\OpLibre\left(\GenTrias\right)$ forms a Poincaré-Birkhoff-Witt
    basis of $\Trias_\gamma$.
\end{Theoreme}
\medskip

\subsubsection{Polytridendriform operads} \label{subsubsec:tdendr_gamma}
Theorem~\ref{thm:presentation_trias_gamma}, by exhibiting a presentation
of $\Trias_\gamma$, shows that this operad is binary and quadratic.
It then admits a Koszul dual, denoted by $\TDendr_\gamma$ and called
{\em $\gamma$-polytridendriform operad}.
\medskip

\begin{Theoreme} \label{thm:presentation_tdendr_gamma}
    For any integer $\gamma \geq 0$, the operad $\TDendr_\gamma$ admits
    the following presentation. It is generated by
    $\GenTDendr := \GenTDendr(2) :=
    \{\GDendrA_a, \MTDendr, \DDendrA_a : a \in [\gamma]\}$ and its space
    of relations $\RelTDendr$ is generated by
    \begin{subequations}
    \begin{equation}\label{equ:relation_presentation_tdendr_gamma_1}
        \MTDendr \circ_1 \MTDendr
         -
        \MTDendr \circ_2 \MTDendr,
    \end{equation}
    \begin{equation}\label{equ:relation_presentation_tdendr_gamma_2}
        \GDendrA_a \circ_1 \MTDendr
         -
        \MTDendr \circ_2 \GDendrA_a,
        \qquad a \in [\gamma],
    \end{equation}
    \begin{equation}\label{equ:relation_presentation_tdendr_gamma_3}
        \MTDendr \circ_1 \DDendrA_a
         -
        \DDendrA_a \circ_2 \MTDendr,
        \qquad a \in [\gamma],
    \end{equation}
    \begin{equation}\label{equ:relation_presentation_tdendr_gamma_4}
        \MTDendr \circ_1 \GDendrA_a
         -
        \MTDendr \circ_2 \DDendrA_a,
        \qquad a \in [\gamma],
    \end{equation}
    \begin{equation}\label{equ:relation_presentation_tdendr_gamma_5}
        \GDendrA_a \circ_1 \DDendrA_{a'}
         -
        \DDendrA_{a'} \circ_2 \GDendrA_a,
        \qquad a, a' \in [\gamma],
    \end{equation}
    \begin{equation}\label{equ:relation_presentation_tdendr_gamma_6}
        \GDendrA_a \circ_1 \GDendrA_b
         -
        \GDendrA_a \circ_2 \DDendrA_b,
        \qquad a < b \in [\gamma],
    \end{equation}
    \begin{equation}\label{equ:relation_presentation_tdendr_gamma_7}
        \DDendrA_a \circ_1 \GDendrA_b
         -
        \DDendrA_a \circ_2 \DDendrA_b,
        \qquad a < b \in [\gamma],
    \end{equation}
    \begin{equation}\label{equ:relation_presentation_tdendr_gamma_8}
        \GDendrA_b \circ_1 \GDendrA_a
         -
        \GDendrA_a \circ_2 \GDendrA_b,
        \qquad a < b \in [\gamma],
    \end{equation}
    \begin{equation}\label{equ:relation_presentation_tdendr_gamma_9}
        \DDendrA_a \circ_1 \DDendrA_b
         -
        \DDendrA_b \circ_2 \DDendrA_a,
        \qquad a < b \in [\gamma],
    \end{equation}
    \begin{equation}\label{equ:relation_presentation_tdendr_gamma_10}
        \GDendrA_d \circ_1 \GDendrA_d -
        \GDendrA_d \circ_2 \MTDendr -
        \left(\sum_{c \in [d]}
        \GDendrA_d \circ_2 \GDendrA_c +
        \GDendrA_d \circ_2 \DDendrA_c
        \right),
        \qquad d \in [\gamma],
    \end{equation}
    \begin{equation}\label{equ:relation_presentation_tdendr_gamma_11}
        \left(\sum_{c \in [d]}
        \DDendrA_d \circ_1 \GDendrA_c +
        \DDendrA_d \circ_1 \DDendrA_c
        \right) +
        \DDendrA_d \circ_1 \MTDendr -
        \DDendrA_d \circ_2 \DDendrA_d,
        \qquad d \in [\gamma].
    \end{equation}
    \end{subequations}
\end{Theoreme}
\medskip

\begin{Proposition} \label{prop:serie_hilbert_tdendr_gamma}
    For any integer $\gamma \geq 0$, the Hilbert series
    $\Hca_{\TDendr_\gamma}(t)$ of the operad $\TDendr_\gamma$ satisfies
    \begin{equation}
        \Hca_{\TDendr_\gamma}(t) =
        t + (2\gamma + 1) t \, \Hca_{\TDendr_\gamma}(t) +
        \gamma (\gamma + 1) t \, \Hca_{\TDendr_\gamma}(t)^2.
    \end{equation}
\end{Proposition}
\medskip

By examining the expression for $\Hca_{\TDendr_\gamma}(t)$ of the
statement of Proposition~\ref{prop:serie_hilbert_tdendr_gamma}, we
observe that for any $n \geq 1$, $\TDendr(n)$ can be seen as the
vector space $\AlgLibre_{\TDendr_\gamma}(n)$ of Schröder trees
with $n + 1$ leaves wherein its edges connecting two internal nodes
are labeled on $[\gamma]$. We call these trees
{\em $\gamma$-edge valued Schröder trees}. For instance,
\begin{equation}
    \begin{split}
    \begin{tikzpicture}[xscale=.25,yscale=.10]
        \node[Feuille](0)at(0.00,-21.60){};
        \node[Feuille](11)at(10.00,-10.80){};
        \node[Feuille](13)at(11.00,-16.20){};
        \node[Feuille](15)at(12.00,-16.20){};
        \node[Feuille](16)at(13.00,-16.20){};
        \node[Feuille](18)at(14.00,-10.80){};
        \node[Feuille](19)at(15.00,-16.20){};
        \node[Feuille](2)at(2.00,-21.60){};
        \node[Feuille](21)at(17.00,-21.60){};
        \node[Feuille](23)at(19.00,-21.60){};
        \node[Feuille](24)at(20.00,-10.80){};
        \node[Feuille](26)at(22.00,-10.80){};
        \node[Feuille](4)at(3.00,-21.60){};
        \node[Feuille](6)at(5.00,-21.60){};
        \node[Feuille](7)at(6.00,-21.60){};
        \node[Feuille](9)at(8.00,-21.60){};
        \node[Noeud](1)at(1.00,-16.20){};
        \node[Noeud](10)at(9.00,-5.40){};
        \node[Noeud](12)at(15.00,0.00){};
        \node[Noeud](14)at(12.00,-10.80){};
        \node[Noeud](17)at(14.00,-5.40){};
        \node[Noeud](20)at(16.00,-10.80){};
        \node[Noeud](22)at(18.00,-16.20){};
        \node[Noeud](25)at(21.00,-5.40){};
        \node[Noeud](3)at(4.00,-10.80){};
        \node[Noeud](5)at(4.00,-16.20){};
        \node[Noeud](8)at(7.00,-16.20){};
        \draw[Arete](0)--(1);
        \draw[Arete](1)edge[]node[EtiqArete]{\begin{math}4\end{math}}(3);
        \draw[Arete](10)edge[]node[EtiqArete]{\begin{math}2\end{math}}(12);
        \draw[Arete](11)--(10);
        \draw[Arete](13)--(14);
        \draw[Arete](14)edge[]node[EtiqArete]{\begin{math}1\end{math}}(17);
        \draw[Arete](15)--(14);
        \draw[Arete](16)--(14);
        \draw[Arete](17)edge[]node[EtiqArete]{\begin{math}4\end{math}}(12);
        \draw[Arete](18)--(17);
        \draw[Arete](19)--(20);
        \draw[Arete](2)--(1);
        \draw[Arete](20)edge[]node[EtiqArete]{\begin{math}2\end{math}}(17);
        \draw[Arete](21)--(22);
        \draw[Arete](22)edge[]node[EtiqArete]{\begin{math}4\end{math}}(20);
        \draw[Arete](23)--(22);
        \draw[Arete](24)--(25);
        \draw[Arete](25)edge[]node[EtiqArete]{\begin{math}4\end{math}}(12);
        \draw[Arete](26)--(25);
        \draw[Arete](3)edge[]node[EtiqArete]{\begin{math}2\end{math}}(10);
        \draw[Arete](4)--(5);
        \draw[Arete](5)edge[]node[EtiqArete]{\begin{math}1\end{math}}(3);
        \draw[Arete](6)--(5);
        \draw[Arete](7)--(8);
        \draw[Arete](8)edge[]node[EtiqArete]{\begin{math}4\end{math}}(3);
        \draw[Arete](9)--(8);
        \node(r)at(15.00,3.5){};
        \draw[Arete](r)--(12);
    \end{tikzpicture}
    \end{split}
\end{equation}
is a $4$-edge valued Schröder tree and a basis element of $\TDendr_4(16)$.
\medskip

We deduce from Proposition~\ref{prop:serie_hilbert_tdendr_gamma} that
\begin{equation}
    \Hca_{\TDendr_\gamma}(t) =
    \frac{1 - \sqrt{1 - (4\gamma + 2)t + t^2} - (2\gamma + 1)t}
    {2(\gamma + \gamma^2)t}.
\end{equation}
Moreover, we obtain that for all $n \geq 1$,
\begin{equation}
    \dim \TDendr_\gamma(n) =
    \sum_{k = 0}^{n - 1} (\gamma + 1)^k \gamma^{n - k - 1} \, \Nar(n, k),
\end{equation}
where $\Nar(n, k)$ is defined in~\eqref{equ:definition_narayana}.
For instance, the first dimensions of $\TDendr_1$, $\TDendr_2$,
$\TDendr_3$, and $\TDendr_4$ are respectively
\begin{equation}
    1, 3, 11, 45, 197, 903, 4279, 20793, 103049, 518859, 2646723,
\end{equation}
\begin{equation}
    1, 5, 31, 215, 1597, 12425, 99955, 824675, 6939769, 59334605, 513972967,
\end{equation}
\begin{equation}
    1, 7, 61, 595, 6217, 68047, 770149, 8939707, 105843409,
    1273241431, 15517824973,
\end{equation}
\begin{equation}
    1, 9, 101, 1269, 17081, 240849, 3511741, 52515549, 801029681,
    12414177369, 194922521301.
\end{equation}
The first one is Sequence~\Sloane{A001003} of~\cite{Slo}. The others
sequences are not listed in~\cite{Slo} at this time.
\medskip

\subsection{Operads of the operadic butterfly}
The {\em operadic butterfly}~\cite{Lod01,Lod06} is a diagram gathering
seven famous operads. We have seen in
Section~\ref{subsec:diagramme_dias_as_dendr_gamma} that this diagram
gathers the diassociative, associative, and dendriform operads. It
involves also the {\em commutative operad} $\Com$, the {\em Lie operad}
$\Lie$, the {\em Zinbiel operad} $\Zin$~\cite{Lod95}, and
the {\em Leibniz operad} $\Leib$~\cite{Lod93}. It is of the form
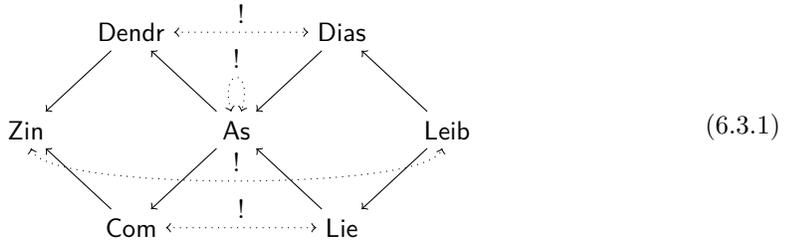
\begin{equation} \label{equ:diagramme_papillon}
    \begin{split}
    \begin{tikzpicture}[xscale=.7,yscale=.65]
        \node(Dendr)at(-2,2){\begin{math}\Dendr\end{math}};
        \node(As)at(0,0){\begin{math}\As\end{math}};
        \node(Dias)at(2,2){\begin{math}\Dias\end{math}};
        \node(Com)at(-2,-2){\begin{math}\Com\end{math}};
        \node(Lie)at(2,-2){\begin{math}\Lie\end{math}};
        \node(Zin)at(-4,0){\begin{math}\Zin\end{math}};
        \node(Leib)at(4,0){\begin{math}\Leib\end{math}};
        \draw[->](Dias)--(As);
        \draw[->](As)--(Dendr);
        \draw[->](Dendr)--(Zin);
        \draw[->](Leib)--(Dias);
        \draw[->](Com)--(Zin);
        \draw[->](As)--(Com);
        \draw[->](Lie)--(As);
        \draw[->](Leib)--(Lie);
        \draw[<->,dotted,loop above,looseness=13](As)
            edge node[anchor=south]{\begin{math} ! \end{math}}(As);
        \draw[<->,dotted](Dendr)
            edge node[anchor=south]{\begin{math} ! \end{math}}(Dias);
        \draw[<->,dotted](Com)
            edge node[anchor=south]{\begin{math} ! \end{math}}(Lie);
        \draw[<->,dotted,loop below,looseness=.3](Zin)
            edge node[anchor=south]{\begin{math} ! \end{math}}(Leib);
    \end{tikzpicture}
    \end{split}
\end{equation}
and as it shows, some operads are Koszul dual of some others
(in particular, $\Com^! = \Lie$ and $\Zin^! = \Leib$).
\medskip

We have to emphasize the fact the operads $\Com$, $\Lie$, $\Zin$, and
$\Leib$ of the operadic butterfly are symmetric operads. The computation
of the Koszul dual of a symmetric operad does not follows what we have
presented in Section~\ref{subsubsec:dual_de_Koszul}. We invite the
reader to consult~\cite{GK94} or~\cite{LV12} for a complete
description.
\medskip

For simplicity, in what follows, we shall consider algebras over
symmetric operads instead of symmetric operads.
\medskip

\subsubsection{A generalization of the operadic butterfly}
A possible continuation to this work consists in constructing a diagram
\begin{equation} \label{equ:diagramme_papillon_gamma}
    \begin{split}
    \begin{tikzpicture}[xscale=.7,yscale=.65]
        \node(Dendr)at(-4,2){\begin{math}\Dendr_\gamma\end{math}};
        \node(As)at(1.5,0){\begin{math}\As_\gamma\end{math}};
        \node(DAs)at(-1.5,0){\begin{math}\DAs_\gamma\end{math}};
        \node(Dias)at(4,2){\begin{math}\Dias_\gamma\end{math}};
        \node(Com)at(-4,-2){\begin{math}\Com_\gamma\end{math}};
        \node(Lie)at(4,-2){\begin{math}\Lie_\gamma\end{math}};
        \node(Zin)at(-6,0){\begin{math}\Zin_\gamma\end{math}};
        \node(Leib)at(6,0){\begin{math}\Leib_\gamma\end{math}};
        \draw[->](Dias)--(As);
        \draw[->](DAs)--(Dendr);
        \draw[->](Dendr)--(Zin);
        \draw[->](Leib)--(Dias);
        \draw[->](Com)--(Zin);
        \draw[->](DAs)--(Com);
        \draw[->](Lie)--(As);
        \draw[->](Leib)--(Lie);
        \draw[<->,dotted](As)
            edge node[anchor=south]{\begin{math} ! \end{math}}(DAs);
        \draw[<->,dotted](Dendr)
            edge node[anchor=south]{\begin{math} ! \end{math}}(Dias);
        \draw[<->,dotted](Com)
            edge node[anchor=south]{\begin{math} ! \end{math}}(Lie);
        \draw[<->,dotted,loop below,looseness=.2](Zin)
            edge node[anchor=south]{\begin{math} ! \end{math}}(Leib);
    \end{tikzpicture}
    \end{split}
\end{equation}
where $\DAs_\gamma$ is the $\gamma$-dual multiassociative operad
defined in Section~\ref{subsubsec:das_gamma} and $\Com_\gamma$,
$\Lie_\gamma$, $\Zin_\gamma$, and $\Leib_\gamma$, respectively are
one-parameter nonnegative integer generalizations of the operads $\Com$,
$\Lie$, $\Zin$, and $\Leib$. Let us now define these operads.
\medskip

\subsubsection{Commutative and Lie operads} \label{subsubsec:com_gamma}
The symmetric operad $\Com$ is the symmetric operad describing the
category of algebras $\Cca$ with one binary operation $\MDAs$, subjected
for any elements $x$, $y$, and $z$ of $\Cca$ to the two relations
\begin{subequations}
\begin{equation}
    x \MDAs y = y \MDAs x,
\end{equation}
\begin{equation}
    (x \MDAs y) \MDAs z = x \MDAs (y \MDAs z).
\end{equation}
\end{subequations}
This operad has the property to be a commutative version of $\As = \DAs_1$.
\medskip

We define the symmetric operad $\Com_\gamma$ by using the same idea of
being a commutative version of $\DAs_\gamma$. Therefore, $\Com_\gamma$
is the symmetric operad describing the category of algebras $\Cca$ with
binary operations $\MDAs_a$, $a \in [\gamma]$, subjected for any elements
$x$, $y$, and $z$ of $\Cca$ to the two sorts of relations
\begin{subequations}
\begin{equation} \label{equ:relation_com_gamma_1}
    x \MDAs_a y = y \MDAs_a x,
    \qquad a \in [\gamma],
\end{equation}
\begin{equation} \label{equ:relation_com_gamma_2}
    (x \MDAs_a y) \MDAs_a z = x \MDAs_a (y \MDAs_a z),
    \qquad a \in [\gamma].
\end{equation}
\end{subequations}
Moreover, we define the symmetric operad $\Lie_\gamma$ as the Koszul dual
of $\Com_\gamma$.
\medskip

\subsubsection{Zinbiel and Leibniz operads} \label{subsubsec:zin_gamma}
The symmetric operad $\Zin$ is the symmetric operad describing
the category of algebras $\Zca$ with one generating binary operation
$\ProdZin$, subjected for any elements $x$, $y$, and $z$ of $\Zca$ to
the relation
\begin{equation} \label{equ:relation_zinbiel}
    (x \ProdZin y) \ProdZin z =
    x \ProdZin (y \ProdZin z) + x \ProdZin (z \ProdZin y).
\end{equation}
This operad has the property to be a commutative version of
$\Dendr = \Dendr_1$. Indeed, Relation~\eqref{equ:relation_zinbiel} is
obtained from Relations~\eqref{equ:relation_dendr_1},
\eqref{equ:relation_dendr_2}, and~\eqref{equ:relation_dendr_3} of
dendriform algebras with the condition that for any elements $x$ and $y$,
$x \GDendr y = y \DDendr x$, and by setting $x \ProdZin y := x \GDendr y$.
\medskip

We define the symmetric operad $\Zin_\gamma$ by using the same idea of
having the property to be a commutative version of $\Dendr_\gamma$.
Therefore, $\Zin_\gamma$ is the symmetric operad describing the
category of algebras $\Zca$ with binary operations $\ProdZin_a$,
$a \in [\gamma]$, subjected for any elements $x$, $y$, and $z$ of $\Zca$
to the relation
\begin{equation} \label{equ:relation_zinbiel_gamma}
    (x \ProdZin_{a'} y) \ProdZin_a z =
    x \ProdZin_{a \Min a'} (y \ProdZin_a z) +
    x \ProdZin_{a \Min a'} (z \ProdZin_{a'} y),
    \qquad a, a' \in [\gamma].
\end{equation}
Relation~\eqref{equ:relation_zinbiel_gamma} is obtained
from Relations~\eqref{equ:relation_dendr_gamma_1_concise},
\eqref{equ:relation_dendr_gamma_2_concise},
and~\eqref{equ:relation_dendr_gamma_3_concise} of
$\gamma$-polydendriform algebras with the condition that for any elements
$x$ and $y$ and $a \in [\gamma]$, $x \GDendr_a y = y \DDendr_a x$, and
by setting $x \ProdZin_a y := x \GDendr_a y$. Moreover, we define the
symmetric operad $\Leib_\gamma$ as the Koszul dual of $\Zin_\gamma$.
\medskip

\begin{Proposition} \label{prop:morphism_com_zin_gamma}
    For any integer $\gamma \geq 0$ and any $\Zin_\gamma$-algebra
    $\Zca$, the binary operations $\MDAs_a$, $a \in [\gamma]$, defined
    for all elements $x$ and $y$ of $\Zca$ by
    \begin{equation}
        x \MDAs_a y := x \ProdZin_a y + y \ProdZin_a x,
        \qquad a \in [\gamma],
    \end{equation}
    endow $\Zca$ with a $\Com_\gamma$-algebra structure.
\end{Proposition}
\begin{proof}
    Since for all $a \in [\gamma]$ and all elements $x$ and $y$ of $\Zca$,
    by~\eqref{equ:relation_zinbiel_gamma}, we have
    \begin{equation}
        x \MDAs_a y - y \MDAs_a x =
        x \ProdZin_a y + y \ProdZin_a x - y \ProdZin_a x - x \ProdZin_a y
        = 0,
    \end{equation}
    the operations $\MDAs_a$ satisfy Relation~\eqref{equ:relation_com_gamma_1}
    of $\Com_\gamma$-algebras. Moreover, since for all $a \in [\gamma]$
    and all elements $x$, $y$, and $z$ of $\Zca$,
    by~\eqref{equ:relation_zinbiel_gamma}, we have
    \begin{equation}\begin{split}
        (x \MDAs_a y) \MDAs_a z
        & - x \MDAs_a (y \MDAs_a z) \\
        & =
        (x \ProdZin_a y + y \ProdZin_a x) \ProdZin_a z +
        z \ProdZin_a (x \ProdZin_a y + y \ProdZin_a x) \\
        & \qquad - x \ProdZin_a (y \ProdZin_a z + z \ProdZin_a y)
        - (y \ProdZin_a z + z \ProdZin_a y) \ProdZin_a x \\
        & = (x \ProdZin_a y) \ProdZin_a z + (y \ProdZin_a x) \ProdZin_a z
        + z \ProdZin_a (x \ProdZin_a y) + z \ProdZin_a (y \ProdZin_a x) \\
        & \qquad - x \ProdZin_a (y \ProdZin_a z) - x \ProdZin_a (z \ProdZin_a y)
        - (y \ProdZin_a z) \ProdZin_a x - (z \ProdZin_a y) \ProdZin_a x \\
        & = (y \ProdZin_a x) \ProdZin_a z - (y \ProdZin_a z) \ProdZin_a x \\
        & = y \ProdZin_a (x \ProdZin_a z) + y \ProdZin_a (z \ProdZin_a x)
        - y \ProdZin_a (z \ProdZin_a x) - y \ProdZin_a (x \ProdZin_a z) \\
        & = 0,
    \end{split}\end{equation}
    the operations $\MDAs_a$ satisfy Relation~\eqref{equ:relation_com_gamma_2}
    of $\Com_\gamma$-algebras.
    Hence, $\Zca$ is a $\Com_\gamma$-algebra.
\end{proof}
\medskip

\begin{Proposition} \label{prop:morphism_dendr_zin_gamma}
    For any integer $\gamma \geq 0$, and any $\Zin_\gamma$-algebra
    $\Zca$, the binary operations $\GDendr_a$, $\DDendr_a$,
    $a \in [\gamma]$ defined for all elements $x$ and $y$ of $\Zca$ by
    \begin{equation}
        x \GDendr_a y := x \ProdZin_a y,
        \qquad a \in [\gamma],
    \end{equation}
    and
    \begin{equation}
        x \DDendr_a y := y \ProdZin_a x,
        \qquad a \in [\gamma],
    \end{equation}
    endow $\Zca$ with a $\gamma$-polydendriform algebra structure.
\end{Proposition}
\begin{proof}
    Since, for all $a, a' \in [\gamma]$ and all elements $x$, $y$,
    and $z$ of $\Zca$, by~\eqref{equ:relation_zinbiel_gamma}, we have
    \begin{equation}\begin{split}
        (x \DDendr_{a'} y) \GDendr_a z
            & - x \DDendr_{a'} (y \GDendr_a z) \\
        & = (y \ProdZin_{a'} x) \ProdZin_a z
            - (y \ProdZin_a z) \ProdZin_{a'} x \\
        & = y \ProdZin_{a \Min a'} (x \ProdZin_a z)
            + y \ProdZin_{a \Min a'} (z \ProdZin_{a'} x)
            - y \ProdZin_{a \Min a'} (z \ProdZin_{a'} x)
            - y \ProdZin_{a \Min a'} (x \ProdZin_a z) \\
        & = 0,
    \end{split}\end{equation}
    the operations $\GDendr_a$ and $\DDendr_a$ satisfy
    Relation~\eqref{equ:relation_dendr_gamma_1_concise} of
    $\gamma$-polydendriform algebras. Moreover, since for all
    $a, a' \in [\gamma]$ and all elements $x$, $y$,
    and $z$ of $\Zca$, by~\eqref{equ:relation_zinbiel_gamma}, we have
    \begin{equation}\begin{split}
        (x \GDendr_{a'} y) \GDendr_a z
            & - x \GDendr_{a \Min a'} (y \GDendr_a z)
            - x \GDendr_{a \Min a'} (y \DDendr_{a'} z) \\
        & = (x \ProdZin_{a'} y) \ProdZin_a z
            - x \ProdZin_{a \Min a'} (y \ProdZin_a z)
            - x \ProdZin_{a \Min a'} (z \ProdZin_{a'} y) \\
        & = x \ProdZin_{a \Min a'} (y \ProdZin_a z)
            + x \ProdZin_{a \Min a'} (z \ProdZin_{a'} y)
            - x \ProdZin_{a \Min a'} (y \ProdZin_a z)
            - x \ProdZin_{a \Min a'} (z \ProdZin_{a'} y) \\
        & = 0,
    \end{split}\end{equation}
    the operations $\GDendr_a$ and $\DDendr_a$ satisfy
    Relation~\eqref{equ:relation_dendr_gamma_2_concise} of
    $\gamma$-polydendriform algebras. Finally, since for all
    $a, a' \in [\gamma]$ and all elements $x$, $y$,
    and $z$ of $\Zca$, we have
    \begin{equation}\begin{split}
        (x \GDendr_{a'} y) \DDendr_{a \Min a'} z
            & + (x \DDendr_a y) \DDendr_{a \Min a'} z
            - x \DDendr_a (y \DDendr_{a'} z) \\
        & = z \ProdZin_{a \Min a'} (x \ProdZin_{a'} y)
            + z \ProdZin_{a \Min a'} (y \ProdZin_a x)
            - (z \ProdZin_{a'} y) \ProdZin_a x \\
        & = z \ProdZin_{a \Min a'} (x \ProdZin_{a'} y)
            + z \ProdZin_{a \Min a'} (y \ProdZin_a x)
            - z \ProdZin_{a \Min a'} (y \ProdZin_a x)
            - z \ProdZin_{a \Min a'} (x \ProdZin_{a'} y) \\
        & = 0,
    \end{split}\end{equation}
    the operations $\GDendr_a$ and $\DDendr_a$ satisfy
    Relation~\eqref{equ:relation_dendr_gamma_3_concise} of
    $\gamma$-polydendriform algebras. Hence $\Zca$ is a
    $\gamma$-polydendriform algebra.
\end{proof}
\medskip

The constructions stated by Propositions~\ref{prop:morphism_com_zin_gamma}
and~\ref{prop:morphism_dendr_zin_gamma} producing from a
$\Zin_\gamma$-algebra respectively a $\Com_\gamma$-algebra and a
$\gamma$-polydendriform algebra are functors from the category of
$\Zin_\gamma$-algebras respectively to the category of
$\Com_\gamma$-algebras and the category of $\gamma$-polydendriform algebras.
These functors respectively translate into symmetric operad morphisms
from $\Com_\gamma$ to $\Zin_\gamma$ and from $\Dendr_\gamma$ to
$\Zin_\gamma$. These morphisms are generalizations of known morphisms
between $\Com$, $\Dendr$, and $\Zin$ of~\eqref{equ:diagramme_papillon}
(see~\cite{Lod01,Lod06,Zin12}).
\medskip

A complete study of the operads $\Com_\gamma$, $\Lie_\gamma$,
$\Zin_\gamma$, and $\Leib_\gamma$, and suitable definitions for all the
morphisms intervening in~\eqref{equ:diagramme_papillon_gamma} is worth
to interest for future works.
\medskip

\bibliographystyle{alpha}
\bibliography{Bibliographie}

\end{document}